\documentclass[11pt]{amsart}
\usepackage{amsfonts,amssymb,amsmath}
\usepackage{geometry}

\usepackage{cite}
\usepackage{xcolor}
\usepackage{tikz-cd}
\usepackage{float}
\usepackage{bm}
\usepackage{pdfpages}
\usepackage{mathtools}
\usepackage{graphicx}
\usepackage{enumerate}
\usepackage{amsfonts,amssymb,verbatim,amsmath,amsthm,latexsym,textcomp,amscd,hyperref}
\usepackage{latexsym,amsfonts,amssymb,epsfig,verbatim}
\usepackage{graphicx}
\usepackage{color}
\usepackage{url}
\usepackage{pgfplots}
\usepackage{tikz}
\usepackage{subcaption}

\graphicspath{ {./images/} }
\newtheorem{theorem}{Theorem}[section]
\newtheorem*{theorem*}{Theorem}
\newtheorem*{lemma*}{Lemma}
\newtheorem*{proposition*}{Proposition}
\newtheorem*{corollary*}{Corollary}
\newtheorem{lemma}[theorem]{Lemma}
\newtheorem{defn}[theorem]{Definition}
\newtheorem{prop}[theorem]{Proposition}
\newtheorem{cor}[theorem]{Corollary}
\newtheorem{example}[theorem]{Example}

\theoremstyle{definition}

\newtheorem{remark}[theorem]{Remark}

\newtheorem{problem}[theorem]{Problem}
\newtheorem{convention}[theorem]{Convention}
\newtheorem{question}[theorem]{Question}
\newtheorem{note}[theorem]{Note}
\newtheorem{notation}[theorem]{Notation}
\def\dl{\delta}

\def\ep{\epsilon}

\def\pr{^\prime}
\def\prr{^{\prime\prime}}
\def\lm{\lambda}
\def\Lm{\Lambda}
\def\ri{\rightarrow}
\def\map{\rightarrow}

\def\sse{\subseteq}
\def\gm{\gamma}
\def\bt{\beta}
\def\al{\alpha}
\def\fa{\forall}
\def\pa{\partial}

\def\L{\mathcal{L}}

\def\F{\mathcal{F}}
\def\H{\mathcal{H}}

\def\Y{\mathcal{Y}}
\def\G{\mathcal{G}}

\def\B{\mathcal{B}}

\def\M{\mathcal{M}}

\def\R{\mathbb {R}}

\def\N{\mathbb {N}}
\def\Z{\mathbb {Z}}

\def\mfY{\mathfrak {Y}}
\def\mfT{\mathfrak {T}}
\def\mfB{\mathfrak {B}}
\def\mfv{\mathfrak {v}}

\def\mfe{\mathfrak {e}}
\def\mfw{\mathfrak {w}}

\begin{document}
	
	\title[A Combination Theorem for Trees of Metric Bundles]{A Combination Theorem for Trees of Metric Bundles}
	\author{Rakesh Halder}
	
\address{Indian Institute of Science Education and Research (IISER) Mohali,Knowledge City,  Sector 81, S.A.S. Nagar 140306, Punjab, India}
\email{rhalder.math@gmail.com}
	
	\thanks{This paper is part of R. Halder’s
		PhD thesis written under the supervision of Dr. P. Sardar.}
	
	
	\maketitle

\begin{abstract}

Motivated by the work of Bestvina--Feighn (\hspace{-.08mm}\cite{BF}) and Mj--Sardar (\hspace{-.08mm}\cite{pranab-mahan}), we
define trees of metric bundles subsuming both the trees of metric spaces and the metric bundles. Then we prove a combination theorem for these spaces. More precisely, we prove that the total space of a tree of metric bundles is hyperbolic if the following hold (see Theorem \ref{main-theorem-com}). $(1)$ The fibers are uniformly hyperbolic metric spaces and the base is also hyperbolic metric space, $(2)$ barycenter maps for the fibers are uniformly coarsely surjective, $(3)$ the edge spaces are uniformly qi embedded in the corresponding fibers and $(4)$ the Bestvina--Feighn hallway flaring condition is satisfied.

As an application, we provide a combination theorem for certain complexes of groups over finite simplicial complex (see Theorem \ref{application-2}).
\end{abstract}
\tableofcontents

\section{Introduction}\label{introduction}
Bestvina--Feighn (\hspace{-.08mm}\cite{BF}) proved that the fundamental group of a finite graph of hyperbolic groups  is hyperbolic provided it satisfies the qi embedded condition and the annuli flare condition (see \cite[Theorem $1.2$]{BF-Adn}). Motivated by the work of Bestvina and Feighn, M. Kapovich posed a question whether their combination theorem for graphs of groups can be extended to complexes of groups (see \cite[Problem $90$]{kap-prob}). (For more detailed exposition in complexes of groups, one is referred to \cite{gersten-stall-tria}, \cite{haefliger-cplx}, \cite{corson1}, \cite{bridson-haefliger}.) One may formulate the problem of M. Kapovich as follows.

\begin{problem}\label{combi-problem}
	Suppose $\G(\Y)$ is a developable complex of groups over a finite connected simplicial complex $\Y$ such that the following hold.
	
	\begin{enumerate}
		
		\item All the local groups are hyperbolic.
		
		\item All the local maps are qi embeddings.
		
		\item The universal cover of $\G(\Y)$ is hyperbolic.
	\end{enumerate}
	
	Under what condition(s) the fundamental group $\pi_1(\G(\Y))$ is hyperbolic.	
\end{problem}
Here is a brief history of the activities around this problem. Suppose $\G(\Y)$ is a complex of groups with the condition as in Problem \ref{combi-problem}. If $\G(\Y)$ is negatively curved and all the local groups are finite then $\pi_1(\G(\Y))$ is hyperbolic due to Gersten--Stallings (\hspace{-.08mm}\cite{gersten-stall-tria}). If the local maps are all isomorphisms onto finite index subgroups of the target groups, local groups are non-elementary hyperbolic and $\G(\Y)$ satisfies the Bestvina--Feighn hallway flaring condition then it follows from the work of Mj--Sardar (\hspace{-.08mm}\cite{pranab-mahan}) that $\pi_1(\G(\Y))$ is hyperbolic. If the universal cover of $\G(\Y)$ is CAT$(0)$ and hyperbolic, and the action of $\pi_1(\G(\Y))$ on the universal cover is acylindrical then $\pi_1(\G(\Y))$  is hyperbolic and local groups are quasiconvex in $\pi_1(\G(\Y))$ due to A. Martin (\hspace{-.08mm}\cite{martin-acycom}). Apart from these extreme cases nothing is known. However, in this article, we attempt this question for yet another type of complexes of groups. Let us first outline the setup.
\begin{defn}[\bf Setup $\mathcal C$]\label{setup C}
\begin{enumerate}\label{setup-combi}
\item Suppose $Y$ is a finite connected graph and $p_Y:\Y\ri Y$ is a graph of spaces $($see \textup{\cite[Section 3.3]{haefliger-cplx}}$)$ where the edge spaces are points. We further assume that $\Y$ is a simplicial complex. Suppose $\G(\Y)$ is a complex of groups over $\Y$ such that all the properties of Problem \ref{combi-problem} hold with the following additional one.
	
\item For all $v\in V(Y)$, $p_Y^{-1}(v)=\Y_v$, say, is a finite connected simplicial complex and the restriction of $\G(\Y)$ on $\Y_v$ is a developable complex of groups, say $\G_v(\Y_v)$, over $\Y_v$. Further, suppose all the local maps in $\G_v(\Y_v)$ are isomorphisms onto finite index subgroups of the target groups.
\end{enumerate}
We denote $\G(\Y)$ in this case as $\G(\Y,Y)$ to emphasize the extra structure on $\Y$.
\end{defn}
Note that if $u,v$ are two vertices in $\Y$ such that $p_Y$ is injective when restricted to the edge $e$ joining $u,v$ then the local homomorphisms $G_e\ri G_u$ and $G_e\ri G_v$ are not necessarily isomorphisms onto finite index subgroups. Then we have the following theorem.



\begin{theorem}\label{application-2}
Suppose $\G(\Y,Y)$ is a complex of groups over $\Y$ as in Definition \ref{setup C}. Further, suppose that in condition $(2)$ of Definition \ref{setup C}, all the local groups of $\G_v(\Y_v)$ are non-elementary (hyperbolic). Moreover, assume that it satisfies the Bestvina--Feighn hallway flaring condition. Then $\pi_1(\G(\Y,Y))$ is hyperbolic.
\end{theorem}

The above Theorem \ref{application-2} follows from a combination theorem for spaces. We now elaborate on this. In \cite{BF}, Bestvina and Feighn proved that a tree of hyperbolic metric spaces is hyperbolic if it satisfies the qi embedded condition and the hallway flaring condition. In \cite{pranab-mahan}, Mj and Sardar proved that if $X$ is a metric bundle over $B$ such that (1) fibers and the base $B$ are uniformly hyperbolic, (2) the barycenter maps for the fibers are uniformly coarsely surjective and (3) the Bestvina--Feighn hallway flaring condition holds then $X$ is hyperbolic. The question that motivated us is
if we can combine these two and still get hyperbolicity. Here is a baby version of the problem we are attempting to solve.

\begin{question}\label{comb-qsn}
	Suppose $\pi_i: X_i\ri B_i$ are metric bundles for $i = 1, 2$. Suppose we join $b_1\in B_1$ and $b_2\in B_2$ by an edge $e$, say. Let $X_e$ be another geodesic metric space. Let $F_{b_1}$ and $F_{b_2}$ be fibers over $b_1$ and $b_2$ of the metric bundles $X_1$ and $X_2$ respectively. Suppose that we have qi embeddings $X_e\ri F_{b_1}$ and $X_e\ri F_{b_2}$, and we form a new space by gluing $X_e\times[0,1]$ to $X_1\sqcup X_2$ as follows: We attach $X_e\times\{0\}$ to $F_{b_1}$ and $X_e\times\{1\}$ to $F_{b_2}$ using the qi embeddings $X_e\ri F_{b_1}$ and $X_e\ri F_{b_2}$ respectively. When is the new space hyperbolic?
\end{question}

In this article, we consider a general version of Question \ref{comb-qsn} and provide the following combination theorem. One is referred to Definition \ref{treesofmetric-bun} for trees of metric bundles.

\begin{theorem}\label{main-theorem-com}
	Suppose $(X,B,T)$ is a tree of metric bundles such that:
	
	\begin{enumerate}
		\item For $v\in V(T)\textrm{ and }a\in B_v$, the fibers, $F_{a,v}$ are uniformly hyperbolic geodesic metric spaces and the barycenter maps $\partial^3F_{a,v}\ri F_{a,v}$ are uniformly coarsely surjective.
		
		\item Let $[v,w]$ be an edge in $T$ and $\mfe=[\mfv,\mfw]$ be the edge joining $\mfv\in B_v$ and $\mfw\in B_w$. Then $\pi_X$ restricted to $\pi_X^{-1}(\mfe)$ is a tree of metric spaces with (uniform) qi embedded condition over $\mfe$.
		
		\item $B$ is hyperbolic geodesic metric space.
		
		\item The Bestvina--Feighn hallway flaring condition is satisfied.
	\end{enumerate}
	Then $X$ is hyperbolic geodesic metric space.
\end{theorem}

We also prove the following combination theorem for trees of metric spaces within an axiomatic framework. This result plays a key role in the proof of Theorem \ref{main-theorem-com}. One is referred to Section \ref{hyp-tree-sps} for the list of hypotheses $(\mathcal{P}0)-(\mathcal{P}4)$.

\begin{theorem}[Theorem \ref{treeofsps-com-thm}]\label{treeofsps-com-thm1}
Suppose $\pi:X\ri T$ is a tree of metric spaces satisfying the properties $(\mathcal{P}0)-(\mathcal{P}4)$. Then $X$ is hyperbolic geodesic metric space.
\end{theorem}

{\bf Necessity of flaring}

Gersten (see \cite[Corollary $6.7$]{gersten}) showed that the annuli flaring is necessary for the fundamental group of a finite graph of hyperbolic groups to be hyperbolic provided the edge groups are qi embedded in the corresponding vertex groups. Mj and Sardar also showed that the (hallway) flaring condition is necessary for metric bundles to be hyperbolic provided fibers are uniformly hyperbolic (see \cite[Proposition $5.8$]{pranab-mahan}). Let us briefly recall the idea of their proof. They first showed that small girth ladders bounded by two qi lifts satisfy flaring condition. Then a general ladder was subdivided into small girth ladders and summing them up it was shown that a general ladder satisfies flaring condition. In doing so they used a crucial lemma (\hspace{-.08mm}\cite[Lemma $5.9$]{pranab-mahan}) which is a specialization, in the context of metric bundles, of the fact that geodesics diverge exponentially in hyperbolic metric spaces. This lemma also holds true in trees of metric bundles. In trees of metric bundles, given two qi lifts over the same base, there is a special ladder (see Definition \ref{K-metric-graph-bundle}) bounded by these qi lifts (see Lemma \ref{getting-metric-graph-bundles}) and this ladder can be subdivided into small girth ladder (see \emph{Subdivision of ladder} in the proof of Theorem \ref{general-ladder-is-hyp}). Therefore, the proof of the following remark is analogous to that of \cite[Proposition $5.8$]{pranab-mahan}, so we omit the full details.

\begin{remark}\label{necessity-of-flaring}
	Suppose $(X,B,T)$ is a tree of metric bundles such that:
	
	\begin{enumerate}
		\item $X$ is hyperbolic.
		
		\item Fibers are uniformly hyperbolic.
		
		\item Edge spaces are uniformly qi embedded in the corresponding fibers.
	\end{enumerate}
	Then $\pi_X:X\ri B$ satisfies the Bestvina--Feighn hallway flaring condition.
\end{remark}

As a consequence of Remark \ref{necessity-of-flaring}, we have the following. 

\begin{cor}\label{necessity-flaring-com-grp}
Suppose $\G(\Y,Y)$ is a complex of groups as in Definition \ref{setup C} such that the fundamental group $\pi_1(\G(\Y,Y))$ of $\G(\Y,Y)$ is hyperbolic. Then $\G(\Y,Y)$ satisfies the Bestvina--Feighn hallway flaring condition. (Note that we do not require the universal cover of $\G(\Y)$ to be hyperbolic.) 
\end{cor}
Finally, we conclude the introduction by explaining why our theorem is nontrivial.

\begin{remark}\label{crusial-words}$\vspace{0.1mm}$
	The overall idea of the proof of Theorem \ref{main-theorem-com} closely follows from that of \cite{ps-kap} and makes crucial use of \cite{pranab-mahan}. We are intellectually indebted to both of these works. However, it is not a direct consequence of the combination theorems of \cite{BF} and \cite{pranab-mahan} in any obvious way for the following reason. Let $v\in V(T)$. As the space $X_v$ over $B_v$ is a metric bundle satisfying all conditions of the main theorem of \cite{pranab-mahan}, $X_v$ is (uniformly) hyperbolic. Now we can think of $X$ as a tree of metric spaces over $T$ where vertex spaces are these bundles and the edge spaces are inverse images under $\pi_B\circ\pi_X$ of the midpoints of the edges in $T$. But we can not apply the main theorem of \cite{BF} to this tree of spaces to conclude our theorem because in this case the edge spaces of the tree of spaces are not, in general, qi embedded in the corresponding vertex spaces.	
\end{remark}

{\bf A few words on the proof of Theorem \ref{main-theorem-com} and organization of the paper:}

(0) In Section \ref{preliminaries} and Section \ref{trees-of-metric-sps}, we recall basics definitions and results which are used in the subsequent sections. We define trees of metric bundles and discuss some of their properties in Section \ref{trees-of-metric-bundles}. Section \ref{hyp-tree-sps} is devoted to proving Theorem \ref{treeofsps-com-thm1}, and this section is independent of later sections. One is not required Section \ref{hyp-tree-sps} to study Section \ref{trees-of-metric-bundles}, \ref{semicontinuous-subgraph}, \ref{hyperbolicity-of-ladder}, \ref{hyp-of-flow-sp}, \ref{union-of-two-flow-sp-is-hyp} but Section \ref{hyp-tmb-sec}.

(1) Motivated by that of \cite{ps-kap}, we construct semicontinuous families, ladders and flow spaces, and more general flow spaces in Section \ref{semicontinuous-subgraph}; whereas ladder was invented by Mitra in \cite{mitra-trees} for trees of metric spaces. Some properties of these subspaces, most importantly, Mitra's retraction of the whole space on these subspaces, are also discussed there. Main construction starts from here in Section \ref{semicontinuous-subgraph}.

(2) The most subtle part was to show the uniform hyperbolicity of ladders and flow spaces. In Section \ref{hyperbolicity-of-ladder}, we prove that the ladders are (uniformly) hyperbolic by
dividing into two cases: small girth and general case. By invoking Bowditch criterion (see Proposition \ref{combing}) for a metric space to be hyperbolic, we prove that small girth ladders are (uniformly) hyperbolic (Subsection \ref{small-ladder}). For general ladder, we first break it up into small girth ladder and then with the help of Proposition \ref{combi-hyp-sps} we conclude its hyperbolicity (Subsection \ref{general-ladder}). Section \ref{hyp-of-flow-sp} is devoted to prove the (uniform) hyperbolicity of flow spaces (Theorem \ref{flow-sp-is-hyp}). To prove this, we follow the strategy elaborated in \cite[Chapter $5$]{ps-kap}.

(3) In Section \ref{union-of-two-flow-sp-is-hyp}, we prove that union of uniform neighborhood of two intersecting flow spaces is uniformly hyperbolic (Proposition \ref{union-of-flow-sp-is-hyp}). In the introduction of this section we elaborate all the properties which are required from the earlier sections to prove this.

(4) In Section \ref{hyp-tmb-sec}, we conclude our main theorem (Theorem \ref{main-theorem-com}) with the help of Theorem \ref{treeofsps-com-thm1}. We think of the tree of metric bundles $(X,B,T)$ as a tree of metric spaces $\pi:X\ri T$ (see Remark \ref{crusial-words}). Then we verify the properties $(\mathcal{P}0)-(\mathcal{P}4)$ of Theorem \ref{treeofsps-com-thm1}, where the flow spaces of metric bundles serve as the specified subspaces as in properties $(\mathcal{P}0)-(\mathcal{P}4)$. This proves the hyperbolicity of $X$ by Theorem \ref{treeofsps-com-thm1}.

(5) In Section \ref{applications}, we explore some applications of Theorem \ref{main-theorem-com} to complexes of groups (Theorem \ref{application-2} and Corollary \ref{necessity-flaring-com-grp}).



{\bf Acknowledgements}. The author is indebted to his supervisor, Dr. Pranab Sardar, for proposing
the problem considered here. Author is also grateful to him for his constant support, encouragement
and many helpful discussions throughout this work. Dr. Sardar helped the author to correct several mistakes and make the paper more readable. Dr. Sardar also suggested the author to formulate Theorem \ref{treeofsps-com-thm1} within an axiomatic framework (see Section \ref{hyp-tree-sps}) and recommanded to work directly with spaces rather than their approximating graphs, which was the approach taken in an earlier version.
The author is very thankful to the anonymous referee also for a thorough reading of the manuscript and many helpful comments, which also greatly contributed to the exposition of this article. Finally, the author would like to acknowledge the support of IISER Mohali institute doctoral fellowship.

\subsection{Flowchart}Reader may find this helpful.
\begin{figure}[h]
	\includegraphics[width=11cm]{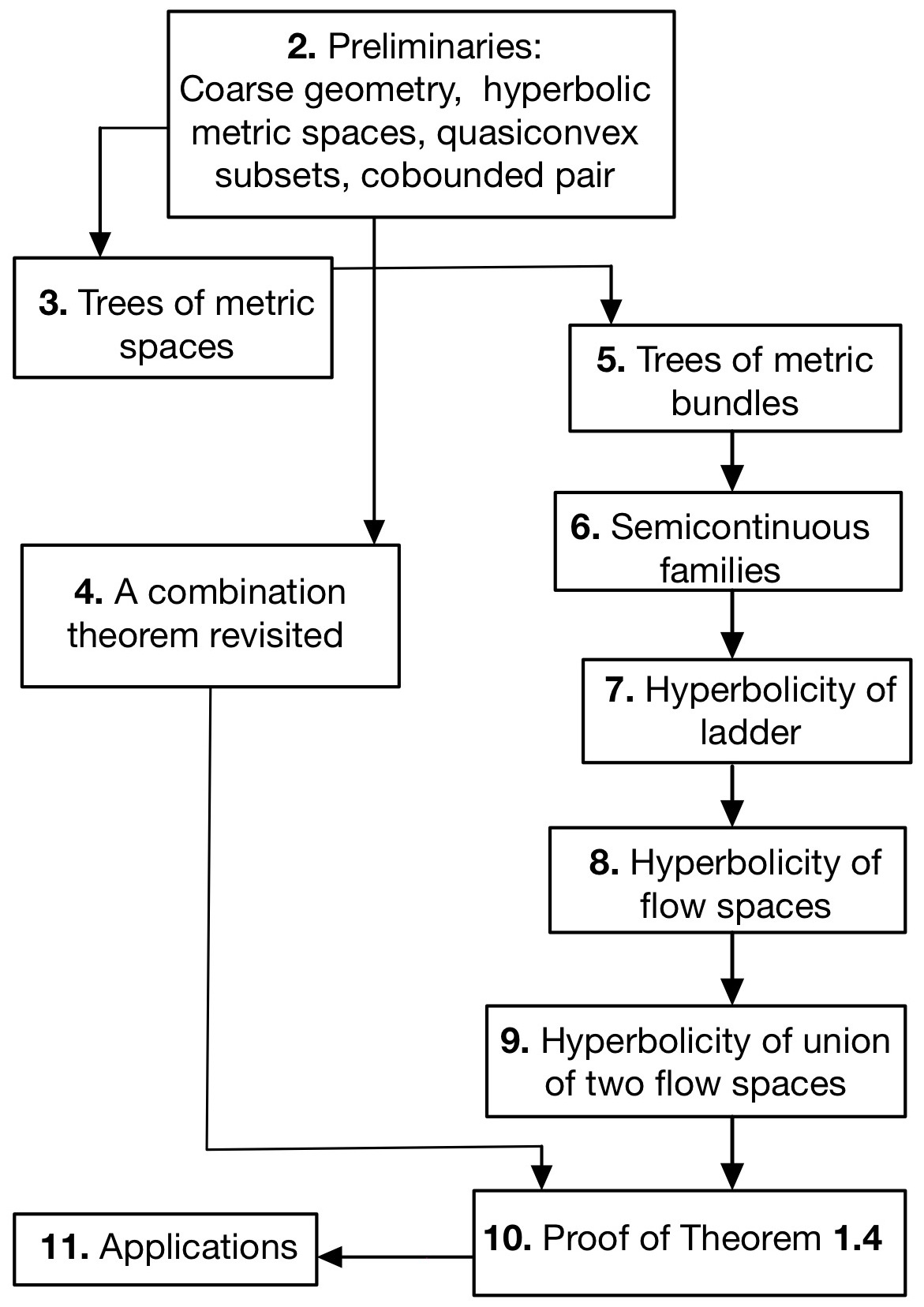}
	\centering
\end{figure}




\section{Preliminaries}\label{preliminaries}

\subsection{Some basic concepts}\label{some-basic-concepts}
We are assuming that the reader is familiar with the basic notions of large-scale geometry (\hspace{-.08mm}\cite{GhH, gromov-hypgps, abc}). However, we will recall some of these for later use. Let $(X,d_X),(Y,d_Y)$ be metric spaces, and let $\epsilon\ge0,r\ge0,C\ge0,D\ge0,R\ge0$. Let $A\sse X$. The subset $A$ is said to be $r$-{\em dense} in $X$ if $X=N_r(A):=\{x\in X:d_X(x,A)\le r\}$. We say that $A$ is {\em rectifiably connected} in $X$ if for all $x, y\in A$, there is a path $\al:[0,1]\ri X$ joining $x,y$ which is of finite length such that $\al([0, 1])\sse A$. A point $a\in A$ is called a {\em nearest point projection} of $x\in X$ if $d_X(x,a)\le d_X(x,a\pr)$ for all $a\pr\in A$. Let $A,B\sse X$. 
The {\bf Hausdorff distance} between $A$ and $B$ is defined to be inf $\{D:A\sse N_D(B), B\sse N_D(A)\}$ and is denoted by $Hd_X(A,B)$. 
Suppose $f:X\ri Y$ is a map.

(1) We say that $f$ is $\epsilon$-{\bf coarsely surjective} if $f(X)$ is $\ep$-dense in $Y$, and $f$ is said to be {\bf coarsely surjective} if it is $\epsilon$-coarsely surjective for some $\epsilon$.

(2) The map $f$ is called $C$-{\bf coarsely Lipschitz} if $d_Y(f(x),f(x\pr))\le Cd_X(x,x\pr)+C$ for all $x,x\pr\in X$. In particular, if $A\sse X$ and there is a map $g:X\ri A$ such that $d_X(g(x),g(x\pr))\le Cd_X(x,x\pr)+C$, $\fa~x,x\pr\in X$ and $g(a)=a$, $\fa~a\in A$ then $g$ is said to be $C$-{\bf coarsely Lipschitz retraction} of $X$ on $A$.

(3) We say that $f$ is a $(k,\epsilon)$-{\bf quasiisometric embedding} (in short $(k,\epsilon)$-{\bf qi embedding}) if $$\frac{1}{k}d_X(x,x\pr)-\epsilon\le d_Y(f(x),f(x\pr))\le kd_X(x,x\pr)+\epsilon \textrm{ for all }x,x\pr\in X.$$The map $f$ is said to be a {\bf quasiisometric (qi) embedding} if it is a $(k,\epsilon)$-qi embedding for some $k\ge1$ and $\epsilon\ge0$. We will refer a $(k,k)$-qi embedding as a $k$-qi embedding. Lastly, we say that $f$ is a $(k,\ep)$-{\bf quasiisometry} (resp. $k$-{\bf quasiisometry}) if it is a $(k,\ep)$-qi embedding (resp. $k$-qi embedding) and $D$-coarsely surjective for some $D\ge0$.

(4) A function $\phi:\R_{\ge0}\ri\R_{\ge0}$ is said to be {\bf proper} if $\phi(r)\ri\infty$ as $r\ri\infty$. Then we say that $f$ is $\phi$-{\bf proper} embedding for some proper function $\phi$ if $d_Y(f(x),f(x\pr))\le r$ implies $d_X(x,x\pr)\le\phi(r)$.


For $x,y\in X$, a {\bf (quasi)geodesic} joining them in $X$ is an (quasi)isometric embedding $\al:[s,t]\sse\R\ri X$ of some interval in $\R$ with $\al(s)=x$ and $\al(t)=y$. We refer to X as a {\bf geodesic metric space} if there exists a geodesic in $X$ joining every pair of points in $X$. A {\bf tree} is a connected graph (see \cite[Section $1.9$, I$.1$]{bridson-haefliger}) without any embedded circle. (Here all the edges are isometric to a closed unit interval of $\R$.) For a tree $T$ and $u,v\in V(T)$ (= the set of vertices of $T$), we usually call the geodesic joining $u,v$ in $T$ as {\bf segment} or {\bf interval} joining $u,v$. 

Suppose $X$ is a geodesic metric space. We denote a geodesic in $X$ joining two points $a,b\in X$ by $[a,b]\sse X$. Suppose $x,y,z\in X$. A {\bf geodesic triangle} in $X$ formed by these three points is the union of chosen geodesic segments $[x,y],[x,z]$ and $[y,z]$, and it is denoted by $\triangle(x,y,z)$. We call those geodesic segments as sides of the triangle. We say that $\triangle(x,y,z)$ is $\dl$-{\bf slim} if any side of $\triangle(x,y,z)$ is contained in the $\dl$-neighborhood of the union of other two sides. Suppose $\triangle(x_1,x_2,x_3)$ is a geodesic triangle in $X$ where $x_1,x_2,x_3\in X$. A point $z\in X$ is called $C$-{\bf center} of this triangle if $z\in N_C([x_i,x_j])$ for $i\ne j$ and $i,j\in\{1,2,3\}$. Sometimes, we call $\cup_{i=1}^{3}[z,x_i]$ a $C$-{\bf tripod} in $X$ with end points $x_1,x_2,x_3$. Let $a,b\in X$ and $x_1,x_2,\cdots,x_n\in[a,b]$. By an {\bf order} $a\le x_1\le x_2\le\cdots\le x_n\le b$, we mean $d_X(a,x_i)\le d_X(a,x_{i+1})$ for $1\le i\le n$. 
\smallskip

{\bf Convention}: {\em Occasionally, we use phrases such as \emph{uniform constants}, \emph{uniformly hyperbolic}, \emph{uniformly Hausdorff close} etc. when we (1) do not require a specific value for the associated parameters, and (2) it is evident that such values exist given the assumptions of Lemma, Proposition or Theorem.

Later on in Sections \ref{semicontinuous-subgraph}, \ref{hyperbolicity-of-ladder}, \ref{hyp-of-flow-sp} and \ref{union-of-two-flow-sp-is-hyp}, we will consider a subset, say $Y$, in a space, say $X$. Suppose $Z=N_D(Y)$ where $D$ is bigger than a certain constant. Then we show that $Z$, with the induced path metric, is hyperbolic (resp. properly embedded or qi embedded in $X$). Now the hyperbolicity constant (resp. proper function or qi embedding constant) depends on some parameters as well as $D$. In such cases, occasionally, we say that a {\em uniform neighborhood of $Y$ is uniformly hyperbolic (resp. uniformly properly embedded or uniformly qi embedded in $X$)}.	

Unless otherwise specified, for a rectifiably connected subset $A$ of a geodesic metric space $X$, the metric on $A$ under consideration is the induced path metric. For a (quasi)geodesic $\al:I\sse\R\ri X$, we often forget the domain of $\al$ and work with the image of $\al$, and use $\al$ to mean $\al(I)\sse X$.}

The following lemmata (Lemmata \ref{imp-coarse-retraction}, \ref{lrape_qi-emb} and \ref{hd-imp-qi}) are standard. So we only sketch their proofs. 

\begin{lemma}\label{imp-coarse-retraction}
Given $D\ge0$ there is a constant $C_{\ref{imp-coarse-retraction}}(D)\ge0$ such that the following holds.
	
Suppose $X$ is a geodesic metric space and $Y$ is a subset of $X$ (not necessarily connected) such that $Y$ is $1$-dense in $X$. Let $U\sse X$ and $\rho:Y\ri U$ be a map such that $d_X(\rho(p),\rho(q))\le D$ for all $p,q\in Y$ with $d_X(p,q)\le1$. Finally, we assume that for any geodesic segment $[x,y]_X$ joining $x,y\in X$, we have points $x=x_0,x_1,\cdots,x_n=y$ on $[x,y]_X$ such that

$(1)$ $d_X(x_i,x_{i+1})=1$ for all $i\in\{1,2,\cdots,n-2\}$,

$(2)$ $d_X(x_i,x_{i+1})\le 1$ for $i=0,n-1$ and

$(3)$ $x_i\in Y$ for all $i$ except possibly for $i=0,n$.

Then $\rho$ can be extended to a map $\rho\pr:X\ri U$ so that $\rho\pr$ is $C_{\ref{imp-coarse-retraction}}$-coarsely Lipschitz.
\end{lemma}
\begin{proof}
Given any $z\in X$, once and for all, we fix $z'\in Y$ such that $d_X(z,z')\le1$; if $z\in Y$, then $z'=z$. Then we define $\rho'(z):=\rho(z')$.

Suppose $x,y\in X$ such that $d_X(x,y)=r$. Then we have $x',y'\in Y$ as mentioned above such that $d_X(x,x')\le 1$, $d_X(y,y')\le 1$ and $\rho'(x)=\rho(x'),~\rho'(y)=\rho(y')$. Thus by triangle inequality, we get $d_X(x',y')\le r+2$. Now by our assumption, we have points $x'=x_0,x_1,\cdots,x_l=y'$ on $[x',y']_X$ such that $d_X(x_i,x_{i+1})=1$ for all $i\in\{1,2,\cdots,l-2\}$ and $d_X(x_0,x_1)\le 1$, $d_X(x_{l-1},x_l)\le 1$ where $l\le r+4$. Note that $x_i\in Y$ for all $i\in\{0,1,2,\cdots,l\}$. Hence, $d_X(\rho'(x),\rho'(y))=d_X(\rho(x'),\rho(y'))\le\sum_{i=0}^{l-1}d_X(\rho(x_i),\rho(x_{i+1}))\le lD\le rD+4D=Dd_X(x,y)+4D$. This completes the proof where $C_{\ref{imp-coarse-retraction}}(D)=4D+1$.
\end{proof}

\begin{lemma}\label{lrape_qi-emb}
	Given a map $\phi:\R_{>0}\ri\R_{>0}$ and constants $C>0$, $R\ge0$ there is a constant $L_{\ref{lrape_qi-emb}}=L_{\ref{lrape_qi-emb}}(\phi,C,R)$ such that we have the following.
	
	Suppose $X$ is a geodesic metric space and $Y\sse X$ such that $N_R(Y)$ is path connected. Let $p:X\ri Y$ be a $C$-coarsely Lipschitz retraction. Further, suppose that the inclusion $i:N_R(Y)\ri X$ is a $\phi$-proper embedding. Then $i:N_R(Y)\ri X$ is a $L_{\ref{lrape_qi-emb}}$-qi embedding.
\end{lemma}
\begin{proof}
We will sketch a proof. (For instance, see \cite[Lemma 1.17]{ps-kap}, \cite[Theorem 11.50]{kap-drutu-book} for particular cases.) Suppose $Z=N_R(Y)$, and the induced path metric on $Z$ is denoted by $d_Z$. Note that $d_X(x,y)\le d_Z(x,y),\fa~ x,y\in Z$. To get the reverse inequality, let $x,y\in Z$ such that $d_X(x,y)=r$. Then there are $x_1,y_1\in Y$ such that $d_Z(x,x_1)\le R$ and $d_Z(y,y_1)\le R$, and so $d_X(x_1,y_1)\le 2R+r$. Now we divide geodesic $[x_1,y_1]_X$ by points on it, say $x_1=w_0,w_1,\cdots,w_l=y_1$, with $d_X(w_{i-1},w_{i})=1$ for $1\le i\le l-1$ and $d_X(w_{l-1},w_{l})\le1$. Then, (1) we apply coarse Lipschitz retraction on these points, and (2) we get a uniform bound on $d_Z(p(w_{i-1}),p(w_i))$ by $\phi$-proper embedding of $Z\map X$. Finally, summation of these, particularly, $l\le\sum_{i=1}^{l} d_Z(p(w_{i-1}),p(w_i))$, will conclude the proof.
\end{proof}

\begin{lemma}\label{hd-imp-qi}
	Given $L\ge1,~D\ge0$, there is a constant $L_{\ref{hd-imp-qi}}=L_{\ref{hd-imp-qi}}(L,D)$ such that we have the following.
	
	Suppose $X$ is a geodesic metric space and $Y\sse Z\sse X$ are geodesic subspaces such that the inclusion $Y\ri X$ is a $L$-qi embedding. Let $Z\sse N_D(Y)$. Then the inclusion $Z\ri X$ is a $L_{\ref{hd-imp-qi}}$-qi embedding.
\end{lemma}
\underline{Sketch of proof of Lemma \ref{hd-imp-qi}}: First of all, note that $Z$ is properly embedded in $X$ since $Z\sse N_D(Y)$ and $Y\map X$ is qi embedding. Then by \cite[Lemma 1.19]{ps-kap}, the inclusion $Y\map Z$ is a quasiisometry as $Hd_X(Y,Z)\le D$. Hence the inclusion $Z\map X$ is a qi embedding as $Y\map X$ is a qi embedding.

\subsection{Hyperbolic metric spaces}
There are several equivalent definitions for hyperbolic geodesic metric spaces (see \cite{abc}, \cite{gromov-hypgps}).
We shall use the following definition.
\begin{defn}
	Suppose $X$ is a geodesic metric space and $\dl\ge0$. We say that $X$ is $\dl$-hyperbolic if all its geodesic triangles are $\dl$-slim. 
	
	A geodesic metric space $X$ is said to be hyperbolic if it is $\dl$-hyperbolic for some $\dl\ge0$.
\end{defn}

{\em Thus, unless otherwise stated, all hyperbolic metric spaces are assumed to be geodesic, as per the definition above}. In a hyperbolic metric space, quasigeodesics and geodesics with the same end points are uniformly Hausdorff close. This is known as Morse lemma or stability of quasigeodesic (see \cite[Theorem $1.7$, III.H]{bridson-haefliger}).

\begin{lemma}[{Stability of quasigeodesic}]\label{ml}
	Given $\dl\ge0,~ k\ge1$ and $\ep\ge0$ there is a constant $D_{\ref{ml}}=D_{\ref{ml}}(\dl,k,\ep)$ such that the following holds.
	
	Suppose $X$ is a $\dl$-hyperbolic geodesic metric space. Then for any geodesic $\alpha$ and $(k,\epsilon)$-quasigeodesic $\beta$ in $X$ with the same end points, we have $Hd_X(\al,\bt)\le D_{\ref{ml}}$.
\end{lemma}

We assume that the reader is familiar with the {\em sequential boundary} or {\em Gromov boundary} of a hyperbolic geodesic metric space that need not be proper (see \cite[Chapter $4$]{abc}, \cite[Chapter III.H]{bridson-haefliger}). We recall that for a hyperbolic space, one can define a barycenter as follows.\smallskip

{\bf The barycenter map} (For more details, one is referred to \cite[Section $2$]{pranab-mahan}): Suppose $X$ is a $\dl$-hyperbolic metric space such that there are more than two points in its Gromov boundary, $\partial X$. Then by \cite[Lemma $2.4$]{pranab-mahan}, for any $\eta\in\partial X$, there is a (uniform) quasigeodesic ray starting at any point of $X$ and representing $\eta$, and for any $\eta\pr,\eta\prr\in \pa X$ with $\eta\pr\ne\eta\prr$, there is a (uniform) bi-infinite quasigeodesic line whose one end represents $\eta\pr$ and the other one represents $\eta\prr$. We denote such a quasigeodesic line by $(\eta\pr,\eta\prr)$. Notice that we do not assume our space to be proper. Let $\partial^3X=\{(\xi_1,\xi_2,\xi_3)\in\partial X\times\partial X\times\partial X:\xi_1\ne\xi_2\ne\xi_3\ne\xi_1\}$. Now for $\xi=(\xi_1,\xi_2,\xi_3)\in\partial^3X$, we consider an ideal quasigeodesic triangle, say $\triangle(\xi_1,\xi_2,\xi_3)$, formed by three (uniform) quasigeodesic lines $\{(\xi_i,\xi_j):i\ne j\textrm{ and }i,j\in\{1,2,3\}\}$. Then by \cite[Lemma $2.7$]{pranab-mahan}, there is a point, say $b_{\xi}$, in $X$ uniformly close to each sides of $\triangle(\xi_1,\xi_2,\xi_3)$, and that $b_{\xi}$ is coarsely well-defined. Thus in this way, we get (coarsely) a map $\psi:\partial^3X\ri X$ sending $\xi$ to $b_{\xi}$. Lastly, by \cite[Lemma $2.9$]{pranab-mahan}, the map $\psi:\partial^3X\ri X$ is coarsely unique and is called \emph{the barycenter map}. Note that for such barycenter map, we always assume that $\partial X$ has more than two elements.\smallskip

For a finitely generated group $G$ with a finite generating set $S$, the Cayley graph of $G$ with respect to $S$ is a graph whose vertex set is $G$ and two vertices, say $g,h\in G$, are joined by an edge if $g^{-1}h\in S\cup S^{-1}$. 

\begin{defn}[{\bf Hyperbolic group}]
A finitely generated group $G$ is said to be hyperbolic if its Cayley graph with respect to some (hence, any) finite generating set is hyperbolic.

A hyperbolic group $G$ is said to be non-elementary if the Gromov boundary of its Cayley graph with respect to some (hence, any) finite generating set contains more than two points.
\end{defn}

It is standard that given two finite generating sets, the Cayley graphs associated with them become quasiisometric to each other. It is easy to prove from the stability of quasigeodesic (see Lemma \ref{ml}) that the hyperbolicity is quasiisometry invariant (see \cite[Theorem $1.9$, III.H]{bridson-haefliger}). Therefore, hyperbolic groups are well-defined.

\begin{remark}\label{nonelementary barycenter map}
It is standard that for a non-elementary hyperbolic group the barycenter map is coarsely surjective.
\end{remark}

Suppose $X$ is a geodesic metric space and $\alpha:[s,t]\sse\R\ri X$ is a continuous injective (rectifiably connected) path. Let $d_{\alpha}$ be the induced path metric on $Im(\alpha)$ from $X$. Then we have the induced order on $Im(\alpha)$ from $[s,t]$. In other words, if $p,q\in[s,t]$ with $p\le q$ then  $\al(p)\le\al(q)$ keeping in mind that $d_{\alpha}(\alpha(s),\alpha(p))\le d_{\alpha}(\alpha(s),\alpha(q))$. 
With this terminology, we have the following.
\begin{lemma}\label{monotonic-map-between-qgs}
	Given $\dl\ge0,~k\ge1$ and $r\ge0$, we have constants $L_{\ref{monotonic-map-between-qgs}}=L_{\ref{monotonic-map-between-qgs}}(\dl,k,r)$ and $k_{\ref{monotonic-map-between-qgs}}=k_{\ref{monotonic-map-between-qgs}}(\dl,k,r)$ such that the following hold.
	
Suppose $(X,d)$ is a $\dl$-hyperbolic metric space. Let $\alpha$ and $\beta$ be continuous injective $k$-quasigeod-\newline esics in $X$ joining points $a_1,a_2$ and $b_1,b_2$ respectively. Further, suppose that $d(a_i,b_i)\le r,~i=1,2$. Let $a_1\le a_2$ and $b_1\le b_2$ be the orders on $\alpha$ and $\beta$ respectively. Then there is a monotonic (piece-wise linear) homeomorphism $\psi:(\alpha,d_{\alpha})\ri(\beta,d_{\beta})$ such that $\psi(a_i)=b_i$ and $d(x,\psi(x))\le k_{\ref{monotonic-map-between-qgs}}$ for all $x\in\al$. Moreover, $\psi$ is a $L_{\ref{monotonic-map-between-qgs}}$-quasiisometry.
	
\end{lemma}
\begin{proof}
	Define a map $\phi:\R_{\ge0}\ri\R_{\ge0}$ such that $\phi(t)=kt+k^2$. Then $\alpha,\beta$ are $\phi$-properly embedded. Now by Lemma \ref{ml} and $\dl$-slimness of geodesic triangles, we have $Hd(\alpha,\beta)\le D_1$, where $D_1=2D_{\ref{ml}}(\dl,k,k)+2\dl+r$. Thus by \cite[Lemma $1.19$]{ps-kap}, we have a map $g:\alpha\ri\beta$ with $d(g(x),x)\le D_1,\fa~ x\in\alpha\setminus\{a_1,a_2\}$ and $g(a_i)=b_i,i=1,2$ such that $g$ is a $L$-quasiisometry, where $L$ depends on $D_1$ and $\phi$. Again, by \cite[Lemma $1.24$]{ps-kap}, we have constants $D_2,D_3$ depending on $L$, and a monotonic (piece-wise linear) homeomorphism $\tilde{g}:\alpha\ri\beta$ such that $\tilde{g}$ is a $D_2$-quasiisometry and $d(g(x),\tilde{g}(x))\le D_3$. So $d(x,\tilde{g}(x))\le d(x,g(x))+d(g(x),\tilde{g}(x))\le D_1+D_3$. Here $\tilde{g}$ serves as the required $\psi$.
	
	
	Therefore, we can take $L_{\ref{monotonic-map-between-qgs}}=D_2$ and $k_{\ref{monotonic-map-between-qgs}}=D_1+D_3$.
\end{proof}

We end this subsection by stating the following results (Lemma \ref{lo_vs_gl}, Proposition \ref{combing} and Proposition \ref{combi-hyp-sps}). One can look at \cite[Theorem $1.4$, Chapter $3$]{CDP} for a proof of Lemma \ref{lo_vs_gl}. Recall that given $k\ge1,~\ep\ge0$, $L>0$ and an interval $I\sse\R$, a map $\al:I\map X$ is said to be $(k,\ep,L)$-{\em local quasigeodesic} if $\al$ restricted to any subinterval of length $\le L$ is $(k,\ep)$-quasigeodesic.

\begin{lemma}[{Local quasigeodesic vs global quasigeodesic}]\label{lo_vs_gl}
Given $\dl\ge0,~k\ge1$ and $\ep\ge0$ there are constants $L_{\ref{lo_vs_gl}}=L_{\ref{lo_vs_gl}}(\dl,k,\ep)$ and $ \lm_{\ref{lo_vs_gl}}=\lm_{\ref{lo_vs_gl}}(\dl,k,\ep)$ such that the following holds.
	
	Suppose $X$ is a $\dl$-hyperbolic metric space. Then any $(k,\ep,L_{\ref{lo_vs_gl}})$-local quasigeodesic in $X$ is a $\lm_{\ref{lo_vs_gl}}$-quasigeodesic.
\end{lemma}

In \cite[Proposition $3.1$]{bowditch-com}, Bowditch provided a criterion for hyperbolicity of a metric graph. Earlier, in \cite{hamenst-teich}, Hamenst\text{\"a}dt also gave a similar criterion for a space to be hyperbolic. In Proposition \ref{combing}, we consider Bowditch’s version for metric spaces. 

\begin{prop}\textup{(\hspace{-.08mm}\cite[Corollary $1.63$]{ps-kap})}\label{combing}
	Given $D_0\ge1,~D\ge0$ and a proper map $\psi:\R_{\ge0}\ri\R_{\ge0}$ there exist $\dl_{\ref{combing}}=\dl_{\ref{combing}}(\psi,D,D_0)$ and $K_{\ref{combing}}=K_{\ref{combing}}(\psi,D,D_0)$ such that the following holds.
	
	Suppose $X$ is a geodesic metric space and $X_0\sse X$ is a $D_0$-dense subset of $X$. Suppose that for any pair $(x,y)$ of distinct points in $X_0$ there is a continuous path $c(x,y)$ joining $x$ and $y$. Further, suppose that for all $x,y,z\in X_0$ and $r\in\R_{\ge0}$, we have
	\begin{enumerate}
		\item $d(x,y)\le r$ implies the length of the path $c(x,y)$ is bounded by $\psi(r)$, and
		
		\item $c(x,y)\sse N_D(c(x,z)\cup c(y,z))$.
	\end{enumerate}	
	
	Then $X$ is $\dl_{\ref{combing}}$-hyperbolic metric space and the paths $c(x,y)$ are $K_{\ref{combing}}$-quasigeodesic.	
\end{prop}

The proposition below is a very special case of the main theorem of \cite{BF} (see also \cite[Theorem $2.59$]{ps-kap}). Here, the space is realized as a tree of metric spaces such that the tree is an interval. (One may look at \cite[Corollary $1.52$]{pranab-mahan} for this result in metric graph.) However, it is true for an arbitrary tree also (see \cite[Theorem $2$]{abhi-sum}).

\begin{prop}\textup{(\hspace{-.08mm}\cite[Theorem $2.59$]{ps-kap})}\label{combi-hyp-sps}
Given $\dl\ge0,~L\ge1,~D\ge0$ there exists $\dl_{\ref{combi-hyp-sps}}=\dl_{\ref{combi-hyp-sps}}(\dl,L,D)$ such that the following holds.
	
Suppose $X=\cup_{i=0}^{n-1}X_i$ is a geodesic metric space with $X_i$'s are geodesic subspaces of $X$ such that:
\begin{enumerate}
\item For $0\le i\le n-1$, $X_i$ is $\dl$-hyperbolic metric space.
		
\item For $0\le i\le n-2$, $Y_{i+1}=X_i\cap X_{i+1}$ is a path connected subspace, and the inclusions $Y_{i+1}\ri X_i$ and $Y_{i+1}\ri X_{i+1}$ are $L$-qi embeddings.
		
\item  For $0\le i\le n-2$, $Y_{i+1}$ separates $X_i$ and $X_{i+1}$ in $X$ in the sense that every path in $X$ joining points in $X_i$ and $X_{i+1}$ passes through $Y_{i+1}$.
		
\item For $1\le i\le n-2$, the pair $(Y_i,Y_{i+1})$ is $D$-cobounded in the metric $X_i$.
		
\item $d_{X_i}(Y_i,Y_{i+1})\ge1$ for $1\le i\le n-2$.
	\end{enumerate}
	Then $X$ is $\dl_{\ref{combi-hyp-sps}}$-hyperbolic metric space.
\end{prop}

\begin{remark}\label{combi-hyp-sps-2}
In Proposition \ref{combi-hyp-sps}, if $n=2$, we only need to check $(1)$ and $(2)$. In that case, $X$ is $\dl_{\ref{combi-hyp-sps-2}}=\dl_{\ref{combi-hyp-sps-2}}(\dl,L)$-hyperbolic (see also Lemma \ref{com-two-hyp-sps}).
\end{remark}

%

\subsection{Quasiconvex subsets}\label{qc subsets}
In this subsection, we will explore various basic results concerning quasiconvex subsets that will be useful in later discussions.
\begin{defn}
	Suppose $X$ is a geodesic metric space and $K\ge0$. A subset $U$ of $X$ is said to be $K$-{\bf quasiconvex} if $[a,b]\sse N_K(U)$ for all $a,b\in U$ and for all geodesics $[a,b]$ joining $a,b$ in $X$. We say that a subset $U$ of $X$ is quasiconvex if it is $K$-quasiconvex for some $K\ge0$.
\end{defn}

In a hyperbolic metric space, a common example of quasiconvex subset is the convex hull of any subset.
\begin{defn}\label{convexhull}
	Suppose $X$ is a geodesic metric space and $U\sse X$. The {\bf quasiconvex hull} of $U$ is defined as hull$(U):=\cup_{a,b\in U}[a,b]$. 
\end{defn}

\begin{remark}\label{hull-is-qc}
	Suppose $X$ is a $\delta$-hyperbolic metric space for some $\dl\ge0$ and $A\subset X$ is any subset. Then hull$(A)$ is $2\dl$-quasiconvex.
\end{remark}

Now we collect some facts related to quasiconvex subsets; some are well known and some are very easy to prove.

\begin{lemma}\label{proj-on-qc}
Given $\dl\ge0$, $K\ge0$ and $D\ge0$, there are constants $C_{\ref{proj-on-qc}}=C_{\ref{proj-on-qc}}(\dl,K)$ and $E_{\ref{proj-on-qc}}=E_{\ref{proj-on-qc}}(\dl,K,D)$ such that the following hold.
	

Suppose $X$ is a $\dl$-hyperbolic metric space. Let $U$ and $V$ be $K$-quasiconvex (closed) subsets of $X$. Then we have the following.
\begin{enumerate}
\item \textup{(\hspace{-.08mm}\cite[Hyperbolic Groups, Lemma $7.3.D$]{gromov-hypgps})} Any map $X\ri U$ sending a point to its nearest point projection is $C_{\ref{proj-on-qc}}$-coarsely Lipschitz retraction.
	
\item \textup{(\hspace{-.08mm}\cite[Corollary 1.105]{ps-kap})} Suppose $x\in X$ and $Hd(U,V)\le D$. If $x_1$ and $x_2$ are nearest point projections of $x$ on $U$ and $V$ respectively, then $d(x_1,x_2)\le E_{\ref{proj-on-qc}}$.
	
\end{enumerate}
\end{lemma}

\begin{lemma}\label{quasi-goes-to-quasi}
For $\dl\ge0,~K\ge0$ and $L\ge1$, we have constants $K_{\ref{quasi-goes-to-quasi}}=K_{\ref{quasi-goes-to-quasi}}(\dl,L,K)$ and  $D_{\ref{quasi-goes-to-quasi}}=D_{\ref{quasi-goes-to-quasi}}(\dl,L,K)$ such that the following hold.
	
Suppose $X$ and $Y$ are $\dl$-hyperbolic metric spaces, and $f:Y\ri X$ is a $L$-qi embedding. Let $U$ be a $K$-quasiconvex subset of $Y$ and $y\in Y$. Then we have the following.
\begin{enumerate}
\item $f(U)$ is $K_{\ref{quasi-goes-to-quasi}}$-quasiconvex in $X$. (For this, we do not need $Y$ to be hyperbolic.)
		
\item \textup{(\hspace{-.08mm}\cite[Lemma $3.5$]{mitra-trees})} If $y\pr$ is a nearest point projection of $y$ on $U$ in $Y$ and $x\pr$ is that of $f(y)$ on $f(U)$ in $X$. Then $d_Y(f(y\pr),x\pr)\le D_{\ref{quasi-goes-to-quasi}}$.
\end{enumerate}
\end{lemma}

\begin{proof}
(1) Suppose $y_1,y_2\in U$ and $x_1=f(y_1)$, $x_2=f(y_2)$. Then by the stability of quasigeodesic (see Lemma \ref{ml}), $Hd_X([x_1,x_2]_X,f([y_1,y_2]_Y))\le D_{\ref{ml}}(\dl,L,L)$. Again $[y_1,y_2]_Y\sse N_K(U)$, and since $f$ is a $L$-qi embedding, $f([y_1,y_2]_Y)\sse N_{(KL+L)}(f(U))$. Hence $[x_1,x_2]_X\sse N_{K_{\ref{quasi-goes-to-quasi}}}(f(U))$ where $K_{\ref{quasi-goes-to-quasi}}:=D_{\ref{ml}}(\dl,L,L)+KL+L$. This completes the proof.
\end{proof}

\begin{lemma}\textup{(\hspace{-.08mm}\cite[Lemma 1.89]{ps-kap})}\label{qi-emb-in-Y-X}
Given $\dl\ge0$, $k\ge0$ and $R\ge k+1$ there is a constant $L_{\ref{qi-emb-in-Y-X}}=L_{\ref{qi-emb-in-Y-X}}(\dl,k,R)$ such that we have the following.
	
Suppose $X$ is a $\dl$-hyperbolic metric space and $A$ is a $k$-quasiconvex subset of $X$. Then $N^X_R(A)$ is path connected and with its induced path metric from $X$, the inclusion $N^X_R(A)\ri X$ is a $L_{\ref{qi-emb-in-Y-X}}$-qi embedding.

\end{lemma}
In reference to the following lemma, it is proven in \cite[Lemma 3.5]{mitra-trees} for a particular case where $A$ is a geodesic segment.
\begin{lemma}\label{for-retraction-flow-space(1)}
Given $\dl\ge0,~L\ge1$ and $K\ge0$, we have constants $K_{\ref{for-retraction-flow-space(1)}}=K_{\ref{for-retraction-flow-space(1)}}(\dl,L,K)$ and $D_{\ref{for-retraction-flow-space(1)}}=D_{\ref{for-retraction-flow-space(1)}}(\dl,L,K)$ such that the following holds.
	
Suppose $X$ is $\dl$-hyperbolic metric space, and $Y\sse X$ is a geodesic subspace such that the inclusion $i:(Y,d_Y)\ri (X,d_X)$ is a $L$-qi embedding  where $d_Y$ is the induced path metric on $Y$ from $X$. Let $A\sse Y$ be $K$-quasiconvex in $Y$. Further, assume that $y\in Y$, and  $y\pr$ is a nearest point projection of $y$ on $A$ in the metric $Y$ and $y\prr$ is that of $y$ on $A$ in the metric $X$. Then $A$ is $K_{\ref{for-retraction-flow-space(1)}}$-quasiconvex in $X$ and $d_X(y\pr,y\prr)\le D_{\ref{for-retraction-flow-space(1)}}$.	
\end{lemma}	
\begin{proof}
By Lemma \ref{quasi-goes-to-quasi} $(1)$, one can take $K_{\ref{for-retraction-flow-space(1)}}=K_{\ref{quasi-goes-to-quasi}}(\dl,L,K)$. For the second part, by \cite[Lemma $1.31$ (2)]{pranab-mahan}, we note that the arc-length parametrization of $[y,y\pr]_Y\cup[y\pr,y\prr]_Y$ is a $(3+2K)$-quasigeodesic in $Y$ and so is $L_1$-quasigeodesic in $X$ for some constant $L_1$ depending on $(3+2K)$ and $L$. Suppose $y_1\in[y,y\prr]_X$ such that $d_X(y\pr,y_1)\le D_{\ref{ml}}(\dl,L_1,L_1)$, and so $d_X(y_1,y\prr)\le D_{\ref{ml}}(\dl,L_1,L_1)$. Hence $d_X(y\pr,y\prr)\le d_X(y\pr,y_1)+d_X(y_1,y\prr)\le2D_{\ref{ml}}(\dl,L_1,L_1)=:D_{\ref{for-retraction-flow-space(1)}}$.
\end{proof}

Given a hyperbolic metric space $X$ and a quasiconvex subset $U$, we denote a map $X\ri U$ sending a point to its nearest point projection by $P_{XU}$ and refer to as a {\em nearest point projection map}. 
Now we recall from \cite[Section $1.18$]{ps-kap}, a small modification in nearest point projection on a path connected quasiconvex subset (\hspace{-.08mm}\cite[Definition $1.121$]{ps-kap}).

\begin{defn}[Modified projection]\label{modified-projection}
Suppose $X$ is a geodesic metric space and $U$ is a path connected quasiconvex subset of $X$. Then for any subset $A\sse X$, the modified projection of $A$ on $U$ is defined as $\bar{P}_{XU}(A):=\textrm{hull}(P_{XU}(A))\sse U$, where the quasiconvex hull is taken in the induced path metric on $U$ from $X$ (see Definition \ref{convexhull} for notation). 
\end{defn}

\begin{lemma}\label{modified-proj-on-qg}
Given $\dl\ge0,~L\ge1$ and $\lm\ge0$ there are constants $\theta_{\ref{modified-proj-on-qg}}=\theta_{\ref{modified-proj-on-qg}}(\dl,L,\lm)$ and $D_{\ref{modified-proj-on-qg}}=D_{\ref{modified-proj-on-qg}}(\dl,L,\lm)$ such that the following hold.
	
Suppose $X$ is a $\dl$-hyperbolic metric space and $Z\sse X$. Further, suppose that with the induced path metric $Z$ is $L$-qi embedded in $X$. Let $Z$ be also $\dl$-hyperbolic. Suppose $x_i\in Z$ ($i=1,2,3$) and $z$ is a $\dl$-center of the triangle $\triangle(x_1,x_2,x_3)$ in $Z$ giving a $\dl$-tripod $Y=\cup_{i=1}^{3}[z,x_i]_Z$ in $Z$. Further, assume that $U$ is a $\lambda$-quasiconvex subset of $X$ and $Y$ is that of $X$. Let $P_{XY}:X\ri Y$ and $P_{XU}:X\ri U$ be nearest point projection maps on $Y$ and on $U$ respectively. Then:
\begin{enumerate}
\item $Hd(P_{XY}(U),\bar{P}_{XY}(U))\le\theta_{\ref{modified-proj-on-qg}}(\dl,L,\lambda)$.
		
\item Let $\bar{Y}=\bar{P}_{XY}(U)$ and $\bar{x}_i\in\bar{Y}$ be the closest to $x_i$ in the intrinsic path metric on $Y$. Then:
		
\begin{enumerate}
\item Let $\bar{Y}\nsubseteq[z,x_i]_Z$ for any $i\in\{1,2,3\}$. Then $d_X(P_{XU}(x_i),P_{XU}(\bar{x}_i))\le D_{\ref{modified-proj-on-qg}}$.
			
\item Let $\bar{Y}\sse [z,x_i]_Z$ for some $i\in\{1,2,3\}$. Note that $\bar{x}_{i+1}=\bar{x}_{i-1}=\bar{z}$ (say). Here $i\pm1$ is calculated in modulo $3$. Then $$d_X(P_{XU}(\bar{x}_i),P_{XU}(x_i)),~d_X(P_{XU}(z),P_{XU}(\bar{z}))\textrm{ and }d_X(P_{XU}(x_{i\pm1}),P_{XU}(\bar{z}))$$ are bounded by $D_{\ref{modified-proj-on-qg}}$.
\end{enumerate}
\end{enumerate} 
\end{lemma}
\begin{proof}
The proof of $(1)$ follows from that of \cite[Lemma $1.125$]{ps-kap}. We only proof $(2)$ $(b)$ since the proof for $(2)$ $(a)$ is a line by line argument of that of $(2)$ $(b)$. In $(2)$ $(b)$, we will specifically address $d_X(P_{XU}(x_{i-1}),P_{XU}(\bar{z}))$ as the other proofs are similar. We fix $i=2$. Then $x_{i-1}=x_1$.
	
Let $P_{XU}(x_1)=x\pr_1$ and $P_{XY}(x\pr_1)=x\prr_1$. Note that $x\prr_1\in[\bar{z},x_2]_Z$. Since $z$ is $\dl$-center of the triangle $\triangle(x_1,x_2,x_3)$ in $Z$, so $d_Z(\bar{z},[x_1,x\prr_1]_Z)\le2\dl$. Thus (by Lemma \ref{ml}) $\exists~z_1\in[x_1,x\prr_1]_X$ such that $d_X(\bar{z},z_1)\le2\dl+D_{\ref{ml}}(\dl,L,L)$. Again, $[x\pr_1,x\prr_1]_X\cup[x\prr_1,x_1]_X$ is $(3+2\lm)$-quasigeodesic in $X$ (\hspace{-.08mm}\cite[Lemma $1.31$ $(2)$]{pranab-mahan}), and so $\exists~z_2\in[x_1,x\pr_1]_X$ such that $d_X(z_1,z_2)\le D_{\ref{ml}}(\dl,3+2\lm,3+2\lm)$. Then by triangle inequality, $d_X(\bar{z},z_2)\le D$, where $D=D_{\ref{ml}}(\dl,3+2\lm,3+2\lm)+2\dl+D_{\ref{ml}}(\dl,L,L)$. Notice that as $P_{XU}(x)=x\pr_1$ and $z_2\in[x_1,x\pr_1]_X$, so $x\pr_1$ is also a nearest point projection of $z_2$ on $U$ in the metric of $X$. Then by Lemma \ref{proj-on-qc} $(1)$, $d_X(P_{XU}(z_2),x\pr_1)\le C_{\ref{proj-on-qc}}(\dl,\lm)$. Hence, $d_X(P_{XU}(\bar{z},x\pr_1)\le d(P_{XU}(\bar{z}),P_{XU}(z_2))+d_X(P_{XU}(z_2),x\pr_1)\le C_{\ref{proj-on-qc}}(\dl,\lm)(D+2)=:D_{\ref{modified-proj-on-qg}}$.
\end{proof}
\begin{remark}\textup{(\hspace{-.08mm}\cite[Remark $1.124$]{ps-kap})}\label{modi-geo}
In the above Lemma \ref{modified-proj-on-qg} $(1)$, if both $U$ and $Y$ are geodesic segments in $X$, one can bound $Hd(P_{XY}(U),\bar{P}_{XY}(U))$ by $4\dl$.
\end{remark}

\subsection{Cobounded pair}
For the rest of this section, we assume that $X$ is a $\dl$-hyperbolic metric space, and $U$ and $V$ are $k$-quasiconvex subsets of $X$ for some $\dl\ge0$ and $k\ge0$.
\begin{defn}
Let $D\ge0$. 
We say that the pair $(U,V)$ is $D$-{\bf cobounded} in $X$ if $$max\{diam(P_{XU}(V)),~diam(P_{XV}(U))\}\le D$$ where $diam(A)$ denotes the diameter of the subset $A\sse X$. The pair $(U,V)$ is said to be cobounded if it is $D$-cobounded for some $D\ge0$.
\end{defn} 

\begin{lemma}\label{R-sep-D-cobdd}\textup{(\hspace{-.08mm}\cite[Lemma $1.139$, Lemma $1.127$]{ps-kap}, \cite[Lemma $1.35$]{pranab-mahan})}
Given $R\ge0$, we have constants $R_{\ref{R-sep-D-cobdd}}= R_{\ref{R-sep-D-cobdd}}(\dl,k)=2k+5\dl$, $D_{\ref{R-sep-D-cobdd}}=D_{\ref{R-sep-D-cobdd}}(\dl,k)=2k+7\dl$ and $R\pr_{\ref{R-sep-D-cobdd}}=2k+3\dl+R$ such that the following hold.
\begin{enumerate}
\item If $d(U,V)>R_{\ref{R-sep-D-cobdd}}$ then the pair $(U,V)$ is $D_{\ref{R-sep-D-cobdd}}$-cobounded.
		
\item If $d(U,V)\le R$ then $P_{XU}(V)\sse N_{R\pr_{\ref{R-sep-D-cobdd}}}(V)\cap U$ and $Hd(P_{XU}(V),P_{XV}(U))\le R\pr_{\ref{R-sep-D-cobdd}}$.
\end{enumerate}
\end{lemma}
One is referred to \cite[Remark $1.142$]{ps-kap} for the upcoming remark. It is possible to minimize those constants similar to what is mentioned in the remark.
\begin{remark}\label{3-re-in-one}
	
\hspace{-8mm}
	
\begin{enumerate}
\item If $U$ and $V$ are geodesic segments in Lemma \ref{R-sep-D-cobdd} $(1)$, then one can take $D_{\ref{R-sep-D-cobdd}}=8\dl$ and $R_{\ref{R-sep-D-cobdd}}=5\dl$.

\item If $U$ and $V$ are geodesic segments in Lemma \ref{R-sep-D-cobdd} $(2)$, then one can take $R\pr_{\ref{R-sep-D-cobdd}}=4\dl+R$.
\end{enumerate} 
\end{remark}

\begin{lemma}\textup{(\hspace{-.08mm}\cite[Corollary $1.140$ (a)]{ps-kap})}\label{small-imp-small}
Let $D\ge0$ and $\{P_{XU}(V)\}\le D$. Then there is a constant $D\pr\ge D$ depending on $\dl,k$ and $D$ such that diam $\{P_{XV}(U)\}\le D\pr$. In particular, the pair $(U,V)$ is $D\pr$-cobounded.
\end{lemma}

\begin{lemma}\label{cobounded-pairs}
Given $\lm\ge0$ and $D\ge0$ there is a constant $C_{\ref{cobounded-pairs}}=C_{\ref{cobounded-pairs}}(\dl,k,\lm,D)$ such that the following holds.
	
Suppose $x,y\in X$ such that $d(P_{XU}(x),P_{XU}(y))\le D$. Let $\al$ be a $\lm$-quasigeodesic in $X$ joining $x$ and $y$. Then the pair $(\al,U)$ is $C_{\ref{cobounded-pairs}}$-cobounded.
\end{lemma}

\begin{proof}
Since quasigeodesics are quasiconvex subsets, so the lemma follows from \cite[Lemma 1.113]{ps-kap} and Lemma \ref{small-imp-small}.
\end{proof}

The proof of the following lemma is standard, so we omit it.

\begin{lemma}\label{fat also cobdd}
Given $D\ge0$ and $R\ge0$ there are constants $D_{\ref{fat also cobdd}}=D_{\ref{fat also cobdd}}(\dl,k,D,R)$ and $K_{\ref{fat also cobdd}}=K_{\ref{fat also cobdd}}(\dl,k,R)$ such that $N_R(U)$ and $N_R(V)$ are $K_{\ref{fat also cobdd}}$-quasiconvex. Moreover, if the pair $(U,V)$ is $D$-cobounded then the pair $(N_R(U),N_R(V))$ is $D_{\ref{fat also cobdd}}$-cobounded.
\end{lemma}

\section{Trees of metric spaces}\label{trees-of-metric-sps}
The notion of trees of metric spaces were introduced by Bestvina and Feighn in \cite{BF}.
A coarsely equivalent definition was given by Mitra in \cite{mitra-trees}. We are going to
adopt the latter definition. 
\begin{defn}\label{tree-of-sps}
Suppose $T$ is a simplicial tree and $X$ is a geodesic metric space. Then a $1$-Lipschitz surjective map $\pi:X\ri T$ is called a {\bf tree of metric spaces} if there is a proper map $\phi: \R_{\ge0}\ri\R_{\ge0}$ with the following properties:
\begin{enumerate}
\item For all $v\in V(T)$, $X_v:=\pi^{-1}(v)$ is a geodesic metric space with the path metric $d_v$ induced from $X$. Moreover, with respect to these metrics, the inclusion $X_v\ri X$ is a $\phi$-proper embedding.
	
\item Suppose $e$ is an edge in $T$ joining $v,w\in V(T)$ and $m_e\in T$ is the midpoint of $e$. Then $X_e:=\pi^{-1}(m_e)$ is a geodesic metric space with respect to the path metric $d_e$ induced from $X$. Moreover, there is a map $\vartheta_e: X_e \times[0,1]\ri\pi^{-1}(e)\sse X$ such that
\begin{enumerate}
\item $\pi\circ\vartheta_e$ is the projection map onto $[v,w]$.
		
\item $\vartheta_e$ restricted to $X_e\times(0,1)$ is an isometry onto $\pi^{-1}(int(e))$ where $int(e)$ denotes the interior of $e$.
		
\item $\vartheta_e$ restricted to $X_e \times \{0\}\simeq X_e$ and $X_e \times \{1\}\simeq X_e$ are $\phi$-proper embeddings from $X_e$ into $X_v$ and $X_w$ respectively with respect to their induced path metrics. We denote these restriction maps by $\vartheta_{e,v}$ and $\vartheta_{e,w}$ respectively and we refer to these maps as incident maps.
	\end{enumerate}
\end{enumerate}

Moreover, we say that $\pi:X\ri T$ is a {\bf tree of hyperbolic metric spaces with the qi embedded condition} if additionally, we have the following. There is $\dl_0\ge0$ and $L_0\ge1$ such that $X_v$ is $\dl_0$-hyperbolic for all $v\in V(T)$ and in $(2)$ $(c)$, $\vartheta_{e,v}$ and $\vartheta_{e,w}$ are $L_0$-qi embeddings.


\end{defn}

In the following lemma we see that if we restrict the tree of metric spaces on some subtree then the inclusion map is uniformly properly embedded.
\begin{lemma}\label{pro-emb}\textup{(\hspace{-.08mm}\cite[Proposition $2.17$]{ps-kap})}
Suppose $\pi:X\ri T$ is a tree of metric spaces. Then there is a function $\eta_{\ref{pro-emb}}=\eta_{\ref{pro-emb}}(\phi):\R_{\ge0}\ri\R_{\ge0}$ depending on $\phi$ as in Definition \ref{tree-of-sps} such that the following holds.

Let $T\pr$ be a subtree of $T$ and $X_{T\pr}:=\pi^{-1}(T\pr)$. Then with respect to the path metric on $X_{T\pr}$ induced from $X$, the inclusion $X_{T\pr}\ri X$ is a $\eta_{\ref{pro-emb}}$-proper embedding.
\end{lemma}

{\bf Convention}: {\em Unless otherwise specified, we always refer to the constants $\dl_0$ and $L_0$ as in Definition \ref{tree-of-sps} for a tree of hyperbolic metric spaces with the qi embedded condition.

Let us fix some secondary constants $\bm{\dl\pr_0,L\pr_0,\lm\pr_0,L\pr_1}$ in the following Lemma \ref{com-two-hyp-sps} and these notations will be used through out the paper.}

\begin{lemma}\textup{(\hspace{-.08mm}\cite[Corollary $2.62$, Lemma $2.27$]{ps-kap})}\label{com-two-hyp-sps}
Suppose $\pi:Z\ri [v,w]$ is a tree of metric spaces over an edge $e=[v,w]$ such that $Z_v$, $Z_w$ are $\dl_0$-hyperbolic and $Z_e$ is $L_0$-qi embedded in both $Z_v$ and $Z_w$. Then

\begin{enumerate}
\item $Z$ is $\bm{\dl\pr_0}$-hyperbolic, and $Z_v$ and $Z_w$ are $\bm{L\pr_0}$-qi embedded in $Z$.
	
\item Suppose $U$ is a $2\dl_0$-quasiconvex subset of the fiber $Z_v$ or $Z_w$. Then $U$ is $\bm{\lambda\pr_0}$-quasiconvex in $Z$ \textrm{$($see Lemma \ref{quasi-goes-to-quasi} $(1))$}, where $\lambda\pr_0=K_{\ref{quasi-goes-to-quasi}}(\dl\pr_0,L\pr_0,2\dl_0)$. In particular, $Z_v,Z_w$ are $\lambda\pr_0$-quasiconvex in $Z$. Thus a nearest point projection map $P_{ZZ_w}:Z\ri Z_w$ in the metric $Z$ is $\bm{L\pr_1}$-coarsely Lipschitz retraction, where $L\pr_1=C_{\ref{proj-on-qc}}(\dl\pr_0,\lambda\pr_0)$.
\end{enumerate}

\end{lemma}

\begin{lemma}\label{need-in-mitra's-proj}
Given $k\ge0,~D\ge0$ and $\ep\ge0$ there is a constant $R_{\ref{need-in-mitra's-proj}}=R_{\ref{need-in-mitra's-proj}}(k,D,\ep)$ 
such that the following holds.

Suppose we have the assumptions of Lemma \ref{com-two-hyp-sps}. Let $A_v$ be a $k$-quasiconvex subsets of $Z_v$ in $Z_v$-metric and $A_w$ be that of $Z_w$ in $Z_w$-metric. Let $x\in Z_v$ and $y\in Z_w$ such that $d_Z(x,y)=1$, and $x'$ be a nearest point projection of $x$ on $A_v$ in $Z_v$-metric and $y'$ be that of $y$ on $A_w$ in $Z_w$-metric. Further, we assume that $Hd_Z(P_{ZZ_w}(A_v),A_w)\le\ep$ and $d_Z(z,A_v)\le D$ for all $z\in A_w$. Then $d_Z(x',y')\le R_{\ref{need-in-mitra's-proj}}$.

\end{lemma}

\begin{proof}
Note that $Z$ is $\dl\pr_0$-hyperbolic, and $Z_v$ and $Z_w$ are $L\pr_0$-qi embedded in $Z$ (see Lemma \ref{com-two-hyp-sps}). Thus $A_v$ and $A_w$ are $K$-quasiconvex in $Z$ for some $K=K_{\ref{quasi-goes-to-quasi}}(
\dl\pr_0,L\pr_0,k)$. Suppose $x_1$ is a nearest point projection of $x$ on $A_v$ and $y_1$ is that of $y$ on $A_w$ in $Z$. Then by Lemma \ref{for-retraction-flow-space(1)}, $d_Z(x\pr,x_1)\le D_{\ref{for-retraction-flow-space(1)}}(\dl\pr_0,L\pr_0,K)$ and $d_Z(y',y_1)\le D_{\ref{for-retraction-flow-space(1)}}(\dl'_0,L\pr_0,K)$. Now we will find a bound on $d_Z(x_1,y_1)$. 

Let $y_2$ and $y_3$ be nearest point projections of $y$ and $y_1$ on $A_v$ respectively in $Z$. By given condition we have $d_Z(y_1,y_3)\le D$ and by Lemma \ref{proj-on-qc} (1) we have $d_Z(x_1,y_2)\le 2C_{\ref{proj-on-qc}}(\dl\pr_0,K)$. Consider the pair $(A_v,Z_w)$ in $Z$. By Lemma \cite[Lemma $1.127$]{ps-kap}, if $y\pr_2$ is a nearest point projection of $y_2$ on $Z_w$,  we have $d_Z(y_2,y\pr_2)\le2K+3\dl\pr_0+D$. Again by given condition we have $d_Z(y\pr_2,A_w)\le\epsilon$, and so $d_Z(y_2,A_w)\le2K+3\dl\pr_0+D+\epsilon$. Now if $y\prr_2$ is a nearest point projection of $y_2$ on $A_w$, then $d_Z(y_2,y\prr_2)\le2K+3\dl\pr_0+D+\epsilon=\epsilon\pr$ (say).

Now we have $d_Z(y,y_3)\le d_Z(y,y_1)+d_Z(y_1,y_3)\le d_Z(y,y\prr_2)+d_Z(y_1,y_3)\le d_Z(y,y_2)+d_Z(y_2,y\prr_2)+D\le d_Z(y,y_2)+\epsilon\pr+D$. Let $z\in[y,y_3]_Z$ such that $d_Z(y_2,z)\le K+2\dl_0\pr$ (see \cite[Lemma $1.102~(i)$]{ps-kap}). Then $d_Z(z,y_3)=d_Z(y,y_3)-d(y,z)\le d_Z(y,y_2)+\epsilon\pr+D-d(y,z)\le d_Z(y_2,z)+\epsilon\pr+D\le K+2\dl_0\pr+\epsilon\pr+D$. So $d_Z(y_2,y_3)\le d_Z(y_2,z)+d_Z(z,y_3)\le 2(K+2\dl_0\pr)+\epsilon\pr+D=D_1$ (say).

Hence combining all inequalities, we have
\begin{align*}
	d_Z(x',y')&\le d_Z(x',x_1)+d_Z(x_1,y_2)+d_Z(y_2,y_3)+d_Z(y_3,y_1)+d_Z(y_1,y')\\
	&\le 2D_{\ref{for-retraction-flow-space(1)}}(\dl\pr_0,L\pr_0,K)+2C_{\ref{proj-on-qc}}(\dl\pr_0,K)+D_1+D=:R_{\ref{need-in-mitra's-proj}}.
\end{align*}

Therefore, we are through.
\end{proof}

\begin{lemma}\label{need-in-mitra's-proj1}
Given $k\ge0$ there is a constant $R_{\ref{need-in-mitra's-proj1}}=R_{\ref{need-in-mitra's-proj1}}(k)$ such that the following holds.

Suppose we have the assumptions of Lemma \ref{com-two-hyp-sps}. Let $A_v$ be a $k$-quasiconvex subset of $Z_v$ in $Z_v$-metric. Let $x\in Z_v$ and $y\in Z_w$ such that $d_Z(x,y)=1$, and $x'$ be a nearest point projection of $x$ on $A_v$ in $Z_v$-metric and $y'$ be that of $y$ on $A_v$ in $Z$-metric. Then $d_Z(x',y')\le R_{\ref{need-in-mitra's-proj1}}$. 
\end{lemma}
\begin{proof}
From the first paragraph of the proof of Lemma \ref{need-in-mitra's-proj}, $A_v$ is $K$-quasiconvex in $Z$ where $K=K_{\ref{quasi-goes-to-quasi}}(
\dl\pr_0,L\pr_0,k)$. If $x_1$ is a nearest point projection of $x$ on $A_v$ in $Z$-metric, then by Lemma \ref{proj-on-qc} (1), $d_Z(x_1,y')\le2 C_{\ref{proj-on-qc}}(\dl\pr_0,K)$. Again by Lemma \ref{for-retraction-flow-space(1)}, $d_Z(x\pr,x_1)\le D_{\ref{for-retraction-flow-space(1)}}(\dl\pr_0,L'_0,K)$. Therefore, we can take $R_{\ref{need-in-mitra's-proj1}}:=D_{\ref{for-retraction-flow-space(1)}}(\dl\pr_0,L'_0,K)+C_{\ref{proj-on-qc}}(\dl\pr_0,K)$.
\end{proof}

%


\section[A combination theorem revisited]{A combination theorem for trees of metric spaces revisited}\label{hyp-tree-sps}

Suppose $\pi:X\ri T$ is a tree of metric spaces (see Definition \ref{tree-of-sps}). In this section, we prove the hyperbolicity of $X$ within an axiomatic framework (Theorem \ref{treeofsps-com-thm}). As a consequence, we get a proof of Theorem \ref{main-theorem-com} in Section \ref{hyp-tmb-sec}. Now we will explain the hypotheses. {\em Unless otherwise specified, by $u\in S$ (or $v\in S$ or $w\in S$) where $S$ is a subtree of $T$, we always mean $u$ (or $v$ or $w$) to be a vertex of $S$}. We use the {\em notation} $X_S:=\pi^{-1}(S)$.

For each vertex $u\in T$ there is a subspace, say $\M(X_u)$, containing $X_u$ and that satisfies the following properties $(\mathcal{P}0)-(\mathcal{P}4)$.

$(\mathcal{P}0)$ Suppose $u,v\in T$ and $e$ is the edge on the geodesic $[u,v]$ incident on $v$. Let $T\pr$ be the maximal subtree of $T$ containing $v$ but not containing $e$. Then $\M(X_u)\cap X_{T\pr}\sse \M(X_v)\cap X_{T\pr}$.

$(\mathcal{P}1)$ Let $L\pr\ge0$. For each $u\in T$, there is a $L\pr$-coarsely Lipschitz retraction $\rho_u:X\ri \M(X_u)$. We also have the following property of $\rho_u$. Let $T_u=\pi(\M(X_u))$ and $e$ be an edge in $T$ intersecting $T_u$ at a vertex. Suppose $v$ is the vertex adjacent to $e$ not in $T_u$ and $S$ is the maximal subtree of $T$ containing $v$ but not containing $e$. Then diam$\{\rho_u(X_S)\}\le C$ for some uniform constant $C\ge0$. 

There is a threshold constant $L\ge0$ and a proper map $\eta:\R_{\ge0}\map\R_{\ge0}$, and a constant $\dl\ge0$ such that the following hold.

$(\mathcal{P}2)$ The subspace $N_L(\M(X_u))$ is path connected, and with the induced path metric from $X$, the inclusion $N_L(\M(X_u))\ri X$ is a $\eta$-proper embedding.


For $u,v\in T$, we say $[u,v]\sse T$ is a \emph{special interval} if either $\M(X_u)\cap X_v\ne\emptyset$ or $X_u\cap \M(X_v)\ne\emptyset$. If $[u,v]$ is a special interval then

$(\mathcal{P}3)$  the inclusion $N_L(\M(X_u))\cup N_L(\M(X_v))\ri X$ is a $\eta$-proper embedding, and

$(\mathcal{P}4)$ $N_L(\M(X_u))\cup N_L(\M(X_v))$ is $\dl$-hyperbolic metric space.

\begin{theorem}\label{treeofsps-com-thm}
Suppose $\pi:X\ri T$ is a tree of metric spaces satisfying properties $(\mathcal{P}0)-(\mathcal{P}4)$. Then $X$ is hyperbolic metric space.
\end{theorem}

{\bf Some remarks on Theorem \ref{treeofsps-com-thm}}: (1) Lemma \ref{gen-abs-flow-qi-emb} below says that $\mathcal{M}_L(X_u)$ are also hyperbolic (by Theorem \ref{treeofsps-com-thm} or by property $(\mathcal{P}4)$).


(2) Suppose $\pi:X\ri T$ is a trees of hyperbolic metric spaces with the qi embedded condition and the Bestvina--Feighn hallway flaring condition (\cite{BF}). Now we think of $\pi:X\ri T$ as trees of metric bundles (such that the natural projection $\pi_B:B\map T$, according to our notation, is isometry). We set $\mathcal{M}(X_u)=\F l_K(X_u)$ (flow space of $X_u$ with certain fixed parameters) (see Subsection \ref{flowsps}). In Section \ref{hyp-tmb-sec}, it is shown, in this case, that $\mathcal{M}(X_u)$ satisfies all the conditions $(\mathcal{P}0)-(\mathcal{P}4)$. This shows that Theorem \ref{treeofsps-com-thm} covers the combination theorem for trees of metric spaces considered in \cite{BF} and particularly, acylindrical trees of metric spaces (\cite{ilya-kap}).


\begin{defn}
	For a subtree $S$ of $T$, we define $\M(X_S):=\cup_{w\in S} \M(X_w)$. Also for finitely many vertices $u_1,u_2,\cdots,u_n$, we define $\M(X_{\{u_1,u_2,\cdots,u_n\}}):=\M(X_{u_1})\cup \M(X_{u_2})\cup\cdots\cup \M(X_{u_n})$.
\end{defn}
{\bf Notation}: {\em For a given subtree $S\sse T$, we denote $N_L(\M(X_S))$ by $\M_L(X_S)$. Unless otherwise specified, in this section, ${\bm L',~\bm C,~\bm L,~\bm\dl}$ and ${\bm \eta}:\R_{\ge0}\map\R_{\ge0}$ are fixed for their values above.}\smallskip

The proof of Theorem \ref{treeofsps-com-thm} is divided into two parts as follows.

$(1)$ Hyperbolicity of $\M_L(X_I)$ where $I$ is a special interval in $T$.

$(2)$ Hyperbolicity of $\M_L(X_I)$ where $I$ is any interval in $T$.

Finally, using $(2)$, we conclude the proof of Theorem \ref{treeofsps-com-thm}. Before going into the proof of $(1)$, let us first prove some lemmata which are required to prove $(1)$ and $(2)$.

\begin{lemma}\label{retraction-gen-abs-flow}
There is a constant $L_{\ref{retraction-gen-abs-flow}}=L_{\ref{retraction-gen-abs-flow}}(L)$ such that for any subtree $S$ of $T$, we have a $L_{\ref{retraction-gen-abs-flow}}$-coarsely Lipschitz retraction $\rho_S:X\ri\mathcal{M}(X_S)$.
\end{lemma}

\begin{proof}
	Let us first define $\rho_S:X\ri\M(X_S)$. Let $x\in X$ and $u$ be the nearest point projection of $\pi(x)$ onto $S$. Then $\rho_S(x)$ is defined to be $\rho_u(x)$. Note that if $x\in\M(X_S)$, then by $(\mathcal{P}1)$, $\rho_S(x)=x$.
	
	Let $X_{vsp}=\cup_{u\in T}X_u$ and $x,y\in X_{vsp}$ such that $d_X(x,y)\le1$. Then by Lemma \ref{imp-coarse-retraction}, we need to show a uniform bound on $d_X(\rho_S(x),\rho_S(y))$. Let $u,v$ be the nearest point projections of $\pi(x),\pi(y)$ on $S$ respectively. If $u=v$, then by definition of $\rho_S$, $\rho_S(x)=\rho_u(x)$ and $\rho_S(y)=\rho_u(y)$. So by $(\mathcal{P}1)$, $d_X(\rho_S(x),\rho_S(y))=d_X(\rho_u(x),\rho_u(y))\le2L\pr$. Now let $u\ne v$. Since $d_T(\pi(x),\pi(y))\le 1$, we have $x,y\in X_S$. So by definition of $\rho_S$, $d_X(\rho_S(x),\rho_S(y))\le 1$. Therefore, we can take $L_{\ref{retraction-gen-abs-flow}}:=C_{\ref{imp-coarse-retraction}}(\textrm{max}\{2L\pr,1\})$.
\end{proof}

\begin{lemma}\label{gen-abs-flow-pro-emb}
There is a proper function $\eta_{\ref{gen-abs-flow-pro-emb}}=\eta_{\ref{gen-abs-flow-pro-emb}}(L):\R_{\ge0}\ri\R_{\ge0}$ such that for any subtree $S$ of $T$, the inclusion $\M_L(X_S)\ri X$ is a $\eta_{\ref{gen-abs-flow-pro-emb}}$-proper embedding.
\end{lemma}

\begin{proof}
We denote the induced path metric on $\M_L(X_S)$ (from $X$) by $d\pr$. Suppose $x_1,y_1\in\M_L(X_S)$ such that $d_X(x_1,y_1)\le r$ for some $r\in\R_{\ge0}$. Then there are points $x,y\in\M(X_S)$ such that $d\pr(x_1,x)\le L$, $d\pr(y_1,y)\le L$, and so $d_X(x,y)\le r+2L$. Let $\pi(x)=u$ and $\pi(y)=v$. Since $\pi$ is $1$-Lipschitz, $d_T(u,v)\le r+2L$. Let $u\pr$ be the nearest point projection of $u$ on $S$ and $v\pr$ be that of $v$ on $S$. We consider the following two cases depending on whether $u\pr=v\pr$ or $u\pr\ne v\pr$.
	
\emph{Case} $1$: Suppose $u\pr\ne v\pr$. Let $x\pr\in X_{u\pr}\cap[x,y]$ and $y\pr\in X_{v\pr}\cap [x,y]$. Then $d_X(x,x\pr)\le r+2L$, $d_X(y,y\pr)\le r+2L$, and also $d_X(x\pr,y\pr)\le r+2L$. Now $x, x\pr\in \M_L(X_{u\pr})$ (see $(\mathcal{P}0)$) implies $d\pr(x,x\pr)\le\eta(r+2L)$ (by $(\mathcal{P}2)$). Similarly, $d\pr(y,y\pr)\le\eta(r+2L)$. Again $x\pr,y\pr\in X_S$ implies $d\pr(x\pr,y\pr)\le d_{X_S}(x\pr,y\pr)\le\eta_{\ref{pro-emb}}(r+2L)$ (see Lemma \ref{pro-emb}). Hence by triangle inequality, $d\pr(x,y)\le2\eta(r+2L)+\eta_{\ref{pro-emb}}(r+2L)$.
	
\emph{Case} $2$: Suppose $u\pr=v\pr$. Then by $(\mathcal{P}0)$, $x,y\in\M(X_{u\pr})$. Hence $d\pr(x,y)\le\eta(r+2L)$ by $(\mathcal{P}2)$.
	
Therefore, by triangle inequality, in both the cases, $d\pr(x_1,y_1)\le2L+2\eta(r+2L)+\eta_{\ref{pro-emb}}(r+2L)=:\eta_{\ref{gen-abs-flow-pro-emb}}(L)(r)$.
\end{proof}

Therefore, combining Lemmata \ref{retraction-gen-abs-flow} and \ref{gen-abs-flow-pro-emb}, we have the following (see Lemma \ref{lrape_qi-emb}).
\begin{lemma}\label{gen-abs-flow-qi-emb}
There is a constant $L_{\ref{gen-abs-flow-qi-emb}}=L_{\ref{gen-abs-flow-qi-emb}}(L)\ge1$ such that for any subtree $S$ of $T$, the inclusion $\M_L(X_S)\ri X$ is a $L_{\ref{gen-abs-flow-qi-emb}}$-qi embedding.	
\end{lemma}


\begin{prop}[Horizontal Subdivision]\label{horizontal-subdivision}
	Let $J=[u,v]\sse T$ be an interval and $n_0\in\N$. Then we can subdivide $J$ into subintervals $J=J_0\cup J_1\cup...\cup J_{n-1}$ such that $J_i=[w_i,w_{i+1}],~w_0=u,~w_{n}=v$, and each $J_i$ is further subdivided into subintervals, $J_i=[w_i,w_{i,1}]\cup[w_{i,1},w_{i,2}]\cup[w_{i,2},w_{i,3}]\cup[w_{i,3},w_{i+1}]$ such that the following hold.
	
	\begin{enumerate}
		\item $\pi(\M(X_{w_i}))\cap J_i=[w_i,w_{i,1}]$, $\fa~0\le i\le n-1$.
		
		\item For all $i$ except possibly $i=n-1$, $d_T(w_{i,1},w_{i,2})\le2n_0$. Also, $[w_i,w_{i,1}],~[w_{i,2},w_{i,3}]$ and $[w_{i,3},w_{i+1}]$ are special intervals. Moreover, $d_T(w_{i,3},w_{i+1})=1$.
		
		\item $d_T(\pi(\M(X_{w_i})),\pi(\M(X_{w_{i+1}})))>2n_0$, $\fa~1\le i\le n-2$.
	\end{enumerate}
\end{prop}

\begin{proof}
	The proof is by induction. Suppose we have constructed $J_{i-1}$ and we want to construct $J_i$.
	
	\emph{Case 1}: Suppose $\M(X_{w_i})\cap X_v\ne\emptyset$. Then we stop the process and set $n-1=i,~J_{n-1}=[w_{n-1},v]$ and $w_{{n-1},s}=v=w_{n}$ for $s=1,2,3$.
	
	\emph{Case 2}: Suppose $\M(X_{w_i})\cap X_v=\emptyset$. Consider the vertex $w_{i,1}\in(w_i,v]$ in $T$, which is the farthest from $w_i$ such that $\pi(\M(X_{w_i}))\cap[w_i,v]=[w_i,w_{i,1}]$. Now we consider the following two subcases.
	
	\emph{Subcase (2A)}: Suppose $d_T(w_{i,1},\pi(\M(X_v)))\le2n_0$. Then we consider $w_{i,2}\in[w_{i,1},v]$ such that $X_{w_{i,2}}\cap\M(X_v)\ne\emptyset$ and $d_T(w_{i,1},w_{i,2})\le2n_0$. Then we stop the process and set $n-1=i$ and $w_{{n-1},3}=v=w_n$.

	\emph{Subcase (2B)}:  Suppose $d_T(w_{i,1},\pi(\M(X_v)))>2n_0$. We take $w_{i+1}\in[w_{i,1},v]$ is the farthest from $v$ such that $d_T(w_{i,1},\pi(\M(X_{w_{i+1}})))>2n_0$. Let $w_{i,3}\in[w_{i,1},w_{i+1}]$ such that $d_T(w_{i,3},w_{i+1})=1$. Then by our choices, $d_T(w_{i,1},\pi(\M(X_{w_{i,3}})))\le2n_0$. Now we fix $w_{i,2}\in[w_{i,1},w_{i,3}]$ such that $d_T(w_{i,1},w_{i,2})\le2n_0$ and $X_{w_{i,2}}\cap\M(X_{w_{i,3}})\ne\emptyset$. We also note that $d_T(\pi(\M(X_{w_i})),\pi(\M(X_{w_{i+1}})))$ $>2n_0$, otherwise, $$d_T(w_{i,1},\pi(\M(X_{w_{i+1}})))\le2n_0.$$
	
	Therefore, we get $J_i=[w_i,w_{i,1}]\cup[w_{i,1},w_{i,2}]\cup[w_{i,2},w_{i,3}]\cup[w_{i,3},w_{i+1}]$ with the required properties.
	
	The induction stops at $(n-1)^{\textrm{th}}$ step if $w_n=v$. Therefore, we are through.
\end{proof}

\begin{lemma}\label{intersection-of-generalized-flow-sps}
Suppose $S_1$ and $S_2$ are two subtrees in $T$ such that $S_1\cap S_2=\{u\}$. Then for all $L\ge0,~ \M_L(X_{S_1})\cap  \M_L(X_{S_2})= \M_L(X_u)$.
	
	
\end{lemma}

\begin{proof}
	It is clear that $ \M_L(X_u)\sse  \M_L(X_{S_1})\cap  \M_L(X_{S_2})$. For the reverse inclusion, let $x\in  \M_L(X_{S_1})\cap  \M_L(X_{S_2})$. Then there exists $x_i\in \M(X_{S_i})$ such that $d_X(x,x_i)\le L$ for $i=1,2$. Let $\pi(x_i)=t_i$ for $i=1,2$. Now by $(\mathcal{P}0)$, if $d_T(t_2,S_1)\le d_T(t_2,S_2)$ then $x_2\in \M(X_u)$ or if $d_T(t_1,S_2)\le d_T(t_1,S_1)$ then $x_1\in \M(X_u)$. In either case, $x\in  \M_L(X_u)$. Now suppose $d_T(t_2,S_1)>d_T(t_2,S_2)$ and $d_T(t_1,S_2)>d_T(t_1,S_1)$. Then since $T$ is a tree, at least one of the geodesics $[x,x_1]_X$ or $[x,x_2]_X$ has to pass through $X_u$. Hence $d_X(x,X_u)\le L$ and so $x\in  \M_L(X_u)$. Therefore, we are done.
\end{proof}

The proof of the upcoming lemma follows from a similar line of reasoning as Lemma \ref{intersection-of-generalized-flow-sps}. So we omit the proof.

\begin{lemma}\label{intersection-of-union-of-flow-sps}
Suppose $u,v$ and $w$ are vertices of $T$ and that $u,v,w$ lie on an interval in $T$ such that $d_T(u,v)\le d_T(u,w)$. Then for all $L\ge0,~ \M_L(X_{\{u,v\}})\cap  \M_L(X_{\{v,w\}})= \M_L(X_v)$.
	
	
\end{lemma}



\begin{lemma}\label{union-uvw-is-hyperbolic}
We have constants $\dl_{\ref{union-uvw-is-hyperbolic}}=\dl_{\ref{union-uvw-is-hyperbolic}}(L)\ge0$ and $K_{\ref{union-uvw-is-hyperbolic}}=K_{\ref{union-uvw-is-hyperbolic}}(L)\ge0$ satisfying the following. Suppose $u,v$ and $w$ lie on an interval in $T$ such that $d_T(u,v)\le d_T(u,w)$ and $ \M(X_u)\cap X_v\ne\emptyset$, $ \M(X_v)\cap X_w\ne\emptyset$. Then $ \M_L(X_{\{u,v,w\}})$ is $\dl_{\ref{union-uvw-is-hyperbolic}}$-hyperbolic metric space with the induced path metric from $X$. Further, the union of any two intersecting spaces among $\{ \M_L(X_u), \M_L(X_v), \M_L(X_w)\}$ is $K_{\ref{union-uvw-is-hyperbolic}}$-quasiconvex in $ \M_L(X_{\{u,v,w\}})$.
\end{lemma}

\begin{proof}
For the first part, we will apply Proposition \ref{combi-hyp-sps} for $n=2$ (see Remark \ref{combi-hyp-sps-2}). Since $\M(X_u)\cap X_v\ne\emptyset$ and $\M(X_v)\cap X_w\ne\emptyset$, by $(\mathcal{P}4)$, $ \M_L(X_{\{u,v\}})$ and $ \M_L(X_{\{v,w\}})$ are $\dl(L)$-hyperbolic. Now by Lemma \ref{intersection-of-union-of-flow-sps}, $ \M_L(X_{\{u,v\}})\cap  \M_L(X_{\{v,w\}})= \M_L(X_v)$. Again, by Lemma \ref{gen-abs-flow-qi-emb}, $ \M_L(X_v)$ is $L_{\ref{gen-abs-flow-qi-emb}}(L)$-qi embedded in $X$ and so is in both $ \M_L(X_{\{u,v\}})$ and $ \M_L(X_{\{v,w\}})$ with respect to their corresponding path metrics.
Thus by Remark \ref{combi-hyp-sps-2}, $ \M_L(X_{\{u,v\}})\cup\M_L(X_{\{v,w\}})= \M_L(X_{\{u,v,w\}})$ is $\dl_{\ref{union-uvw-is-hyperbolic}}$-hyperbolic, where $\dl_{\ref{union-uvw-is-hyperbolic}}:=\dl_{\ref{combi-hyp-sps-2}}(\dl,L_{\ref{gen-abs-flow-qi-emb}}(L))$.
	
For the second part, we prove that $ \M_L(X_{\{u,w\}})$ is a quasiconvex in $ \M_L(X_{\{u,v,w\}})$ provided $ \M_L(X_u)\cap  \M_L(X_w)\ne\emptyset$, and the proof is same for other intersections. By Lemma \ref{gen-abs-flow-qi-emb}, $ \M_L(X_u)$ and $ \M_L(X_w)$ are $L_{\ref{gen-abs-flow-qi-emb}}(L)$-qi embedded in $X$ and so are in $ \M_L(X_{\{u,v,w\}})$. Then by Lemma \ref{quasi-goes-to-quasi} $(1)$, $ \M_L(X_u)$ and $ \M_L(X_w)$ are $K_1$--quasiconvex in $ \M_L(X_{\{u,v,w\}})$, where $K_1=K_{\ref{quasi-goes-to-quasi}}(\dl_{\ref{union-uvw-is-hyperbolic}}(L),L_{\ref{gen-abs-flow-qi-emb}}(L),0)$. Therefore, $ \M_L(X_u)\cup  \M_L(X_w)$ is $K_{\ref{union-uvw-is-hyperbolic}}$-quasiconvex in $ \M_L(X_{\{u,v,w\}})$, where $K_{\ref{union-uvw-is-hyperbolic}}:=K_1+\dl_{\ref{union-uvw-is-hyperbolic}}$.
\end{proof}

\underline{{\bf Hyperbolicity of $\boldmath{\M_L(X_I)}$ where $\boldmath{I}$ is a special interval}}:

\begin{prop}\label{special-int-sp-is-hyp}
We have a constant $\dl_{\ref{special-int-sp-is-hyp}}=\dl_{\ref{special-int-sp-is-hyp}}(L)$ such that for any special interval $I$ in $T$ the space $ \M_L(X_I)$ is $\dl_{\ref{special-int-sp-is-hyp}}$-hyperbolic with the induced path metric from $X$.
\end{prop}

\begin{proof}
We will apply Proposition \ref{combing}. Let $I=[u\pr,v\pr]$. Without loss of generality, we assume that $ \M(X_{u\pr})\cap X_{v\pr}\ne\emptyset$.
	
\emph{Choices:} For a given $x\in \M(X_I)$, we fix once and for all $u_x\in I$ corresponding to $x$ such that $x\in \M(X_{u_x})$. For a pair $(x,y)$ of distinct points in $ \M(X_I)$, without loss of generality, we assume that $d_T(u\pr,u_x)\le d_T(u\pr,u_y)$. Since $ \M(X_{u\pr})\cap X_{v\pr}\ne\emptyset$, so by $(\mathcal{P}0)$, $ \M(X_{u_x})\cap X_{u_y}\ne\emptyset$. We take $c(x,y)$ a fixed geodesic path in $\M_L(X_{\{u_x,u_y\}})$. These paths serve as family of paths for Proposition \ref{combing}.
	
Note that $ \M(X_I)$ is $L$-dense subset in $ \M_L(X_I)$. Let $x,y,z\in \M(X_I)$. Without loss of generality, we assume that $x\in \M(X_u),~y\in \M(X_v)$ and $z\in \M(X_w)$ for some $u,v,w\in[u\pr,v\pr]$, and $d_T(u\pr,u)\le d_T(u\pr,v)\le d_T(u\pr,w)$. So by $(\mathcal{P}0)$, $ \M(X_u)\cap X_v\ne\emptyset,~ \M(X_v)\cap X_w\ne\emptyset$ and $ \M(X_u)\cap X_w\ne\emptyset$.
	
{\bf Condition (1)}: Suppose $s,t\in\{u,v,w\}$. Then $[s,t]$ is a special interval. Hence by $(\mathcal{P}3)$, $ \M_L(X_{\{s,t\}})$ is $\eta$-properly embedded in $X$ and so is in $ \M_L(X_I)$. Hence the family of paths are $\eta$-properly embedded in $ \M_L(X_I)$.
	
{\bf Condition (2)}: We want to prove that the triangle formed by the paths $c(x,y),~c(y,z)$ and $c(x,z)$ is uniformly slim.
Then by Lemma \ref{union-uvw-is-hyperbolic}, $ \M_L(X_{\{u,v,w\}})\sse  \M_L(X_I)$ is $\dl_{\ref{union-uvw-is-hyperbolic}}(L)$-hyperbolic with the induced path metric and for all $s,t\in\{u,v,w\}$, $ \M_L(X_{\{s,t\}})$ is $K_1$-quasiconvex in $ \M_L(X_{\{u,v,w\}})$, where $K_1=K_{\ref{union-uvw-is-hyperbolic}}(L)$. Now we will show that $ \M_L(X_{\{s,t\}})$ is uniformly qi embedded in $ \M_L(X_{\{u,v,w\}})$ where $s,t\in\{u,v,w\}$. Let $N\pr_{K_1+1}(\M_L(X_{\{s,t\}}))\sse \M_L(X_{\{u,v,w\}})$ denote $(K_1+1)$-neighborhood of $\M_L(X_{\{s,t\}})$ in metric of $ \M_L(X_{\{u,v,w\}})$. Hence by Lemma \ref{qi-emb-in-Y-X}, $N\pr_{(K_1+1)}( \M_L(X_{\{s,t\}}))$ is $L_1$-qi embedded in $ \M_L(X_{\{u,v,w\}})$ for some $L_1$ depending on $\dl_{\ref{union-uvw-is-hyperbolic}}(L)$ and $K_1$. Hence by Lemma \ref{hd-imp-qi}, the inclusion $ \M_L(X_{\{s,t\}})\ri  \M_L(X_{\{u,v,w\}})$ is a $L_2$-qi embedding, where $L_2=L_{\ref{hd-imp-qi}}(\dl_{\ref{union-uvw-is-hyperbolic}}(L),K_1+1)$.
So by the stability of quasigeodesic (see Lemma \ref{ml}), there is a constant $D=2D_{\ref{ml}}(\dl_{\ref{union-uvw-is-hyperbolic}},L_2,L_2)+\dl_{\ref{union-uvw-is-hyperbolic}}(L)$ such that the triangle formed by paths $c(x,y)$, $c(y,z)$, $c(x,z)$ is $D$-slim in the metric of $ \M_L(X_{\{u,v,w\}})$ and so is in the metric of $ \M_L(X_I)$. 
	
Therefore, by Proposition \ref{combing}, $ \M_L(X_I)$ is $\dl_{\ref{special-int-sp-is-hyp}}$-hyperbolic, where $\dl_{\ref{special-int-sp-is-hyp}}=\dl_{\ref{combing}}(\eta,D,L)$.
\end{proof}
As an iterated application of Proposition \ref{special-int-sp-is-hyp} along with Proposition \ref{combi-hyp-sps} for $n=2$ (see Remark \ref{combi-hyp-sps-2}), we obtain the following. The proof is omitted.

\begin{lemma}\label{hyp-of-depending-on-length}
Given $l\in\N$ there is a constant $\dl_{\ref{hyp-of-depending-on-length}}=\dl_{\ref{hyp-of-depending-on-length}}(L,l)$ satisfying the following. Let $I=[u,v]$ for $u,v\in T$ such that $d_T(u,v)\le l$. Then $ \M_L(X_I)$ is a $\dl_{\ref{hyp-of-depending-on-length}}$-hyperbolic metric space with the induced path metric from $X$.
\end{lemma}

\underline{{\bf Hyperbolicity of $\M_L(X_I)$ where $I$ is any interval}}: Before going into the proof, we first prove the following two lemmata which will be used in the proof.

Let $l\in\N$. Suppose $J=\cup_{i=1}^{4}J_i\sse T$ is an interval in $T$ such that the length of $J_2\le l$. Further $J_1,J_3,J_4$ are special intervals, and $J_i\cap J_{i+1}$ is a single vertex for $1\le i\le 3$. 

\begin{lemma}\label{union-of-three-special-int-is-hyp}
Suppose $J=\cup_{i=1}^{4}J_i$ is an interval as described above. For all $l\in\N$ there exists $\dl_{\ref{union-of-three-special-int-is-hyp}}=\dl_{\ref{union-of-three-special-int-is-hyp}}(L,l)$ such that $ \M_L(X_J)$ is $\dl_{\ref{union-of-three-special-int-is-hyp}}$-hyperbolic metric space with the induced path metric from $X$.
\end{lemma}
\begin{proof}
We will apply Proposition \ref{combi-hyp-sps} for $n=2$ (see Remark \ref{combi-hyp-sps-2}) three times, successively on the pairs $( \M_L(X_{J_1}), \M_L(X_{J_2}))$, $( \M_L(X_{J_1\cup J_2}), \M_L(X_{J_3}))$ and $( \M_L(X_{J_1\cup J_2\cup J_3}), \M_L(X_{J_4}))$. Since $J_1,J_3,J_4$ are special interval, by Proposition \ref{special-int-sp-is-hyp}, $ \M_L(X_{J_i})$ is $\dl_{\ref{special-int-sp-is-hyp}}(L)$-hyperbolic for $i=1,3,4$; and by Lemma \ref{hyp-of-depending-on-length}, $ \M_L(X_{J_2})$ is $\dl_{\ref{hyp-of-depending-on-length}}(L,l)$-hyperbolic. Let $\dl_1=max\{\dl_{\ref{special-int-sp-is-hyp}}(L),\dl_{\ref{hyp-of-depending-on-length}}(L,l)\}$. Suppose $\{u\}=J_1\cap J_2$. Then by Lemma \ref{intersection-of-generalized-flow-sps}, $ \M_L(X_u)= \M_L(X_{J_1})\cap  \M_L(X_{J_2})$. Again by Lemma \ref{gen-abs-flow-qi-emb}, $ \M_L(X_u)$ is $L_{\ref{gen-abs-flow-qi-emb}}(L)$-qi embedded in $X$ and so is in both $ \M_L(X_{J_1})$ and $ \M_L(X_{J_2})$ in their respective path metric.
Therefore, by Remark \ref{combi-hyp-sps-2}, $ \M_L(X_{J_1})\cup  \M_L(X_{J_2})$ is $\dl_{\ref{combi-hyp-sps-2}}(\dl_1,L_{\ref{gen-abs-flow-qi-emb}}(L))$-hyperbolic.
	
Applying the similar argument as above on the remaining pairs mentioned above, we conclude that $ \M_L(X_{J_1\cup J_2\cup J_3})\cup  \M_L(X_{J_4})= \M_L(X_J)$ is uniformly hyperbolic metric space with the induced path metric from $X$. Let the uniform constant be $\dl_{\ref{union-of-three-special-int-is-hyp}}=\dl_{\ref{union-of-three-special-int-is-hyp}}(L,l)$.
\end{proof}	

\begin{lemma}\label{cobounded-flow-sps}
Given $\dl\ge0$ and a proper function $g:\R_{\ge0}\ri\R_{\ge0}$, there is a constant $D_{\ref{cobounded-flow-sps}}=D_{\ref{cobounded-flow-sps}}(\dl,L,g)\ge0$ such that the following holds.
	
Let $Y\sse X$ be a $\dl$-hyperbolic subspace of $X$ such that $\M_L(X_u)\cup \M_L(X_v)\sse Y$ and $\pi(\M(X_u))\cap\pi(\M(X_v))=\emptyset$. Suppose that $Y$ is $g$-properly embedded in $X$. Then the pair $(\M_L(X_u),\M_L(X_v))$ is $D_{\ref{cobounded-flow-sps}}$-cobounded in $Y$.
\end{lemma}

\begin{proof}
We first note that by Lemma \ref{gen-abs-flow-qi-emb}, $\M_L(X_u)$ is $L_{\ref{gen-abs-flow-qi-emb}}(L)$-qi embedded in $X$ and so is in $Y$. Then by Lemma \ref{quasi-goes-to-quasi} $(1)$, $\M_L(X_u)$ is $K_1$-quasiconvex in $Y$, where $K_1=K_{\ref{quasi-goes-to-quasi}}(\dl_1,L_{\ref{gen-abs-flow-qi-emb}}(L),0)$. Hence $\M(X_u)$ is $K_2$-quasiconvex in $Y$, where $K_2=K_1+L$. Similarly, $\M(X_v)$ is $K_2$-quasiconvex in $Y$.
	
Let $p:Y\ri \M(X_u)$ be a nearest point projection map on $\M(X_u)$ in the metric of $Y$. Let $x,y\in \M(X_v)$ such that $p(x)=x_1,~p(y)=y_1$. By Lemma \ref{fat also cobdd} and the symmetry of the proof, it is enough to show that $d_Y(x_1,y_1)$ is uniformly bounded. 
	
Now by \cite[Lemma $1.31$ $(2)$]{pranab-mahan}, the arc-length parametrizations of $[x,x_1]_Y\cup[x_1,y_1]_Y$ and $[y,y_1]_Y\cup[y_1,x_1]_Y$ are $(3+2K_2)$-quasigeodesic in $Y$. If $d_Y(x_1,y_1)\le L_{\ref{lo_vs_gl}}(\dl,3+2K_2,3+2K_2)$, then we are done. Suppose $d_Y(x_1,y_1)>L_{\ref{lo_vs_gl}}(\dl,3+2K_2,3+2K_2)$.  So by Lemma \ref{lo_vs_gl}, $[x,x_1]_Y\cup[x_1,y_1]_Y\cup[y_1,y]_Y$ is $\lambda_1$-quasigeodesic in $Y$, where  $\lambda_1=\lambda_{\ref{lo_vs_gl}}(\dl,3+2K_2,3+2K_2)$. Therefore, by the stability of quasigeodesic (see Lemma \ref{ml}) and $K_2$-quasiconvexity of $\M(X_v)$ in $Y$, there exist $x_2,y_2\in \M(X_v)$ such that $d_Y(x_1,x_2)\le D_{\ref{ml}}(\dl,\lambda_1,\lambda_1)+K_2=L_1$ (say) and $d_Y(y_1,y_2)\le L_1$. Since $\pi(\M(X_u))\cap\pi(\M(X_v))=\emptyset$, by $(\mathcal{P}1)$, diam$\{\rho_u(\M(X_v))\}\le C$ in $X$. Therefore, $x_2,y_2\in \M(X_v)$ implies $d_X(\rho_u(x_2),\rho_u(y_2))\le C$. Again $\rho_u$ is $L\pr$-coarsely Lipschitz retraction of $X$ on $\M(X_u)$ (see $(\mathcal{P}1)$). Since $x_1,y_1\in \M(X_u)$, so $\rho_u(x_1)=x_1,~\rho_u(y_1)=y_1$. Then $d_X(x_1,\rho_u(x_2))=d_X(\rho_u(x_1),\rho_u(x_2))\le L\pr L_1+L\pr=L_2$ (say). Similarly, $d_X(y_1,\rho_u(y_2))\le L_2$. Therefore, $d_X(x_1,y_1)\le d_X(x_1,\rho(x_2))+d_X(\rho_u(x_2),\rho_u(y_2))+d_X(\rho_u(y_2),y_1)\le2L_2+C$ $\Rightarrow d_Y(x_1,y_1)\le g(2L_2+C)$ since $Y$ is $g$-properly embedded in $X$.
	
Hence diam$\{p(\M(X_v))\}\le L_3$ in $Y$, where $L_3=\textrm{max}\{g(2L_2+C),~L_{\ref{lo_vs_gl}}(\dl,3+2K_2,3+2K_2)\}$. Therefore, (by the symmetry of the proof) the pair $(\M(X_u),\M(X_v))$ is $L_3$-cobounded in $Y$. Then by Lemma \ref{fat also cobdd}, the pair $(\M_L(X_u),\M_L(X_v))$ is $D_{\ref{cobounded-flow-sps}}$-cobounded in $Y$, where $D_{\ref{cobounded-flow-sps}}=D_{\ref{fat also cobdd}}(\dl_1,K_2,L_3,L)$.
\end{proof}

Now we are ready to proof the main result.

\begin{prop}\label{general-int-sp-is-hyp}
We have a constant $\dl_{\ref{general-int-sp-is-hyp}}=\dl_{\ref{general-int-sp-is-hyp}}(L)$ such that for any interval $I$ in $T$ the space $ \M_L(X_I)$ is $\dl_{\ref{general-int-sp-is-hyp}}$-hyperbolic with the induced path metric from $X$.
\end{prop}

\begin{proof}
Let $I=J_0\cup J_1\cup\cdots\cup J_{n-1}$ be a subdivision of the interval $I$ coming from horizontal subdivision, i.e., Proposition \ref{horizontal-subdivision}, with $n_0=[L]+2$, where $[L]$ is the greatest integer not greater than $L$. We refer to Proposition \ref{horizontal-subdivision} for the description of $J_i=[w_i,w_{i+1}]$. Then by Lemma \ref{union-of-three-special-int-is-hyp}, $ \M_L(X_{J_i})$ is $\dl_{\ref{union-of-three-special-int-is-hyp}}(L,2n_0)$-hyperbolic metric space for all $i\in\{0,1,\cdots,n-1\}$.
	
Now we will verify all the conditions of Proposition \ref{combi-hyp-sps}. Let $X_i= \M_L(X_{J_i})$ for $0\le i\le n-1$ and $Y_{i+1}=X_i\cap X_{i+1}= \M_L(X_{w_{i+1}})$ for $0\le i\le n-2$ (see Lemma \ref{intersection-of-generalized-flow-sps}).
	
$(1)$ For $0\le i\le n-1$, $X_i$ is $\dl_1$-hyperbolic metric space, where $\dl_1=\dl_{\ref{union-of-three-special-int-is-hyp}}(L,2n_0)$.
	
$(2)$ By Lemma \ref{gen-abs-flow-qi-emb}, $Y_{i+1}=\M_L(X_{w_{i+1}})$ is $L_{\ref{gen-abs-flow-qi-emb}}(L)$-qi embedded in $X$ so is in both $X_i$ and $X_{i+1}$ for $0\le i\le n-2$.
	
$(3)$ Let $i\in\{0,1,\cdots,n-2\}$. Note that by $(\mathcal{P}0)$, if $x\in X_i\setminus Y_{i+1}$ and $y\in X_{i+1}\setminus Y_{i+1}$ then $w_{i+1}\in[\pi(x),\pi(y)]\setminus\{\pi(x),\pi(y)\}$. Hence every path in $\M_L(X_I)$ joining points $X_i$ and $X_{i+1}$ passes through $Y_{i+1}$.
	
$(4)$ Suppose $i\in\{1,2,\cdots,n-2\}$. Note that $\pi( \M(X_{w_{i}}))\cap\pi( \M(X_{w_{i+1}}))=\emptyset$ (by Proposition \ref{horizontal-subdivision} $(3)$). Again $X_i$ is $\eta_{\ref{gen-abs-flow-pro-emb}}(L)$-properly embedded in $X$ and so is in $\M_L(X_i)$. Also, $ \M_L(X_{w_{i}})\cup \M_L(X_{w_{i+1}})\sse X_i$. Then by Lemma \ref{cobounded-flow-sps}, there is $D$ depending on $\dl_1$, $L$ and $\eta_{\ref{gen-abs-flow-pro-emb}}(L)$ such that the pair $(Y_i,Y_{i+1})$ is $D$-cobounded in $X_i$.
	
$(5)$ Suppose $1\le i\le n-2$. On contrary, let $d_{X_i}(Y_i,Y_{i+1})<1$. Then $d_X(\M(X_{w_i}),\M(X_{w_{i+1}}))\le2L+1$, and so $d_T(\pi(\M(X_{w_i})),\pi(\M(X_{w_{i+1}})))\le2L+1\le2n_0$ (by our choice of $n_0$). This contradicts to $(3)$ of Proposition \ref{horizontal-subdivision}.

	%
	%
	
Hence by Proposition \ref{combi-hyp-sps}, $ \M_L(X_I)$ is $\dl_{\ref{general-int-sp-is-hyp}}$-hyperbolic, where $\dl_{\ref{general-int-sp-is-hyp}}=\dl_{\ref{combi-hyp-sps}}(\dl_1,L_{\ref{gen-abs-flow-qi-emb}}(L),D)$.
\end{proof}

As a consequence of Proposition \ref{general-int-sp-is-hyp} along with Proposition \ref{combi-hyp-sps}, we obtain following. We omit the proof. 
\begin{lemma}\label{tripod-sp-is-hyp}
We have a constant $\dl_{\ref{tripod-sp-is-hyp}}=\dl_{\ref{tripod-sp-is-hyp}}(L)$ satisfying the following. Let $u,v,w\in T$ and $T_{uvw}$ be the tripod in $T$ with vertices $u,v,w$. Then $ \M_L(X_{T_{uvw}})$ is $\dl_{\ref{tripod-sp-is-hyp}}$-hyperbolic metric space with the induced path metric from $X$.
\end{lemma}

\subsection{Proof of Theorem \ref{treeofsps-com-thm}} We show that $X$ satisfies all the conditions of Proposition \ref{combing}. Let $X_{vsp}=\cup_{u\in T}X_u$. Note that $X_{vsp}$ is a $1$-dense subset of $X$. So given any two points $x,y\in X_{vsp}$, we define path joining them as follows.

Let $x\in X_u$ and $y\in X_v$ for some $u,v\in T$. Note that $X_{[u,v]}=\pi^{-1}([u,v])$. We fix once and for all, a geodesic path $c(x,y)$ in $\M_L(X_{[u,v]})$ joining $x$ and $y$. These paths serve as family of paths for Proposition \ref{combing}.

Let $x,y,z\in X_{vsp}$ such that $\pi(x)=u,~\pi(y)=v$ and $\pi(z)=w$.

{\bf Condition (1)}: For all $s,t\in\{u,v,w\}$, by Proposition \ref{gen-abs-flow-pro-emb}, $\M_L(X_{[s,t]})$ is $\eta_{\ref{gen-abs-flow-pro-emb}}(L)$-properly embedded in $X$ and so are the paths.

{\bf Condition (2)}: Let $T_{uvw}$ be the tripod in $T$ with vertices $\{u,v,w\}$. By Lemma \ref{tripod-sp-is-hyp}, $\M_L(X_{T_{uvw}})$ is $\dl_{\ref{tripod-sp-is-hyp}}(L)$-hyperbolic metric space. Again by Lemma \ref{gen-abs-flow-qi-emb}, $\M_L(X_{[u,v]}),~\M_L(X_{[v,w]})$ and $\M_L(X_{[u,w]})$ are $L_{\ref{gen-abs-flow-qi-emb}}(L)$-qi embedded subspaces of $X$ and so are of $\M_L(X_{T_{uvw}})$. Then the hyperbolicity of $\M_L(X_{T_{uvw}})$ and the stability of quasigeodesic (see Lemma \ref{ml}) in $\M_L(X_{T_{uvw}})$ imply that the triangle formed by paths $c(x,y),~c(y,z)$ and $c(x,z)$ is $D$-slim in $\M_L(X_{T_{uvw}})$ and so is in $X$, where $$D=2D_{\ref{ml}}(\dl_{\ref{tripod-sp-is-hyp}}(L),L_{\ref{gen-abs-flow-qi-emb}}(L),L_{\ref{gen-abs-flow-qi-emb}}(L))+\dl_{\ref{tripod-sp-is-hyp}}(L).$$

Hence by Proposition \ref{combing}, $X$ is $\dl_{\ref{combing}}(\eta_{\ref{gen-abs-flow-pro-emb}}(L),D,1)$-hyperbolic. This completes the proof.\qed



\section{Trees of metric bundles and their properties}\label{trees-of-metric-bundles}
The notion of metric bundles (see Definition \ref{metric-bundle-defn}) were introduced by Mj and Sardar in \cite{pranab-mahan}. 
Subsuming both metric bundles and trees of metric spaces, we define trees of metric bundles in Definition \ref{treesofmetric-bun}.

\begin{defn}\textup{(\hspace{-.08mm}\cite[Definition $1.2$]{pranab-mahan})}\label{metric-bundle-defn}
Suppose $(X,d)$ and $(B,d_B)$ are geodesic metric spaces. Let $c_0\ge1$, and let $\phi:\R_{\ge0}\ri\R_{\ge0}$ be a proper map. We say that $X$ is a $(\phi,c_0)$-{\bf metric bundle} over $B$ if there is a $1$-Lipschitz and surjective map $p:X\ri B$ such that the following hold.

\begin{enumerate}
\item Let $z\in B$. Then $F_z:=p^{-1}(z)$, called fiber, is a geodesic metric space with the induced path metric from $X$ and the inclusion $F_z\ri X$ is a $\phi$-proper embedding.
	
\item Let $z_1,z_2\in B$ such that $d_B(z_1,z_2)\le 1$. Let $\al$ be a geodesic joining $z_1$ and $z_2$. Then for all $z\in\al$ and $x\in F_z$, there are paths in $p^{-1}(\al)$ of length at most $c_0$ joining $x$ to points in $F_{z_1}$ and $F_{z_2}$.
\end{enumerate}
\end{defn}

\begin{defn}[\bf Trees of metric bundles]\label{treesofmetric-bun}
Let $(X,d)$ be a geodesic metric space. Suppose $\pi_B:(B,d_B)\ri T$ is a tree of metric spaces over a tree $T$ such that the edge spaces are points. Let $c_0\ge1$, and let $\phi:\R_{\ge0}\ri\R_{\ge0}$ be a proper map. A tree of metric bundles is a $1$-Lipschitz surjective map $\pi_X:X\ri B$ such that the following hold (see Figure \ref{grps}).
\begin{enumerate}
	
	\item For all $u\in V(T)$ let $B_u:=\pi_B^{-1}(u)$ and $X_u:=\pi_X^{-1}(B_u)$. Then $X_u$ is geodesic metric space with the induced path metric and the restriction of $\pi_X$ to $X_u$ gives a $(\phi,c_0)$-metric bundle $X_u\ri B_u$ (see Definition \ref{metric-bundle-defn}).

	
	\item Let $e=[v,w]$ be an edge in $T$, and $\mfe=[\mfv,\mfw]$ be the lift of $e$ joining $\mfv\in B_v$ and $\mfw\in B_w$. Then $\pi_X$ restricted to $\pi_X^{-1}(\mfe)$ is a tree metric spaces over $\mfe$ with parameter $\phi$ (see Definition \ref{tree-of-sps})
	
	\item For $u\in V(T)$ and $a\in B_u$, we denote the fiber corresponding to $a$ by $F_{a,u}(:=\pi_X^{-1}(a))$. Then the inclusion $F_{a,u}\ri X$ is a $\phi$-proper embedding.
\end{enumerate}
Most of the time, we say that $(X,B,T)$ is a tree of metric bundles keeping the structural maps $\pi_X$, $\pi_B$ and the parameters $\phi$, $c_0$, and other structures implicit. We denote the composition of $\pi_X:X\ri B$ and $\pi_B:B\ri T$ by $\pi:X\ri T$. 
\end{defn}

\begin{figure}[h]
\includegraphics[width=10cm]{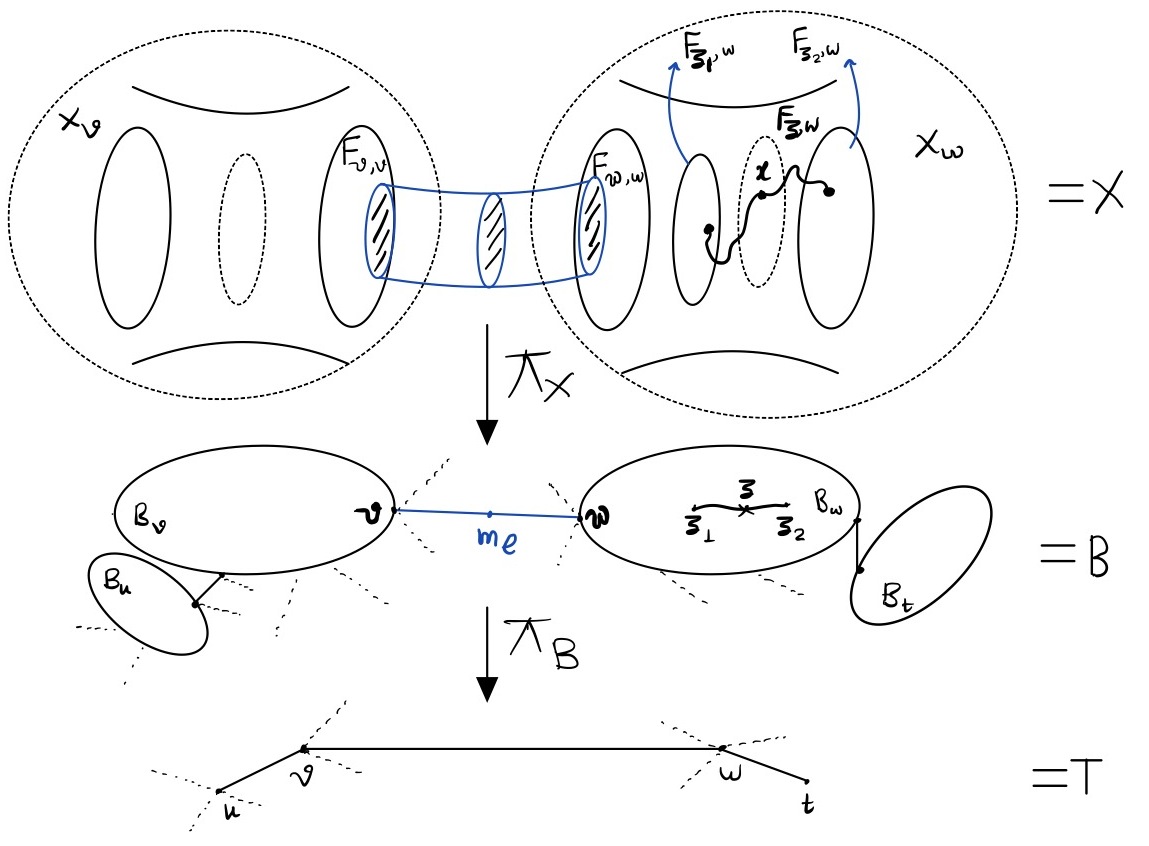}
\centering
\caption{Trees of Metric Bundles}
\label{grps}
\end{figure}

\underline{{\bf Disclaimer}}: {\em The term `trees of metric bundles' may be misleading for the map $\pi_X:X\ri B$ since $B$ is not a tree in general. However, it is not misleading for $\pi=\pi_B\circ\pi_X:X\ri T$. To maintain consistency with existing literature, we will adhere to our chosen nomenclature.}

In this section, we will see some properties of a tree of metric bundles $(X,B,T)$ that are used in our proofs. In our statements, we make the structural parameters $\phi,c_0$ implicit.

For $u\in V(T)$, since fibers are $\phi$-properly embedded in $X$, we can show (along the same line of arguments given in the proof of \cite[Proposition $2.17$]{ps-kap}; see Lemma \ref{pro-emb}) that $X_u:=\pi^{-1}(u)$ is uniformly properly embedded in $X$. Then considering $(X,B,T)$ as a tree of metric spaces $\pi:=\pi_B\circ\pi_X:X\ri T$, we get the following as corollary of Lemma \ref{pro-emb}.

\begin{prop}\label{subsp-is-proper-emb}
	Suppose $(X,B,T)$ is a tree of metric bundles. Then there exists a proper function $\eta_{\ref{subsp-is-proper-emb}}:\R_{\ge0}\ri\R_{\ge0}$ depending only on the structural parameters of $(X,B,T)$ such that for any subtree $S$ of $T$, the inclusion $X_S:=\pi^{-1}(S)\ri X$ is a $\eta_{\ref{subsp-is-proper-emb}}$-proper embedding.
\end{prop}

\subsection{An example of tree of metric bundles}

Now we will see a motivational example of a tree of metric bundles in Example \ref{trees of metric bundles exps}. Construction of our example is based on the free product with amalgamation of groups. We start by recalling some results from \cite[Section $5.2$]{NirMaMj-commen}.

Recall that a subgroup $N$ of a group $G$ is said to be {\bf {\em commensurated}} if for all $g\in G$, the subgroup $N\cap gNg^{-1}$ has finite index in both $N$ and $gNg^{-1}$.

\begin{defn}\label{cayley-abels graph}
Let $G$ be a finitely generated group with finite generating set $S$, and let $N<G$ be a commensurated subgroup. The Cayley--Abels graph $\G(G,N,S)$ of the pair $(G,N)$ with respect to $S$ is a simplicial $($connected$)$ graph whose vertices are cosets of $N$ in $G$ and edges are $\{(gN,gsN):g\in G,~s\in S\}$.
\end{defn}

Note that if $N$ is normal in $G$ in above definition, then $\G(G,N,S)$ is precisely the Cayley graph of $G/N$. Again, for any group $G$ with a generating set $S$, the Cayley--Abels graph $\G(G,1,S)$ for the pair $(G,1)$ is precisely a Cayley graph of $G$ with respect to the generating set $S$.

Suppose $G,~N$ and $S$ are as in Definition \ref{cayley-abels graph}. Then there is a natural simplicial {\em quotient map} $p:\G(G,1,S)\map\G(G,N,S)$ that sends a vertex $g$ to a vertex $gN$ and an edge $(g,gs)$ to an edge $(gH,gsH)$ if $gH\ne gsH$ or $gH$ if $gH=gsH$.

We do not use {\em metric graph bundle} in this paper except in the following proposition. Interested reader may have a look at \cite[Definition 1.5]{pranab-mahan}. Metric graph bundle is a graph version of metric bundle where spaces are metric graphs and fibers are defined only over vertex of the base space.
\begin{prop}\textup{(\cite[Proposition 5.12]{NirMaMj-commen})}\label{metric bundle}
Suppose $G$ is a finitely generated group and $N$ is a finitely generated commensurated subgroup. Given finite generating sets $S_N$ and $S_G$ respectively for $N$ and $G$ with $S_N\sse S_G$, we can find a finite generating set $S$ of $G$ such that $S_G\sse S$ and the map $p:\G(G,1,S)\map\G(G,N,S)$ is a metric graph bundle with fibers are isometric to $\G(H,1,H\cap S)$.
\end{prop}

Suppose $G=\langle S_G|R_G \rangle$ is a group and $H=\langle S_H|R_H \rangle$ is a subgroup of $G$. We say that these presentations is a {\bf {\em compatible presentation}} for the pair $(H,G)$ if $S_H\sse S_G$ and $R_H\sse R_G$.

Let $\bigvee_{\al}S^1_{\al}$ denote the wedge of circles. Given a group $G=\langle S_G|R_G \rangle$, a CW complex $\mathcal X_G$ is obtained from $\bigvee_{\al\in S_G} S^1_{\al}$ by attaching $2$-cells by reading off relators. The CW complex $\mathcal X_G$ is called the {\bf {\em presentation complex}} of $G$ for the given presentation $\langle S_G|R_G \rangle$.

Note that given groups $H<G$ and a compatible presentation of the pair $(H,G)$, there is a natural inclusion $\mathcal X_H\map \mathcal X_G$, and we denote this inclusion by $i_{H,G}$.

Now we will see the construction of trees of metric bundles from amalgamations of groups.

\begin{example}\label{trees of metric bundles exps}
Suppose $G_i$ is a finitely generated group and $N_i$ is a finitely generated commensurated subgroup of $G_i$, $i=1,2$. Let $H$ be a common finitely generated subgroup of $N_1$ and $N_2$ giving the free product with amalgamation $G=G_1*_HG_2$. Now we will show that the tree of metric spaces corresponding to $G$ admits a tree of metric bundles structure as follows. The construction is two folded. (Note that the more general construction of trees of metric bundles that come from complexes of groups is given in Subsection \ref{graphs-tree-gra-grphs-comp}.)

{\bf Tree of metric spaces corresponding to $G$:} Suppose $S_H$, $S_{N_i}$ and $S_{G_i}$ are finite generating sets of compatible presentations of the pairs $(H,N_i)$ and $(N_i,G_i)$ where $i=1,2$. Let $\mathcal X_H$, $\mathcal X_{N_i}$ and $\mathcal X_{G_i}$ be the corresponding presentation complexes where $i=1,2$. Suppose $$\mathcal X=\mathcal X_{G_1}\sqcup\mathcal X_H\times[0,1]\sqcup\mathcal X_{G_2}/\sim$$ where $(x,0)\sim i_{H,G_1}(x)$ and $(x,1)\sim i_{H,G_2}(x)$ for all $x\in\mathcal X_H$ $($see \textup{\cite[Chapter $1.B$]{hatcher-book},\cite{scott-wall})}. Note that $\mathcal X$ is a finite CW complex and by the Seifert--Van Kampen theorem, the fundamental group of $\mathcal X$ is isomorphic to $G$. Let $X$ is the universal cover of $\mathcal X$ with the standard CW complex structure coming from $\mathcal X$. Note that by the theorem of Bridson \textup{(\cite{bridson-com})}, $X$ is a complete geodesic metric space. Suppose $T$ is the Bass--Serre tree of the amalgamation $G$ \textup{(\cite{serre-trees})} giving the tree of metric spaces $\pi:X\map T$ \textup{(\cite{BF,mitra-trees})}.

{\bf Blowing up of the Bass-Serre tree:} Now we will blow up $T$ to the base $B$ of the tree of metric bundles $\pi_X:X\map B$ we are looking for. Let $X_{N_i}$ and $X_{G_i}$ are the universal covers of $\mathcal X_{N_i}$ and $\mathcal X_{G_i}$ respectively with the standard CW complex structures coming from $\mathcal X_{N_i}$ and $\mathcal X_{G_i}$ where $i=1,2$. By Proposition \ref{metric bundle}, after extending the generating set, if necessary, we assume that $p_i:X_{G_i}\map\mathcal G(G_i,N_i,S_{G_i})$ is a metric bundle with fibers are isometric to $X_{N_i}$. Let $\mathcal G_i=\mathcal G(G_i,N_i,S_{G_i})$, $i=1,2$ and $$B'=\{g_1\mathcal G_1:g_1\in G\}\sqcup\{g_2\mathcal G_2:g_2\in G\}/\sim$$where $g\mathcal G_i\sim h\mathcal G_i$, $g,h\in G$ if and only if $g^{-1}h\in G_i$, $i\in\{1,2\}$. We obtain a connected graph $B$ from $B'$ by attaching an edge joining $gN_1\in V(\G_1)$ and $gN_2\in V(\G_2)$ corresponding to each edge $gH$ joining $gG_1\in V(T)$ and $gG_2\in V(T)$ where $g\in G$.

So we have a natural map $\pi_B:B\map T$ by collapsing the space $g\mathcal G_i$ to the vertex $gG_i\in T$ and sending an edge joining $gN_1$ and $gN_2$ isometrically onto the edge joining $gG_1$ and $gG_2$.

Note that $\pi^{-1}(gG_i)$ is isometric to $X_{G_i}$ where $g\in G$, $gG_i\in V(T)$ and $i=1,2$. Finally, we also have a map $\pi_X:X\map B$ induced by restricting it to $p_i$ on $\pi^{-1}(gG_i)$, $i=1,2$. It follows that $\pi_X:X\map B$ satisfies conditions $(1)$ and $(2)$ of Definition \ref{treesofmetric-bun}. Since $p_i:X_{G_i}\map\mathcal G_i$ is a metric bundle where $i=1,2$ and $\pi:X\map T$ is a tree of metric spaces, it follows that fibers for $\pi_X:X\map B$ are uniformly properly embedded whence condition $(3)$ of Definition \ref{treesofmetric-bun} is satisfied.
\end{example}

\subsection{On the Bestvina--Feighn hallway flaring}
In a hyperbolic metric space, (quasi)geodesics diverge exponentially. In a tree of metric bundles $X$, qi lifts (see Definition \ref{qi-lift}) are quasigeodesics. So, they diverge exponentially provided $X$ is hyperbolic. This property is captured in Definition \ref{flaring-defn} for special types of quasigeodesics, namely, qi lifts. (See also \emph{necessity of flaring} in Introduction \ref{introduction} to get more on this.) This definition is a generalization of the Bestvina--Feighn hallway flaring condition (\hspace{-.08mm}\cite{BF}) in a natural way as defined in \cite[Definition $1.12$]{pranab-mahan} for metric bundles. To define that we need the following definitions. 

\begin{defn}[Quasiisometric (qi) section]\label{K-qi-section}
	Let $K\ge1$. Suppose $(X,B,T)$ is a tree of metric bundles. Let $B_1$ be an isometrically embedded subspace in $B$ and $X_1\sse X$. We say $X_1$ is $K$-qi section in $X$ over $B_1$ if there is a $K$-qi embedding $s:B_1\ri X$ such that $\pi_X \circ s=id$ on $B_1$ and $X_1=Im(s)$. Further, we say that it is compatible if the following hold.
	\begin{enumerate}
		
		\item For all vertex $w\in\pi_B(B_1)$, $X_1\cap X_w$ is a $K$-qi section over $B_1\cap B_w$ in the path metric of $X_w$ and $X_1\cap X=Im(s|_{B_1\cap B_w})$.
		
		\item Suppose $[v,w]\sse \pi_B(B_1)$ is an edge, and $[\mfv,\mfw]$ is the edge joining $\mfv\in B_v$ and $\mfw\in B_w$. Then $s(\mfv)$ and $s(\mfw)$ are $K$-apart in the path metric on $\pi_X^{-1}([\mfv,\mfw])$ induced from $X$.
	\end{enumerate}
\end{defn}

Here we mention that existence of uniform qi section in metric bundle was one of the difficult jobs in \cite{pranab-mahan}. We are going to use it frequently in our paper (see Lemma \ref{qi-sec-inside-lad-len-lift} $(1)$). For a short exact sequence, existence of such qi section was proved earlier by Mosher \cite{mosher-hypextns}.

\begin{defn}\label{qi-lift}
	If $B_1$, in Definition \ref{K-qi-section}, is a geodesic segment, say $\alpha:[0,r]\sse\R\ri B$, then we call the section a $K$-{\bf qi lift} of the geodesic $\alpha$.
	
	According to our definition, a $K$-qi lift of a geodesic $\al:[0,r]\sse\R\ri B$ is $Im(\tilde{\al})$ where $\tilde{\al}:Im(\al)\sse B\ri X$ is a $K$-qi embedding. We will simultaneously use $Im(\al)$ and $[0,r]$ as the domain of $\tilde{\al}$.
\end{defn}

\begin{defn}\label{flaring-defn}
Suppose $k\ge1$. A tree of metric bundles $(X,B,T)$ is said to satisfy $k$-{\bf flaring condition} $($see Figure \ref{picture4}$)$ if there are constants $M_k>0$, $n_k\in\N$ and $\lm_k>1$ depending on $k$ such that the following holds.

For every pair $(\gm_0, \gm_1)$ of $k$-qi lifts of a geodesic $\gm:[-n_k,n_k]\ri[a,b]\sse B$ joining $a,b\in B$ with $d^f(\gm_0(0),\gm_1(0))>M_k$, we have $$\lm_kd^f(\gm_0(0),\gm_1(0))<max\{d^f(\gm_0(-n_k),\gm_1(-n_k)),d^f(\gm_0(n_k),\gm_1(n_k))\}$$
where $d^f$ denotes the fiber distance in the corresponding fiber. Occasionally, we say that $(X,B,T)$ satisfies $k$-flaring condition suppressing the constants $M_k,n_k,\lm_k$. We say that $(X,B,T)$ satisfies a flaring condition if it satisfies $k$-flaring condition for all $k\ge1$.
\end{defn}
\begin{figure}[h]
\includegraphics[width=8cm]{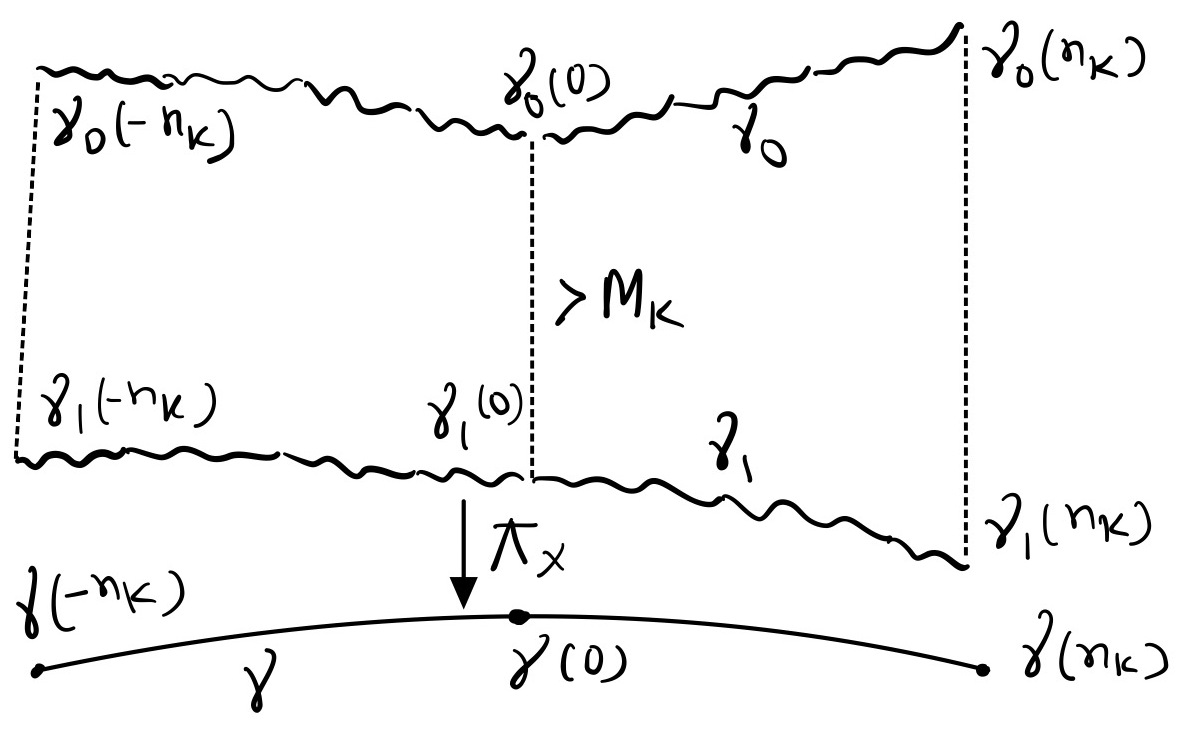}
\centering
\caption{Flaring Condition}
\label{picture4}
\end{figure}

Now we will state Lemma \ref{flaring-lemma} without a proof since these results correspond to \cite[Lemma $2.17$, Lemma $2.18$]{pranab-mahan} in metric bundles situation and one has a similar proofs for these in trees of metric bundles. The result $(1)$ of Lemma \ref{flaring-lemma} is defined as \emph{uniform flaring condition} in the book \cite{ps-kap}, and $(2)$ says that the neck are quasiconvex subset of the base in the sense of \cite{pranab-mahan}. Throughout the paper, we directly will not use flaring condition above in proving results instead Lemma \ref{flaring-lemma}.

\begin{lemma}\label{flaring-lemma}
Let $k\ge1$. Suppose $(X,B,T)$ is a tree of metric bundles satisfying $k$-flaring condition with constants $M_k,n_k,\lm_k$. Then for all $D\ge0$ there are constants $\tau_{\ref{flaring-lemma}}=\tau_{\ref{flaring-lemma}}(k,D)$ and $R_{\ref{flaring-lemma}}=R_{\ref{flaring-lemma}}(k,D)$ satisfying the following.

Let $\gm_0,\gm_1$ be two $k$-qi lifts of a geodesic $\gm:[0,r]\ri[a,b]\sse B$ joining $a,b\in B$. Let $0=r_0<r_1<\cdots<r_n=r$ such that $r_{i+1}-r_i=1$ for $0\le i\le n-2$ and $r_n-r_{n-1}\le 1$. Then:

\begin{enumerate}
\item $d^f(\gm_0(r_i),\gm_1(r_i))>M_k$ for all $1\le i\le n-1$ and $$max\{d^f(\gm_0(0),\gm_1(0)),d^f(\gm_0(r),\gm_1(r))\}\le D$$ together imply $r=d_B(a,b)\le\tau_{\ref{flaring-lemma}}$.
	
\item $max\{d^f(\gm_0(0),\gm_1(0)),d^f(\gm_0(r),\gm_1(r))\}\le D$ implies for all $0\le i\le n$, $$d^f(\gm_0(r_i),\gm_1(r_i))\le R_{\ref{flaring-lemma}}.$$
\end{enumerate}
\end{lemma}

We make the following remark for later use, and it follows from the Definition \ref{flaring-defn}.  
\begin{remark}\label{bigsmall-subsp}
Let $k\ge1$. Suppose $(X,B,T)$ is a tree of metric bundles satisfying $k$-flaring condition with constants $M_k,n_k,\lm_k$. Let $S$ be a subtree of $T$, and let $X_S=\pi_X^{-1}(S)$ and $B_S=\pi_B^{-1}(S)$. Note that the restriction $\pi_X|_{X_S}:X_S\ri B_S$ is a tree of metric bundles by Proposition \ref{subsp-is-proper-emb}. Then $(a)$ $(X,B,T)$ satisfies $k\pr$-flaring condition for all $k\pr\le k$ with the same constants, and $(b)$ the restriction $\pi_X|_{X_S}:X_S\ri B_S$ also satisfies $k$-flaring condition with the same constants. 
\end{remark}

\subsection{Some properties}

Before discussing properties, we first make the following definition, and then fix some conventions and notations which are used throughout the paper.

Motivated by the main theorems of \cite{BF} and \cite{pranab-mahan}, we define the following.
\begin{defn}\label{axiomH}
	We say that a tree of metric bundles $(X,B,T)$ satisfies the hyperbolic axioms, in short axiom {\bf H}, with parameters $\dl_0\ge0,N\ge0$ and $L_0\ge1$ if the following hold.
	\begin{enumerate}
		\item Let $u\in V(T)$ and $a\in B_u$. Then $F_{a,u}$ is $\dl_0$-hyperbolic and the barycenter map $\partial^3F_{a,u}\ri F_{a,u}$ is $N$-coarsely surjective.
		
		\item  Let $e=[v,w]$ be an edge in $T$. Let $\mfe=[\mfv,\mfw]$ be the edge joining $\mfv\in B_v$ and $\mfw\in B_w$, and $m_{\mfe}$ be the mid point of $\mfe$. Then the incident maps $\vartheta_{\mfe\mfv}:\pi_X^{-1}(m_{\mfe})\ri F_{\mfv, v}$ and $\vartheta_{\mfe\mfw}:\pi_X^{-1}(m_{\mfe})\ri F_{\mfw, w}$ are $L_0$-qi embeddings $($see Definition \ref{tree-of-sps} $2~(c))$.
		
		\item Lastly, let $B$ be $\dl_0$-hyperbolic. This assumption is the same as $B_v$ is $\dl_0$-hyperbolic for all $v\in V(T)$. 
	\end{enumerate}	
\end{defn}

\underline{{\bf Convention and Notations}}
\begin{convention}\label{convention}
	\begin{enumerate}
		\item {\em Unless otherwise specified, our tree of metric bundles $(X,B,T)$ always satisfies the axiom {\bf H} with constants $\dl_0,N,L_0$ as in Definition \ref{axiomH}. We will use $\dl\pr_0$, $L\pr_0$, $\lm\pr_0$ and $L\pr_1$ for the constants depending on $\dl_0$ and $L_0$ as defined in Lemma \ref{com-two-hyp-sps}.}
		
		\item {\em Let $B\pr$ be a path connected subspace of $B$, and let $S$ be a subtree of $T$. Unless otherwise specified, by $b\in B\pr$ and $u\in S$, we always mean $b\in B\pr\cap B_{\pi_B(b)}$ and $u\in V(S)$ respectively.}
		
		\item {\em If $X_1$ is a $K$-qi section over $B_1$, then it is a compatible $\eta_{\ref{subsp-is-proper-emb}}(2K)$-qi section over $B_1$. Thus, now onward, by a $K$-qi section, we always mean a compatible $K$-qi section.}
	\end{enumerate}
\end{convention}

\begin{notation}\label{connecting-edge-spaces}
	{\em We use these notations for the rest of the paper. We denote the composition map $\bm{\pi_B\circ\pi_X}$ by $\bm{\pi}$. For a subtree $S\sse T$, $\bm{X_S:=\pi^{-1}(S)}$ and $\bm{B_S:=\pi_B^{-1}(S)}$. In particular, for $v\in V(T)$, $\bm{X_v:=\pi^{-1}(v)}$ and $\bm{B_v:=\pi_B^{-1}(v)}$. Let $v,w\in T$, $d_T(v,w)=1$, and let $[\mfv,\mfw]$ be the edge joining $\mfv\in B_v$ and $\mfw\in B_w$. We denote $\bm{F_{\mfv\mfw}:=\pi_X^{-1}([\mfv,\mfw])}$. The induced path metric on $F_{\mfv\mfw}$ is denoted by $\bm{d_{\mfv\mfw}}$, and we use $\bm{N^{\mfv\mfw}}$ to mean neighborhood of subsets of $F_{\mfv\mfw}$ in the $d_{\mfv\mfw}$-metric. For a fiber $F_{a,u}$, where $u\in V(T)$ and $a\in B_u$, we simply use $\bm{d^f}$, $\bm{diam^f}$, $\bm{N^f}$ and $\bm{[x,y]^f}$ $($or $\bm{[x,y]_{F_{a,u}}})$ respectively, to denote the induced path metric on $F_{a,u}$, the diameter of a subset of $F_{a,u}$ in $d^f$-metric, the neighborhood of a subset of $F_{a,u}$ in $d^f$-metric and a geodesic inside $F_{a,u}$ joining $x,y\in F_{a,u}$. Since it will be clear from the context which fiber we are working with, we are not being more specific on $d^f,~N^f$ etc. Lastly, $\bm{P_{\mfw}:=P_{F_{\mfv\mfw},F_{\mfw,w}}}$.}
\end{notation} 

Now we put metric bundle (see Definition \ref{metric-bundle-defn}) structure on subspace of $X$. 

\begin{defn}[K-metric bundle and special $K$-ladder]\label{K-metric-graph-bundle}
Suppose $X_1\sse X$ and $S$ is a subtree of $T$ such that the restriction map $\pi_X|_{X_1}:X_1\ri B_S$ is surjective. We say that $X_1$ forms a {\bf $K$-metric bundle} over $B_S$ if there is a $K$-qi section through each point of $X_1$ over $B_S$ such that the image lies inside $X_1$. Further, we say that $X_1$ forms a {\bf special $K$-ladder} if $X_1$ forms a $K$-metric bundle along with two $K$-qi sections, say $\Sigma_1$ and $\Sigma_2$, over $B_S$ such that $X_1=\bigcup\limits_{v\in S,~b\in B_v}[F_{b,v}\cap\Sigma_1,F_{b,v}\cap\Sigma_2]_{F_{b,v}}$. In this case, occasionally, we denote $X_1$ by $\L_K(\Sigma_1,\Sigma_2)$ or simply by $\L(\Sigma_1,\Sigma_2)$ when $K$ is understood.
\end{defn}

Lemma \ref{qi-sec-inside-lad-len-lift} (2) is proven for metric graph bundles (see \cite[Definition $1.5$]{pranab-mahan}) in \cite[Lemma 3.1]{pranab-mahan} and the proof works for metric bundles as well. So we omit its proof.

%

\begin{lemma}\textup{(\hspace{-.08mm}\cite[Proposition $2.10$, Proposition $2.12$, Lemma $3.1$, Lemma $3.3$]{pranab-mahan})}\label{qi-sec-inside-lad-len-lift}
Suppose $(X,B,T)$ is a tree of metric bundles satisfying axiom {\bf H}. (For $(3)$ below, axiom {\bf H} is not required.) Given $K\ge1$ and $R\ge2K$ there is a constant $C_{\ref{qi-sec-inside-lad-len-lift}}=C_{\ref{qi-sec-inside-lad-len-lift}}(K)>K$ such that the following hold.

$(1)$ There exists a constant $K_{\ref{qi-sec-inside-lad-len-lift}}\ge1$ depending only on $\dl_0,N$ (constants of axiom {\bf H}) such that through each point $x\in X_v$ there is a $K_{\ref{qi-sec-inside-lad-len-lift}}$-qi section over $B_v$ in the path metric of $X_v$, where $v\in V(T)$.

$(2)$ Suppose $v\in V(T)$, and $\Sigma_1$ and $\Sigma_2$ are two $K$-qi sections over $B_v$. Then $\L(\Sigma_1,\Sigma_2)$ is a special $C_{\ref{qi-sec-inside-lad-len-lift}}(K)$-ladder over $B_v$.

$(3)$ Suppose $\Sigma$ is a $K$-qi section over an isometrically embedded subspace $B_1\sse B$. Let $s:[a,b]\ri X$ be a qi lift  of a geodesic segment $[a,b]\sse B_1$ such that $Im(s)\sse \Sigma$. Then $N_R(\Sigma)$ is path connected and if the induced path metric on $N_R(\Sigma)$ is $d\pr$ then $d\pr(s(a),s(b))\le2K d_B(a,b)$. Moreover, $N_{2K}(\Sigma)$ is $K(2K+1)$-qi embedded in any geodesic subspace that contains $N_R(\Sigma)$, with respect to the induced path metric from $X$.
\end{lemma}

\begin{remark}
Note that the sole purpose of the condition that the barycenter map $\pa^3F_{a,u}\map F_{a,u}$ is (uniformly) coarsely surjective (in Definition \ref{axiomH} $(1)$) is to get a (uniform) qi section through each point of $X_u$ over $B_u$. This is assured by Lemma \ref{qi-sec-inside-lad-len-lift} $(1)$.

\noindent We will not use this condition explicitly anywhere instead Lemma \ref{qi-sec-inside-lad-len-lift} $(1)$.
\end{remark}

\begin{remark}\label{iteration}
In the view of Lemma \ref{qi-sec-inside-lad-len-lift} $~(2)$, for $i\in\N$, we denote the $i^{th}$ iteration by $C^{(i)}_{\ref{qi-sec-inside-lad-len-lift}}(K):=C_{\ref{qi-sec-inside-lad-len-lift}}(C^{(i-1)}_{\ref{qi-sec-inside-lad-len-lift}}(K)),$ where $C^{(0)}_{\ref{qi-sec-inside-lad-len-lift}}(K)=K$. In other words, if $\L(\Sigma_1,\Sigma_2)$ is bounded by two $C^{(i-1)}_{\ref{qi-sec-inside-lad-len-lift}}(K)$-qi sections, say $\Sigma_1$ and $\Sigma_2$, over $B_v$, then $\L(\Sigma_1,\Sigma_2)$ is a special $C^{(i)}_{\ref{qi-sec-inside-lad-len-lift}}(K)$-ladder over $B_v$.
\end{remark}
As an application of Lemma \ref{qi-sec-inside-lad-len-lift} $(2)$ along with the fact that quadrilaterals are slim in hyperbolic spaces, we have the following. We omit the proof. 
\begin{lemma}\label{getting-metric-graph-bundles}
Given $K\ge1$, there is a constant $K_{\ref{getting-metric-graph-bundles}}=K_{\ref{getting-metric-graph-bundles}}(K)$ such that the following holds. 

Suppose $(X,B,T)$ is a tree of metric bundles satisfying axiom {\bf H}. Let $S$ be a subtree of $T$. Let $\mathcal{G}_{K,S}=\{\gm:\gm \textrm{ is a }K\textrm{-qi section over } B_S\}\ne\emptyset$. For $w\in S$, $b\in B_w$, let $H_{b,w}=\textrm{hull}\{\gm(b):\gm\in\mathcal{G}_{K,S}\}\sse F_{b,w}$ and $H=\bigcup_{w\in S,~b\in B_w} H_{b,w}$. (Here quasiconvex hull is considered in the corresponding fiber.) Then $H$ is $K_{\ref{getting-metric-graph-bundles}}$-metric bundle over $B_S$. In particular, if $\mathcal{G}_{K,S}=\{\gm_1,\gm_2\}$ then $\bigcup_{w\in S,~b\in B_w}[\gm_1(b),\gm_2(b)]^f$ forms a special $K_{\ref{getting-metric-graph-bundles}}$-ladder over $B_S$.
\end{lemma}

\begin{lemma}\label{centers-forms-qi-section}
Given $k\ge1$, there is a constant $k_{\ref{centers-forms-qi-section}}=k_{\ref{centers-forms-qi-section}}(k)$ such that the following holds.

Suppose $(X,B,T)$ is a tree of metric bundles satisfying axiom {\bf H}. Let $v\in T$ and $\Sigma_i$ be $k$-qi section over $B_v$ for $i=1,2,3$. Suppose  $\{x^{(i)}_{b,v}\}=\Sigma_i\cap F_{b,v},i=1,2,3$ and $z_{b,v}$ is a $\dl_0$-center of geodesic triangle $\triangle(x^{(1)}_{b,v},x^{(2)}_{b,v},x^{(3)}_{b,v})\sse F_{b,v},~b\in B_v$. Then the map $s:B_v\ri X$ defined by $b\mapsto z_{b,v}$ is a $k_{\ref{centers-forms-qi-section}}$-qi section over $B_v$.
\end{lemma}

\begin{proof}
Let $b_1,b_2\in B_v$ such that $d_B(b_1,b_2)\le1$. Since $\pi_X:X\ri B$ is $1$-Lipschitz, we only have to prove that $d_{X_v}(z_{b_1,v},z_{b_2,v})$ is uniformly bounded. Define a map $\psi:F_{b_1,v}\ri F_{b_2,v}$ by $\psi(x^{(i)}_{b_1,v})=x^{(i)}_{b_2,v}$ for $i=1,2,3$ and $\fa~z\in F_{b_1,v}\setminus\{x^{(i)}_{b_1,v}:i=1,2,3\}$, we take $\psi(z)\in F_{b_2,v}$ such that $d(z,\psi(z))\le c_0$ (as in Definition \ref{treesofmetric-bun}). Note that $d_{X_v}(x,\psi(x))\le 2k+c_0$ for all $x\in F_{b_1,v}$. Then by \cite[Lemma $1.15$]{pranab-mahan}, $\psi$ extends to a $g(2k+c_0)$-quasiisometry from $F_{b_1,v}$ to $F_{b_2,v}$ for some function $g:\R_{\ge0}\ri\R_{\ge0}$. Therefore, by \cite[Lemma $1.29$ $(2)$]{pranab-mahan}, $d^f(\psi(z_{b_1,v}),z_{b_2,v})$ is bounded by a constant $D$, depending only on $\dl_0$ (hyperbolicity constant of $F_{b_i,v}$) and $g(2k+c_0)$. Hence, we can take $k_{\ref{centers-forms-qi-section}}:=D+2k$.
\end{proof}

\begin{lemma}\label{modified-projection-gives-qi-sections}
Given $K\ge1$, there is a constant $K_{\ref{modified-projection-gives-qi-sections}}=K_{\ref{modified-projection-gives-qi-sections}}(K)$ such that the following holds.

Suppose $(X,B,T)$ is a tree of metric bundles satisfying axiom {\bf H}. Let $v\in T$ and $\Sigma_i$ be $K$-qi section over $B_v$ for $i=1,2,3,4$. Let $\L_1=\L(\Sigma_1,\Sigma_2)$ and $\L_2=\L(\Sigma_3,\Sigma_4)$ be special ladders over $B_v$ formed by these sections (see Lemma \ref{qi-sec-inside-lad-len-lift} (2)). Let $a\in B_v$ and $\bar{P}_a:\L_1\cap F_{a,v}\ri\L_2\cap F_{a,v}$ be modified projection in the metric $F_{a,v}$ (see Definition \ref{modified-projection}). Further, suppose $\bar{P}_a(\L_1\cap F_{a,v})=[p_a,q_a]_{F_{a,v}}\sse \L_2\cap F_{a,v}$ such that $p_a$ is closest to $\Sigma_3\cap F_{a,v}$ and $q_a$ is that to $\Sigma_4\cap F_{a,v}$ in the metric $F_{a,v}$. Moreover, we define $s_1:B_v\ri \L_2$ and $s_2:B_v\ri\L_2$ by $a\mapsto p_a$ and $a\mapsto q_a$ respectively, where $a\in B_v$.

Then $s_1$ and $s_2$ are $K_{\ref{modified-projection-gives-qi-sections}}$-qi sections over $B_v$ such that their images lie inside $\L_2$.
\end{lemma}
\begin{proof}
Suppose $b,c\in B_v$ such that $d_B(b,c)\le1$. Let $\L_i\cap F_{b,v}=[x_i,y_i]$ and $\L_i\cap F_{c,v}=[s_i,t_i]$ for $i=1,2$, where $\Sigma_1(b)=x_1,~\Sigma_1(c)=s_1$, $\Sigma_2(b)=y_1,~\Sigma_2(c)=t_1$ and $\Sigma_3(b)=x_2,~\Sigma_3(c)=s_2$, $\Sigma_4(b)=y_2,~\Sigma_4(c)=t_2$. Suppose $\bar{P}_b([x_1,y_1])=[p_1,q_1]$ and $\bar{P}_c([s_1,t_1])=[p_2,q_2]$ such that $p_1$, $p_2$ are closest to $x_2$, $s_2$ respectively in the metric $F_{b,v}$ and $q_1$, $q_2$ are closest to $y_2$, $t_2$ respectively in the metric $F_{c,v}$. Since $\pi_X:X\ri B$ is $1$-Lipschitz, we only need to show that $d_{X_v}(p_1,p_2)$ and $d_{X_v}(q_1,q_2)$ are uniformly bounded. We will show only the former one as a similar proof works for the later one.

Let $\bar{P}_b(x_1)=x_1\pr,\bar{P}_b(y_1)=y_1\pr,\bar{P}_c(s_1)=s_1\pr,\bar{P}_c(t_1)=t_1\pr$. Note that $x\pr_1\in[p_1,y\pr_1]\sse[x_2,y_2]$ or $x\pr_1\in[y\pr_1,q_1]\sse[x_2,y_2]$ and $s\pr_1\in[p_2,t\pr_1]\sse[s_2,t_2]$ or $s\pr_1\in[t\pr_1,q_2]\sse[s_2,t_2]$. Depending on the position on $x\pr_1$ and $s\pr_1$, we consider the following four cases.

Like in Lemma \ref{centers-forms-qi-section}, we define a map $\psi:F_{b,v}\ri F_{c,v}$ such that $\psi(x_i)=s_i,~\psi(y_i)=t_i,~i=1,2$ and for all other points $x\in F_{b,v}$, we take $\psi(x)\in F_{c,v}$ such that $d_{X_v}(x,\psi(x))\le c_0$ (as in Definition \ref{treesofmetric-bun}). Then $d_{X_v}(x,\psi(x))\le2K+c_0,\fa~x\in F_{b,v}$, and so by \cite[Lemma $1.15$]{pranab-mahan}, $\psi$ extends to a $g(2K+c_0)$-quasiisometry from $F_{b,v}$ to $F_{c,v}$ for some function $g:\R_{\ge0}\ri\R_{\ge0}$. Thus by Lemma \ref{quasi-goes-to-quasi} $(2)$, there is $k=2K+D_{\ref{quasi-goes-to-quasi}}(\dl_0,g(2K+c_0),\dl_0)$ such that
\begin{eqnarray}\label{same-same}
d_X(x\pr_1,s\pr_1)\le k \textrm{ and } d_X(y\pr_1,t\pr_1)\le k.
\end{eqnarray}

\emph{Case} 1: Suppose $x_1\pr\in[p_1,y\pr_1]$ and $s\pr_1\in[p_2,t\pr_1]$. Then by \cite[Corollary $1.116$]{ps-kap}, there is constant $C_1$ depending on $\dl_0$ such that $d^f(p_1,x\pr_1)\le C_1$ and $d^f(p_2,s\pr_1)\le C_1$. Combining with inequation \ref{same-same}, we have $d_{X_v}(p_1,p_2)\le k+2C_1$.

\begin{figure}[h]
	\includegraphics[width=12cm]{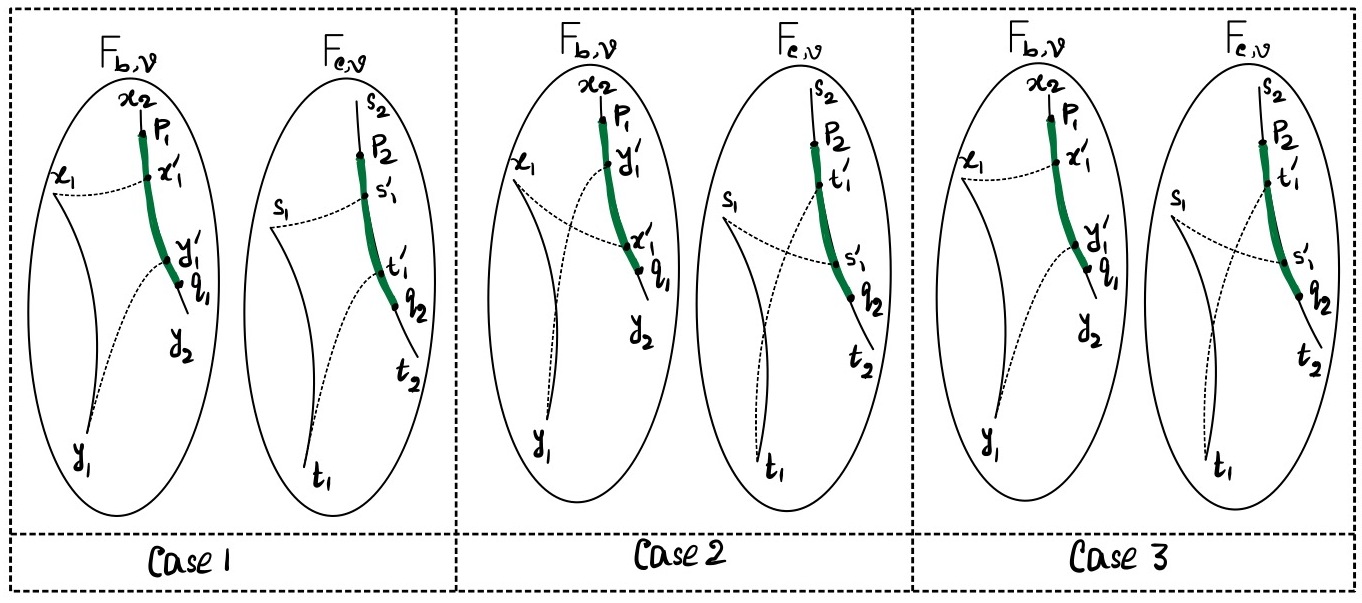}
	\centering
	\caption{}
\end{figure}
\emph{Case} 2: Suppose $x_1\pr\in[y\pr_1,q_1]$ and $s\pr_1\in[t\pr_1,q_2]$. In this case, $y\pr_1\in[p_1,x\pr_1]$ and $t\pr_1\in[p_2,s\pr_1]$. Again by \cite[Corollary $1.116$]{ps-kap}, $d^f(p_1,y\pr_1)\le C_1$ and  $d^f(p_2,t\pr_1)\le C_1$. Combining with inequation \ref{same-same}, we get $d_{X_v}(p_1,p_2)\le k+2C_1$.

\emph{Case} 3: Suppose $x_1\pr\in[p_1,y\pr_1]$ and $s\pr_1\in[t\pr_1,q_2]$. In this case, $t\pr_1\in[p_2,s\pr_1]$. Applying \cite[Corollary $1.116$]{ps-kap}, we get $d_{X_v}(p_1,x\pr_1)\le C_1$ and $d_{X_v}(p_2,t\pr_1)\le C_1$. We define a map (like above) $\psi: F_{b,v}\ri F_{c,v}$ such that $\psi(x\pr_1)=s\pr_1,~\psi(y_2)=t_2$ and for all other points $x\in F_{b,v}$, we take $\psi(x)\in F_{c,v}$ such that $d_{X_v}(x,\psi(x))\le c_0$ (as in Definition \ref{treesofmetric-bun}). Then $\fa~x\in F_{b,v}$, $d_{X_v}(\psi(x),x)\le k+c_0$ (see inequation \ref{same-same}). Thus by \cite[Lemma $1.15$]{pranab-mahan}, $\psi$ extends to a $g(k+c_0)$-quasiisometry from $F_{b,v}$ to $F_{c,v}$ for some function $g:\R_{\ge0}\ri\R_{\ge0}$. Applying Morse Lemma \ref{ml}, $\exists~\zeta\in[s\pr_1,t_2]\sse[s_2,t_2]$ such that $d_{X_v}(y\pr_1,\zeta)\le  c_0+D_{\ref{ml}}(\dl_0,g(k+c_0),g(k+c_0))=k\pr$ (say). Then $d_{X_v}(\zeta,t\pr_1)\le d_{X_v}(\zeta,y\pr_1)+d_{X_v}(y\pr_1,t\pr_1)\le k+k\pr$. Since fibers are $\phi$-properly embedded in $X$, so $d^f(\zeta,t\pr_1)\le \phi(k+k\pr)$. In particular, $d^f(s\pr_1,t\pr_1)\le \phi(k+k\pr)$. Thus $d_{X_v}(p_1,p_2)\le d_{X_v}(p_1,x\pr_1)+d_{X_v}(x\pr_1,s\pr_1)+d_{X_v}(s\pr_1,t\pr_1)+d_{X_v}(t\pr_1,p_2)\le2C_1+k+\phi(k+k\pr)$.

\emph{Case} 4: Lastly, we assume that $x_1\pr\in[y\pr_1,q_1]$ and $s\pr_1\in[p_2,t\pr_1]$. But this is same as Case $3$.

Therefore, we can take $K_{\ref{modified-projection-gives-qi-sections}}:=2C_1+k+\phi(k+k\pr)$ (maximum of all constants we get in the above four cases).
\end{proof}

\section{Semicontinuous families: flow space and ladder}\label{semicontinuous-subgraph}
Our standard assumptions in this section are that the tree of metric bundles $(X,B,T)$ must satisfy axiom {\bf H} and the Bestvina--Feighn hallway flaring condition. However, by Remark \ref{bigsmall-subsp} $(a)$ one can observe that $k$-flaring condition for a large $k$ is enough for our combination theorem to hold (see introduction of Sections \ref{hyperbolicity-of-ladder}, \ref{hyp-of-flow-sp}, \ref{union-of-two-flow-sp-is-hyp}).

Motivated by the construction of the semicontinuous families of spaces in \cite[Chapter $3$]{ps-kap}, we build subspaces analogous to that in trees of metric bundles. We follow the same terminology used in \cite{ps-kap}. Also, following \cite{ps-kap}, (\hspace{-.08mm}\cite{mitra-trees} and \cite{pranab-mahan}), we will see two special kinds of subspaces: flow spaces and ladders. These are the building blocks, which will be shown to be hyperbolic, towards proving the main theorem (Theorem \ref{main-theorem-com}).

Suppose $(X,B,T)$ is a tree of metric bundles as in Definition \ref{treesofmetric-bun}. Suppose $T_{\mfY}$ is a (connected, nonempty) subtree of $T$ and $\mfY=\bigcup_{v\in V(T_{\mfY}),b\in B_v}Q_{b,v}$ where $Q_{b,v}\sse F_{b,v}$ is nonempty. With this, we define the following.

\begin{defn}[Semicontinuous families]\label{semicts-subgrphs}
Let $K\ge1,C\ge0$ and $\epsilon\ge0$. We say that $\mfY\sse X$ is a $(K,C,\epsilon)$-semicontinuous family in $X$ with a central base $\mfB=\pi_B^{-1}(\mfT)$ for some central subtree $\mfT\sse T_{\mfY}$ if the following hold.
\begin{enumerate}
		
\item Let $v\in T_{\mfY}$ and $b\in B_v$. Then $Q_{b,v}$ is a $2\dl_0$-quasiconvex subset of $F_{b,v}$ and $\bigcup_{b\in B_v}^{}Q_{b,v}\sse X_v$ forms a $K$-metric bundle (see Definition \ref{K-metric-graph-bundle}) over $B_v$. Moreover, $\bigcup_{v\in T_{\mfY},~b\in B_v} Q_{b,v}$ forms a $K$-metric bundle over $\mfB$ in $X$.
		
\item Let $v,w\in T_{\mfY}$ such that $w\notin {\mfT},d_T(v,w)=1,d_T({\mfT},w)>d_T({\mfT},v)$. Let $[\mfv,\mfw]$ be the edge joining $\mfv\in B_v$ and $\mfw\in B_w$. Then $Hd_{\mfv\mfw}(P_{\mfw}(Q_{\mfv,v}),Q_{\mfw,w})\le\epsilon$  (see Notation \ref{connecting-edge-spaces} for $P_{\mfw}$) and $d_{\mfv\mfw}(x,Q_{\mfv,v})\le K,\fa~x\in Q_{\mfw,w}$. Moreover, if both $v,w\in {\mfT}$ then $Hd_{\mfv\mfw}(Q_{\mfv,v},Q_{\mfw,w})\le K$.
		
\item Suppose $w\notin T_{\mfY},v\in T_{\mfY}$ such that $d_T(v,w)=1$. Let $[\mfv,\mfw]$ be the edge joining $\mfv\in B_v$ and $\mfw\in B_w$. Then the pair $(Q_{\mfv,v},F_{\mfw.w})$ is $C$-cobounded in the metric $F_{\mfv\mfw}$.
		
\item Additionally, let $B\pr\sse \pi_B^{-1}(T_{\mfY})$ be a $(1,6\dl_0)$-qi embedded subspace in $B$. Suppose $v\in T_{\mfY}$ and $B\pr_v:=B_v\cap B\pr$. Then $\fa~v\in T_{\mfY}\textrm{ and }\fa~b\in B_v\setminus B\pr_v$, diam$^f(Q_{b,v})\le C$. Let $\mfY\pr:=\pi_X^{-1}(B\pr)\cap\mfY$.
	\end{enumerate}
\end{defn}

\begin{remark}
$(a)$ {\em The condition $(4)$ is used in Section \ref{hyp-of-flow-sp} (more precisely, in Lemma \ref{Li-bar-is-qi-emb-in-Li}), otherwise, all the time $B\pr=\pi_B^{-1}(T_{\mfY})$ and so $\mfY\pr=\mfY$.}
	
$(b)$ If $T$ is a single vertex, say $\{u\}$, then $\mfY$ is $K$-metric bundle over $B_u$. Existence of such $K$-metric bundle, for some $K\ge1$, is ensured by Lemma \ref{qi-sec-inside-lad-len-lift} (1).
	
$(c)$ If $\pi_B:B\ri T$ is a graph isomorphism, then $\mfY$ is the same as the semicontinuous family defined in the book \cite[Chapter $3$]{ps-kap}.
	
$(d)$ ({\bf Maximality}) `Moreover part' in conditions $(1)$ and $(2)$ are equivalent provided first parts of $(1)$ and $(2)$ hold. Let $z\in\mfY$ and $t_z$ be the nearest point projection of $\pi(z)$ on $\mfT$. Suppose $B_z=\pi_B^{-1}([t_z,\pi(z)])\cup\mfB$. Then it follows from the conditions $(1)$ and $(2)$ that there is a (compatible) $K$-qi section (see Definition \ref{K-qi-section}), say $\Sigma_z$, over $B_z$ lying inside $\mfY$.  Sometimes (more precisely, in  Subsection \ref{general-ladder}), we work with maximal qi sections through points in $\mfY$ in following sense. Let $S$ be a subtree of $T_{\mfY}$ containing $\mfT\cup[t_z,\pi(z)]$. Note that $B_z\sse B_S$. Let $\mathcal{G}=\{\eta:\eta\textrm{ is a }K\textrm{-qi section}\textrm{ over } B_S \textrm{ through } z\textrm{ lying inside }\mfY\}$. We put an order `$\le$' (inclusion) on $\mathcal{G}$ as follows. For $\eta,\eta\pr\in\mathcal{G}$, we say $\eta\le\eta\pr$ if and only if $\eta\sse\eta\pr$. This order `$\le$' makes $\mathcal{G}$ a poset. It is easy to see every chain has an upper bound in $\mathcal{G}$. Therefore, by Zorn's lemma, we get a maximal compatible $K$-qi section through $z$ over a base, say $B_S$, containing $B_z$ and contained in $\pi_B^{-1}(T_{\mfY})$. By abusing notation, we still denote the base for this maximal section by $B_z$.
	
$(e)$ Note that $Q_{b,v}$ are $2\dl_0$-quasiconvex in $F_{b,v}$ and $B\pr$ is $(1,6\dl_0)$-qi embedded in $\pi_B^{-1}(T_{\mfY})$. One also can introduce uniform constants for these. However, for simplicity, we will exclusively work with these constants.
	
$(f)$ {\em Later on, in our statements, we suppress the dependence on the constants $C,\epsilon$ and the other structural constants of the tree of metric bundles when dealing with semicontinuous family.}
	
\end{remark}

Now we will see nice properties (see Theorem \ref{mitra's-retraction-on-semicts-subsp}, Proposition \ref{semicts-is-proper-emb} and Corollary \ref{semicts-is-qi-emb}) enjoyed by semicontinuous families. In the proof, we use the same notations as in Definition \ref{semicts-subgrphs}. First, we prove that there is uniformly coarsely Lipschitz retraction of $X$ on semicontinuous families. This is motivated by Mitra's retraction in \cite{mitra-ct} (see also \cite[Theorem $3.2$]{pranab-mahan} and \cite[Theorem $3.3$]{ps-kap}). In this paper, we refer to this retraction as \emph{Mitra's retraction}. 


\begin{theorem}\label{mitra's-retraction-on-semicts-subsp}
Given $K\ge1,~C\ge0$ and $\ep\ge0$ there is a constant $L_{\ref{mitra's-retraction-on-semicts-subsp}}=L_{\ref{mitra's-retraction-on-semicts-subsp}}(K)$ such that the following holds.
	
If $\mfY$ is a $(K,C,\ep)$-semicontinuous family (as in Definition \ref{semicts-subgrphs}) in $X$, then there is a $L_{\ref{mitra's-retraction-on-semicts-subsp}}$-coarsely Lipschitz retraction $\rho_{\ref{mitra's-retraction-on-semicts-subsp}}=\rho_{\mfY\pr}:X\ri\mfY\pr$ of $X$ on $\mfY\pr$.
\end{theorem}

\begin{proof}	
	Let $X_{vsp}=\cup_{u\in T}X_u$ and $x,y\in X_{vsp}$ such that $d_X(x,y)\le1$. Then by Lemma \ref{imp-coarse-retraction}, it is enough to define a map $\rho:X_{vsp}\ri\mfY\pr$ for which $d_X(\rho(x),\rho(y))$ is uniformly bounded. 
	
	Let us define $\rho:X_{vsp}\ri\mfY\pr$ as follows. Suppose $x\in X_{vsp}$ and $b=\pi_X(x),u=\pi(x)$. 
	If $b\in B\pr$, then $\rho(x)$ is defined to be a nearest point projection on $Q_{b,u}$ in the metric $F_{b,u}$. Now suppose $b\notin B\pr$. Let $a$ be a nearest point projection of $b$ on $B\pr$ and $\pi_B(a)=v$. Since $B\pr$ is $(1,6\dl_0)$-qi embedded in $\pi_B^{-1}(T_{\mfY})$, so is in $B$. Then $B\pr$ is $K\pr$-quasiconvex in $B$, where $K\pr=K_{\ref{quasi-goes-to-quasi}}(\dl_0,max\{1,6\dl_0\},0)$. Note that $a$ is coarsely well defined. We also assume that $a_-\in[a,b]_B$ such that $a\ne a_-$ and $d_B(a,a_-)\le1$, and let $\pi_B(a_-)=w$. Let $x\pr$ be a nearest point projection of $x$ on $F_{a_-,w}$ in the metric $X$. Then we define $\rho(x)$ as nearest point projection of $x\pr$ on $Q_{a,v}$ in the path metric $F_{aa_-}:=\pi_X^{-1}([a,a_-]_B)$.
	
	Now we prove $d_X(\rho(x),\rho(y))$ is uniformly bounded where $x,y\in X_{vsp}$ and $d_X(x,y)\le1$. Let $\pi_X(x)=a,~\pi_X(y)=b$ and $\pi(x)=v,~\pi(y)=w$. We consider the following cases, depending on the position of $a,b,v$ and $w$.
	
	{\bf Case 1}: Suppose $a,b\in B\pr$. We consider two subcases, depending on whether $v=w$ or $v\ne w$.
	
	\emph{Subcase} (1A): Suppose $v=w$. We proof it by dividing into two parts, when $a=b$ and $a\ne b$.
	
	\emph{Subsubcase} (1AA): Suppose $a=b$. Since $Q_{a,v}$ is $2\dl_0$-quasiconvex in $F_{a,v}$, by Lemma \ref{proj-on-qc} $(1)$, $d^f(\rho(x),\rho(y))\le2C_{\ref{proj-on-qc}}(\dl_0,2\dl_0)=L_1$ (say).
	
	\emph{Subsubcase} (1AB): Suppose $a\ne b$. Note that $d_B(a,b)\le d_X(x,y)\le1$. Now through each point in $\mfY\cap X_v$ there is $K$-qi section over $B_v$ lying inside $\mfY\cap X_v$. Define a map $\psi:F_{a,v}\ri F_{b,v}$ as follows. For $z\in Q_{a,v}$ take $\psi(z)\in Q_{b,v}$ such that $d_{X_v}(\psi(z),z)\le2K$. For $z\notin Q_{a,v}$ take $\psi(z)\in F_{b,v}$ such that $d(\psi(z),z)\le c_0$ (as in Definition \ref{treesofmetric-bun}). In either case, $d_{X_v}(\psi(z),z)\le2K,\fa~z\in F_{a,v}$. Then by \cite[Lemma $1.15$]{pranab-mahan}, $\psi$ is a $g(2K+c_0)$-quasiisometry for some function $g:\R_{\ge0}\ri \R_{\ge0}$. Again $\fa~\xi\in Q_{b,v},\exists~\eta\in Q_{a,v}$ such that $d_{X_v}(\xi,\eta)\le 2K$ and $\eta$ is further $2K$-close to a point in $\psi(Q_{a,v})$, i.e. $Q_{b,v}\sse N_{4K}(\psi(Q_{a,v}))$ in the metric of $X_v$. Since fibers are $\phi$-properly embedded, $Q_{b,v}\sse N^f_{\phi(4K)}(Q_{a,v})$ (see \ref{connecting-edge-spaces} for notation). Then $\psi(Q_{a,v})\sse Q_{b,v}$ implies $Hd^f(\psi(Q_{a,v}),Q_{b,v})\le \phi(4K)$ in the metric of $F_{b,v}$.
	
	Let $y_1$ be a nearest point projection of $y$ on $\psi(Q_{a,v})$ in the metric of $F_{b,v}$. Now by Lemma \ref{quasi-goes-to-quasi} $(1)$, there is a constant $K_1=K_{\ref{quasi-goes-to-quasi}}(\dl_0,g(2K+c_0),2\dl_0)\ge2\dl_0$ such that $\psi(Q_{a,v})$ is $K_1$-quasiconvex in $F_{a,v}$, and so by Lemma \ref{proj-on-qc} $(2)$, $d^f(y_1,\rho(y))\le E_{\ref{proj-on-qc}}(\dl_0,K_1,\phi(4K))$. Again Lemma \ref{quasi-goes-to-quasi} $(2)$ says that there is $D=D_{\ref{quasi-goes-to-quasi}}(\dl_0,g(2K+c_0),2\dl_0)$ for which $d_{X_v}(\psi(\rho(x)),y_1)\le D$. By the definition of $\psi$, we also have $d_{X_v}(\rho(x),\psi(\rho(x)))\le2K$. Hence combining these four inequalities, we have $L_2=2K+D_{\ref{quasi-goes-to-quasi}}(\dl_0,g(2K+c_0),2\dl_0)+E_{\ref{proj-on-qc}}(\dl_0,K_1,\phi(4K))$ such that $d_{X_v}(\rho(x),\rho(y))\le L_2$.
	
	\emph{Subcase} (1B): Suppose $v\ne w$. Then it follows that $d_X(x,y)=1$, $d_B(a,b)=1$, $x\in F_{\mfv,v},~y\in F_{\mfw,w}$ and $a\in B_v$, $b\in B_w$. To make things notationally consistent, we assume that $a=\mfv,b=\mfw$. Note that $[\mfv,\mfw]\sse B\pr\sse\pi_B^{-1}(T_{\mfY})$. Now irrespective of whether $[\mfv,\mfw]$ is an edge in $\mfB$ or not, we have $Hd_{\mfv\mfw}(P_{\mfw}(Q_{\mfv,v}),Q_{\mfw,w})\le max\{2K,\ep\}$. Then by Lemma \ref{need-in-mitra's-proj}, $d_X(\rho(x),\rho(y))\le d_{\mfv\mfw}(\rho(x),\rho(y))\le R_{\ref{need-in-mitra's-proj}}(2\dl_0,K,max\{2K,\ep\})=L_3$ (say).

	{\bf Case 2}: Suppose one of $a,b$ belongs to $B\pr$. Without loss of generality, we assume that $a\in B\pr$ and $b\notin B\pr$. Here we also consider the following subcases, depending on whether $v=w$ or $v\ne w$.
	
	\emph{Subcase} (2A): Suppose $v=w$. Let $\pi_X(\rho(y))=a\pr$. Then $d_B(a,a\pr)\le2$, and so $Hd_{X_v}(Q_{a,v},Q_{a\pr,v})$ $\le2K+K=3K$. Again, since diam$^f(Q_{b,v})\le C$ then diam$(Q_{a,v})\le4K+C$. Thus diam$^f(Q_{a,v})\le \phi(4K+C)$, and so $d_X(\rho(x),\rho(y))\le3K+\phi(4K+C)=L_4$ (say).
	
	\emph{Subcase} (2B): Suppose $v\ne w$. For the consistency of notation, we assume that $a=\mfv,b=\mfw$. Without loss of generality, we let $d_T(\pi_B(\mfB\cap B\pr),v)<d_T(\pi_B(\mfB\cap B\pr),w)$. Note that in this case, $\rho(y)$ is nearest point projection of $y$ on $Q_{\mfv,v}$ in the metric of $F_{\mfv\mfw}$. Then by Lemma \ref{need-in-mitra's-proj1}, $d_X(\rho(x),\rho(y))\le d_{\mfv\mfw}(\rho(x),\rho(y))\le R_{\ref{need-in-mitra's-proj1}}(2\dl_0)=L_5$ (say). 
	
{\bf Case 3}: Suppose $a,b\notin B\pr$. Let $a\pr=\pi_X(\rho(x)),b\pr=\pi_X(\rho(y))$ and $a\pr_-\in[a\pr,a]_B$, $b\pr_-\in[b\pr,b]_B$ such that $a\pr_-\ne a\pr$, $b\pr_-\ne b\pr$ and $d_B(a\pr,a\pr_-)\le1,d_B(b\pr,b\pr_-)\le1$. Since $d_X(x,y)\le1$, then $\pi_B(a\pr)=\pi_B(b\pr)$ and $\pi_B(a\pr_-)=\pi_B(b\pr_-)$. Let us rename $\pi_B(a\pr)$ as $v$ and $\pi_B(a\pr_-)$ as $w$ not to make notation-heavy. We consider the following subcases depending on whether $v=w$ or $v\ne w$.
	
\emph{Subcase} (3A): Suppose $v=w$. Note that $K\pr=K_{\ref{quasi-goes-to-quasi}}(\dl_0,max\{1,6\dl_0\},0)$. Since $B\pr$ is $K\pr$-quasiconvex in $B$ and $d_B(a,b)\le d_X(x,y)\le1$, so by Lemma \ref{proj-on-qc} $(1)$, $d_B(a\pr,b\pr)\le2C_{\ref{proj-on-qc}}(\dl_0,K\pr)$. Note that $\mfY\cap X_{v}$ forms a $K$-metric bundle and so $Hd_{X_{v}}(Q_{a\pr,v},Q_{b\pr,v})\le2KC_{\ref{proj-on-qc}}(\dl_0,K\pr)+K$. Since $a\pr_-,b\pr_-\in B_v$ then diam$^f(Q_{a\pr,v})\le \phi(4K+C)$ and diam$^f(Q_{b\pr,v})\le \phi(4K+C)$ (see \emph{Subcase} (2A) for instance). Hence $d_X(\rho(x),\rho(y))\le2KC_{\ref{proj-on-qc}}(\dl_0,K\pr)+K+\phi(4K+C)=L_6$ (say).
	
\emph{Subcase} (3B): Suppose $v\ne w$. For the consistency of notation, we assume that $a\pr=b\pr=\mfv$ and $a\pr_-=b\pr_-=\mfw$. In the same way, we consider the following two subcases.
	
\emph{Subsubcase} (3BA): Let $[\mfv,\mfw]$ be an edge in $\pi_B^{-1}(T_{\mfY})$. Without loss of generality, we assume that $d_B(\pi_B(\mfB\cap B\pr),v)<d_B(\pi_B(\mfB\cap B\pr),w)$. By the assumption diam$^f(Q_{\mfw,w})\le C$ and so $diam(P_{\mfw}(Q_{\mfv,v}))\le2\epsilon+C$ in the metric of $F_{\mfv\mfw}$. Then by Lemma \ref{small-imp-small}, there is a constant $C_1$ depending on $\lm\pr_0,\dl\pr_0$ and $2\ep+C$ such that the pair $(Q_{\mfv,v},F_{\mfw,w})$ is $C_1$-cobounded in the metric of $F_{\mfv\mfw}$. Therefore, $d_X(\rho(x),\rho(y))\le C_1=L_7$ (say).
	
\emph{Subsubcase} (3BB): Suppose $[\mfv,\mfw]$ is not an edge in $\pi_B^{-1}(T_{\mfY})$. Then by definition of semicontinuous family $\mfY$, the pair $(Q_{\mfv,v},F_{\mfw,w})$ is $C$-cobounded in the path metric of $F_{\mfv\mfw}$. So $d_X(\rho(x),\rho(y))\le C$.
	
Suppose $L=max\{L_i, C:1\le i\le 7\}=max\{L_i:1\le i\le7\}$. Therefore, by Lemma \ref{imp-coarse-retraction}, one can take $L_{\ref{mitra's-retraction-on-semicts-subsp}}=C_{\ref{imp-coarse-retraction}}(L)$.
\end{proof}

Next we show that a uniform neighborhood of semicontinuous families are path connected in $X$ and with the induced path metric from $X$, they are uniformly properly embedded in $X$. As a consequence, we will see that they are also (uniformly) qi embedded in $X$ (see Corollary \ref{semicts-is-qi-emb}).

\begin{prop}\label{semicts-is-proper-emb}
Suppose $K\ge1,~C\ge0$ and $\epsilon\ge0$. Then for all $L\ge max\{2\dl_0+1,2K\}$ there exists a proper map $\eta_{\ref{semicts-is-proper-emb}}=\eta_{\ref{semicts-is-proper-emb}}(K,L):\R_{\ge0}\ri\R_{\ge0}$ such that the following holds.
	
If $\mfY$ is a $(K,C,\epsilon)$-semicontinuous family (as in Definition \ref{semicts-subgrphs}) in $X$, then $N_L(\mfY\pr)$ is path connected and with the path metric on $N_L(\mfY')$ induced from $X$, the inclusion $N_L(\mfY\pr)\ri X$ is a $\eta_{\ref{semicts-is-proper-emb}}$-proper embedding.
\end{prop}

\begin{proof}
It is clear that $N_L(\mfY\pr)$ is path connected. We denote the path metric on $N_L(\mfY\pr)$ induced from $X$ by $d\pr$.
	
For second part, we first show that for $r\in\R_{\ge0}$, $x,y\in\mfY\pr$ and $d_X(x,y)\le r$ we have a bound on $d\pr(x,y)$ in terms of $r$, and in the end, we show the same for points in $N_L(\mfY\pr)$. Fix $u\in\pi_B(\mfB\cap B\pr)$. We take $t$, the center of the tripod in $T$ with vertices $\pi(x)$, $\pi(y)$, $u$ if $[\pi(x),\pi(y)]_T\cap\pi_B(\mfB\cap B\pr)=\emptyset$, otherwise, $t\in[\pi(x),\pi(y)]_T\cap\pi_B(\mfB\cap B\pr)$ arbitrary. Let $a=\pi_X(x),b=\pi_X(y)$. Then $d_B(a,b)\le d_X(x,y)\le r$. Since the inclusion $B\pr\ri \pi_B^{-1}(T_{\mfY})$ is a $(1,6\dl_0)$-qi embedding, $d_{B\pr}(a,b)\le r+6\dl_0$. Let $c\in B_t\cap [a,b]_{B\pr}$ be arbitrary. Then $d_{B\pr}(a,c)\le r+6\dl_0$ and $d_{B\pr}(c,b)\le r+6\dl_0$. By taking $K$-qi lifts of geodesics $[a,c]_{B\pr}$ and $[c,b]_{B\pr}$ in $\mfY$ (more precisely, in $\mfY\pr$), we get $x_1,y_1\in Q_{c,t}$ such that $d\pr(x,x_1)\le2K(r+6\dl_0)$ and $d\pr(y,y_1)\le2K(r+6\dl_0)$ (see Lemma \ref{qi-sec-inside-lad-len-lift} $(3)$). Now $d_X(x_1,y_1)\le d_X(x_1,x)+d_X(x,y)+d_X(y,y_1)\le d\pr(x_1,x)+d_X(x,y)+d\pr(y,y_1)\le4K(r+6\dl_0)+r$ $\Rightarrow d^f(x_1,y_1)\le \phi(4K(r+6\dl_0)+r)$.
	
Note that $N^f_{2\dl_0+1}(Q_{c,t})\sse {N_L(\mfY\pr)}\cap F_{c,t}$  as $L\ge2\dl_0+1$, and $x_1,y_1\in Q_{c,t}$. Then by Lemma \ref{qi-emb-in-Y-X}, there is $D(r)$ depending on $r$ such that $d\pr(x_1,y_1)\le D(r)$. Hence $d\pr(x,y)\le d\pr(x,x_1)+d\pr(x_1,y_1)+d\pr(y_1,y)\le4K(r+6\dl_0)+D(r)$.
	
Now suppose $x,y\in N_L(\mfY\pr)$ such that $d_X(x,y)\le r$. Then $\exists~x_1,y_1\in\mfY\pr$ such that $d\pr(x,x_1)\le L$ and $d\pr(y,y_1)\le L$. So $d_X(x_1,y_1)\le2L+r$. Thus by above, $d\pr(x_1,y_1)\le 4K(2L+r+6\dl_0)+D(2L+nr)$. Hence combining these inequalities, we get $d\pr(x,y)\le 4K(2L+r+6\dl_0)+D(2L+r)+2L$.
	
Therefore, we can take $\eta_{\ref{semicts-is-proper-emb}}:\R_{\ge0}\ri\R_{\ge0}$ sending $r\mapsto4K(2L+r+6\dl_0)+D(2L+r)+2L$.
\end{proof}
As a consequence we have the following corollary (see Lemma \ref{lrape_qi-emb}).
\begin{cor}\label{semicts-is-qi-emb}
Suppose $K\ge1,~C\ge0$ and $\epsilon\ge0$. Then for all $L\ge max\{2\dl_0+1,2K\}$ there exists a constant $L_{\ref{semicts-is-qi-emb}}=L_{\ref{semicts-is-qi-emb}}(K,L):=L_{\ref{lrape_qi-emb}}(\eta_{\ref{semicts-is-proper-emb}}(K,L),L_{\ref{mitra's-retraction-on-semicts-subsp}}(K),L)$ such that the following holds.
	
If $\mfY$ is a $(K,C,\epsilon)$-semicontinuous family (as in Definition \ref{semicts-subgrphs}) in $X$, then the inclusion $N_L(\mfY\pr)\ri X$ is a $L_{\ref{semicts-is-qi-emb}}$-qi embedding in $X$.
\end{cor}


\begin{remark}\label{3-in-one}
Conclusion of Theorem \ref{mitra's-retraction-on-semicts-subsp}, Proposition \ref{semicts-is-proper-emb} and Corollary \ref{semicts-is-qi-emb} holds for $\mfY$ as well.
\end{remark}

\subsection{Flow space}\label{flowsps}
Suppose $(X,B,T)$ is a tree of metric bundles as in Definition \ref{treesofmetric-bun}. Let $u\in T,a\in B_u$. Given a subset $A_{a,u}$ of $F_{a,u}$, we define (rather, construct) the flow space of $A_{a,u}$, which is a semicontinuous family in $X$ with a central base (possibly bigger than) $B_u$. The construction of this flow space is by induction as follows. Let $k\ge K_{\ref{qi-sec-inside-lad-len-lift}}$ be fixed.

{\bf Step 1}: $\mathcal{G}_{a,u}=\{\gm:\gm \textrm{ is a } k\textrm{-qi section over }B_u\textrm{ through a point in }A_{a,u}\}$. Let $b\in B_u$ and $Q_{b,u}=\textrm{hull}\{\gm(b):\gm\in\mathcal{G}_{a,u}\}\sse F_{b,u}$, where quasiconvex hull is considered in the corresponding fiber. Note that $Q_{b,u}$ is $2\dl_0$-quasiconvex in $F_{b,u}$ (see Remark \ref{hull-is-qc}) and by Lemma \ref{qi-sec-inside-lad-len-lift} $(2)$, $\bigcup_{b\in B_u}Q_{b,u}$ forms a $C_{\ref{qi-sec-inside-lad-len-lift}}(k)$-metric bundle over $B_u$.

{\bf Step 2}: We extend this to other $X_v$ by induction on $d_T(u,v)$. Suppose we have extended it till $X_v$, where $d_T(u,v)=n$. Let $w\in T$ such that $d_T(u,w)=n+1$ and $d_T(v,w)=1$. Let $[\mfv,\mfw]$ be the edge joining $\mfv\in B_v$ and $\mfw\in B_w$. We denote $Q_{b,t}$ as the intersection of the flow space we are constructing with $F_{b,t}$ for $t\in T,b\in B_t$. We first flow $Q_{\mfv,v}$ in $F_{\mfw,w}$ and then by Step $1$ above in the entire $X_w$, provided $Q_{\mfw,w}\ne\emptyset$.

Let us fix $R\ge R_{\ref{R-sep-D-cobdd}}(\dl\pr_0,\lambda\pr_0)$, and let $R\pr=R\pr_{\ref{R-sep-D-cobdd}}(\dl\pr_0,\lambda\pr_0,R)$. Note that $Q_{\mfv,v}$ is $2\dl_0$-quasiconvex in $F_{\mfv,v}$ and so is $\lm\pr_0$-quasiconvex in $F_{\mfv\mfw}$ (see Lemma \ref{com-two-hyp-sps}$~(2)$). Suppose $N^{\mfv\mfw}_R(Q_{\mfv,v})\cap F_{\mfw,w}\ne\emptyset$. Then by Lemma \ref{R-sep-D-cobdd} $(2)$, $P_{\mfw}(Q_{\mfv,v})\sse N^{\mfv\mfw}_{R\pr}(Q_{\mfv,v})\cap F_{\mfw,w}=:Q\pr_{\mfw,w}$ (say). Let $Q_{\mfw,w}:=\textrm{hull}(Q\pr_{\mfw,w})\sse F_{\mfw,w}$, where quasiconvex hull is considered in $F_{\mfw,w}$. Note that $Q_{\mfw,w}$ is $2\dl_0$-quasiconvex in $F_{\mfw,w}$. 
Now we apply Step $1$ to $Q_{\mfw,w}$ by considering all $k$-qi section over $B_w$ through points in $Q_{\mfw,w}$.

If $N^{\mfv\mfw}_R(Q_{\mfv,v})\cap F_{\mfw,w}=\emptyset$, then we will not `flow' $Q_{\mfv,v}$ in that direction. In other words, let $S$ be the connected component of $T\setminus\{v\}$ containing $w$. Then $\fa~ t\in S\textrm{ and }\fa~b\in B_t$, we have $Q_{b,t}=\emptyset$.

Now we prove the following properties which verify that the subspace we are constructing is a semicontinuous family. Let $v,w\in T$ such that $d_T(u,v)<d_T(u,w)$.

{\bf Property 1}: Suppose $Q_{\mfv,v}\ne\emptyset$ and $N^{\mfv\mfw}_R(Q_{\mfv,v})\cap F_{\mfw,w}=\emptyset$. Then the pair $(Q_{\mfv,v},F_{\mfw,w})$ is $C:=D_{\ref{R-sep-D-cobdd}}(\dl\pr_0,\lambda\pr_0)$-cobounded in $F_{\mfv\mfw}$. Indeed, $Q_{\mfv,v}$ and $F_{\mfw,w}$ are $\lambda\pr_0$-quasiconvex in $F_{\mfv\mfw}$ (see Lemma \ref{com-two-hyp-sps}$~(2)$). So by Lemma \ref{R-sep-D-cobdd} (1), we are done.

{\bf Property 2}: Suppose both $Q_{\mfv,v}$ and $Q_{\mfw,w}$ are nonempty. Then $Q_{\mfw,w}\sse N^{\mfv\mfw}_{K\pr}(Q_{\mfv,v})$ for some uniform constant $K\pr\ge0$.

\emph{Proof.} Let $x\in Q_{\mfw,w}$. Then $\exists~x_1,x_2\in Q\pr_{\mfw,w}$ and $x\pr\in[x_1,x_2]_{F_{\mfw,w}}$ such that $d^f(x,x\pr)\le\dl_0$. Let $y_1,y_2\in Q_{\mfv,v}$ such that $d_{\mfv\mfw}(x_i,y_i)\le R\pr,~i=1,2$. Note that $Q_{\mfv,v}$ is $L\pr_0$-qi embedded in $F_{\mfv\mfw}$ (Lemma \ref{com-two-hyp-sps}). Then by slimness of quadrilateral in $F_{\mfv\mfw}$ with vertices $x_1,x_2,y_1$ and $y_2$, there is $x\prr\in Q_{\mfv,v}$ such that $d_{\mfv\mfw}(x\pr,x\prr)\le 2D_{\ref{ml}}(\dl\pr_0,L\pr_0,L\pr_0)+R\pr+2\dl\pr_0+2\dl_0$. Thus $d(x,x\prr)\le2D_{\ref{ml}}(\dl\pr_0,L\pr_0,L\pr_0)+R\pr+2\dl\pr_0+2\dl_0+\dl_0=:{K\pr}$.

{\bf Property 3}: Suppose both $Q_{\mfv,v}$ and $Q_{\mfw,w}$ are nonempty. Then $Hd_{\mfv\mfw}(P_{\mfw}(Q_{\mfv,v}),Q_{\mfw,w})\le\ep$ for some uniform constant $\ep\ge0$.

\emph{Proof.} Property (2) tells that $Q_{\mfw,w}\sse N^{\mfv\mfw}_{2K\pr}(P_{\mfw}(Q_{\mfv,v}))$. Again by construction $P_{\mfw}(Q_{\mfv,v})\sse Q_{\mfw,w}$. So $Hd_{\mfv\mfw}(P_{\mfw}(Q_{\mfv,v}),Q_{\mfw,w})\le2K\pr=:\ep$.

\smallskip
We denote $\mathcal{F}l_K(A_{a,u}):=\bigcup\limits_{v\in T,~b\in B_v}Q_{b,v}$, where $K=max\{K\pr,C_{\ref{qi-sec-inside-lad-len-lift}}(k)\}$.

\begin{defn}[\bf Flow space]\label{flow-space-def}
We say that $\F l_K(A_{a,u})=\bigcup_{v\in T,~b\in B_v}Q_{b,v}$ is the flow space of $A_{a,u}$ with parameters $k\ge K_{\ref{qi-sec-inside-lad-len-lift}}$ and $R\ge R_{\ref{R-sep-D-cobdd}}(\dl\pr_0,\lambda\pr_0)$. It is clear from the construction that $\F l_K(A_{a,u})$ is a $(K,C,\ep)$-semicontinuous family, where $K=K_{\ref{flow-space-def}}(k,R)=max\{K\pr,C_{\ref{qi-sec-inside-lad-len-lift}}(k)\},~C=D_{\ref{R-sep-D-cobdd}}(\dl\pr_0,\lambda\pr_0)$ and $\epsilon=\epsilon_{\ref{flow-space-def}}(R)$ are as in the above properties.
	
In particular, for any $u\in T$ and $a\in B_u$, suppose $\F l_K(F_{a,u})$ is the flow space of $F_{a,u}$ with parameters $k\ge K_{\ref{qi-sec-inside-lad-len-lift}},~R\ge R_{\ref{R-sep-D-cobdd}}(\dl\pr_0,\lambda\pr_0)$, where $K=K_{\ref{flow-space-def}}(k,R)$. Then by Lemma \ref{qi-sec-inside-lad-len-lift} $(1)$, $X_u\sse\F l_K(F_{a,u})$. In this case, we denote the flow space by $\F l_K(X_u)$ and we say that $\F l_K(X_u)$ is the flow space of $X_u$ with parameters $k,R$.

Although $\F l_K(A_{a,u})$ $($or $\F l_K(X_u))$ depends on the constants $C,\epsilon$ and the other structural constants of $(X,B,T)$, we make them implicit in our notation.
\end{defn}

We have defined the flow space of a subset of a fiber and of $X_u$ for $u\in T$ in Definition \ref{flow-space-def}. Below we make it a bit general, and this will be used in Section \ref{union-of-two-flow-sp-is-hyp}.


\begin{defn}[\bf Flow space of metric bundles]\label{gen-flow-sp-1}
Fix $k\ge K_{\ref{qi-sec-inside-lad-len-lift}}$. Let $S$ be a subtree of $T$ and $R\ge max\{R_{\ref{R-sep-D-cobdd}}(\dl\pr_0,\lambda\pr_0),k\}$. 
Suppose $H$ is a $k$-metric bundle over $B_S$ (see Definition \ref{K-metric-graph-bundle}). Let $H_{b,u}:=H\cap F_{b,u}$ for $u\in S,~b\in B_u$. We also assume that $H_{b,u}$ is $2\dl_0$-quasiconvex in $F_{b,u}$. Suppose $\mathcal{S}(S,1)=\{w\in T:d_T(S,w)=1\}$. Let $w\in\mathcal{S}(S,1)$ and $v\in S$ such that $d_T(v,w)=1$. Let $T_{wv}$ be the union of the connected component of $T\setminus \{v\}$ containing $w$ and the edge $[v,w]$. Suppose $[\mfv,\mfw]$ is the edge joining $\mfv\in B_v$ and $\mfw\in B_w$. Let $\F l^{T_{wv}}_K(H_{\mfv,v})$ be the flow space of $H_{\mfv,v}$ considered only in $X_{T_{wv}}$ (with the parameters $k,R$) such that $\F l^{T_{wv}}_K(H_{\mfv,v})\cap X_{v}=H\cap X_{v}$, where $K=K_{\ref{flow-space-def}}(k,R)$. We define $\F l_K(H):=\bigcup_{w\in\mathcal{S}(S,1)}\F l^{T_{wv}}_K(H_{\mfv,v})$ as flow space of $H$.
	
It is clear that $\F l_K(H)$ is a $(K,C,\epsilon)$-semicontinuous family with a central base $B_S$, where $K=K_{\ref{flow-space-def}}(k,R)$,  $C=D_{\ref{R-sep-D-cobdd}}(\dl\pr_0,\lambda\pr_0)$, $\epsilon=\epsilon_{\ref{flow-space-def}}(R)$.
	
\end{defn}


\begin{remark}
$(1)$ In the construction of flow space above, in Step $(2)$, when we go from one metric bundle to the next metric bundle away from $u$, we took $R$-neighborhood of $Q_{\mfv,v}$ in $F_{\mfv\mfw}$ and intersected it with $F_{\mfw,w}$. If the intersection is nonempty then we proceed in the metric bundle $X_w$, otherwise, we do not. Equivalently, we can do the following in Step $(2)$ instead: we will check if the pair $(Q_{\mfv,v},F_{\mfw,w})$ is not $C$-cobounded in $F_{\mfv\mfw}$ for some pre-assumed constant $C\ge0$. If so, we will proceed in the metric bundle $X_w$, otherwise, we will not. With the help of Lemma \ref{R-sep-D-cobdd} $(1)$, one can easily check that these two definitions are equivalent. In this definition, flow spaces, in some special cases, looks very simple as in $(2)$ below. Of course, through out the paper, we will stick to Definitions \ref{flow-space-def} and \ref{gen-flow-sp-1} for flow spaces.

$(2)$ In Example \ref{trees of metric bundles exps}, suppose $G_i$ and $N_i$ are hyperbolic groups where $i=1,2$. Let $\pi_X:X\map B$ be the corresponding tree of metric bundles. In addition, if we take $H=\{1\}$, then flow space of $X_u$ coincide with $X_u$ for all $u\in V(T)$ with respect to the above equivalent definition for any fixed $C>0$. On the other hand, if $R$ as in Definition \ref{gen-flow-sp-1} is greater than or equal to $1$, then according to Definition \ref{gen-flow-sp-1}, $X_u$ is strictly contained in $\F l_K(X_u)$, and if $R<1$ then $\F l_K(X_u)=X_u$. Note that $R<1$ can possibly happen when the fibers are trees.
\end{remark} 

Now we will record some constants from the above discussion in the following lemma for later use.

\begin{lemma}\label{const-of-flow-sp}
Given $k\ge K_{\ref{qi-sec-inside-lad-len-lift}}$ and $R\ge max\{R_{\ref{R-sep-D-cobdd}}(\dl\pr_0,\lambda\pr_0),k\}$ there are constants $K_{\ref{const-of-flow-sp}}=K_{\ref{const-of-flow-sp}}(k,R)=K_{\ref{flow-space-def}}(k,R),~ \epsilon_{\ref{const-of-flow-sp}}=\epsilon_{\ref{const-of-flow-sp}}(R)=\epsilon_{\ref{flow-space-def}}(R)$ and $C_{\ref{const-of-flow-sp}}=D_{\ref{R-sep-D-cobdd}}(\dl\pr_0,\lambda\pr_0)$ such that the following hold.
\begin{enumerate}
\item Let $u\in T$. Then $\F l_{K_{\ref{const-of-flow-sp}}}(X_u)$ is a $(K_{\ref{const-of-flow-sp}},C_{\ref{const-of-flow-sp}},\epsilon_{\ref{const-of-flow-sp}})$-semicontinuous family with a central base $B_u$.
		
		
\item Let $S$ be a subtree of $T$ and $H$ be a $k$-metric bundle over $B_S$. Then $\F l_{K_{\ref{const-of-flow-sp}}}(H)$ is a $(K_{\ref{const-of-flow-sp}},C_{\ref{const-of-flow-sp}},\epsilon_{\ref{const-of-flow-sp}})$-semicontinuous family with a central base $B_S$.
		
	\end{enumerate}
\end{lemma}

Note that in Lemma \ref{const-of-flow-sp}, $(2)$ needs the condition $R\ge max\{R_{\ref{R-sep-D-cobdd}}(\dl\pr_0,\lambda\pr_0),k\}$ whereas $(1)$ holds for $R\ge R_{\ref{R-sep-D-cobdd}}(\dl\pr_0,\lambda\pr_0)$.

Consider the flow spaces with parameters $k,R$ as taken in Lemma \ref{const-of-flow-sp}. 
Let $\bm{K}=K_{\ref{const-of-flow-sp}}(k,R)$, $\bm{C}=C_{\ref{const-of-flow-sp}}$ and $\bm{\epsilon}=\epsilon_{\ref{const-of-flow-sp}}(R)$. 
Flow spaces being semicontinuous families, we have and restate the following three results for flow spaces (see Proposition \ref{mitra's-retraction-on-semicts-subsp}, Proposition \ref{semicts-is-proper-emb}, Corollary \ref{semicts-is-qi-emb} and also Remark \ref{3-in-one}) as they will be utilized extensively in Sections \ref{hyp-of-flow-sp} and \ref{union-of-two-flow-sp-is-hyp}.

\begin{prop}\label{mitra's-retraction-on-fsandgss}
	There exists $L_{\ref{mitra's-retraction-on-fsandgss}}=L_{\ref{mitra's-retraction-on-fsandgss}}(K)$-coarsely Lipschitz retraction $\rho_{\ref{mitra's-retraction-on-fsandgss}}:X\ri\F l_K(Z)$ where $Z\in\{X_u,H\}$.
\end{prop}


\begin{prop}\label{fsandgss-are-proper-emb}
Given $L\ge max\{2K,2\dl_0+1\}$, there is a proper map $\eta_{\ref{fsandgss-are-proper-emb}}=\eta_{\ref{fsandgss-are-proper-emb}}(K,L):\R_{\ge0}\ri\R_{\ge0}$ such that the inclusion $N_L(\F l_K(Z))\ri X$ is a $\eta_{\ref{fsandgss-are-proper-emb}}$-proper embedding in $X$, where $Z\in\{X_u,H\}$.
\end{prop}

\begin{cor}\label{fsandgss-are-qi-emb}
Given $L\ge max\{2K,2\dl_0+1\}$, there is a constant $L_{\ref{fsandgss-are-qi-emb}}=L_{\ref{fsandgss-are-qi-emb}}(K,L)$ such that the inclusion $N_L(\F l_K(Z))\ri X$ is a $L_{\ref{fsandgss-are-qi-emb}}$-qi embedding, where $Z\in\{X_u,H\}$.
\end{cor}

Flow space being semicontinuous family, the fundamental and crucial property is the existence of qi sections through each point over the respective domain. If one carefully analyses the construction, one will realize the following. Given a qi section, one can construct, by taking larger neighborhood at the junction (more precisely, in $F_{\mfv\mfw}$), a flow space containing the qi section. The following lemma captures this property. Since the proof is straightforward, we omit it.
\begin{lemma}\label{flow-sp-contains-qi-sec}
Given $K\ge K_{\ref{const-of-flow-sp}}(K_{\ref{qi-sec-inside-lad-len-lift}},R_{\ref{R-sep-D-cobdd}}(\dl\pr_0,\lambda\pr_0))$ there are constants $K_{\ref{flow-sp-contains-qi-sec}}=K_{\ref{flow-sp-contains-qi-sec}}(K)=K_{\ref{const-of-flow-sp}}(K,K)$, $C_{\ref{flow-sp-contains-qi-sec}}=C_{\ref{const-of-flow-sp}}(\dl\pr_0,\lambda\pr_0)$ and $\epsilon_{\ref{flow-sp-contains-qi-sec}}=\epsilon_{\ref{flow-sp-contains-qi-sec}}(K)=\epsilon_{\ref{const-of-flow-sp}}(K)$ such that the following holds.
	
Suppose $S\sse T$ is a subtree and $u\in S$. Let $\gm$ be a $K$-qi section over $B_S$. Suppose $\F l_{K_{\ref{flow-sp-contains-qi-sec}}}(X_u)$ is the flow space of $X_u$ with parameters $K$ and $K$. Then $\gm\sse\F l_{K_{\ref{flow-sp-contains-qi-sec}}}(X_u)$ and $\F l_{K_{\ref{flow-sp-contains-qi-sec}}}(X_u)$ is a $(K_{\ref{flow-sp-contains-qi-sec}},C_{\ref{flow-sp-contains-qi-sec}},\epsilon_{\ref{flow-sp-contains-qi-sec}})$-semicontinuous family. 
\end{lemma}

\begin{notation}\label{iteration-fn-for-qi-sec}
We define a function $\kappa^{(i)}\longmapsto\kappa^{(i+1)}$ to measure the iteration in Lemma \ref{flow-sp-contains-qi-sec}. In other words, suppose $\kappa\ge K_{\ref{const-of-flow-sp}}(K_{\ref{qi-sec-inside-lad-len-lift}},R_{\ref{R-sep-D-cobdd}}(\dl\pr_0,\lambda\pr_0))$. Define $\kappa^{(0)}=\kappa,~\kappa^{(i+1)}=K_{\ref{flow-sp-contains-qi-sec}}(\kappa^{(i)},\kappa^{(i)})$. Then for the flow space $\F l_{\kappa^{(i)}}(X_u)$ of $X_u$ with parameters $k=\kappa^{(i-1)}$ and $R=\kappa^{(i-1)}$, we have $\F l_{\kappa^{(i)}}(X_u)\sse\F l_{\kappa^{(i+1)}}(X_u)$.  
\end{notation}

%

\subsection{Ladder}\label{discussion-on-ladder}
First, we will outline the additional properties (L1 and L2 below) that a semicontinuous family must satisfy in order to be considered as ladder. Then we will formally define a ladder in Definition \ref{ladder} for future reference.

Suppose $(X,B,T)$ is a tree of metric bundles. Let $K\ge1,~C\ge0$ and $\ep\ge0$. Let $\L\sse X$ be a $(K,C,\ep)$-semicontinuous family with a central base, say $\mfB$, such that the fibers are geodesic segments. In addition, suppose we have the following properties L1 and L2 in $\L$.

{\bf (L1)}: For all $v\in T_{\L}$, where $T_{\L}=\textrm{hull}(\pi(\L))$, $\L\cap X_v$ is a special $K$-ladder (see Definition \ref{K-metric-graph-bundle}) over $B_v$. Moreover, $\L\cap X_{\mfB}$ forms a special $K$-ladder over $\mfB$.\smallskip

We will need a family of monotonic maps at the junction. In doing so, we will put an orientation on the fiber geodesics as follows (see Figure \ref{ladderF}).

\begin{notation}\label{notation-ladder} $\mfT:=\pi_B(\mfB)$, $\L_{a,v}:=\L\cap F_{a,v}~\fa~a\in B_v,\fa~v\in T_{\L}$.
\end{notation}

Let $\Sigma$ and $\Sigma'$ be $K$-qi sections over $\mfB$ which bounds $\L\cap X_{\mfB}$ (see $(L1)$ above). We set, say $\Sigma$, as $top$ and $\Sigma'$ as $bot$ to give an orientation on $\L\cap X_{\mfB}$, where the abbreviation $top$ and $bot$ is coming from `top' and `bottom' respectively. Then we have an orientation for each fiber geodesics of $\L\cap X_{\mfB}$ as `$bot$ to $top$'. In other words, let $t\in\mfT$ and $a\in B_t$. Then $bot(\L_{a,t}):=\Sigma'(a)$ and $top(\L_{a,t}):=\Sigma(a)$, and the geodesic $\L_{a,t}$ is oriented from `$bot(\L_{a,t})$ to $top(\L_{a,t})$'. 

Now we fix $u\in \mfT$ once and for all.  We define monotonic maps by induction on $d_T(u,v)$ as follows, where $v\in T_{\L}$. Let $v,w\in T_{\L}$ such that $d_T(u,v)<d_T(u,w)$ and $d_T(v,w)=1$. We also assume that $w\notin\mfT$, and it is mentioned in the end for the case $w\in\mfT$. Again the orientation to fiber geodesics of $\L\cap X_w$ depend on that of $\L\cap X_v$. Let $[\mfv,\mfw]$ be the edge joining $\mfv\in B_v$ and $\mfw\in B_w$. Let $\L_{\mfv,v}:=[x_{\mfv,v},y_{\mfv,v}]^f$ and $\L_{\mfw,w}:=[x_{\mfw,w},y_{\mfw,w}]^f$ such that $top(\L_{\mfv,v})=x_{\mfv,v}$ and $bot(\L_{\mfv,v})=y_{\mfv,v}$. Let $\bar{x}_{\mfv,v},~\bar{y}_{\mfv,v}\in\L_{\mfv,v}$ such that $d_{\mfv\mfw}(\bar{x}_{\mfv,v},x_{\mfw,w})\le K$ and $d_{\mfv\mfw}(\bar{y}_{\mfv,v},y_{\mfw,w})\le K$. (The existence of such points clear as $\L$ is a $(K,C,\ep)$-semicontinuous family.) Let $h_{wv}:\L_{\mfw,w}\ri\L_{\mfv,v}$ be a monotonic map (see Lemma \ref{monotonic-map-between-qgs}) sending $x_{\mfw,w}$ to $\bar{x}_{\mfv,v}$ and $y_{\mfw,w}$ to $\bar{y}_{\mfv,v}$ such that $d_{\mfv\mfw}(h_{wv}(x),x)\le k_{\ref{monotonic-map-between-qgs}}(\dl\pr_0,L\pr_0,K)$ for all $x\in\L_{\mfw,w}$. We fix this $h_{wv}$ once and for all for such $v,w$. The orientation in $\L\cap X_w$ depends on the order of the appearance of $\bar{x}_{\mfv,v}$ and $\bar{y}_{\mfv,v}$ in $\L_{\mfv,v}$. Let the $K$-qi sections $\gm_1$ and $\gm_2$ bound $\L\cap X_w$ such that $\gm_1(\mfw)=x_{\mfw,w}$ and $\gm_2(\mfw)=y_{\mfw,w}$. If there is an order $y_{\mfv,v}\le \bar{y}_{\mfv,v}<\bar{x}_{\mfv,v}\le x_{\mfv,v}$, then we set $\gm_1$ to be $top$ and $\gm_2$ to be $bot$ for $\L\cap X_w$. In other words, $top(\L_{a,w})=\gm_1(a)$ and $bot(\L_{a,w})=\gm_2(a)$, $a\in B_w$. If the order is $y_{\mfv,v}\le \bar{x}_{\mfv,v}<\bar{y}_{\mfv,v}\le x_{\mfv,v}$, then we set $\gm_2$ to be $top$ and $\gm_1$ to be $bot$ for $\L\cap X_w$. In other words, $top(\L_{a,w})=\gm_2(a)$ and $bot(\L_{a,w})=\gm_1(a)$, $a\in B_w$. However, by renaming, we always denote $\bar{x}_{\mfv,v}\in Im(h_{wv})$ for the closest point (in the induced path metric on $\L_{\mfv,v}$) to $x_{\mfv,v}$ and $\bar{y}_{\mfv,v}\in Im(h_{wv})$ for the closest point to $y_{\mfv,v}$. Otherwise, i.e., if $\bar{x}_{\mfv,v}=\bar{y}_{\mfv,v}$, then we set any one of $\gm_1,\gm_2$ as $top$ and the other one as $bot$. If $w\in\mfT$ then $v\in\mfT$. Then the monotonic map $h_{wv}:\L_{\mfw,w}\ri \L_{\mfv,v}$ sends $x_{\mfw,w}$ to $x_{\mfv,v}$ and $y_{\mfw,w}$ to $y_{\mfv,v}$. We let $top(\L):=\cup_{a\in B_v,~v\in T_{\L}} top(\L_{a,v})$ and $bot(\L):=\cup_{a\in B_v,~v\in T_{\L}}bot(\L_{a,v})$.

\begin{figure}[h]
	\includegraphics[width=10cm]{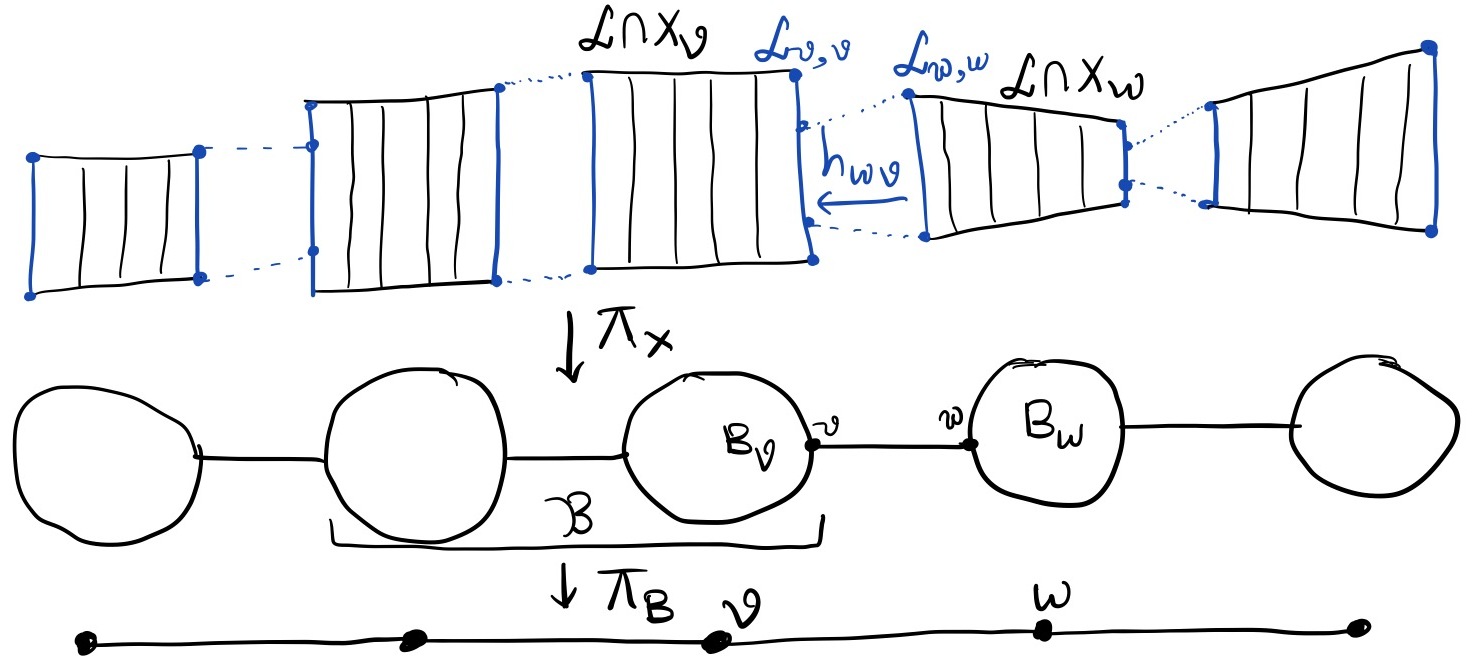}
	\centering
	\caption{Ladder}
	\label{ladderF}
\end{figure}

{\bf (L2)} \underline{{\bf Quasiisometric (qi) section in $\L$}}: Let $x\in \L$ such that $t=\pi(x)$. Suppose $s$ is the nearest point projection of $t$ on $\mfT$. Then there is a $K$-qi section through $x$ lying inside $\L$ over $B_x:=\mfB\cup B_{[s,t]}$. By a qi section in $\L$, we always mean that it obeys the order at the junction between two metric bundles given by the family $\{h_{wv}\}$. In other words, suppose $u\in\mfT$ is fixed (as mentioned above) and $\Sigma$ is a qi section in $\L$ over $B_1$, say. Let $[\mfv,\mfw]$ be the edge in $B_1$ joining $\mfv\in B_v$ and $\mfw\in B_w$ corresponding to the edge $[v,w]$ in $T$ such that $d_T(u,v)<d_T(u,w)$. Then $\Sigma(\mfv)=h_{wv}(\Sigma(\mfw))$. Thus $d_{\mfv\mfw}(\Sigma(\mfv),\Sigma(\mfw))\le k_{\ref{monotonic-map-between-qgs}}(\dl\pr_0,L\pr_0,K)$. As $max\{K,k_{\ref{monotonic-map-between-qgs}}(\dl\pr_0,L\pr_0,K)\}=k_{\ref{monotonic-map-between-qgs}}(\dl\pr_0,L\pr_0,K)$, $\Sigma$ would be a $k_{\ref{monotonic-map-between-qgs}}(\dl\pr_0,L\pr_0,K)$-qi section. By abusing notation, we still refer to $\Sigma$ as a $K$-qi section.\smallskip

For our reference, we state the following definition as explained above.
\begin{defn}[\bf Ladder]\label{ladder}
A ladder $\L\sse X$ with parameters $K\ge1,~C\ge0$ and $\epsilon\ge0$ is a $(K,C,\epsilon)$-semicontinuous family with a central base, say $\mfB$, along with a family of monotonic maps $\{h_{wv}\}$ (as described above) and satisfying $(L1)$ and $(L2)$.
	
Most of the time, we say that $\L$ is a $(K,C,\epsilon)$-ladder with a central base $\mfB$ keeping other things implicit. Occasionally, we denote $\L$ by $\L_K$ to emphasise $K$.
\end{defn}

One can think of $x_{\mfw,w}$ and $y_{\mfw,w}$ as uniformly close to nearest point projections (in $d_{\mfv\mfw}$-metric) of $x_{\mfv,v}$ and $y_{\mfv,v}$ on $F_{\mfw,w}$ respectively. (This uniform bound is measured in terms of $\epsilon$.) Again it follows from the definition of ladder that the pairs $([x_{\mfv,v},\bar{x}_{\mfv,v}]^f,F_{\mfw,w})$ and $([y_{\mfv,v},\bar{y}_{\mfv,v}]^f,F_{\mfw,w})$ are uniformly cobounded in $F_{\mfv\mfw}$. This is proved in the lemma below.

\begin{lemma}\label{for-subladder}
Given $K\ge1,~C\ge0$ and $\epsilon\ge0$ there are constants $C_{\ref{for-subladder}}=C_{\ref{for-subladder}}(K,C,\epsilon)$ and $\epsilon_{\ref{for-subladder}}=\epsilon_{\ref{for-subladder}}(K,C,\epsilon)$ such that the following hold.
	
Suppose $\L$ is a $(K,C,\epsilon)$-ladder with a central base $\mfB$ (as in Definition \ref{ladder}). Let $[v,w]$ be an edge in $T$. Suppose $[\mfv,\mfw]$ is the edge joining $\mfv\in B_v$ and $\mfw\in B_w$ such that $w\notin\mfT$ and $d_T(\mfT,v)<d_T(\mfT,w)$. Let $z\in[\bar{x}_{\mfv,v},\bar{y}_{\mfv,v}]^f\sse \L_{\mfv,v}$ and $z\pr\in\L_{\mfw,w}$ such that $h_{wv}(z\pr)=z$. (The notations are same as in Definition \ref{ladder}). Then:
	
$(1)$ The pairs $([x_{\mfv,v},\bar{x}_{\mfv,v}]^f,F_{\mfw,w})$ and $([y_{\mfv,v},\bar{y}_{\mfv,v}]^f,F_{\mfw,w})$ are $C_{\ref{for-subladder}}$-cobounded in $F_{\mfv\mfw}$.
	
$(2)$ $Hd_{\mfv\mfw}(P_{\mfw}([x_{\mfv,v},z]^f),[x_{\mfw,w},z\pr]^f)\le\epsilon_{\ref{for-subladder}}$ and $Hd_{\mfv\mfw}(P_{\mfw}([y_{\mfv,v},z]^f),[y_{\mfw,w},z\pr]^f)\le\epsilon_{\ref{for-subladder}}$.
\end{lemma}

\begin{proof}	
$(1)$ We prove only for the pair $([x_{\mfv,v},\bar{x}_{\mfv,v}]^f,F_{\mfw,w})$ as the other one has a similar proof. For ease of notation, let $x_1=x_{\mfv,v},x_2=x_{\mfw,w},x_3=\bar{x}_{\mfv,v}$. Let $x\pr_1=P_{\mfw}(x_1)$ and $x\prr_1\in\L_{\mfw,w}$ such that $d_{\mfv\mfw}(x\pr_1,x\prr_1)\le\epsilon$. Suppose $x_4=h_{wv}(x\prr_1)$ and $x\pr_4=P_{\mfw}(x_4)$. Note that $x_3\in[x_4,x_1]^f$. Now $d_{\mfv\mfw}(x\prr_1,x_4)\le K$ implies $d_{\mfv\mfw}(x\prr_1,x\pr_4)\le2K$. So $d_{\mfv\mfw}(x\pr_1,x\pr_4)\le\epsilon+2K$. Again, $[x_1,x_4]^f$ is a $L\pr_0$-quasigeodesic in $F_{\mfv\mfw}$ (see Lemma \ref{com-two-hyp-sps}). Then by \cite[Corollary $1.116$]{ps-kap}, there is a constant $C_1$ depending on $\dl\pr_0,\lm\pr_0,L\pr_0$ such that  $Hd_{\mfv\mfw}(P_{\mfw}([x_1,x_4]^f),[x\pr_1,x\pr_4]_{F_{\mfv\mfw}})\le C_1$. Let $x\pr_3=P_{\mfw}(x_3)$. Then $d_{\mfv\mfw}(x\pr_3,x\pr_1)\le d_{\mfv\mfw}(x\pr_3,[x\pr_1,x\pr_4]_{F_{\mfv\mfw}})+d_{\mfv\mfw}(x\pr_1,x\pr_4)\le C_1+\epsilon+2K$. Again applying \cite[Corollary $1.116$]{ps-kap} to $[x_1,x_3]^f$, the diameter of $P_{\mfw}([x_1,x_3]^f)$ in the metric $F_{\mfv\mfw}$ is bounded by $\le2C_1+(C_1+\epsilon+2K)=3C_1+\epsilon+2K$. Therefore, by Lemma \ref{small-imp-small}, we can take a constant depending on $\lm\pr_0,\dl\pr_0$ and $3C_1+\ep+2K$ as our required constant $C_{\ref{for-subladder}}$. So, we are done.
	
	$(2)$ We only prove that $Hd_{\mfv\mfw}(P_{\mfw}([x_{\mfv,v},z]^f),[x_{\mfw,w},z\pr]^f)$ is uniformly bounded as the other one has a similar proof. We continue with the notations used in $(1)$. From the above proof, we note that $d_{\mfv\mfw}(x\pr_1,x\pr_3)\le C_1+\epsilon+2K$. So $d_{\mfv\mfw}(x\pr_1,x_2)\le d_{\mfv\mfw}(x\pr_1,x\pr_3)+d_{\mfv\mfw}(x\pr_3,x_3)+d_{\mfv\mfw}(x_3,x_2)\le C_1+\epsilon+2K+K+K=C_1+\epsilon+4K$. Let $x\pr=P_{\mfw}(x)$ for $x\in[x_1,z]^f\sse\L_{\mfv,v}$. If $x\in[x_3,z]^f$, then $\exists~y\in[x_2,z\pr]^f\sse\L_{\mfw,w}$ such that $h_{wv}(y)=x$ and $d_{\mfv\mfw}(y,x)\le K$. So $d_{\mfv\mfw}(y,x\pr)\le2K$. If $x\in[x_1,x_3]$, then $d_{\mfv\mfw}(x\pr,x_2)\le d_{\mfv\mfw}(x\pr,x\pr_1)+d_{\mfv\mfw}(x\pr_1,x_2)\le C_{\ref{for-subladder}}+C_1+\epsilon+4K=2(2C_1+\epsilon+3K)$. So $P_{\mfw}([x_{\mfv,v},z]^f)$ is contained in $2(2C_1+\epsilon+3K)$-neighborhood of $[x_2,z\pr]^f\sse\L_{\mfw,w}$ in the metric of $F_{\mfv\mfw}$. For the other inclusion, let $\xi\pr\in[x_2,z\pr]^f\sse\L_{\mfw,w}$. Then there is $\xi\in[x_3,z]^f\sse\L_{\mfv,v}$ such that $h_{wv}(\xi\pr)=\xi$ and $d_{\mfv\mfw}(\xi,\xi\pr)\le K$. Then $\xi\pr$ is contained in $2K$-neighborhood of $P_{\mfw}([x_1,z]^f)$ in the path metric of $F_{\mfv\mfw}$. As $\xi\pr$ is arbitrary in $[x_2,z\pr]^f\sse\L_{\mfw,w}$. So, we take $\epsilon_{\ref{for-subladder}}=2(2C_1+\epsilon+3K)$.
\end{proof}

\begin{defn}[\bf Subladder]
Suppose $\L$ is a $(K,C,\ep)$-ladder. A subladder $\L\pr$ in $\L$ is a $(K\pr,C\pr,\epsilon\pr)$-ladder whose fiber geodesics are subsegments of the corresponding fiber geodesics of the ladder $\L$ and the family of monotonic maps are restrictions of the given family $\{h_{wv}\}$. The constants $K\pr$, $C\pr$ and $\ep\pr$ depend on the given ones (see Subsection \ref{general-ladder} for examples).
\end{defn}

\begin{defn}[\bf Girth and neck]\label{girth-neck}
Suppose $\L$ is a $(K,C,\epsilon)$-ladder with a central base $\mfB$. Let $B_1\sse \mfB$ and let $\ell(\al)$ denote the length of a fiber geodesic $\al$ in the corresponding fiber metric. The girth of the ladder $\L$ over $B_1$ is denoted by $\L^g|_{B_1}$ and defined as inf$\{\ell(\L_{b,v}):v\in \pi_B(B_1),~b\in B_v\cap B_1\}$. For $A\ge0$, the $A$-neck of the ladder $\L$ inside $B_1$ is denoted by $\L^n(A)|_{B_1}$ and defined as $\{b\in B_1:\ell(\L_{b,v})\le A,~\pi_B(b)=v\}$.
\end{defn}

Let $\L$ be a $(K,C,\epsilon)$-ladder. Let $x,y\in\L$ and  $\Sigma_x,\Sigma_y$ be $K$-qi sections through $x,y$ over $B_x,B_y$ respectively. Suppose $B_{xy}=B_x\cap B_y$. Then  $\bigcup\limits_{v\in\pi_B(B_{xy}),~b\in B_v}[F_{b,v}\cap\Sigma_x,F_{b,v}\cap\Sigma_y]_{F_{b,v}}$ forms a special $K_{\ref{getting-metric-graph-bundles}}(K)$-ladder over $B_{xy}$ (see Definition \ref{K-metric-graph-bundle}). 
We denote this special ladder by $\L_{xy}$. With these notations, we have the following lemma.
\begin{lemma}\label{neck-is-qc}
Given $K\ge1$, $C\ge0,~\epsilon\ge0$ and $A\ge M_K$ there is a constant $K_{\ref{neck-is-qc}}=K_{\ref{neck-is-qc}}(K,A)$ such that the following holds, where $M_K$ is coming from $K$-flaring condition.
	
Let $\L$ be a $(K,C,\epsilon)$-ladder with a central base $\mfB$. Then for $x,y\in\L,~\L_{xy}^n(A)|_{B_{xy}}$ is $K_{\ref{neck-is-qc}}$-quasiconvex in $B_{xy}$, and consequently, in $B$ as well.
\end{lemma}

\begin{proof}
If $\L_{xy}^n(A)|_{B_{xy}}=\emptyset$, then there is nothing to prove. Suppose $\L_{xy}^n(A)|_{B_{xy}}\ne\emptyset$ and $a,b\in\L_{xy}^n(A)|_{B_{xy}}$. Without loss of generality, we assume that $d^f(\Sigma_x(s),\Sigma_y(s))>A\ge M_K,~\fa~ s\in[a,b]\setminus\{a,b\}$. So by Lemma \ref{flaring-lemma} $(1)$, $d_B(a,b)\le\tau_{\ref{flaring-lemma}}(K,A)$. Hence one can take $K_{\ref{neck-is-qc}}:=\tau_{\ref{flaring-lemma}}(K,A)$.
\end{proof}

We finish this subsection by noting an interesting fact, which gives a criterion for a family of geodesic segments in the fibers to form a ladder.

\begin{lemma}\label{promoting-of-ladder}
Given $K\pr\ge1,~C\pr\ge0$ and $\epsilon\pr\ge0$, we have constants $k_{\ref{promoting-of-ladder}}=k_{\ref{promoting-of-ladder}}(K\pr)$, $c_{\ref{promoting-of-ladder}}=c_{\ref{promoting-of-ladder}}(C\pr)$ and $\varepsilon_{\ref{promoting-of-ladder}}=\varepsilon_{\ref{promoting-of-ladder}}(\epsilon\pr)$ such that the following holds.
	
Suppose $\L$ is a collection of geodesic segments in the fibers such that:
	
$(1)$ $T_{\L}:=\textrm{hull}(\pi(\L))$. For all $v\in T_{\L}$, $\L\cap X_v$ forms a special $K\pr$-ladder in $X_v$ over $B_v$ bounded by two $K\pr$-qi sections (see Definition \ref{K-metric-graph-bundle}). We also have a subtree $\mfT$ in $T_{\L}$ with the following. Suppose $v,w\in T_{\L}$ with $d_T(v,w)=1$. Let $[\mfv,\mfw]$ be the edge joining $\mfv\in B_v$ and $\mfw\in B_w$. Let $\L_{\mfv,v}:=\L\cap F_{\mfv,v}=[x_{\mfv,v},y_{\mfv,v}]^f$ and $\L_{\mfw,w}:=\L\cap F_{\mfw,w}=[x_{\mfw,w},y_{\mfw,w}]^f$. If $v,w\in\mfT$, then $d_{\mfv\mfw}(x_{\mfv,v},x_{\mfw,w})$ $\le K\pr$ and $d_{\mfv\mfw}(y_{\mfv,v},y_{\mfw,w})\le K\pr$. Otherwise, if $d_T(v,\mfT)<d_T(w,\mfT)$, then $x_{\mfw,w},y_{\mfw,w}\in N^{\mfv\mfw}_{K\pr}(\L_{\mfv,v})$.
	
$(2)$ In the second part of $(1)$, where $d_T(v,\mfT)<d_T(w,\mfT)$, we have $Hd_{\mfv\mfw}(P_{\mfw}(\L_{\mfv,v}),\L_{\mfw,w})\le\epsilon\pr.$
	
$(3)$ Let $v\in T_{\L},w\notin T_{\L}$ such that $d_T(v,w)=1$. Let $[\mfv,\mfw]$ be the edge joining $\mfv\in B_v$ and $\mfw\in B_w$. Then the pair $(\L_{\mfv,v},F_{\mfw,w})$ is $C\pr$-cobounded in the path metric of $F_{\mfv\mfw}$.
	
	Then $\L$ is a $(k_{\ref{promoting-of-ladder}},c_{\ref{promoting-of-ladder}},\varepsilon_{\ref{promoting-of-ladder}})$-ladder with a central base $\mfB:=\pi_B^{-1}(\mfT)$.
\end{lemma}

\begin{proof}
We only need to find $k_{\ref{promoting-of-ladder}}$ and set an orientation on the fiber geodesics along with the family $\{h_{wv}\}$ of monotonic maps. Fix $u\in\mfT$. Suppose $[v,w]$ is an edge in $T_{\L}$ with $d_T(u,v)<d_T(u,w)$. Let us fix once and for all  $\bar{x}_{\mfv,v},\bar{y}_{\mfv,v}\in\L_{\mfv,v}$ such that $d_{\mfv\mfw}(x_{\mfw,w},\bar{x}_{\mfv,v})\le K\pr,~d_{\mfv\mfw}(y_{\mfw,w},\bar{y}_{\mfw,w})\le K\pr$ with $\bar{x}_{\mfv,v}\in\L_{\mfv,v}\cap[x_{\mfv,v},\bar{y}_{\mfv,v}]^f$, and in the case if $v,w\in\mfT$, then $\bar{x}_{\mfv,v}=x_{\mfv,v}$ and $\bar{y}_{\mfv,v}=y_{\mfv,v}$. We inductively fix an orientation as discussed in Subsection \ref{discussion-on-ladder}.
	
Now we apply Lemma \ref{monotonic-map-between-qgs} on $F_{\mfv\mfw}$ and $L\pr_0$-quasigeodesics $\L_{\mfv,v}$ and $\L_{\mfw,w}$. So we get a monotonic map, say $h_{wv}:\L_{\mfw,w}\ri\L_{\mfv,v}$, such that $h_{wv}$ is a $L_{\ref{monotonic-map-between-qgs}}(\dl\pr_0,L\pr_0,K\pr)$-quasiisometry. Also, we have $d_{\mfv\mfw}(x,h_{wv}(x))\le k_{\ref{monotonic-map-between-qgs}}(\dl\pr_0,L\pr_0,K\pr)$ for all $x\in \L_{\mfw,w}$, and $h_{wv}(x_{\mfw,w})=\bar{x}_{\mfv,v},~~h_{wv}(y_{\mfw,w})=\bar{y}_{\mfv,v}$. We fix once and for all such maps $h_{wv}$.
	
Therefore, one can take $k_{\ref{promoting-of-ladder}}=max\{k_{\ref{monotonic-map-between-qgs}}(\dl\pr_0,L\pr_0,K\pr),K\pr\}$ and $c_{\ref{promoting-of-ladder}}=C\pr$, $\varepsilon_{\ref{promoting-of-ladder}}=\epsilon\pr$.
\end{proof}


\section{Hyperbolicity of ladder}\label{hyperbolicity-of-ladder}

In this section, we show that a uniform neighborhood of a ladder (see Definition \ref{ladder}) is uniformly hyperbolic with the induced path metric. We divide the proof into two cases. $(1)$ Ladder with small girth (see Definition \ref{girth-neck}); here we construct paths for any pair of points in the ladder and prove that they satisfy the conditions of Proposition \ref{combing} (see Proposition \ref{hyp_small_ladder}). $(2)$ For a general ladder, we subdivide it into (uniformly) small (but not too small) girth ladders and show that they satisfy all conditions of Proposition \ref{combi-hyp-sps} (see Theorem \ref{general-ladder-is-hyp}). 
Let $K\ge1,~C\ge0,~\epsilon\ge0$, and let $\L_K$ be a $(K,C,\ep)$-ladder with a central base $\mfB$. In this section, we fix the {\bf notation $\bm{L}_{\bm{Kr}}:=\bm{N}_{\bm{r}}(\bm{\L}_{\bm{K}})$} for $r\ge0$. Additionally, we will use the same notations as introduced in Definition \ref{ladder} and in Notation \ref{notation-ladder} for ladders. In view of Remark \ref{bigsmall-subsp}, in this section, we require the tree of metric bundles $(X,B,T)$ to satisfy $C^{(9)}_{\ref{qi-sec-inside-lad-len-lift}}(K)$-flaring condition (see below). 

\subsection{Hyperbolicity of ladders (small girth)}\label{small-ladder}
We refer to Remark \ref{iteration} for the notation of $C^{(i)}_{\ref{qi-sec-inside-lad-len-lift}}(K)$, $i\in\N\cup\{0\}$. Fix $\kappa=C^{(3)}_{\ref{qi-sec-inside-lad-len-lift}}(K)$ for this Subsection \ref{small-ladder}. Given $A_0\ge0$, we set $A=max\{M_{C^{(i)}_{\ref{qi-sec-inside-lad-len-lift}}(K)},A_0:i=0,1,2,3\}$, where $M_{C^{(i)}_{\ref{qi-sec-inside-lad-len-lift}}(K)}$ is coming from $C^{(i)}_{\ref{qi-sec-inside-lad-len-lift}}(K)$-flaring condition. Note that $C^{(i+1)}_{\ref{qi-sec-inside-lad-len-lift}}(K)\ge C^{(i)}_{\ref{qi-sec-inside-lad-len-lift}}(K)$.
\begin{prop}\label{hyp_small_ladder}
	Suppose $R\ge2\kappa$ and $A$ as above. Then there exists $\dl_{\ref{hyp_small_ladder}}=\dl_{\ref{hyp_small_ladder}}(K,A_0,R)$ such that the following holds.
	
	Suppose $\L^g|_{\mfB}\le A_0$ (see Definition \ref{girth-neck}). Then $L_{KR}:=N_R(\L_K)$ is $\dl_{\ref{hyp_small_ladder}}$-hyperbolic with respect to the path metric induced from $X$.
\end{prop}

\begin{proof}
	{\bf Idea of the proof.}
The proof of this proposition is long. So we will break it up into several lemmata. We first define a family of discrete paths, say $c(x,y)$, for a pair of distinct points $x,y\in\mathcal{L}$. These paths are concatenation of three subpaths: two of them lie in two qi sections in $\L$ through $x$ and $y$, and the other one is a uniform bounded length path in fiber. Now we fix such paths once and for all. Then we show that this family of paths (after making them continuous) satisfies Proposition \ref{combing}. Hence the hyperbolicity of $L_{KR}$ follows. This construction of paths would very be helpful when it comes to the computation of their lengths since it would easily come from qi sections. For instance, showing Hausdorff closeness or slimness of triangles of such paths would become easy when their parts lie in the same qi sections. But for the parts which do not lie in the same qi sections, we use the geometry of their necks in $B$.
	
	\begin{notation}\label{notation-hyp-of-small-ladder}
		Let $x,y,z\in\L$. We use the following notations for Proposition \ref{hyp_small_ladder}. For a fixed $u\in\mfT$,  $B_s=\mfB\cup\pi_B^{-1}([u,\pi(s)]),~s\in\{x,y,z\};~B_{xy}=B_x\cap B_y,~B_{xyz}=B_x\cap B_y\cap B_z$. We use $\bar{x}$ to denote $\pi_X(x)$ (projection under $\pi_X$), and the same way we have $\bar{y}=\pi_X(y)$ and $\bar{z}=\pi_X(z)$. 
		We will denote the path metric on $L_{KR}$ induced from $X$ by $d\pr$.
	\end{notation}
\begin{defn}[Family of paths]\label{family-of-paths}
	Let $x,y\in\L$. Suppose $\Sigma_x,\Sigma_y$ are $K$-qi sections in $\L$ over $B_x, B_y$ through $x,y$ respectively. Let $u_{xy}$ be the center of the triangle $\triangle(u,\pi(x),\pi(y))$ for some $u\in\mfT$ provided $[\pi(x),\pi(y)]\cap\mfT=\emptyset$, otherwise, $u_{xy}$ be the nearest point projection of $\pi(x)$ on $\mfT$. Let $U_{xy}=\L_{xy}^n(A)|_{B_{xy}}$ be the $A$-neck of the special ladder $\L_{xy}$ bounded by $\Sigma_x,\Sigma_y$ (see Definition \ref{girth-neck}) over the common base $B_{xy}$. 
	Then $U_{xy}$ is $K_{\ref{neck-is-qc}}(K,A)$-quasiconvex (see Lemma \ref{neck-is-qc}). Let $\bm{t}_{\bm{xy}}$ be a nearest point projection of $\bar{x}$ on $U_{xy}$ and let  $\bm{v}_{\bm{xy}}:=\pi_B(t_{xy})$. We take $K$-qi lifts $\tilde{\al}_{xy}$ and $\tilde{\gm}_{xy}$ of geodesics $\bm{\al}_{\bm{xy}}:=[\bar{x},t_{xy}]_B$ and $\bm{\gm}_{\bm{xy}}:=[\bar{y},t_{xy}]_B$ in $\Sigma_x$ and $\Sigma_y$ respectively. Denote $\mu_{xy}=[\Sigma_x(t_{xy}),\Sigma_y(t_{xy})]^f\sse\L_{t_{xy},v_{xy}}$.
	
In general, $\tilde{\al}_{xy}$ and $\tilde{\gm}_{xy}$ are not continuous paths. According to our requirement in Proposition \ref{combing}, we make them continuous as following. Fix points $\bar{x}=a_1,a_2,\cdots,a_n=t_{xy}$ on $[\bar{x},t_{xy}]_B$ such that $d_B(a_i,a_{i+1})=1$ for $1\le i\le n-2$ and $d_B(a_{n-1},a_n)\le 1$.
	Consider the path $[\tilde{\al}_{xy}]=[\Sigma_x(a_1),\Sigma_x(a_2)]\cup[\Sigma_x(a_2),\Sigma_x(a_3)]\cup\cdots\cup[\Sigma_x(a_{n-1}),\Sigma_x(a_n)]$. Now $[\tilde{\al}_{xy}]$ is a continuous path. Similarly, we have the continuous path $[\tilde{\gm}_{xy}]$ path corresponding to $\tilde{\gm}_{xy}$. 
	
We define $c(x,y):=\tilde{\al}_{xy}\cup\mu_{xy}\cup\tilde{\gm}_{xy}$ and $[c(x,y)]=[\tilde{\al}_{xy}]\cup\mu_{xy}\cup[\tilde{\gm}_{xy}]^{-}$, where $[\tilde{\gm}_{xy}]^-$ denotes the path corresponding to $[\tilde{\gm}_{xy}]$ with opposite orientation. We see that there is an asymmetry in the definition of $c(x,y)$ $($resp. $[c(x,y)])$ and the number of choices are involved. However, for each unordered pair $\{x,y\}$, we fix once and for all a choices and choose either $[c(x,y)]$ or $[c(y,x)]$ $($corresponding to $c(x,y)$ or $c(y,x))$ as the path joining $x$ and $y$. So our family of paths consists of path $[c(x,y)]$ for all distinct $x, y\in\L$.
\end{defn}
	
{\bf What we will do.} {\em Note that $Hd_{L_{KR}}(c(x,y),[c(x,y)])$ is uniformly bounded. Therefore, it is enough to prove $(a)$ the arc-length parametrization of $[c(x,y)]$ is properly embedded for condition $(1)$ and $(b)$ the slimness of paths $c(x,y)$ for condition $(2)$ of Proposition \ref{combing} to conclude the hyperbolicity of $L_{KR}$. So all the time we will work with discrete paths $c(x,y)$ except for the condition $(1)$ as mentioned above}.\medskip

	
	Different choices of geodesics joining $\bar{x},t_{xy}$ and $\bar{y},t_{xy}$ give rise to path joining $x,y$, that are $2K\dl_0$ (uniformly) Hausdorff close to $c(x,y)$. However, we will have to consider the following two natural questions.
	\begin{enumerate}
		\item  Are $c(x,y)$ and $c(y,x)$ uniformly Hausdorff close?
		
		\item Suppose $\Sigma\pr_x$ and $\Sigma\pr_y$ are two different $K$-qi sections through $x$ and $y$ respectively lying inside $\L$. Let $c\pr(x,y)$ be a path joining $x$ and $y$ coming from the construction above for the qi sections $\Sigma\pr_x,~\Sigma\pr_y$. Are $c(x,y)$ and $c\pr(x,y)$ uniformly Hausdorff close?
	\end{enumerate}

These two questions are proven in \cite[Section $3$]{pranab-mahan} for the case of metric graph bundles (see \cite[Definition $1.5$]{pranab-mahan}). However, we will establish that these are also true in our case (see Lemma \ref{Hd-close-of-c(x,y)-c(y,x)} and Corollary \ref{Hd-close-of-c(x,y)-c(x,y)}). The proof idea involves playing with quasiconvex subsets $U_{xy}$ and lifts in qi sections.
	
	\begin{lemma}\label{Hd-close-of-c(x,y)-c(y,x)}
With the hypothesis of Proposition \ref{hyp_small_ladder}, there exists a constant $D_{\ref{Hd-close-of-c(x,y)-c(y,x)}}=D_{\ref{Hd-close-of-c(x,y)-c(y,x)}}(\kappa,A)$ such that $$Hd\pr(c(x,y),c(y,x))\le D_{\ref{Hd-close-of-c(x,y)-c(y,x)}}.$$
	\end{lemma}
	\begin{proof}
We can think of $\Sigma_x,\Sigma_y$ as $(K\le)~\kappa$-qi sections in $\L$ through $x,y$ respectively. By our notation, $t_{yx}$ is a nearest point projection of $\bar{y}$ on $U_{yx}(=U_{xy})$. Let $\al=[t_{xy},t_{yx}]$. Also, $\al_{yx}=[\bar{y},t_{yx}],\gm_{yx}=[t_{yx},\bar{x}]$; and $\tilde{\al}_{yx},\tilde{\gm}_{yx}$ are lifts of $\al_{yx},\gm_{yx}$ in $\Sigma_y,\Sigma_x$ respectively. Further, suppose  $v_{yx}:=\pi_B(t_{yx})$ and $\mu_{yx}:=[\Sigma_y(t_{yx}),\Sigma_x(t_{yx})]^f\sse \L_{t_{yx},v_{yx}}$. Finally, $c(y,x)=\tilde{\al}_{yx}\cup\mu_{yx}\cup\tilde{\gm}_{yx}$. Since $U_{xy}$ is $~K_{\ref{neck-is-qc}}(\kappa,A)$-quasiconvex, the arc-length parametrizations of $\al_{xy}\cup\al$ and $\al_{yx}\cup\al$ are $(3+2K_{\ref{neck-is-qc}}(\kappa,A))$-quasigeodesics (by \cite[Lemma $1.31$ $(2)$]{pranab-mahan}). So by Lemma \ref{ml}, there is $D$ depending on $\dl_0$, $3+2K_{\ref{neck-is-qc}}(\kappa,A)$ such that $Hd_B(\gm_{yx},\al_{xy}\cup\al)\le D$ and $Hd_B(\gm_{xy},\al_{yx}\cup\al)\le D$. Again, $\al\sse B_{xy}$ and $t_{xy},t_{yx}\in U_{xy}$. So by Lemma \ref{flaring-lemma} $(2)$, $\fa ~s\in\al,~d^f(\Sigma_x(s),\Sigma_y(s))\le R_{\ref{flaring-lemma}}(\kappa,A)$. Below, we prove that $c(x,y)$ lies inside uniform neighborhood of $c(y,x)$. Then by the symmetry of proof we will be done.
		
Let $\xi\in c(x,y)\cap\tilde{\al}_{xy}$ and $\eta=\pi_X(\xi)$. Then $\exists~\eta\pr\in\gm_{yx}$ such that $d_B(\eta,\eta\pr)=d_{B_{xy}}(\eta,\eta\pr)\le D$. So by taking $\kappa$-qi lift of $[\eta,\eta\pr]$ in $\Sigma_x$ (see Lemma \ref{qi-sec-inside-lad-len-lift} $(3)$), we get $d\pr(\xi,c(y,x))\le d\pr(\xi,\tilde{\gm}_{yx})\le2\kappa D$.
		
Let $\xi\in c(x,y)\cap\mu_{xy}$. Then from above, $d\pr(\xi,c(y,x))\le2\kappa D+A$.
		
Finally, let $\xi\in c(x,y)\cap\tilde{\gm}_{xy}$ and $\eta=\pi_X(\xi)$. Then $\exists~\eta\pr\in\al_{yx}\cup\al$ such that $d_B(\eta,\eta\pr)\le\dl_0$. If $\eta\pr\in\al_{yx}$, then by taking $\kappa$-qi lift of $[\eta,\eta\pr]$ in $\Sigma_y$, we get $d\pr(\xi,c(y,x))\le d\pr(\xi,\tilde{\al}_{yx})\le2\kappa\dl_0$. Again if $\eta\pr\in\al$, then $\eta\pr$ is further $D$-close to $\gm_{yx}$, i.e. $\exists~\eta\prr\in\gm_{yx}$ such that $d_B(\eta\pr,\eta\prr)\le D$. Therefore, by taking lifts of geodesics $[\eta,\eta\pr]$ and $[\eta\pr,\eta\prr]$ in $\Sigma_y$ and $\Sigma_x$ respectively, we get
		\begin{eqnarray*}
			d\pr(\xi,\tilde{\gm}_{yx})&\le&d\pr(\Sigma_y(\eta),\Sigma_x(\eta\prr))\\
			&\le&d\pr(\Sigma_y(\eta),\Sigma_y(\eta\pr))+d\pr(\Sigma_y(\eta\pr),\Sigma_x(\eta\pr))+d\pr(\Sigma_x(\eta\pr),\Sigma_x(\eta\prr))\\
			&\le&2\kappa\dl_0+R_{\ref{flaring-lemma}}(\kappa,A)+2\kappa D \textrm{ (since $\eta\pr\in\al$)}\\
			&=&2\kappa(\dl_0+D)+R_{\ref{flaring-lemma}}(\kappa,A)
		\end{eqnarray*}
		Therefore, we can take $D_{\ref{Hd-close-of-c(x,y)-c(y,x)}}:=2\kappa(\dl_0+D)+R_{\ref{flaring-lemma}}(\kappa,A)$ so that $Hd\pr(c(x,y),c(y,x))\le D_{\ref{Hd-close-of-c(x,y)-c(y,x)}}$.
	\end{proof}
	
	To prove (2), we first show that if we change one of the qi sections, then the path we get is uniformly Hausdorff close to the other one. In other words, suppose $\Sigma\pr_x$ is another qi section through $x$. Let $c(x,y)$ and $c_1(x,y)$ be paths coming from the pairs $(\Sigma_x,\Sigma_y)$ and $(\Sigma\pr_x,\Sigma_y)$ respectively. Then $Hd\pr(c(x,y),c_1(x,y))\le D$ for some uniform constant $D\ge0$. Hence we complete (2) by applying twice this process. Indeed, $Hd\pr(c(x,y),c\pr(x,y))\le Hd\pr(c(x,y),c_1(x,y))+Hd\pr(c_1(x,y),c\pr(x,y))\le2D$, where $\Sigma\pr_y$ is another qi section through $y$ and the path $c\pr(x,y)$ is coming from the pair $(\Sigma\pr_x,\Sigma\pr_y)$.
	
	\begin{lemma}\label{Hd-close-of-c(x,y)-c_1(x,y)}
		With the hypothesis of Proposition \ref{hyp_small_ladder}, there is a constant $D_{\ref{Hd-close-of-c(x,y)-c_1(x,y)}}=D_{\ref{Hd-close-of-c(x,y)-c_1(x,y)}}(\kappa,A)$ such that $$Hd\pr(c(x,y),c_1(x,y))\le D_{\ref{Hd-close-of-c(x,y)-c_1(x,y)}}.$$
	\end{lemma}
	\begin{proof}
		Here also we consider $\Sigma_x,\Sigma\pr_x$ and $\Sigma_y$ as $\kappa$-qi sections. Let the special ladder formed by pair $(\Sigma\pr_x,\Sigma_y)$ restricted over $B_{xy}$  be $\L\pr_{xy}$ (see Lemma \ref{neck-is-qc}). Let $V$ be the $A$-neck of the ladder $\L\cap X_{\mfB}$ (restriction of $\L$ on $\mfB$, see also Notation \ref{notation-ladder}) and $U\pr_{xy}$ be that of $\L\pr_{xy}$. Notice that $V\sse U_{xy}\cap U\pr_{xy}$. We assume that $t\pr_{xy}$ is a nearest point projection of $\bar{x}$ on $U\pr_{xy}$. Let  $\al\pr_{xy}=[\bar{x},t\pr_{xy}]$ and $\gm\pr_{xy}=[\bar{y},t\pr_{xy}]$. First, we prove that $d_B(t_{xy},t\pr_{xy})$ is uniformly small.
		
		\emph{$d_B(t_{xy},t\pr_{xy})$ is uniformly small}: We fix a point $t\in V$ and a geodesic $\al=[\bar{x},t]$. Now for $s\in\{t,\bar{x}\}$, $d^f(\Sigma_x(s),\Sigma\pr_x(s))\le A$.  So by Lemma \ref{flaring-lemma} $(2)$, for all $s\in\al$, $d^f(\Sigma_x(s),\Sigma\pr_x(s))\le R_{\ref{flaring-lemma}}(\kappa,A)$. Again $U_{xy},U\pr_{xy}$ are $K_{\ref{neck-is-qc}}(\kappa,A)$-quasiconvex and so by \cite[Lemma $1.31$ $(2)$]{pranab-mahan}, the arc-length parametrizations of $\al_{xy}\cup[t_{xy},t]$ and $\al\pr_{xy}\cup[t\pr_{xy},t]$ are $(3+2K_{\ref{neck-is-qc}}(\kappa,A))$-quasigeodesics. Thus by Lemma \ref{ml}, there is $D$ depending on $\dl_0$, $3+2K_{\ref{neck-is-qc}}(\kappa,A)$ such that $Hd_B(\al,\al_{xy}\cup[t_{xy},t])\le D$ and $Hd_B(\al,\al\pr_{xy}\cup[t\pr_{xy},t])\le D$. Then $Hd(\al_{xy}\cup[t_{xy},t],\al\pr_{xy}\cup[t\pr_{xy},t])\le2D$. Hence $\exists~t_0\in\al_{xy}$ such that $d_B(t\pr_{xy},t_0)\le3D+\dl_0$ or $\exists~t_0\in\al\pr_{xy}$  such that $d_B(t_{xy},t_0)\le3D+\dl_0$. Without loss of generality, we assume that $d_B(t_{xy},t_0)\le3D+\dl_0$ for $t_0\in\al\pr_{xy}$. Since $T$ is tree and $B_v$'s are isometrically embedded in $B$, we can take $t_0\in B_y$. In particular, $[t_{xy},t_0]\sse B_{xy}$.
		
		Again for $s\in\al\pr_{xy},~\exists~s\pr\in\al$ such that $d_B(s,s\pr)\le D$. By taking lifts of geodesic $[s,s\pr]$ in $\Sigma_x$ and $\Sigma\pr_x$ (see Lemma \ref{qi-sec-inside-lad-len-lift} $(3)$), $d_X(\Sigma_x(s),\Sigma\pr_x(s))\le d_X(\Sigma_x(s),\Sigma_x(s\pr))+d_X(\Sigma_x(s\pr),\Sigma\pr_x(s\pr))+d_X(\Sigma\pr_x(s\pr),\Sigma\pr_x(s))\le2D\dl_0+R_{\ref{flaring-lemma}}(\kappa,A)+2D\dl_0=D_1$ (say). As fibers are $\phi$-properly embedded, $d^f(\Sigma_x(s),\Sigma\pr_x(s))$ $\le \phi(D_1)$. In particular, $$d_X(\Sigma_x(t_0),\Sigma\pr_x(t_0))\le D_1\textrm{ and }d^f(\Sigma_x(t_0),\Sigma\pr_x(t_0))\le \phi(D_1).$$
		
Note that $[t_{xy},t_0]\sse B_{xy}$. Now taking lifts of $[t_0,t_{xy}]$ in $\Sigma_x$ and $\Sigma_y$, we have $d_X(\Sigma_x(t_0),\Sigma_x(t_{xy}))$ $\le2\kappa(3D+\dl_0)$ and $d_X(\Sigma_y(t_{xy}),\Sigma_y(t_0))\le2\kappa(3D+\dl_0)$. Again, $d^f(\Sigma_y(t_{xy}),\Sigma_x(t_{xy}))\le A$. Therefore, combining all these inequalities, we have $d_X(\Sigma\pr_x(t_0),\Sigma_y(t_0))\le D_1+4\kappa(3D+\dl_0)+A=D_2$ (say). So $d^f(\Sigma\pr_x(t_0),\Sigma_y(t_0))\le \phi(D_2)$. Then by Lemma \ref{flaring-lemma} $(1)$, $d_B(t_0,t\pr_{xy})\le\tau_{\ref{flaring-lemma}}(\kappa,\phi(D_2))$. Hence $d_B(t_{xy},t\pr_{xy})\le d_B(t_{xy},t_0)+d_B(t_0,t\pr_{xy})\le D_3$ where $D_3=3D+\dl_0+\tau_{\ref{flaring-lemma}}(\kappa,\phi(D_2))$.
		
		Let us come back to the proof of Hausdorff closeness of $c(x,y)$ and $c_1(x,y)$. We only prove that $c(x,y)$ lies inside uniform neighborhood of $c_1(x,y)$. Then by the symmetry of the proof we will be done.
		
		Let $\xi\in c(x,y)\cap\tilde{\al}_{xy}$ and $\eta=\pi_X(\xi)$. Then $d_B(\eta,\eta\pr)\le D_3+\dl_0$ for some $\eta\pr\in\al\pr_{xy}$. Since $\eta\pr\in\al\pr_{xy}$, from the above paragraph, $d^f(\Sigma_x(\eta\pr),\Sigma\pr_x(\eta\pr))\le \phi(D_1)$. Therefore, by taking lift of geodesic $[\eta,\eta\pr]$ in $\Sigma_x$ (see Lemma \ref{qi-sec-inside-lad-len-lift} $(3)$), we get $d\pr(\xi,c_1(x,y))\le d\pr(\Sigma_x(\eta),\Sigma\pr_x(\eta\pr))\le d\pr(\Sigma_x(\eta),\Sigma_x(\eta\pr))+d^f(\Sigma_x(\eta\pr),\Sigma\pr_x(\eta\pr))\le2\kappa(D_3+\dl_0)+\phi(D_1)$.
		
		Now let $\xi\in c(x,y)\cap\tilde{\gm}_{xy}$ and $\eta=\pi_X(\xi)$. Then $\exists~\eta\pr\in\gm\pr_{xy}$ such that $d_B(\eta,\eta\pr)\le D_3+\dl_0$. Taking lift of $[\eta,\eta\pr]$ in $\Sigma_y$, we get $d\pr(\xi,c_1(x,y))\le d\pr(\Sigma_y(\eta),\Sigma_y(\eta\pr))\le2\kappa(D_3+\dl_0)$.
		
		Finally, we assume that $\xi\in c(x,y)\cap\mu_{xy}$. Then $d\pr(\xi,c_1(x,y))\le2\kappa(D_3+\dl_0)+A$.
		
		We note that $\phi(D_1)>A$. Therefore, $c(x,y)\sse N_{2\kappa(D_3+\dl_0)+\phi(D_1)}(c_1(x,y))$. Hence, we can take $D_{\ref{Hd-close-of-c(x,y)-c_1(x,y)}}:=2\kappa(D_2+\dl_0)+\phi(D_1)$.
	\end{proof}
	
	\begin{cor}\label{Hd-close-of-c(x,y)-c(x,y)}
With the hypothesis of Proposition \ref{hyp_small_ladder}, there exists a constant $D_{\ref{Hd-close-of-c(x,y)-c(x,y)}}=D_{\ref{Hd-close-of-c(x,y)-c(x,y)}}(\kappa,A)$ such that Hausdorff distance between any two paths joining $x,y\in\L$ coming from the path-construction with $\kappa$-qi sections through $x,y$ is bounded by $D_{\ref{Hd-close-of-c(x,y)-c(x,y)}}$ in the path metric of $L_{KR}$.
	\end{cor}
	\begin{proof}
		We can take $D_{\ref{Hd-close-of-c(x,y)-c(x,y)}}=D_{\ref{Hd-close-of-c(x,y)-c(y,x)}}(\kappa,A)+2D_{\ref{Hd-close-of-c(x,y)-c_1(x,y)}}(\kappa,A)$.
	\end{proof}
	Now we show that this family of paths satisfies the fellow-traveling property (\hspace{-.08mm}\cite[Definition $1.60$]{ps-kap}). In other words, any two such paths whose starting points are same and ending points are at uniform distance are uniformly Hausdorff close.
	
\begin{prop}[Fellow-traveling property]\label{fellow-travelling-property}
For all $r\ge0$ there exists a constant $D_{\ref{fellow-travelling-property}}=D_{\ref{fellow-travelling-property}}(\kappa,A,r)$ such that the following holds.
		
With the hypothesis of Proposition \ref{hyp_small_ladder}, if $x,y,z\in\L$ such that $d_X(x,y)\le r$, then $$Hd\pr(c(x,z),c(y,z))\le D_{\ref{fellow-travelling-property}}.$$
	\end{prop}
	
	\begin{proof}
We will be working with $\Sigma_x,\Sigma_y,\Sigma_z$ as $\kappa$-qi sections respectively through $x,y,z$ over $B_x$, $B_y$, $B_z$ inside the ladder $\L$ (explained in Case $1$ below). We consider the following three cases depending on the position of $\bar{x}=\pi_X(x)$ and $\bar{y}=\pi_X(y)$.
		
{\bf Case 1}: Let $\bar{x}=\bar{y}$. Since the fibers are $\phi$-properly embedded, without loss of generality, we assume that $d^f(x,y)\le r$. Applying Lemma \ref{qi-sec-inside-lad-len-lift} $(2)$, we may assume that these sections satisfy the following inclusion property when we restrict them to $B_{xyz}$. We have three possibilities $\Sigma_y|_{B_{xyz}}\sse\L_{xz}|_{B_{xyz}},$ $\Sigma_x|_{B_{xyz}}\sse\L_{yz}|_{B_{xyz}}$ or $\Sigma_z|_{B_{xyz}}\sse\L_{xy}|_{B_{xyz}}$. (To get this one has to consider $\Sigma_x,\Sigma_y,\Sigma_z$ as $\kappa=C^{(3)}_{\ref{qi-sec-inside-lad-len-lift}}(K)$-qi sections instead $K$-qi sections.) In the Subcase $(1A)$ below, we will see that the proof for the inclusions $\Sigma_y|_{B_{xyz}}\sse\L_{xz}|_{B_{xyz}}$ and $\Sigma_x|_{B_{xyz}}\sse\L_{yz}|_{B_{xyz}}$ are similar. So we proof this proposition when $\Sigma_y|_{B_{xyz}}\sse\L_{xz}|_{B_{xyz}}$ and $\Sigma_z|_{B_{xyz}}\sse\L_{xy}|_{B_{xyz}}$ in the following subcases.
		
\emph{Subcase} (1A): Suppose $\Sigma_y|_{B_{xyz}}\sse\L_{xz}|_{B_{xyz}}$. Recall that $U_{xy}=\L^n_{xy}(A)|_{B_{xy}},~U_{xz}=\L^n_{xz}(A)|_{B_{xz}}$ and $U_{yz}=\L^n_{yz}(A)|_{B_{yz}}$ are $K_{\ref{neck-is-qc}}(\kappa,A)$-quasiconvex (Lemma \ref{neck-is-qc}). Since $\Sigma_y|_{B_{xyz}}\sse\L_{xz}|_{B_{xyz}}$, so we have $U_{xz}\sse U_{xy}\cap U_{yz}$. By our notation, $t_{xz}$ and $t_{yz}$ are nearest point projections of $\bar{x}$ and $\bar{y}$ on $U_{xz}$ and $U_{yz}$ respectively. Hence for $M=max\{r,A\}$, $\fa~s\in\al_{xz},~d^f(\Sigma_x(s),\Sigma_y(s))\le R_{\ref{flaring-lemma}}(\kappa,M)$ (Lemma \ref{flaring-lemma} $(2)$).
		
		\emph{Claim}: $d_B(t_{xz},t_{yz})$ is uniformly bounded.
		
\emph{Proof of the claim}: Since $\al_{yz}\cup[t_{yz},t_{xz}]$ is $(3+2K_{\ref{neck-is-qc}}(\kappa,A))$-quasigeodesic, $\exists$ $t\in\al_{xz}$ such that $d_B(t_{yz},t)\le D$ for some $D$ depending on $\dl_0$ and $3+2K_{\ref{neck-is-qc}}(\kappa,A)$ (Lemma \ref{ml}). Since $T$ is a tree and $B_v$'s are isometrically embedded in $B$, we can take $t\in\al_{xz}\cap B_{yz}$. Then by taking lifts of geodesic $[t_{yz},t]$ in $\Sigma_y$ and $\Sigma_z$, we get $d\pr(\Sigma_y(t),\Sigma_z(t))\le d\pr(\Sigma_y(t),\Sigma_y(t_{yz}))+d^f(\Sigma_y(t_{yz}),\Sigma_z(t_{yz}))+d\pr(\Sigma_z(t_{yz}),\Sigma_z(t))\le2\kappa D+A+2\kappa D=4\kappa D+A$ $\Rightarrow d^f(\Sigma_y(t),\Sigma_z(t))\le\phi(4\kappa D+A)$. So, $d^f(\Sigma_x(t),\Sigma_z(t))\le d^f(\Sigma_x(t),\Sigma_y(t))+d^f(\Sigma_y(t),\Sigma_z(t))\le R_{\ref{flaring-lemma}}(\kappa,M)+\phi(4\kappa D+A)=R_1$ (say). Then by Lemma \ref{flaring-lemma} $(1)$, $d_B(t,t_{xz})\le\tau_{\ref{flaring-lemma}}(\kappa,R_1)$. Hence by triangle inequality, $d_B(t_{yz},t_{xz})\le d_B(t_{yz},t)+d_B(t,t_{xz})\le D_1$, where $D_1=D+\tau_{\ref{flaring-lemma}}(\kappa,R_1)$.\qed\smallskip
		
Now we show the Hausdorff closeness of paths. We only prove that $c(x,z)$ lies in uniform neighborhood of $c(y,z)$. Then by the symmetry of the proof we will be done.
		
Let $\xi\in c(x,z)\cap\tilde{\al}_{xz}$ and $\eta=\pi_X(\xi)$. Then $\exists~\eta\pr\in\al_{yz}$ such that $d_B(\eta,\eta\pr)\le D_1+\dl_0$. Note that $[\eta,\eta\pr]\sse B_y$. Then by taking lift of geodesic $[\eta,\eta\pr]$ in $\Sigma_y$, we get $d\pr(\xi,c(y,z))\le d\pr(\Sigma_x(\eta),\Sigma_y(\eta\pr))\le d\pr(\Sigma_x(\eta),\Sigma_y(\eta))+d\pr(\Sigma_y(\eta),\Sigma_y(\eta\pr))\le R_{\ref{flaring-lemma}}(\kappa,M)+2\kappa(D_1+\dl_0)$.
		
Now let $\xi\in c(x,z)\cap\tilde{\gm}_{xz}$ and $\eta=\pi_X(\xi)$. Then $\exists~\eta\pr\in\gm_{yz}$ such that $d_B(\eta,\eta\pr)\le D_1+\dl_0$. Note that $[\eta,\eta\pr]\sse B_z$. So, by taking lift of $[\eta,\eta\pr]$ in $\Sigma_z$, $d\pr(\xi,c(y,z))\le d\pr(\Sigma_z(\eta),\Sigma_z(\eta\pr))\le2\kappa(D_1+\dl_0)$ (Lemma \ref{qi-sec-inside-lad-len-lift} $(3)$).
		
Finally, let $\xi\in c(x,z)\cap\mu_{xz}$. Then $d\pr(\xi,c(y,z))\le2\kappa(D_1+\dl_0)+A$.
		
Let $R_2:=max\{R_{\ref{flaring-lemma}}(\kappa,M)+2\kappa(D_1+\dl_0),2\kappa(D_1+\dl_0)+A\}=R_{\ref{flaring-lemma}}(\kappa,M)+2\kappa(D_1+\dl_0)$. Hence $c(x,z)\sse N_{R_2}(c(y,z))$. Therefore,  $Hd\pr(c(x,z),c(y,z))\le R_2$.
		
\emph{Subcase} (1B): Suppose $\Sigma_z|_{B_{xyz}}\sse\L_{xy}|_{B_{xyz}}$. Here also we will do the same as in Subcase $(1A)$. Let $a$ be the nearest point projection of $\bar{x}=\bar{y}$ on $B_{xyz}$. Since $\L\cap X_{\mfB}$ has girth $\le A_0\le A$, by Lemma \ref{flaring-lemma} $(2)$, $d^f(\Sigma_x(s),\Sigma_y(s))\le R_{\ref{flaring-lemma}}(\kappa,M),\fa~s\in[a,\bar{x}]$. (Note that $[a,\bar{x}]$ could be $\{\bar{x}\}=\{a\}$ if $\bar{x}\in B_{xyz}$.) In particular, $d^f(\Sigma_x(a),\Sigma_y(a))\le R_{\ref{flaring-lemma}}(\kappa,M)$. Again, since $\Sigma_z\sse\L_{xy}|_{B_{xyz}}$, we have $d^f(\Sigma_x(a),\Sigma_z(a))\le R_{\ref{flaring-lemma}}(\kappa,M)$ and $d^f(\Sigma_y(a),\Sigma_z(a))\le R_{\ref{flaring-lemma}}(\kappa,M)$. Note that $t_{yz}$ and $t_{xz}$ are also nearest point projections of $a$ on $U_{yz}$ and $U_{xz}$ respectively. So by Lemma \ref{flaring-lemma} $(1)$, we have $D_2=\tau_{\ref{flaring-lemma}}(\kappa,R_{\ref{flaring-lemma}}(\kappa,M))$ such that $d_B(a,t_{xz})\le D_2$ and $d_B(a,t_{yz})\le D_2$. Thus $d_B(t_{xz},t_{yz})\le2D_2$. Now we only show that $c(x,z)$ lies in uniform neighborhood of $c(y,z)$. Then by symmetry of the proof we will be done.
		
Let $\xi\in c(x,z)\cap\tilde{\al}_{xz}$ and $\eta=\pi_X(\xi)$. Note that $\eta\in[t_{xz},a]\cup[a,\bar{x}]$. If $\eta\in[a,\bar{x}]$, then $d\pr(\xi,c(y,z))\le d^f(\Sigma_x(\eta),\Sigma_y(\eta))\le R_{\ref{flaring-lemma}}(\kappa,M)$. If $\eta\in[t_{xz},a]$, then $d\pr(\xi,c(y,z))\le d\pr(\Sigma_x(\eta),\Sigma_x(a))$ $+d^f(\Sigma_x(a),\Sigma_y(a))\le2\kappa D_2+R_{\ref{flaring-lemma}}(\kappa,M)$.
		
Now let $\xi\in c(x,z)\cap\tilde{\gm}_{xz}$ and $\eta=\pi_X(\xi)$. Then $\exists~\eta\pr\in\gm_{yz}$ such that $d_B(\eta,\eta\pr)\le2D_2+\dl_0$. So taking lift of $[\eta,\eta\pr]$ in $\Sigma_z$, we get $d\pr(\xi,c(y,z))\le d\pr(\Sigma_z(\eta),\Sigma_z(\eta\pr))\le2\kappa(2D_2+\dl_0)$.
		
Finally, if $\xi\in c(x,z)\cap\mu_{xz}$, then $d\pr(\xi,c(y,z))\le2\kappa(2D_2+\dl_0)+A$. 
		
Therefore, $Hd\pr(c(x,z),c(y,z))\le max\{2\kappa(2D_2+\dl_0)+A,2\kappa D_2+R_{\ref{flaring-lemma}}(\kappa,M)\}$ $=R_3$ (say).
		
Let $\bm{R_4(\kappa,A,r)}:=max\{R_2,R_3\}$.
		
Now for the rest of the proof for this proposition, we assume that all the paths $c(\zeta,\zeta\pr)$ are constructed using the qi sections $\Sigma_x,\Sigma_y$ and $\Sigma_z$, where $\zeta,\zeta\pr\in\Sigma_x\cup\Sigma_y\cup\Sigma_z$.

{\bf Case 2}: Let $\pi(x)=\pi(y)$. Suppose $\Sigma_y(\bar{x})=y_1$. We also assume that $\Sigma_{y_1}=\Sigma_y$. Since $d_X(x,y)\le r$, so $d_B(\bar{x},\bar{y})\le r$. Now by taking lift of geodesic $[\bar{x},\bar{y}]$ in $\Sigma_y$, we get $d\pr(y,y_1)\le2\kappa r$. Thus $d_X(y_1,x)\le2\kappa r+r$ and so $d^f(y_1,x)\le\phi(2\kappa r+r)$. Therefore, by Case $1$, $Hd\pr(c(y_1,z),c(x,z))$ $\le R_4(\kappa,A,\phi(2\kappa r+r))$.
		
Now we investigate on $Hd\pr(c(y_1,z),c(y,z))$. For the consistency of notation, we let $\bar{y}_1=\pi_X(y_1)$. Let $t_{y_1z}$ be a nearest point projection of $\bar{y}_1$ on $U_{y_1z}=U_{yz}$ (since $\Sigma_y=\Sigma_{y_1}$). Again $U_{y_1z}$ is $K_{\ref{neck-is-qc}}(\kappa,A)$-quasiconvex and $d_B(\bar{y}_1,\bar{y})\le2\kappa r$, so by lemma \ref{proj-on-qc} $(1)$, we have $d_B(t_{y_1z},t_{yz})\le D_3$, where $D_3=(2\kappa r+1)C_{\ref{proj-on-qc}}(\dl_0,K_{\ref{neck-is-qc}}(\kappa,A))$. Note that $\bar{z}=\pi_X(z)$ and $\al_{y_1z}=[\bar{y}_1,t_{y_1z}]$, $\gm_{y_1z}=[\bar{z},t_{y_1z}]$. Then $Hd_B(\gm_{y_1z},\gm_{yz})\le D_3+\dl_0$ and $Hd_B(\al_{y_1z},\al_{yz})\le D_3+2\dl_0$ (note that $D_3>2\kappa r$). Thus $Hd\pr(\tilde{\al}_{yz},\tilde{\al}_{y_1z})\le2\kappa(D_3+2\dl_0)$ and $Hd\pr(\tilde{\gm}_{yz},\tilde{\gm}_{y_1z})\le2\kappa(D_3+\dl_0)$. Hence, $Hd\pr(c(y,z),c(y_1,z))\le2\kappa(D_3+2\dl_0)+A$.
		
Therefore, $Hd\pr(c(x,z),c(y,z))\le R_4(\kappa,A,\phi(2\kappa r+r))+2\kappa(D_3+2\dl_0)+A=:\bm{R_5(\kappa,A,r)}$ (say).
		
{\bf Case 3}: Now we consider the general case. Let $\bar{x}_1$ and $\bar{y}_1$ be the nearest point projections of $\bar{x}$ and $\bar{y}$ on $B_{u_{xy}}$ respectively (see Definition of family of paths for $u_{xy}$). Let $\Sigma_x(\bar{x}_1)=x_1$ and $\Sigma_y(\bar{y}_1)=y_1$. Since $d_B(\bar{x},\bar{y})\le r$, so $d_B(\bar{x},\bar{x}_1)\le r$ and $d_B(\bar{y}_1,\bar{y})\le r$. Thus by taking lifts of geodesics $[\bar{x},\bar{x}_1]$ and $[\bar{y},\bar{y}_1]$ in $\Sigma_x$ and $\Sigma_y$ respectively, we get $d\pr(x,x_1)\le2\kappa r$ and $d\pr(y,y_1)\le2\kappa r$. So by triangle inequality, $d_X(x_1,y_1)\le4\kappa r+r$. Note that $\pi(x_1)=\pi(y_1)$. Hence by Case $2$, $Hd\pr(c(x_1,z),c(y_1,z))\le R_5(\kappa,A,4\kappa r+r)$. Since $d\pr(x,x_1)\le2\kappa r$ and $d\pr(y,y_1)\le2\kappa r$, we have $Hd\pr(c(x,z),c(y,z))\le R_5(\kappa,A,4\kappa r+r)+2\kappa r=:\bm{R_6(\kappa,A,r)}$ (say).
		
Therefore, we can take $D_{\ref{fellow-travelling-property}}=max\{R_i(\kappa,A,r):i=4,5,6\}=R_6(\kappa,A,r)$.
\end{proof}
The proof for slimness of triangle formed by three paths (as in path construction) inside special $K$-ladder (see Definition \ref{K-metric-graph-bundle}) was done in \cite[Lemma $3.11$ for small girth ladder]{pranab-mahan} in case of metric graph bundles (see \cite[Definition $1.5$]{pranab-mahan}). The same proof works in case of metric bundles (Lemma \ref{slimness-in-metric-ladder(special case)}); which we will use in Condition $(2)$ below. We state it (for small girth ladder) without proof.
\begin{lemma}\textup{(\hspace{-.08mm}\cite[Lemma $3.11$]{pranab-mahan})}\label{slimness-in-metric-ladder(special case)}
Given $k\ge1,\mathcal{A}\ge0$, there is a constant $D_{\ref{slimness-in-metric-ladder(special case)}}=D_{\ref{slimness-in-metric-ladder(special case)}}(k,\mathcal{A})$ such that the following holds.
		
Suppose $\L(\Sigma,\Sigma\pr)$ is a special $k$-ladder (in a tree of metric bundles $(X,B,T)$) bounded by two $k$-qi sections $\Sigma,\Sigma\pr$ over an isometrically embedded subspace $B_1\sse B$ such that inf $\{d^f(\Sigma(a),\Sigma\pr(a)):a\in B_1\}\le \mathcal{A}$. Let $x,y,z\in \L(\Sigma,\Sigma\pr)$. Then the triangle formed by paths $c(x,y),c(x,z)$ and $c(y,z)$, coming from the path construction, are $D_{\ref{slimness-in-metric-ladder(special case)}}$-slim in the induced path metric on $N_{2C^{(3)}_{\ref{qi-sec-inside-lad-len-lift}}(k)}(\L(\Sigma,\Sigma\pr))\sse X$.
	\end{lemma}
	
{\bf \emph{Continuation of the proof of Proposition \ref{hyp_small_ladder}}}: We verify the condition $(1)$ and $(2)$ of Proposition \ref{combing} for our family of paths. Here $\L$ is $R$-dense in $L_{KR}$. We will be working with $\Sigma_x,\Sigma_y,\Sigma_z$ as $\kappa$-qi sections (explained in Condition $(2)$, Case $1$ below). 
	
{\bf Condition (1)}: Let $x,y\in\L$ such that $d_X(x,y)=r$ for $r\in\R_{\ge0}$. We want to show that the length of $[c(x,y)]$ in the path metric of $(L_{KR},d\pr)$ is bounded in terms of $r$. Let $c\in[\bar{x},\bar{y}]\cap B_{u_{xy}}$ and $c_1\in[\bar{x},t_{xy}]\cap B_{u_{xy}},c_2\in[\bar{y},t_{xy}]\cap B_{u_{xy}}$ such that $d_B(c,c_i)\le\dl_0,~i=1,2$. (We refer to the `definition of family of paths' for $u_{xy}$.) Since $d_B(\bar{x},\bar{y})\le d_X(x,y)\le r$, so $d_B(\bar{x},c)\le r$ and $d_B(\bar{y},c)\le r$. Let $\Sigma_x(c)=\{x_1\}$ and $\Sigma_y(c)=\{y_1\}$. Now taking lifts of $[\bar{x},c]$ and $[\bar{y},c]$ in $\Sigma_x$ and $\Sigma_y$ respectively, we have $d\pr(x,x_1)\le 2r\kappa$ and $d\pr(y,y_1)\le2r\kappa$ (see Lemma \ref{qi-sec-inside-lad-len-lift} $(3)$). Then by triangle inequality, $d_X(x_1,y_1)\le r(4\kappa+1)$. So $d_X(\Sigma_x(c_1),\Sigma_y(c_1))\le d\pr(\Sigma_x(c_1),\Sigma_x(c))+d_X(\Sigma_x(c),\Sigma_y(c))+d\pr(\Sigma_y(c),\Sigma_y(c_1))\le2\kappa\dl_0+r(4\kappa+1)+2\kappa\dl_0=4\kappa\dl_0+r(4\kappa+1)$. Thus $d^f(\Sigma_x(c_1),\Sigma_y(c_1))\le \phi(4\kappa\dl_0+r(4\kappa+1))$. Since $t_{xy}$ is a nearest point projection of $\bar{x}$ on $U_{xy}$, by Lemma \ref{flaring-lemma} $(1)$, $d_B(c_1,t_{xy})\le D$, where $D=\tau_{\ref{flaring-lemma}}(\kappa,\phi(4\kappa\dl_0+n(4\kappa+1)))$. So $d_B(c_2,t_{xy})\le d_B(c_2,c_1)+d_B(c_1,t_{xy})\le2\dl_0+D$. Again $d_B(\bar{x},c_1)\le r+\dl_0$ and $d_B(\bar{y},c_2)\le r+\dl_0$. Hence $d_B(\bar{x},t_{xy})\le d_B(\bar{x},c_1)+d_B(c_1,t_{xy})\le r+\dl_0+D$ and $d_B(\bar{y},t_{xy})\le d_B(\bar{y},c_2)+d_B(c_2,t_{xy})\le r+3\dl_0+D$. Note that $\al_{xy}=[\bar{x},t_{xy}]$ and $\gm_{xy}=[\bar{y},t_{xy}]$. Therefore, by taking lifts of $\al_{xy}$ and $\gm_{xy}$ in $\Sigma_x$ and $\Sigma_y$ respectively, we see that the length of $[c(x,y)]$ is bounded by $2\kappa(r+\dl_0+D)+A+2\kappa(r+3\dl_0+D)=4\kappa(r+D+2\dl_0)+A$. So the paths $[c(x,y)]$ are $\psi$-properly embedded, where $\psi:\R_{\ge0}\ri\R_{\ge0}$ is a function such that
	\begin{eqnarray}\label{proper-fn}
		\psi(r)=4\kappa(r+D+2\dl_0)+A
	\end{eqnarray}
	
	{\bf Condition (2)}:
	Recall $U_{xy}=\L^n_{xy}(A)|_{B_{xy}},~U_{xz}=\L^n_{xz}(A)|_{B_{xz}}$ and $U_{yz}=\L^n_{yz}(A)|_{B_{yz}}$ are $K_{\ref{neck-is-qc}}(\kappa,A)$-quasiconvex and so is in $B$. We show that paths $c(x,y),c(x,z),c(y,z)$, coming from the above qi sections, are uniformly slim in the path metric of $(L_{KR},d\pr)$. Depending on the position of $t_{xy},t_{xz},t_{yz}$ with respect to $B_{xyz}$, we consider the following two cases. Note that by the definition of $B_{xyz}$, either all of $t_{xy},t_{xz},t_{yz}$ are in $B_{xyz}$ or at most one of them is outside of $B_{xyz}$.
	
	{\bf Case 1}: All of $t_{xy},t_{xz},t_{yz}$ are in $B_{xyz}$.
	
	In this case, without loss of generality, we assume that $\Sigma_y|_{B_{xyz}}\sse\L_{xz}|_{B_{xyz}}$, i.e. in the ladder $\L$, we have the order, $bot(\L_{a,v})\le \Sigma_x(a)\le\Sigma_y(a)\le\Sigma_z(a)\le top(\L_{a,v})$, where $a\in B_{xyz}$ and $v=\pi_B(a)$. (To get this one has to consider $\Sigma_x,\Sigma_y,\Sigma_z$ as $\kappa=C^{(3)}_{\ref{qi-sec-inside-lad-len-lift}}(K)$-qi sections instead $K$-qi sections, see Lemma \ref{qi-sec-inside-lad-len-lift} $(2)$). Let $\bar{x}_1,\bar{y}_1,\bar{z}_1$ be (the) nearest point projections of $\bar{x},\bar{y},\bar{z}$ on $B_{xyz}$ respectively. Let $\Sigma_x(\bar{x}_1)=x_1,\Sigma_y(\bar{y}_1)=y_1$ and $\Sigma_z(\bar{z}_1)=z_1$. Further, we assume that the restriction of $c(x,y)$ from $x_1$ to $y_1$ is $c(x_1,y_1)$. Likewise, we have $c(x_1,z_1)$ and $c(y_1,z_1)$. Note that restriction of $\Sigma_x$ and $\Sigma_z$ over $B_{xyz}$ form a special $C_{\ref{qi-sec-inside-lad-len-lift}}(\kappa)=C^{(4)}_{\ref{qi-sec-inside-lad-len-lift}}(K)$-ladder over $B_{xyz}$ bounded by two qi sections $\Sigma_x|_{B_{xyz}}$ and $\Sigma_z|_{B_{xyz}}$ such that $inf\{d^f(\Sigma_x(s),\Sigma_z(s)):s\in B_{xyz}\}\le A$. Since $(X,B,T)$ satisfies $C^{(7)}_{\ref{qi-sec-inside-lad-len-lift}}(K)=C^{(3)}_{\ref{qi-sec-inside-lad-len-lift}}(C^{(4)}_{\ref{qi-sec-inside-lad-len-lift}}(K))$-flaring condition, so by Lemma \ref{slimness-in-metric-ladder(special case)}, the triangle formed by the paths $c(x_1,y_1),c(x_1,z_1)$ and $c(y_1,z_1)$ are $D_{\ref{slimness-in-metric-ladder(special case)}}(C^{(4)}_{\ref{qi-sec-inside-lad-len-lift}}(K),A)$-slim in the path metric of $L_{KR}$. Let $D_1=D_{\ref{slimness-in-metric-ladder(special case)}}(C^{(4)}_{\ref{qi-sec-inside-lad-len-lift}}(K),A)$.
	
	For the point $\xi\in c(x,y)$ such that $\xi\notin c(x_1,y_1)$, $\xi$ is $2\kappa\dl_0$-close to $c(x,z)\cup c(y,z)$ in the path metric of $L_{KR}$. Same for others.
	
	Note that $2\kappa\dl_0\le D_1$. Therefore, the triangle formed by the paths $c(x,y),~c(x,z)$ and $c(y,z)$ are $D_1$-slim in the path metric of $L_{KR}$.
	
	{\bf Case 2}:
	One of $t_{xy},t_{xz},t_{yz}$ is not in $B_{xyz}$ (see Figure \ref{one-of-three}). Without loss of generality, we assume that $t_{xy}\notin B_{xyz}$. In this case, we do not need to consider what exactly is happening to $\Sigma_x,\Sigma_y$ and $\Sigma_z$ over $B_{xyz}$.
	
	Note that $\pi_B(t_{xy})=v_{xy}$. Let $t\in B_{v_{xy}}$ such that $d_B(B_{xyz},B_{v_{xy}})=d_B(B_{xyz},t)$. Let us fix $s\in\mfB\cap U_{xy}$ such that $t\in[s,t_{xy}]$. (We can get such $s$ as $\L^g|_{\mfB}\le A_0\le A$, see Definition \ref{girth-neck} for notation.) Since $s,t_{xy}\in U_{xy}$, by Lemma \ref{flaring-lemma} $(2)$, for all $\zeta\in[s,t_{xy}]$, $d^f(\Sigma_x(\zeta),\Sigma_y(\zeta))\le R_{\ref{flaring-lemma}}(\kappa,A)$. In particular, $d^f(\Sigma_x(t),\Sigma_y(t))\le R_{\ref{flaring-lemma}}(\kappa,A)$. Then by the fellow-travelling property (see Proposition \ref{fellow-travelling-property}), we have $$Hd\pr(c(x,z)|_{[t,t_{xz}]\cup[t_{xz},\bar{z}]},c(y,z)|_{[t,t_{yz}]\cup[t_{yz},\bar{z}]})\le D_{\ref{fellow-travelling-property}}(\kappa,A,R_{\ref{flaring-lemma}}(\kappa,A))=D_2\textrm{ (say)}.$$
	
	Again let $t_x,t_y$ be the nearest point projections of $\bar{x}$ and $\bar{y}$ respectively on $B_{v_{xy}}$. Then $$Hd\pr(c(x,y)|_{[t_x,\bar{x}]},c(x,z)|_{[t_x,\bar{x}]})\le2\kappa\dl_0\textrm{ and }Hd\pr(c(x,y)|_{[t_y,\bar{y}]},c(y,z)|_{[t_y,\bar{y}]})\le2\kappa\dl_0.$$
	
	Now we only need to analyse what is happening to the paths $c(x,y),c(x,z)$ and $c(y,z)$ over $B_{v_{xy}}$ to conclude the slimness, and here we go.
	
	\begin{figure}[h]
		\includegraphics[width=10cm]{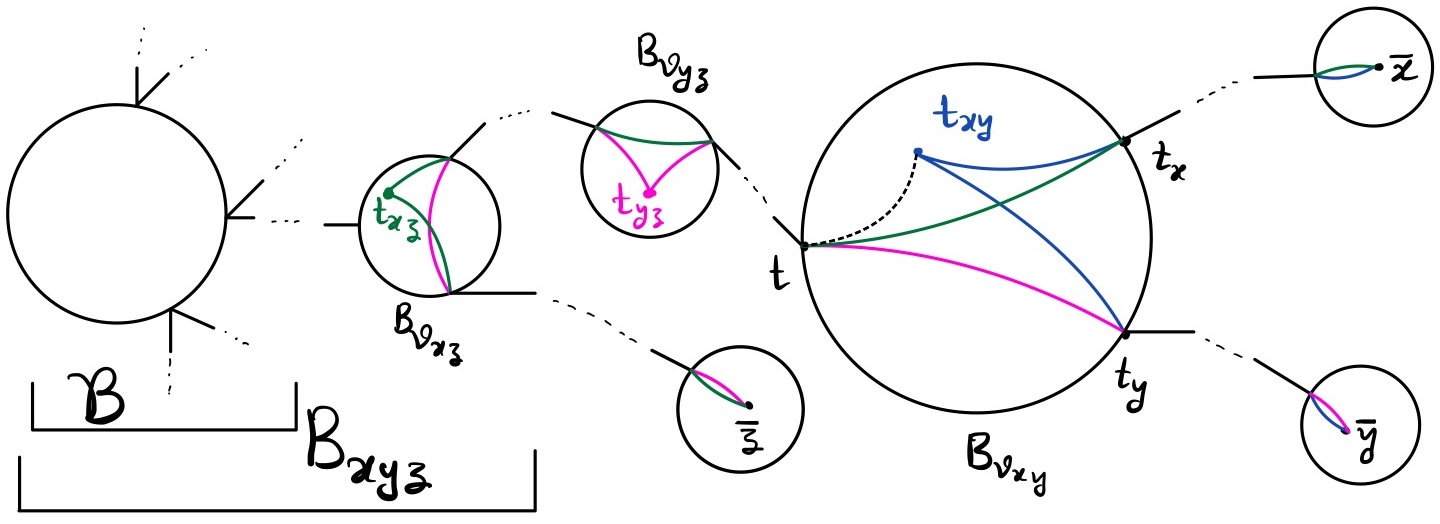}
		\centering
		\caption{Case $2$}
		\label{one-of-three}
	\end{figure}
	
{\bf\textit{A. The portion of the path $\bm{c(x,y)}$ over $\bm{B_{v_{xy}}}$ is uniformly close to $\bm{c(x,z)\cup c(y,z)}$}}:
	
Let $\xi\in c(x,y)\cap(\tilde{\al}_{xy}\cup\tilde{\gm}_{xy})$ such that $\eta=\pi_X(\xi)$. Then $\eta\in[t_{xy},t_x]\cup[t_{xy},t_y]$. First, we consider $\eta\in[t_{xy},t_x]$. Since $U_{xy}$ is $K_{\ref{neck-is-qc}}(\kappa,A)$-quasiconvex, $t_{xy}$ nearest point projection of $\bar{x}$ and $s\in U_{xy}$, so by \cite[Lemma $1.31 (2)$]{pranab-mahan}, the arc-length parametrization of $[t_x,t_{xy}]\cup[t_{xy},s]$ is a $(3+2K_{\ref{neck-is-qc}}(\kappa,A))$-quasigeodesic. In particular, $[t_x,t_{xy}]\cup[t_{xy},t]$ is $(3+2K_{\ref{neck-is-qc}}(\kappa,A))$-quasigeodesic. Therefore, by Lemma \ref{ml}, there is $D_3$ depending on $\dl_0$ and $3+2K_{\ref{neck-is-qc}}(\kappa,A)$ such that $d_B(\eta,\eta\pr)\le D_3$ for some  $\exists~\eta\pr\in[t_x,t]$. So by taking lift of $[\eta,\eta\pr]$ in $\Sigma_x$ (see Lemma \ref{qi-sec-inside-lad-len-lift} $(3)$), we get $d\pr(\xi,c(x,z))\le d\pr(\Sigma_x(\eta),\Sigma_x(\eta\pr))\le2\kappa D_3$.
	
Now suppose $\eta\in[t_{xy},t_y]$. Then the slimness of $\triangle(t_{xy},t,t_y)$ says that $\eta\in N_{\dl_0}([t_{xy},t]\cup[t,t_y])$. Let $\exists~\eta\pr\in[t,t_y]$ such that $d_B(\eta,\eta\pr)\le\dl_0$. Then by taking lift of $[\eta,\eta\pr]$ in $\Sigma_y$ (see Lemma \ref{qi-sec-inside-lad-len-lift} $(3)$), we get $d\pr(\xi,c(y,z))\le d\pr(\Sigma_y(\eta),\Sigma_y(\eta\pr))\le2\kappa\dl_0$. Now let $\eta\pr\in[t,t_{xy}]$ such that $d_B(\eta,\eta\pr)\le\dl_0$. Again, since $[t_x,t_{xy}]\cup[t_{xy},t]$ is $(3+2K_{\ref{neck-is-qc}}(\kappa))$-quasigeodesic, $\exists~\eta\prr\in[t,t_x]$ such that $d_B(\eta\pr,\eta\prr)\le D_3$. Taking lift of the geodesic $[\eta,\eta\pr]$ in $\Sigma_y$ and that of the geodesic $[\eta\pr,\eta\prr]$ in $\Sigma_x$, we get $d\pr(\Sigma_y(\eta),\Sigma_y(\eta\pr))\le2\kappa\dl_0$ and $d\pr(\Sigma_x(\eta\pr),\Sigma_x(\eta\prr))\le2\kappa D_3$. Recall that $\fa~\zeta\in[s,t_{xy}]$, $d^f(\Sigma_x(\zeta),\Sigma_y(\zeta))\le R_{\ref{flaring-lemma}}(\kappa,A)$; in particular, $d^f(\Sigma_x(\eta\pr),\Sigma_y(\eta\pr))\le R_{\ref{flaring-lemma}}(\kappa,A)$. So, by triangle inequality, $d\pr(\xi,c(x,z))\le d\pr(\Sigma_y(\eta),\Sigma_x(\eta\prr))\le 2\kappa\dl_0+R_{\ref{flaring-lemma}}(\kappa,A)+2\kappa D_3=D_4\textrm{ (say)}$.
	
	Again, if $\xi\in\mu_{xy}$, then $d\pr(\xi,c(x,z))\le2\kappa D_3+A\le2\kappa D_3+R_{\ref{flaring-lemma}}(\kappa,A)\le D_4$.

	{\bf\textit{B. The portion of the path $\bm{c(y,z)}$ over $\bm{B_{v_{xy}}}$ is uniformly close to $\bm{c(x,y)\cup c(x,z)}$}}: 
	
Note that the portion of $c(y,z)$ over $B_{v_{xy}}$ is $c(y,z)\cap\tilde{\al}_{yz}$. Let $\xi\in c(y,z)\cap\tilde{\al}_{yz}$ such that $\eta=\pi_X(\xi)$. Then $\eta\in[t_y,t]$, and the slimness of $\triangle(t_{xy},t_y,t)$ says that $\eta\in N_{\dl_0}([t_y,t_{xy}]\cup[t_{xy},t])$. First, we consider that $\exists~\eta\pr\in[t_y,t_{xy}]$ such that $d_B(\eta,\eta\pr)\le\dl_0$. So by taking lift of $[\eta,\eta\pr]$ in $\Sigma_y$ (see Lemma \ref{qi-sec-inside-lad-len-lift} $(3)$), we get $d\pr(\xi,c(x,y))\le d\pr(\Sigma_y(\eta),\Sigma_y(\eta\pr))\le2\kappa\dl_0$. Now, suppose $\exists~\eta\pr\in[t_{xy},t]$ such that $d_B(\eta,\eta\pr)\le\dl_0$. Recall that $[t_x,t_{xy}]\cup[t_{xy},t]$ is $(3+2K_{\ref{neck-is-qc}}(\kappa))$-quasigeodesic and so $\exists~\eta\prr\in[t_x,t]$ such that $d_B(\eta\pr,\eta\prr)\le D_3$ where $D_3$ is defined above in {\bf A}. Taking lifts of the geodesic $[\eta,\eta\pr]$ in $\Sigma_y$ and that of the geodesic $[\eta\pr,\eta\prr]$ in $\Sigma_x$, we get $d\pr(\Sigma_y(\eta),\Sigma_y(\eta\pr))\le2\kappa\dl_0$ and $d\pr(\Sigma_x(\eta\pr),\Sigma_x(\eta\prr))\le2\kappa D_3$. Also $\fa~\zeta\in[t_{xy},t],d^f(\Sigma_x(\zeta),\Sigma_y(\zeta))\le R_{\ref{flaring-lemma}}(\kappa,A)$; in particular, $d^f(\Sigma_y(\eta\pr),\Sigma_x(\eta\pr))\le R_{\ref{flaring-lemma}}(\kappa,A)$. Therefore, by triangle inequality, $d\pr(\xi,c(x,z))\le d\pr(\Sigma_y(\eta),\Sigma_x(\eta\prr))\le 2\kappa\dl_0+R_{\ref{flaring-lemma}}(\kappa,A)+2\kappa D_3=D_4$ where $D_4$ is defined above in {\bf A}.
	
	{\bf\emph{C. The portion of the path $\bm{c(x,z)}$ over $\bm{B_{v_{xy}}}$ is uniformly close to $\bm{c(x,y)\cup c(y,z)}$}}:
	
Note that the portion of $c(x,z)$ over $B_{v_{xy}}$ is $c(x,z)\cap\tilde{\al}_{xz}$. Let $\xi\in c(x,z)\cap\tilde{\al}_{xz}$ such that $\eta=\pi_X(\xi)$. Then $\eta\in[t_x,t]$, and so $\eta\in N_{\dl_0}([t_x,t_{xy}]\cup[t_{xy},t])$. If $\exists~\eta\pr\in[t_x,t_{xy}]$ such that $d_B(\eta,\eta\pr)\le\dl_0$, then by taking lift of the geodesic $[\eta,\eta\pr]$ in $\Sigma_x$, we get $d\pr(\xi,c(x,y))\le d\pr(\Sigma_x(\eta),\Sigma_x(\eta\pr))\le2\kappa \dl_0$. Now, let $\exists~\eta\pr\in[t_{xy},t]$ such that $d_B(\eta,\eta\pr)\le\dl_0$. Recall that $\fa~\zeta\in[t,t_{xy}]$, $d^f(\Sigma_x(\zeta),\Sigma_y(\zeta))\le R_{\ref{flaring-lemma}}(\kappa,A)$. Again, if we look at $\triangle(t_{xy},t,t_y),\exists~\eta\prr\in[t_{xy},t_y]\cup[t_y,t]$ such that $d_B(\eta\pr,\eta\prr)\le\dl_0$. If  $\eta\prr\in[t_{xy},t_y]$, then by taking lifts of geodesics $[\eta,\eta\pr]$ and $[\eta\pr,\eta\prr]$ in $\Sigma_x$ and $\Sigma_y$ respectively, we get $d\pr(\xi,c(x,y))\le d\pr(\Sigma_x(\eta),\Sigma_y(\eta\prr))\le d\pr(\Sigma_x(\eta),\Sigma_x(\eta\pr))+d^f(\Sigma_x(\eta\pr),\Sigma_y(\eta\pr))+d\pr(\Sigma_y(\eta\pr),\Sigma_y(\eta\prr))\le2\kappa\dl_0+R_{\ref{flaring-lemma}}(\kappa,A)+2\kappa\dl_0=D_5$ (say). If $\eta\prr\in[t_y,t]$, then the same inequality would imply that $\xi$ is $D_5$-close to $c(y,z)$ in the path metric of $L_{KR}$.
	
	Let $D\pr=max\{D_1,D_2,D_4,D_5\}+2D_{\ref{Hd-close-of-c(x,y)-c(x,y)}}(\kappa,A)$, representing the maximum of all constants obtained in Case $1$, Case $2$; additionally, considering  Corollary \ref{Hd-close-of-c(x,y)-c(x,y)}, we add $2D_{\ref{Hd-close-of-c(x,y)-c(x,y)}}(\kappa,A)$. Therefore, the triangle formed by the paths $c(x,y),c(x,z)$ and $c(y,z)$, which we started with to show the combing criterion, are $D\pr$-slim. Hence, by Proposition \ref{combing}, $L_{KR}$ is $\dl_{\ref{hyp_small_ladder}}=\dl_{\ref{combing}}(\psi,D\pr,R)$-hyperbolic, where $\psi$ is defined in Condition $(1)$, equation \ref{proper-fn}.
\end{proof}

\subsection{Hyperbolicity of ladders (general case)}\label{general-ladder}


\begin{lemma}[Bisection of ladders]\label{bisection-of-ladder}
There are constants $K_{\ref{bisection-of-ladder}}=K_{\ref{bisection-of-ladder}}(K)=C_{\ref{qi-sec-inside-lad-len-lift}}(K),~C_{\ref{bisection-of-ladder}}=C_{\ref{bisection-of-ladder}}(K,C,\epsilon)\ge C,~\epsilon_{\ref{bisection-of-ladder}}=\epsilon_{\ref{bisection-of-ladder}}(K,C,\epsilon)\ge\epsilon$ such that the following holds.
	
Suppose $z\in\L\cap X_{\mfB}$ and $\Sigma_z$ is a maximal $K$-qi section in $\L$ through $z$. Then $\Sigma_z$ divide the ladder $\L$ into two $(K_{\ref{bisection-of-ladder}},C_{\ref{bisection-of-ladder}},\epsilon_{\ref{bisection-of-ladder}})$-subladders, say $\L^+$ and $\L^-$, with a central base $\mfB$ such that
	\begin{eqnarray*}
		top(\L^+)&\sse&top(\L),\ \Sigma_z\sse bot(\L^+)\textrm{ and }\\ 
		bot(\L^-)&\sse&bot(\L),\ \Sigma_z\sse top(\L^-)
	\end{eqnarray*}   
\end{lemma}

\begin{proof}
	Since the proofs are similar, we prove only for $\L^+$, say. Note that $\Sigma_z$ is a maximal $K$-qi section over some base $B_z\sse \pi^{-1}_B(T_{\L})$, say. Let $T_z=\pi_B(B_z)$. There are two kinds of segments in the fibers of $\L^+$ as follows.
	
	\emph{First kind}: For all $v\in T_z$ and for all $b\in B_v,~\L^+_{b,v}=[top(\L_{b,v}),\Sigma_z(b)]\sse \L_{b,v}$.
	
	\emph{Second kind}: Let $w\in T_{\L}\setminus T_z$ and $v\in T_z$ such that $d_T(v,w)=1$. Let $S$ be the connected component of $T\setminus\{v\}$ containing $\{w\}$. If we have an order $h_{wv}(top(\L_{\mfw,w}))<\Sigma_z(\mfv)\le top(\L_{\mfv,v})$ (see Figure \ref{order} left one), then $\L^+_{b,t}=\emptyset$ for $t\in S\textrm{ and } b\in B_t$. If the order is $bot(\L_{\mfv,v})<\Sigma_z(\mfv)\le h_{wv}(bot(\L_{\mfw,w}))$ (see Figure \ref{order} right one), then $\L^+_{b,t}=\L_{b,t}$ for $t\in T_{\L}\cap S\textrm{ and }b\in B_t$ with the same orientation as it was for $\L$. Also, the family of maps $\{h_{wv}\}$ for $\L^+$ are the restriction of that of $\L$.
	
	\begin{figure}[h]
		\includegraphics[width=10cm]{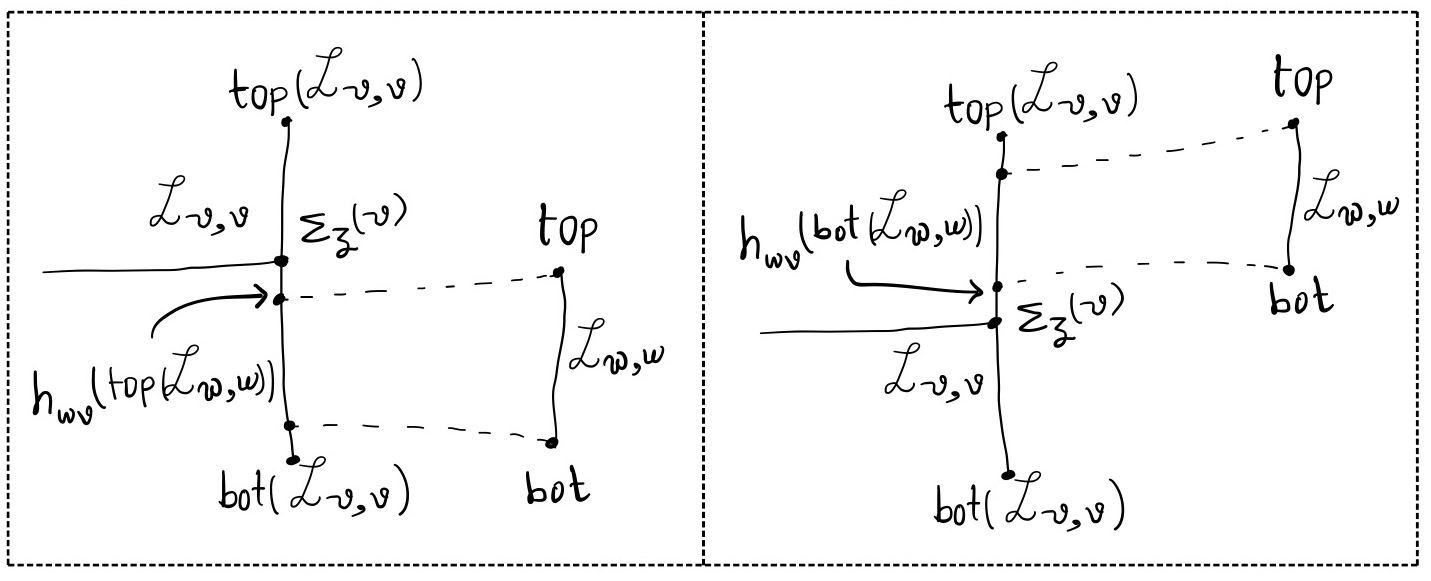}
		\centering
		\caption{}
		\label{order}
	\end{figure}
	
	Now with the help of Lemma \ref{promoting-of-ladder}, we show that union of these fiber geodesics form a ladder. In view of Lemma \ref{promoting-of-ladder}, we have to find $K\pr,C\pr,\epsilon\pr$. By Lemma \ref{qi-sec-inside-lad-len-lift} $(2)$, one observes that $K\pr=C_{\ref{qi-sec-inside-lad-len-lift}}(K)$, $C\pr=C_{\ref{for-subladder}}(K,C,\epsilon)$ and $\epsilon\pr=\epsilon_{\ref{for-subladder}}(K,C,\epsilon)$ serve our purpose.
	
Therefore, by Lemma \ref{promoting-of-ladder}, $\L^+$ is a $(K_{\ref{bisection-of-ladder}},C_{\ref{bisection-of-ladder}},\epsilon_{\ref{bisection-of-ladder}})$-subladder in $\L$, where $K_{\ref{bisection-of-ladder}}=k_{\ref{promoting-of-ladder}}(K\pr)$, $C_{\ref{bisection-of-ladder}}=c_{\ref{promoting-of-ladder}}(C\pr)$ and $\epsilon_{\ref{bisection-of-ladder}}=\varepsilon_{\ref{promoting-of-ladder}}(\epsilon\pr)$ for the above $K\pr,C\pr,\epsilon\pr$. Since the family of maps $\{h_{wv}\}$ for $\L^+$ are restriction, so $k_{\ref{promoting-of-ladder}}(K\pr)=C_{\ref{qi-sec-inside-lad-len-lift}}(K)$.
\end{proof}	
In the same line, we also have the following lemma. Since the proof is similar to that of Lemma  \ref{bisection-of-ladder}, we omit the proof.
\begin{lemma}[Trisection of ladders]\label{trisection-of-ladder}
There are constants $K_{\ref{trisection-of-ladder}}=K_{\ref{trisection-of-ladder}}(K)=C_{\ref{qi-sec-inside-lad-len-lift}}(K)$, $C_{\ref{trisection-of-ladder}}=C_{\ref{trisection-of-ladder}}(K,C,\epsilon)$ and $\epsilon_{\ref{trisection-of-ladder}}=\epsilon_{\ref{trisection-of-ladder}}(K,C,\epsilon)$ such that the following holds.
	
Suppose $x,y\in\L\cap X_{\mfB}$. Let $\Sigma_x,\Sigma_y$ be maximal $K$-qi sections through $x,y$ over $B_x,B_y$ respectively. We assume that $\fa~v\in\pi_B(B_x\cap B_y)\textrm{ and }\fa~b\in B_v,$ we have an order $bot(\L_{b,v})\le\Sigma_x(b)\le\Sigma_y(b)\le top(\L_{b,v})$ in $\L_{b,v}$. Then we have $(K_{\ref{trisection-of-ladder}},C_{\ref{trisection-of-ladder}},\epsilon_{\ref{trisection-of-ladder}})$-subladder in $\L$ bounded by $\Sigma_x,\Sigma_y$ with central base $B_x\cap B_y$.
\end{lemma}


\begin{lemma}\label{union_of_two_ladders}
	For all $R\ge2K_{\ref{bisection-of-ladder}}(K)$ there exists a constant $R_{\ref{union_of_two_ladders}}=R_{\ref{union_of_two_ladders}}(K,R)$ such that the following holds.
	
Let $x\in\L\cap X_{\mfB}$, and let $\Sigma_x$ be a maximal $K$-qi section in $\L$. Now we have two $(K_{\ref{bisection-of-ladder}},C_{\ref{bisection-of-ladder}},\epsilon_{\ref{bisection-of-ladder}})$-subladders, say $\L^+$ and $\L^-$, coming from Lemma \ref{bisection-of-ladder}. Then $N_R(\L^+)\cap N_R(\L^-)\sse N_{R_{\ref{union_of_two_ladders}}}(\Sigma_x)$ in both the path metric of $N_R(\L^+)$ and $N_R(\L^-)$.
\end{lemma}
\begin{proof}
For ease of notation, let $\L^{(1)}=\L^+,~\L^{(2)}=\L^-$. Let $d_i$ be the induced path metric on $N_R(\L^i),~i=1,2$. Suppose $y\in N_R(\L^{(1)})\cap N_R(\L^{(2)})$ and $y_i\in\L^{(i)}$ such that $d_i(y,y_i)\le R,~i=1,2$. Then $d_X(y_1,y_2)\le2R$ and so $d_B(\bar{y}_1,\bar{y}_2)\le2R$, where $\pi_X(y_i)=\bar{y}_i,~i=1,2$. If $\mfT\cap[\pi(y_1),\pi(y_2)]=\emptyset$, we let $u\in T$ such that $d_T(\mfT,[\pi(y_1),\pi(y_2)])=d_T(\mfT,u)$; otherwise, we take $u\in\mfT\cap[\pi(y_1),\pi(y_2)]$ arbitrary. Fix $c\in B_u\cap[\bar{y}_1,\bar{y}_2]_B$. Then $d_B(\bar{y}_i,c)\le2R$. Let $t_i$ be the nearest point projection of $\pi(y_i)$ on $\mfT$ and $B_{y_i}=\mfB\cup\pi_B^{-1}([t_i,\pi(y_i)])$ for $i=1,2$. Then there is $K_{\ref{bisection-of-ladder}}(K)$-qi section, say $\Sigma_{y_i}$, over $B_{y_i}$ through $y_i$ in $\L^{(i)},~i=1,2$. Let $K_{\ref{bisection-of-ladder}}(K)=K_1$ Taking lifts of geodesic $[\bar{y}_i,c]_B$ in $\Sigma_{\bar{y}_i}$, we get $d_i(y_i,\Sigma_{y_i}(c))\le2K_1.2R=4K_1R,~i=1,2$. Then $d_X(\Sigma_{y_1}(c),\Sigma_{y_2}(c))\le d_X(\Sigma_{y_1}(c),y_1)+d_X(y_1,y_2)+d_X(y_2,\Sigma_{y_2}(c))\le2R(4K_1+1)$. So $d^f(\Sigma_{y_1}(c),\Sigma_{y_2}(c))\le \phi(2R(4K_1+1))$. Since $\Sigma_x(c)\in[\Sigma_{y_1}(c),\Sigma_{y_2}(c)]^f\sse \L_{c,u}$, so $d^f(\Sigma_{y_i}(c),\Sigma_x(c))\le \phi(2R(4K_1+1))$. Hence $d_i(y,\Sigma_x)\le d_i(y,\Sigma_x(c))\le d_i(y,y_i)+d_i(y_i,\Sigma_{y_i}(c))+d_i(\Sigma_{y_i}(c),\Sigma_x(c))\le R+4K_1 R+\phi(2R(4K_1+1)),~ i=1,2$. So, we can take $R_{\ref{union_of_two_ladders}}:=R(4K_1+1)+\phi(2R(4K_1+1))$.
\end{proof}

Now we are ready to state and proof of the main result of this subsection.

\begin{theorem}\label{general-ladder-is-hyp}
Let $K\ge1,~C\ge0$ and $\ep\ge0$. Suppose $\L_K$ is a $(K,C,\ep)$-ladder with a central base $\mathfrak B$. Then for all $R\ge2C^{(9)}_{\ref{qi-sec-inside-lad-len-lift}}(K)$, there exists $\dl_{\ref{general-ladder-is-hyp}}=\dl_{\ref{general-ladder-is-hyp}}(K,R)$ such that $L_{KR}:=N_R(\L_K)$ is $\dl_{\ref{general-ladder-is-hyp}}$-hyperbolic with respect to the path metric induced from $X$.
\end{theorem}

\begin{proof}
	
\underline{{\bf \emph{Subdivision of ladder}}}: We fix a fiber geodesic $\L_{a,u}$ for some $u\in\mfT$ and $a\in B_u$. Let $K_1=K_{\ref{trisection-of-ladder}}(K)$. We also fix $A_0>max\{\phi(2K+k_{\ref{monotonic-map-between-qgs}}(\dl\pr_0,L\pr_0,K)),\phi(4K(2R+1)+2R+1),\phi(8KR+2R),\phi(4KD_2+D_2)\}$ where $D_2$ is defined below in the verification of condition $(4)$ of Proposition \ref{combi-hyp-sps}. For $x\in\L_{a,u}$, $\Sigma_x$ denote a maximal $K$-qi section in $\L$ over $B_x$, say. Let $\gm:[0,l]\ri\L_{a,u}$ be the arc length parametrization such that $\gm(0)=bot(\L_{a,u})$ and $\gm(l)=top(\L_{a,u})$. Now we inductively subdivide $\L$ into small girth ladders as follows. First, inductively we construct a finite sequence of points on $\L_{a,u}$ and $K$-qi sections through that, which will help in subdivision. Note that set map from $\mfB$ to $bot(\L)\cap\pi_X^{-1}(\mfB)$ is $K$-qi section in $\L$. Set $x_0=\gm(0)$ and $\Sigma_{x_0}=bot(\L)\cap\pi_X^{-1}(\mfB)$. Suppose $x_i=\gm(t_i)$ has been constructed. Let $$\Omega_{i+1}=\{t\in(t_i,l]:\gm(t)=x\textrm{ and }d^f(\Sigma_{x_i}(s),\Sigma_x(s))>A_0,\fa~s\in B_{x_i}\cap B_x\}.$$If $\Omega_{i+1}=\emptyset$, then we define $x_{i+1}=\gm(l)$ and stop the process. Otherwise, we take $x_{i+1}= \gm(\textrm{min }\{\textrm{ inf }\Omega_{i+1} +A_0/2,l\})$, and $x\pr_{i+1}=\gm(\textrm{ inf }\Omega_{i+1}-A_0/2)$. The construction of these points and sections stop at $n$-th step if $x_n=\gm(l)$.
	
\emph{Claim}: Let $i>j$ and $d^f(\Sigma_{x_i}(t),\Sigma_{x_{j}}(t))>A_0$, $\fa~t\in B_{x_i}\cap B_{x_{j}}$. Then for $v\in\pi_B(B_{x_i}\cap B_{x_{j}})$ and $b\in B_v$, we have the order $bot(\L_{b,v})\le\Sigma_{x_i}(b)\le\Sigma_{x_{j}}(b)\le top(\L_{b,v})$ in the fiber geodesic $\L_{b,v}$.
	
\emph{Proof of the claim}: Indeed, because we have the family of order preserving monotonic maps $\{h_{wv}\}$, if $B_v$ is single vertex, then we are done. Otherwise, let $b,b\pr\in B_v$ such that $d_B(b,b\pr)=1$ and $bot(\L_{b\pr,v})\le\Sigma_{x_i}(b\pr)\le\Sigma_{x_{j}}(b\pr)\le top(\L_{b\pr,v})$ but $bot(\L_{b,v})\le\Sigma_{x_{j}}(b)<\Sigma_{x_i}(b)\le top(\L_{b,v})$. Let $\al=[bot(\L_{b\pr,u}),\Sigma_{x_{j}}(b\pr)]^f$$\sse \L_{b\pr,u}$ and $\bt=[bot(\L_{b,u}),\Sigma_{x_{j}}(b)]^f\sse \L_{b,u}$. Consider the $\dl\pr_0$-hyperbolic space $F_{b\pr b}:=\pi_X^{-1}([b\pr,b])$ (see Lemma \ref{com-two-hyp-sps}). Then we apply Lemma \ref{monotonic-map-between-qgs}, to $L\pr_0$-quasigeodesic $\al,~\bt$ in $\dl\pr_0$-hyperbolic space $F_{b\pr b}$. So there is a point $z\in\bt$ such that $d_{X_u}(\Sigma_{x_i}(b\pr),z)\le d_{F_{b\pr b}}(\Sigma_{x_i}(b\pr),z)\le k_{\ref{monotonic-map-between-qgs}}(\dl\pr_0,L\pr_0,K)$. Thus by triangle inequality we have $d_{X_u}(z,\Sigma_{x_i}(b))\le 2K+k_{\ref{monotonic-map-between-qgs}}(\dl\pr_0,L\pr_0,K)$, and so $d^f(z,\Sigma_{x_i}(b))\le \phi(2K+k_{\ref{monotonic-map-between-qgs}}(\dl\pr_0,L\pr_0,K))$. Since $\Sigma_{x_{j}}(b)\in[z,\Sigma_{x_i}(b)]^f$, so $d^f(\Sigma_{x_{j}}(b),\Sigma_{x_i}(b))\le d^f(\Sigma_{x_i}(b),z)\le \phi(2K+k_{\ref{monotonic-map-between-qgs}}(\dl\pr_0,L\pr_0,K))<A_0$ which contradicts to the fact that $d^f(\Sigma_{x_i}(t),\Sigma_{x_{j}}(t))>A_0$ for all $t\in B_{x_i}\cap B_{x_{j}}$.\qed\smallskip
	
Hence by Lemma \ref{trisection-of-ladder}, if $d^f(\Sigma_{x_i}(t),\Sigma_{x_{j}}(t))>A_0~\fa~t\in B_{x_i}\cap B_{x_{j}}$, the $K$-qi sections $\Sigma_{x_i}$ and $\Sigma_{x_{j}}$ bounds a $(K_1,C_1,\ep_1)$-subladder in $\L_K$ over the central base $B_{x_i}\cap B_{x_{j}}$, where $K_1=K_{\ref{trisection-of-ladder}}(K),~C_1=C_{\ref{trisection-of-ladder}}(K,C,\epsilon)$ and $\ep_1=\epsilon_{\ref{trisection-of-ladder}}(K,C,\epsilon)$. If $j=i+1$, we denote this subladder by $\L^{(i)}=\L(\Sigma_{x_i},\Sigma_{x_{i+1}})$.
	
	Again, if $d^f(\Sigma_{x_i}(t),\Sigma_{x_{i+1}}(t))>A_0$, $\fa~t\in B_{x_i}\cap B_{x_{i+1}}$, then $\Sigma_{x\pr_{i+1}}$ is a maximal $K$-qi section in $\L_K$ through $x\pr_{i+1}$ over $B_{x\pr_{i+1}}$. Also from the construction, $\exists~a\in B_{x_i}\cap B_{x\pr_{i+1}}$ and $\exists~b\in B_{x\pr_{i+1}}\cap B_{x_{i+1}}$ such that
	\begin{eqnarray}\label{girth-condition}
		d^f(\Sigma_{x_i}(a),\Sigma_{x\pr_{i+1}}(a))\le A_0\textrm{ and }d^f(\Sigma_{x\pr_{i+1}}(b),\Sigma_{x_{i+1}}(b))\le A_0.
	\end{eqnarray}
	It is very well possible that $\Sigma_{x\pr_{i+1}}$ does not lie fully in $\L^{(i)}$. In that case, considering Lemma \ref{qi-sec-inside-lad-len-lift} $(2)$, we adjust $\Sigma_{x\pr_{i+1}}$ to lie inside $\L^{(i)}$, turning it into a $C_{\ref{qi-sec-inside-lad-len-lift}}(K)$-qi section over possibly a smaller base than $B_{x\pr_{i+1}}$. (We refer to the proof of  Lemma \ref{qi-sec-inside-lad-len-lift} $(2)$, i.e., \cite[Lemma $3.1$]{pranab-mahan}.) We still denote this modified qi section as $\Sigma_{x\pr_{i+1}}$ and its base as $B_{x\pr_{i+1}}$. We note that this modification will not effect to the girth condition \ref{girth-condition}; and $K_1=K_{\ref{trisection-of-ladder}}(K)=C_{\ref{qi-sec-inside-lad-len-lift}}(K)$.
	
	Therefore, by Lemma \ref{bisection-of-ladder}, the $K_1$-qi section $\Sigma_{x\pr_{i+1}}$ subdivides the ladder $\L^{(i)}$ into two $(K_2,C_2,\ep_2)$-subladders, where $K_2=K_{\ref{bisection-of-ladder}}(K_1),C_2=C_{\ref{bisection-of-ladder}}(K_1,C_1,\ep_1)$ and $\ep_2=\epsilon_{\ref{bisection-of-ladder}}(K_1,C_1,\ep_1)$. Let us denote these subladders of $\L^{(i)}$ by  $\L^{i1}=\L(\Sigma_{x_i},\Sigma_{x\pr_{i+1}})$ and $\L^{i2}=\L(\Sigma_{x\pr_{i+1}},\Sigma_{x_{i+1}})$. Note that $K_2=C^{(2)}_{\ref{qi-sec-inside-lad-len-lift}}(K)$ and the ladders $\L^{i1}$ and $\L^{i2}$ satisfy the small girth condition \ref{girth-condition}.
	
	Therefore, the ladder $\L$ is subdivided into $(K_1,C_1,\ep_1)$-subladders $\L^{(i)},~0\le i\le n-1$. Also, $\L^{(i)}$'s are further subdivided into two $(K_2,C_2,\ep_2)$-subladders $\L^{i1},\L^{i2}$ in $\L^{(i)}$ except possibly for $i=n-1$.
	
	\begin{lemma}\label{x_i-x_{i+1}}
		Let $x\in\Sigma_{x_{i}}$ and $y\in\Sigma_{x_j}$ such that $d_X(x,y)\le D$ and $i\ne j$. Then there is a point $c\in B_{x_i}\cap B_{x_j}$ such that $d^f(\Sigma_{X_i}(c),\Sigma_{x_j}(c))\le \phi(4KD+D)$.
	\end{lemma} 
	
	\begin{proof}
		Let $\pi_X(x)=a$ and $\pi_X(y)=b$. Suppose $c\in[a,b]$ such that $a\in B_{x_i}\cap B_{x_j}$. Since $B_{x_i}$'s are isometrically embedded in $B$, $[a,c]_B\sse B_{x_i}$ and $[c,b]_B\sse B_{x_j}$. Now $d_B(a,b)\le d_X(x,y)\le D$ implies $d_B(a,c)\le D$ and $d_B(c,b)\le D$. By taking $K$-qi lift of $[a,c]_B$ and $[c,b]_B$ in $\Sigma_{x_i}$ and $\Sigma_{x_j}$ respectively, we have $d_X(x,\Sigma_{x_i}(c))\le2KD$ and $d_X(y,\Sigma_{x_j}(c))\le2KD$. Again by triangle inequality, $d_X(\Sigma_{x_i}(c),\Sigma_{x_j}(c))\le4KD+D$. Hence $d^f(\Sigma_{X_i}(c),\Sigma_{x_j}(c))\le\phi(4KD+D)$.
	\end{proof}	
	
\underline{{\bf \emph{Continuation of the proof of Theorem \ref{general-ladder-is-hyp}}}}: We use the following notations for the proof. $$X_i:=N_R(\L^{(i)}),~L^{i1}:=N_R(\L^{i1}),~L^{i2}:=N_R(\L^{i2}),~0\le i\le n-1$$
	
	From the construction, it follows that $L_{KR}=\cup_{i=0}^{n-1} X_i$. We will verify the conditions of Proposition \ref{combi-hyp-sps}.
	
	$(1)$ \emph{$X_i$'s are uniformly hyperbolic, $0\le i\le n-1$}.
	
	Note that $(\L^{(i)})^g|_{B_{x_i}\cap B_{x_{i+1}}}>A_0$ except possibly for $i=n-1$ (see Definition \ref{girth-neck} for notation). If $(\L^{(n-1)})^g|_{B_{x_{n-1}}\cap B_{x_n}}\le A_0$ then by Proposition \ref{hyp_small_ladder}, $X_{n-1}$ is $\dl_{\ref{hyp_small_ladder}}(K_1,A_0,R)$-hyperbolic. Otherwise, the ladder $\L^{(i)}$ is subdivided by a $K_1$-qi section $\Sigma_{x\pr_{i+1}}$ into two $(K_2,C_2,\ep_2)$-subladders, $\L^{i1}$ and $\L^{i2}$ such that their girth over central base $\le A_0$ (see inequation \ref{girth-condition}). Since $(X,B,T)$ satisfies flaring condition, by Proposition \ref{hyp_small_ladder}, $L^{i1}$ and $L^{i2}$ are $\dl_{\ref{hyp_small_ladder}}(K_2,A_0,R)$-hyperbolic. Note that $N_{2K_1}(\Sigma_{x\pr_{i+1}})$ is a connected subspace in $L^{i1}\cap L^{i2}$, and by Lemma \ref{union_of_two_ladders}, $L^{i1}\cap L^{i2}\sse N_{R_{\ref{union_of_two_ladders}}(K_1,R)}(N_{2K_1}(\Sigma_{x\pr_{i+1}}))$. Again the inclusions $N_{2K_1}(\Sigma_{x\pr_{i+1}})\ri L^{i1}$ and $N_{2K_1}(\Sigma_{x\pr_{i+1}})\ri L^{i2}$ are $K_1(2K_1+1)$-qi embeddings (see Lemma \ref{qi-sec-inside-lad-len-lift} $(3)$). So by Lemma \ref{hd-imp-qi}, $L^{i1}\cap L^{i2}$ is $L_1$-qi embedded in both $L^{i1}$ and $L^{i2}$ for some $L_1$ depending on $K_1(2K_1+1)$ and $R_{\ref{union_of_two_ladders}}(K_1,R)$. Therefore, by Remark \ref{combi-hyp-sps-2}, $L^{(i)}$ is $\dl\pr=\dl_{\ref{combi-hyp-sps-2}}(\dl_{\ref{hyp_small_ladder}}(K_2,A_0,R),L_1)$. Therefore, for $0\le i\le n-1$, $X_i$ is $\dl_1$-hyperbolic metric space, where  $\dl_1=max\{\dl\pr,\dl_{\ref{hyp_small_ladder}}(K_1,A_0,R)\}$.\smallskip
	
$(2)$ Let $0\le i\le n-2$. By Lemma \ref{qi-sec-inside-lad-len-lift} $(3)$, $N_{2K}(\Sigma_{x_{i+1}})$ is $K(2K+1)$-qi embedded in both $X_i$ and $X_{i+1}$. By \emph{Fact} $1$, $\Sigma_{x_i}$ and $\Sigma_{x_{i+2}}$ bounds $(K_1,C_1,\epsilon_1)$-ladder. So by Lemma \ref{union_of_two_ladders}, $X_i\cap X_{i+1}\sse N_{R_{\ref{union_of_two_ladders}}(K_1,R)}(\Sigma_{x_{i+1}})$. So by Lemma \ref{hd-imp-qi}, $X_i\cap X_{i+1}$ is $L_2$-qi embedded in both $X_i$ and $X_{i+1}$ for some $L_2$ depending on $K(2K+1)$ and $R_{\ref{union_of_two_ladders}}(K_1,R)$.\smallskip
	
$(3)$ Let $x\in X_i$, $y\in X_{i+1}$ and $\al$ be a path in $L_{KR}$ joining $x$ and $y$. 
	
\emph{Claim}: There is a point in $\al$ which is $R$-close to $\L^{(i)}$ and $\L^{(i+1)}$. This is enough for (3).
	
\emph{Proof of the claim}: Suppose this is not the case. Then there are points $z\in\al$, $z_i\in\L^{(i)}$ and $z_j\in\L^{(j)}$ such that $d_{X_i}(z,z_i)\le R$, $d_{X_j}(z,z_j)\le R$ and $j-i\ge2$. So $d_X(z_i,z_j)\le 2R$. Then by Lemma \ref{x_i-x_{i+1}}, $\exists~c\in B_{x_i}\cap B_{x_j}$ such that $d^f(\Sigma_{x_i}(c),\Sigma_{x_j}(c))\le\phi(8KR+2R)<A_0$ which contradicts to the construction of $\Sigma_{x_i}$'s.\qed\smallskip 
	
$(4)$ Now we want to prove that the pair $(Y_i,Y_{i+1})$ is uniformly cobounded for $1\le i\le n-2$ where $Y_i=X_{i-1}\cap X_i$ and $Y_{i+1}=X_i\cap X_{i+1}$. Since $X_i$'s are $\dl_1$-hyperbolic and the inclusion $N_{2K}(\Sigma_{x_i})\ri X_i$ is a $K(2K+1)$-qi embedding (see Lemma \ref{qi-sec-inside-lad-len-lift} $(3)$), then $\Sigma_{x_i}$'s are $K\pr$-quasiconvex in $X_i$, where $K\pr=K_{\ref{quasi-goes-to-quasi}}(\dl_1,K(2K+1),0)+2K$ (see Lemma \ref{quasi-goes-to-quasi} $(1)$). By similar argument, we have that $\Sigma_{x_{i+1}}$ is also $K\pr$-quasiconvex in $X_i$.
	
	We prove that the set of nearest point projections of $\Sigma_{x_i}$ on $\Sigma_{x_{i+1}}$ in the metric of $X_i$ is uniformly bounded; which will complete the proof. Indeed, let $\rho:\Sigma_{x_i}\ri\Sigma_{x_{i+1}}$ be a nearest point projection map in $X_i$ such that the diameter of $\rho(\Sigma_{x_i})$ is bonded by $D$ in the metric of $X_i$. Then by Lemma \ref{small-imp-small} there is $D_1$ depending on $\dl_1,K\pr$ and $D$ such that the pair $(\Sigma_{x_i},\Sigma_{x_{i+1}})$ is $D_1$-cobounded in $X_i$. By \emph{Fact} $1$, $\Sigma_{x_{i-1}}$ and $\Sigma_{x_{i+1}}$ bounds a $(K_1,C_1,\epsilon_1)$-ladder. So by Lemma \ref{union_of_two_ladders}, $Hd(Y_i,\Sigma_{x_i})$ and $Hd(Y_{i+1},\Sigma_{x_{i+1}})$ are bounded by $R_{\ref{union_of_two_ladders}}(K_1,R)$. Hence by Lemma \ref{proj-on-qc} $(2)$, the pair $(Y_i,Y_{i+1})$ is $D\pr$-cobounded where $D\pr=D_1+2E_{\ref{proj-on-qc}}(\dl_1,K\pr,R_{\ref{union_of_two_ladders}}(K_1,R))$.

	Let $\rho(y_j)=p_j$ for $y_j\in\Sigma_{x_i}$ and $p_j\in\Sigma_{x_{i+1}},~j=1,2$. We prove that $d_{X_i}(p_1,p_2)$ is bounded by $D$. By \cite[Lemma $1.31 (2)$]{pranab-mahan}, the arc-length parametrizations of $[y_1,p_1]_{X_i}\cup[p_1,p_2]_{X_i}$ and $[y_2,p_2]_{X_i}\cup[p_2,p_1]_{X_i}$ are $(3+2K\pr)$-quasigeodesic in $X_i$.
	
	\emph{Claim}: $d_{X_i}(p_1,p_2)\le L_{\ref{lo_vs_gl}}(\dl_1,3+2K\pr,3+2K\pr)=:D$.
	
	\emph{Proof of the claim}: On contrary, suppose $d_{X_i}(p_1,p_2)> L_{\ref{lo_vs_gl}}(\dl_1,3+2K\pr,3+2K\pr)$. Then by Lemma \ref{lo_vs_gl}, $[y_1,p_1]_{X_i}\cup[p_1,p_2]_{X_i}\cup[p_2,y_2]_{X_i}$ is $\lm$-quasigeodesic in $X_i$, where  $\lm=\lm_{\ref{lo_vs_gl}}(\dl_1,3+2K\pr,3+2K\pr)$. Now by the stability of quasigeodesic (see Lemma \ref{ml}) in $X_i$ and $K\pr$-quasiconvexity of $\Sigma_{x_i}$ in $X_i$, $\exists~z_1,z_2\in\Sigma_{x_i}$ such that $d_{X_i}(p_j,z_j)\le D_2$, where $D_2=D_{\ref{ml}}(\dl_1,\lm,\lm)+K\pr,~j=1,2$. In particular, $d_X(\Sigma_{x_i},\Sigma_{x_{i+1}})\le D_2$. Then by Lemma \ref{x_i-x_{i+1}}, there is  $c\in B_{x_i}\cap B_{x_{i+1}}$ such that $d^f(\Sigma_{x_i}(c),\Sigma_{x_{i+1}}(c))\le \phi(4KD_2+D_2)<A_0$  which contradicts to our construction of $\Sigma_{x_i}$'s.\qed\smallskip
	
	$(5)$ On contrary, suppose $d_{X_i}(Y_i,Y_{i+1})<1$. Then $d_{X_i}(\Sigma_{x_i},\Sigma_{x_{i+1}})\le 2R+1$. Then by Lemma \ref{x_i-x_{i+1}}, $\exists~c\in B_{x_i}\cap B_{x_{i+1}}$ such that $d^f(\Sigma_{x_i}(c),\Sigma_{x_{i+1}}(c))\le \phi(4K(2R+1)+2R+1)<A_0$ which contradicts to our construction of $\Sigma_{x_i}$'s.

	Therefore, we have shown that the collection $\{X_i:0\le i\le n-1\}$ satisfies all conditions of Proposition \ref{combi-hyp-sps}. Hence, $N_R(\L_K)=L_{RK}$ is $\dl_{\ref{general-ladder-is-hyp}}$-hyperbolic, where $\dl_{\ref{general-ladder-is-hyp}}=\dl_{\ref{combi-hyp-sps}}(\dl_1,K(2K+1),D\pr)$.
\end{proof}


\section{Hyperbolicity of flow spaces}\label{hyp-of-flow-sp}

Suppose $R=6\dl_0+\theta_{\ref{modified-proj-on-qg}}(\dl\pr_0,L\pr_0,\lambda\pr_0)+4\lambda\pr_0+8\dl\pr_0>R_{\ref{R-sep-D-cobdd}}(\dl\pr_0,\lambda\pr_0)=2\lambda\pr_0+5\dl\pr_0$ and $k=K_{\ref{qi-sec-inside-lad-len-lift}}$. Let $u\in T$ and $\F l_K(X_u)$ be the flow space of $X_u$ obtained for the parameters $k$ and $R$ (see Definition \ref{flow-space-def}). More precisely, $\F l_K(X_u)$ is a $(K,C,\epsilon)$-semicontinuous family, where $K=K_{\ref{const-of-flow-sp}}(k,R),~C=C_{\ref{const-of-flow-sp}}$ and $\epsilon=\epsilon_{\ref{const-of-flow-sp}}(R)$. This section is devoted to proving the uniform hyperbolicity of a uniform neighborhood of $\F l_K(X_u)$ with the induced path metric. {\bf In this section, we work with these flow spaces and these parameters}. {\bf So we reserve} $\bm{K,C}$ {\bf and} $\bm{\epsilon}$ {\bf for the above values}. The idea is to apply Bowditch's criterion (see Proposition \ref{combing}) to show that $N_L(\F l_K(X_u))$ is hyperbolic (see Theorem \ref{flow-sp-is-hyp}). Given a pair of points, we first find a ladder inside $\F l_K(X_u)$ containing those points (see Corollary \ref{ladder-in-flow-sp-for-two-pts}), and in that ladder, we take a fixed geodesic path joining those points for the family of paths to apply Proposition \ref{combing}. 
This strategy is elaborated in \cite[Chapter $5$]{ps-kap} when $X$ is a tree of metric spaces. In the line of finding ladder, we prove something more in the following Proposition which is kind of heart of this section. In view of Remark \ref{bigsmall-subsp}, for this section, we require the tree of metric bundles $(X,B,T)$ to satisfy $max\{C^{(9)}_{\ref{qi-sec-inside-lad-len-lift}}(k_{\ref{existence-of-tripod-ladder}}),\mathcal{R}_0(2k_{\ref{ladder-in-flow-sp-for-two-pts}}+1)\}=:\bm{k}_{\bm{*}}$
-flaring condition, where $\mathcal{R}_0=L_{\ref{mitra's-retraction-on-semicts-subsp}}(k_{\ref{ladder-in-flow-sp-for-two-pts}},c_{\ref{ladder-in-flow-sp-for-two-pts}},\varepsilon_{\ref{ladder-in-flow-sp-for-two-pts}})$ is defined in the proof of Lemma \ref{Li's-bar-are-Hd-close}, Case $(1)$.

\subsection{Existence of ladder in $\F l_K(X_u)$}

\begin{prop}\label{existence-of-tripod-ladder}
There are constants $k_{\ref{existence-of-tripod-ladder}}=k_{\ref{existence-of-tripod-ladder}}(K),~c_{\ref{existence-of-tripod-ladder}}=c_{\ref{existence-of-tripod-ladder}}(K)$ and $\varepsilon_{\ref{existence-of-tripod-ladder}}=\varepsilon_{\ref{existence-of-tripod-ladder}}(K)$ such that the following hold.
	
Suppose $x^i\in\F l_K(X_u)$ and $\Sigma_i$ is a $K$-qi section through $x^i$ over $B_{x^i}:=B_{[u,\pi(x^i)]}$ lying inside $\F l_K(X_u),$ $i=1,2,3$. Let $\mfB=\bigcap\limits_{i=1}^{3}B_{x^i}$ and $\mfT=\pi_B(\mfB)$. Then we have the following.
\begin{enumerate}
\item There is $(k_{\ref{existence-of-tripod-ladder}},c_{\ref{existence-of-tripod-ladder}},\varepsilon_{\ref{existence-of-tripod-ladder}})$-ladder $\L^i,~i=1,2,3$ containing $\Sigma_i$ with a central base $\mfB$ (possibly bigger) such that:
		
\begin{enumerate}
\item Let $S_i=\textrm{hull}(\pi(\L^i))$ and $B_i=\pi_B^{-1}(S_i),~i=1,2,3$, and $B_{123}=\cap_{i=1}^{3}B_i$ and $S_{123}=\cap_{i=1}^{3}S_i$. Then $\Xi=\{\bigcap\limits_{i=1}^{3}\L^i_{b,v}:v\in S_{123},~b\in B_v\}$ is a $k_{\ref{existence-of-tripod-ladder}}$-qi section over $B_{123}$ and $\Xi\sse N^f_{5\dl_0}(\F l_K(X_u))$.
			
\item $\Sigma_i\sse bot(\L^i)\sse\F l_K(X_u)$ and $\Xi\sse top(\L^i),~i=1,2,3$.
			
\item $\L^i\sse N^f_{6\dl_0}(\F l_K(X_u)),~i=1,2,3$.
\end{enumerate}
		
\item There exist $(k_{\ref{existence-of-tripod-ladder}},c_{\ref{existence-of-tripod-ladder}},\varepsilon_{\ref{existence-of-tripod-ladder}})$-ladder $\L^{ij}$ with central base $\mfB$ containing $\Sigma_i$ and $\Sigma_j$ such that $bot(\L^i)\sse top(\L^{ij})$, $bot(\L^j)\sse bot(\L^{ij})$. Also, $\L^{ij}\sse N^f_{2\dl_0}(\F l_K(X_u))$.
\end{enumerate}
	
Although $k_{\ref{existence-of-tripod-ladder}},~c_{\ref{existence-of-tripod-ladder}}$ and $\varepsilon_{\ref{existence-of-tripod-ladder}}$ depend on the constants $C,\epsilon$ and the other structural constants, we keep them implicit.
\end{prop}

\begin{proof}
The construction of $\L^i$, $\L^{ij}$ and $\Xi$ are by induction on $d_T(u,v)$, where $v\in T$. The proof is divided into four steps. In Step (0), we have explained initial step of induction; in Step (1), we have fixed some notations and mentioned some results which will be used in the subsequent construction; in Step (2), we have explained how to proceed the induction step; and finally, we conclude the proposition in Step (3).

{\bf Step (0): Initial step of induction}

As an initial step, we first explain how to get $\L^i$, $\L^{ij}$ and $\Xi$ in $X_u$. Note that $\Sigma_i\cap X_u$ is a $K$-qi section over $B_u$ in the metric of $X_u,~i=1,2,3$. Let $\Sigma_i\cap F_{b,u}=\{x^i_{b,u}\}$ where $b\in B_u$ and $i=1,2,3$. Then by Lemma \ref{centers-forms-qi-section}, for all $b\in B_u$, a $\dl_0$-center, say $z_{b,u}$, of geodesic triangle $\triangle_{b,u}=\triangle(x^1_{b,u},x^2_{b,u},x^3_{b,u})$ in the fiber $F_{b,u}$, forms a $k_{\ref{centers-forms-qi-section}}(K)$-qi section over $B_u$ in the metric of $X_u$. Let $Y_{b,u}:=\cup_{i=1}^{3}[z_{b,u},x^i_{b,u}]_{F_{b,u}},~b\in B_u$. We refer to $\cup_{b\in B_u}Y_{b,u}$ as a {\em tripod of ladder} over $B_u$ and $[z_{b,u},x^i_{b,u}]_{F_{b,u}}$ as legs of the tripod $Y_{b,u}$ with vertices $\{x^i_{b,u}:i=1,2,3\}$. Recall that $Q_{b,u}=\F l_K(X_u)\cap F_{b,u}$. Now $Q_{b,u}(=F_{b,u})$ is $2\dl_0$-quasiconvex in $F_{b,u}$ implies $[x^i_{b,u},x^j_{b,u}]_{F_{b,u}}\sse N^f_{2\dl_0}(\F l_K(X_u))$ for all distinct $i,j\in\{1,2,3\}$. So $\dl_0$-centers, $z_{b,u}$ of geodesic triangles $\triangle_{b,u}$ belong to $N^f_{5\dl_0}(\F l_K(X_u))$ and $Y_{b,u}=\cup_{i=1}^{3}[z_{b,u},x^i_{b,u}]_{F_{b,u}}\sse N^f_{6\dl_0}(Q_{b,u})$. Here $\L^i_{b,u}=[z_{b,u},x^i_{b,u}]_{F_{b,u}}$ with $top(\L^i_{b,u})=z_{b,u},~bot(\L^i_{b,u})=x^i_{b,u}$ for $i\in\{1,2,3\}$ and $\L^{ij}_{b,u}=[x^i_{b,u},x^j_{b,u}]_{F_{b,u}}$ with $top(\L^{ij}_{b,u})=bot(\L^i_{b,u})=x^i_{b,u},bot(\L^{ij}_{b,u})=bot(\L^j_{b,u})=x^j_{b,u}$ for all distinct $i,j\in\{1,2,3\}$. We note that $\{z_{b,u}:b\in B_u\}\sse\Xi$ (which we are constructing).

{\bf Step (1): Some facts and lemmata}
	
Now we assume the induction hypothesis. In other words, let $v,w\in T$ such that $d_T(u,v)=n,~d_T(u,w)=n+1$ and $d_T(v,w)=1$. Suppose we have constructed $\L^i,~L^{ij}$ and $\Xi$ over $B_t$ for all the vertices $t\in[u,v]$. Now we will explain how and when to extend $\L^i,~\L^{ij}$ and $\Xi$ inside $X_w$. Let $[\mfv,\mfw]$ be the edge joining $\mfv\in B_v$ and $\mfw\in B_w$.

We divide the construction into several cases and subcases. Before going into demonstration, let us fix some notations, collect some facts and lemmata (Lemmata \ref{other-pairs-are-cobounded} and \ref{Lij-is-cobounded}) which will be used in the construction.\smallskip

{\bf Notations}: {\em We use the following notations $\L^i_{a,t}:=\L^i\cap F_{a,t},~\L^{ij}_{a,t}:=\L^{ij}\cap F_{a,t}$ for $t\in T$ and $a\in B_t$. We denote the nearest point projection maps by $P_{\mfw}:F_{\mfv\mfw}\ri F_{\mfw,w}$ and $P_Y:F_{\mfv\mfw}\ri Y_{\mfv,v}$, and the modified projection $($see Definition \ref{modified-projection}$)$ map by $\bar{P}_Y:F_{\mfw,w}\ri Y_{\mfv,v}$, and let $\bar{Y}_{\mfv,v}:=\bar{P}_Y(F_{\mfw,w})$. Specifically, if $Y_{\mfv,v}\ne\emptyset$ then $\bar{Y}_{\mfv,v}$ is the quasiconvex hull $($in $Y_{\mfv,v}$-metric$)$ of three points $\bar{x}^i_{\mfv,v}$, $i=1,2,3$, such that $\bar{x}^i_{\mfv,v}$ is the nearest point projection of $x^i_{\mfv,v}$ $($in $Y_{\mfv,v}$-metric$)$ onto $\bar{Y}_{\mfv,v}$. Suppose $\tilde{x}^i_{\mfw,w}:=P_{\mfw}(\bar{x}^i_{\mfv,v})$ for all $i\in\{1,2,3\}$.

In the construction, we will see that either $Y_{\mfv,v}$ is a genuine tripod or a degenerate tripod (i.e., a geodesic segment). We denote the tripod, in the former case, by $Y_{\mfv,v}:=\cup_{i=1}^{3}[x^{i}_{\mfv,v},z_{\mfv,v}]^f$ $($in Case $1$ and Case $2$ below$)$, and in the later case, by $Y_{\mfv,v}:=[x^i_{\mfv,v},y^i_{\mfv,v}]^f$ $($in Case $3$ and Case $4$ below$)$ with $top(\L^i_{\mfv,v})=y^i_{\mfv,v}$ and $bot(\L^i_{\mfv,v})=x^i_{\mfv,v}$. Let $T_{wv}$ be the connected component of $T\setminus\{v\}$ containing $w$, and $B_{T_{wv}}=\pi_B^{-1}(T_{wv})$ and $X_{T_{wv}}=\pi^{-1}(T_{wv})$.}\smallskip

{\bf Some facts}:\smallskip

\emph{Fact} $(1)$: Suppose $\bar{Y}_{\mfv,v}\sse\L^i_{\mfv,v}$. Then $\bar{x}^{i-1}_{\mfv,v}=\bar{x}^{i+1}_{\mfv,v}$ ($i\pm1$ is considered in module $3$) and $\bar{x}^{i-1}_{\mfv,v}$ is the point closest to $z_{\mfv,v}$ in the induced path metric of $Y_{\mfv,v}$. Hence by Lemma \ref{modified-proj-on-qg} $(2)$ $(b)$, for $t\in\{i\pm1\}$, we have $$d_{\mfv\mfw}(P_{\mfw}(x^i_{\mfv,v}),\tilde{x}^i_{\mfw,w}),~ d_{\mfv\mfw}(P_{\mfw}(z_{\mfv,v}),P_{\mfw}(\bar{x}^{t}_{\mfw,w}))\textrm{ and }d_{\mfv\mfw}(P_{\mfw}(x^{i\pm1}_{\mfv,v}),P_{\mfw}(\bar{x}^{t}_{\mfw,w})$$ are bounded by $D_{\ref{modified-proj-on-qg}}(\dl\pr_0,L\pr_0,\lm\pr_0)=D$ (say) where $\tilde{x}^i_{\mfw,w}=P_{\mfw}(\bar{x}^i_{\mfv,v})$.\smallskip

\emph{Fact} $(2)$: Suppose $\bar{Y}_{\mfv,v}\nsubseteq\L^i_{\mfv,v}$ for any $i\in\{1,2,3\}$. Then by Lemma \ref{modified-proj-on-qg} $(2)$ $(a)$, we have $d_{\mfv\mfw}(P_{\mfw}(x^i_{\mfv,v}),\tilde{x}^i_{\mfw,w})\le D$, where $\tilde{x}^i_{\mfw,w}=P_{\mfw}(\bar{x}^i_{\mfv,v})$.\smallskip

\underline{Now for the following facts, we assume that the pair $(Y_{\mfv,v},F_{\mfw,w})$ is not $C$-cobounded in $F_{\mfv\mfw}$.} Then by Lemma \ref{R-sep-D-cobdd} (1), we have $d_{\mfv\mfw}(Y_{\mfv,v},F_{\mfw,w})\le2\lm\pr_0+5\dl\pr_0$ and so $Hd_{\mfv\mfw}(P_Y(F_{\mfw,w}),P_{\mfw}(Y_{\mfv,v}))\le 4\lm\pr_0+8\dl\pr_0$. 
\smallskip

\emph{Fact} $(3)$:  {\em For all $i\in\{1,2,3\}$, we have $d_{\mfv\mfw}(\bar{x}^i_{\mfv,v},\tilde{x}^i_{\mfw,w})\le K$ and $\tilde{x}^i_{\mfw,w}\in Q_{\mfw,w}$ where $\tilde{x}^i_{\mfw,w}=P_{\mfw}(\bar{x}^i_{\mfv,v})$.} Indeed, $\bar{x}^i_{\mfv,v}\in P_Y(F_{\mfw,w})$, and so $d_{\mfv\mfw}(\bar{x}^i_{\mfv,v},\tilde{x}^i_{\mfw,w})\le4\lm\pr_0+8\dl\pr_0\le K$. Again $d^f(\bar{x}^i_{\mfv,v},Q_{\mfv,v})\le6\dl_0$ (see initial step of induction) and since $R>6\dl_0+4\lambda\pr_0+8\dl\pr_0$, hence $\tilde{x}^i_{\mfw,w}\in Q_{\mfw,w}$ (see construction of flow spaces \ref{flowsps}).\smallskip

\emph{Fact} $(4)$: {\em If $\bar{Y}_{\mfv,v}\sse\L^i_{\mfv,v}$ for some $i\in\{1,2,3\}$ and $d_{\mfv\mfw}(x^{j}_{\mfv,v},Q_{\mfw,w})\le K$ for some $j\ne i$ then $d_{\mfv\mfw}(z_{\mfv,v},Q_{\mfw,w})\le K+\phi(2K)+2\dl_0$.} Indeed, $d_{\mfv\mfw}(P_Y\circ P_{\mfw}(x^{j}_{\mfv,v}),x^{j}_{\mfv,v})\le 2K$ and $P_Y\circ P_{\mfw}(x^{j}_{\mfv,v})\in\bar{Y}_{\mfv,v}$. So $d^f(P_Y\circ P_{\mfw}(x^{j}_{\mfv,v}),x^{j}_{\mfv,v})\le\phi(2K)$. Again $z_{\mfv,v}$ is $\dl_0$-center of geodesic triangle with vertices $\{x^i_{\mfv,v}:i=1,2,3\}$ in the fiber $F_{\mfv,v}$ so $d^f(z_{\mfv,v},[P_Y\circ P_{\mfw}(x^{j}_{\mfv,v}),x^{j}_{\mfv,v}]^f)\le2\dl_0$. Thus $d^f(z_{\mfv,v},x^{j}_{\mfv,v})\le\phi(2K)+2\dl_0$. Hence $d_{\mfv\mfw}(z_{\mfv,v},Q_{\mfw,w})\le K+\phi(2K)+2\dl_0$.\smallskip

\emph{Fact} $(5)$: {\em Fact $(3)$ also says that if $Y_{\mfv,v}=\L^i_{\mfv,v}=[x^i_{\mfv,v},y^i_{\mfv,v}]^f$ $($i.e., $Y_{\mfv,v}$ has one leg$)$ then $d_{\mfv\mfw}(\bar{x}^i_{\mfv,v},\tilde{x}^i_{\mfv,v})\le K$, $d_{\mfv\mfw}(\bar{y}^i_{\mfv,v},\tilde{y}^i_{\mfv,v})\le K$} where $\bar{Y}_{\mfv,v}=[\bar{x}^i_{\mfv,v},\bar{y}^i_{\mfv,v}]^f\sse\L^i_{\mfv,v}$ such that $\bar{x}^i_{\mfv,v}$ is close to $x^i_{\mfv,v}$ and $\tilde{y}^i_{\mfv,v}=P_{\mfw}(\bar{y}^i_{\mfv,v})$. Again by Fact $(1)$, $$d_{\mfv\mfw}(P_{\mfw}(x^i_{\mfv,v}),\tilde{x}^i_{\mfw,w})\le D\textrm{ and }d_{\mfv\mfw}(P_{\mfw}(y^i_{\mfv,v}),\tilde{y}^i_{\mfw,w})\le D.$$

\begin{lemma}\label{other-pairs-are-cobounded}
	Suppose $\bar{Y}_{\mfv,v}\sse\L^i_{\mfv,v}$ for some $i\in\{1,2,3\}$. Then there is a constant $C_{\ref{other-pairs-are-cobounded}}$ satisfying the following.
	
	$(1)$ \textup{(\hspace{-.08mm}\cite[Corollary $5.4$]{ps-kap})} For $t\in\{i\pm1\}$, the pair $(\L^t_{\mfv,v},F_{\mfw,w})$ is $C_{\ref{other-pairs-are-cobounded}}$-cobounded in the path metric of $F_{\mfv\mfw}$. $($Here $i\pm1$ is considered in module $3.)$
	
	$(2)$ The pair $(\L^{i+1i-1}_{\mfv,v},F_{\mfw,w})$ is $C_{\ref{other-pairs-are-cobounded}}$-cobounded in the path metric of $F_{\mfv\mfw}$.
\end{lemma}
\begin{proof}
Fix $t\in\{i\pm1\}$. By Fact $(1)$, $d_{\mfv\mfw}(P_{\mfw}(x^t_{\mfv,v}),P_{\mfw}(z_{\mfv,v}))\le2D$ (by triangle inequality). Again $d_{\mfv,\mfw}(P_{\mfw}(x^{i-1}_{\mfv,v}),P_{\mfw}(x^{i+1}_{\mfv,v}))\le2D$. Therefore, by Lemma \ref{cobounded-pairs}, we can take $C_{\ref{other-pairs-are-cobounded}}:=C_{\ref{cobounded-pairs}}(\dl\pr_0,L\pr_0,\lm\pr_0,2D)$.
\end{proof}	

	\begin{lemma}\label{Lij-is-cobounded}
	There is a constant $C_{\ref{Lij-is-cobounded}}$ satisfying the following.
	
	If the pair $(Y_{\mfv,v},F_{\mfw,w})$ is $C$-cobounded in $F_{\mfv\mfw}$, then the pairs $(\L^{ij}_{\mfv,v},F_{\mfw,w})$ and $(\L^i_{\mfv,v},F_{\mfw,w})$ are $C_{\ref{Lij-is-cobounded}}$-cobounded in $F_{\mfv\mfw}$ for all distinct $i,j\in\{1,2,3\}$.
\end{lemma}
\begin{proof}
	By Lemma \ref{cobounded-pairs}, we can take $C_{\ref{Lij-is-cobounded}}=C_{\ref{cobounded-pairs}}(\dl\pr_0,L\pr_0,\lm\pr_0,C)$.
\end{proof}
{\bf Step (2): Induction step}	
	
Now we return to the construction of $\L^i$, $\L^{ij}$ and $\Xi$. Note that we have induction hypothesis mentioned above. Depending on the coboundedness of the pair $(Y_{\mfv,v},F_{\mfw,w})$ and the emptiness of $\Sigma_i\cap F_{\mfw,w}$, we consider several cases and subcases. In those cases and subcases, we explain how and when to extend the tripod $Y_{\mfv,v}$, in particular, $\L^i,\L^{ij}$ and $\Xi$, first in $F_{\mfw,w}$ and then in the entire $X_w$. (Then in Lemmata \ref{Li} and \ref{Lij}, we will prove that $\L^i$ and $\L^{ij}$ are ladders, and in Corollary \ref{Xi-form-qi-section}, we will prove that $\Xi$ is a qi section.) In the end of some cases and subcases, we make some note which will be used in the proofs of Lemma \ref{Li} and Lemma \ref{Lij}. We recommend the reader first to read the construction and then look at those notes while reading Lemma \ref{Li} and Lemma \ref{Lij}. Also, all the time we refer to the Figure \ref{heart}.\smallskip

{\bf Case 1}: Suppose $Y_{\mfv,v}$ has all its three legs and the pair $(Y_{\mfv,v},F_{\mfw,w})$ is $C$-cobounded in $F_{\mfv\mfw}$. Depending on the emptiness of $\Sigma_i\cap F_{\mfw,w}$, we consider the following subcases.

\emph{Subcase} $(1A)$: Suppose $\Sigma_i\cap X_w=\emptyset$ for all $i\in\{1,2,3\}$. Then we set $\L^i$ and $\L^{ij}$ to be empty over $B_{T_{wv}}$ for all distinct $i,j\in\{1,2,3\}$.
	
\emph{Subcase} $(1B)$: Suppose $\Sigma_i\cap F_{\mfw,w}\ne\emptyset$ for all $i\in\{1,2,3\}$. For all $i\in\{1,2,3\}$, we set $\bar{x}^i_{\mfv,v}$ and $\tilde{x}^i_{\mfw,w}$ to be $\Sigma_i\cap F_{\mfv,v}$ and $\Sigma_i\cap F_{\mfw,w}$ respectively. Then we have tripod of ladder inside $X_w$ formed by three qi sections $\Sigma_i\cap X_w,~i=1,2,3$ as described in initial step of the induction.\smallskip
	
	
\emph{Subcase} $(1C)$: Suppose $\Sigma_{i\pm1}\cap X_w\ne\emptyset$ and $\Sigma_i\cap X_w=\emptyset$. We set $\bar{x}^{i\pm1}_{\mfv,v}=\Sigma_{i\pm1}(\mfv)$ and $\tilde{x}^{i\pm1}_{\mfw,w}=\Sigma_{i\pm1}(\mfw)$. We also set $\bar{x}^i_{\mfv,v}\in\L^i_{\mfv,v}$ to be the point closest to $x^i_{\mfv,v}$ in the induced path metric of $\L^i_{\mfv,v}$ such that $\bar{x}^i_{\mfv,v}$ is $D_1$-close (in $d_{\mfv\mfw}$-metric) to a point $\tilde{x}^i_{\mfv,v}\in Q_{\mfw,w}$. (Existence of $D_1$ and $\bar{x}^i_{\mfv,v}$ are explained below.) Let $\gm$ be a $K$-qi section through $\tilde{x}^i_{\mfw,w}$ lying inside $\F l_K(X_u)\cap X_w$ over $B_w$. Then we have tripod of ladder inside $X_w$ (as described in the initial step of the induction) formed by three qi sections $\Sigma_{i\pm1}\cap X_w$ and $\gm$.
	
{\em Existence of $D_1$ and $\bar{x}^i_{\mfv,v}$}: It is enough to prove that $d_{\mfv\mfw}(z_{\mfv,v},\Sigma_{i-1}(\mfw))\le D_1$. Note that the pair $(Y_{\mfv,v},F_{\mfw,w})$ is $C$-cobounded in $F_{\mfv\mfw}$ and $d_{\mfv\mfw}(\Sigma_{i\pm1}(\mfw),Y_{\mfv,v})\le K$. Hence by triangle inequality, we have $d_{\mfv\mfw}(\Sigma_{i+1}(\mfv),\Sigma_{i-1}(\mfv))\le4K+C$. Then $d^f(\Sigma_{i+1}(\mfv),\Sigma_{i-1}(\mfv))\le\phi(4K+C)$ and so $d^f(\Sigma_{i-1}(\mfv),z_{\mfv,v})\le \phi(4K+C)+\dl_0$ (as $z_{\mfv,v}$ is $\dl_0$-center geodesic triangle with vertices $\{\Sigma_i(\mfv):i=1,2,3\}$ in the fiber $F_{\mfv,v}$). Therefore, again by triangle inequality, we have $d_{\mfv\mfw}(\Sigma_{i-1}(\mfw),z_{\mfv,v})\le K+\phi(4K+C)+\dl_0=:D_1$.

		\begin{figure}[h]
		\includegraphics[width=10cm]{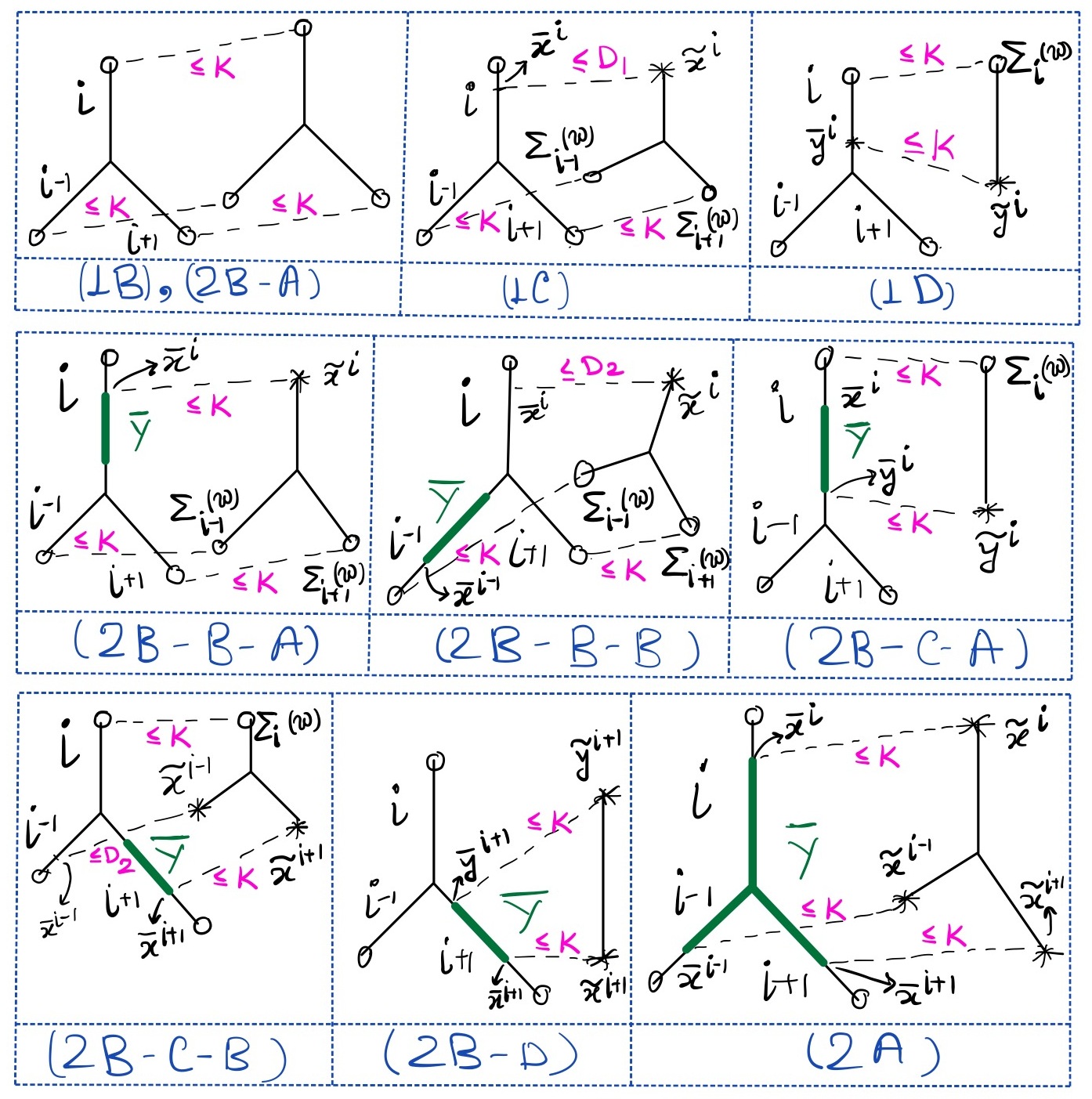}
		\centering
		\caption{\small For ease of notation, we only use $i$ and $\bar{x}^i$ to denote $\L^i_{\mfv,v}$ and $\bar{x}^i_{\mfv,v}$ respectively in the figure and so on. We also omit some not to make it clumsy. The left tripods are in $F_{\mfv,v}$ and the right tripods or degenerated tripods are in $F_{\mfw,w}$. Dotted path joining left and right tripods are geodesics in $F_{\mfv\mfw}$ whose distance is bounded above by $K$, $D_1$ or $D_2$ as shown in the picture.}
		\label{heart}
	\end{figure}
\vspace{-4mm}
\emph{Subcase} $(1D)$: Suppose $\Sigma_i\cap F_{\mfw,w}\ne\emptyset$ and $\Sigma_{i\pm1}\cap F_{\mfw,w}=\emptyset$. We set $\bar{x}^i_{\mfv,v}=\Sigma_i(\mfv)$ and $\tilde{x}^i_{\mfw,w}=\Sigma_i(\mfw)$. We also set $\bar{y}^i_{\mfv,v}\in\L^i_{\mfv,v}$ to be closest to $z_{\mfv,v}$ (in the fiber metric) which is $K$-close in $F_{\mfv\mfw}$-metric to a point in $\tilde{y}^i_{\mfw,w}\in Q_{\mfw,w}$. (Existence of such point is clear as $\Sigma_i\cap F_{\mfw,w}\ne\emptyset$.) 
Let $\gm$ be a $K$-qi section through $\tilde{y}^i_{\mfw,w}$ lying inside $\F l_K(X_u)\cap X_w$ over $B_w$.

In this case, we have degenerate tripod of ladder (i.e., special ladder) inside $X_w$ formed by two qi sections $\Sigma_i\cap X_w$ and $\gm$.

We set this as part of the ladder $\L^i$ with $top(\L^i_{a,w})=\gm(a)$ and $bot(\L^i_{a,w})=\Sigma_i(a), ~a\in B_w$. Also, we set $\L^{ii\pm1}\cap X_w=\L^i\cap X_w$ with same orientation as $\L^i$ has and $\L^{i-1i+1}$ is empty over $B_{T_{wv}}$.\smallskip
	
	\emph{Note $(1D)$}: Since the pair $(Y_{\mfv,v},F_{\mfw,w})$ is $C$-cobounded in $F_{\mfv\mfw}$-metric, so by the construction of $\tilde{y}^i_{\mfw,w}$, we have $d_{\mfv\mfw}(P_{\mfw}(z_{\mfv,v}),\tilde{y}^i_{\mfw,w})\le2K+C$ and $d_{\mfv\mfw}(P_{\mfw}(x^{i\pm1}_{\mfv,v}),\tilde{y}^i_{\mfw,w})\le2K+C$.\smallskip

{\bf Case 2}: Suppose $Y_{\mfv,v}$ has all its three legs and the pair $(Y_{\mfv,v},F_{\mfw,w})$ is not $C$-cobounded in $F_{\mfv\mfw}$. We consider the following two subcases depending on whether $\bar{Y}_{\mfv,v}\sse\L^j_{\mfv,v}$ for some $j\in\{1,2,3\}$ or $\bar{Y}_{\mfv,v}\nsubseteq\L^j_{\mfv,v}$ for any $j\in\{1,2,3\}$.\smallskip
	
\emph{Subcase} $(2A)$: Suppose $\bar{Y}_{\mfv,v}\nsubseteq\L^j_{\mfv,v}$ for any $j\in\{1,2,3\}$. We set $\bar{x}^i_{\mfv,v}=\Sigma_i(\mfv)$ if $\Sigma_i\cap X_v\ne\emptyset$. We also set $\tilde{x}^i_{\mfw,w}=\Sigma_i(\mfw)$ if $\Sigma_i\cap X_w\ne\emptyset$, otherwise, as mentioned in {\em notations}, we set $\tilde{x}^i_{\mfw,w}=P_{\mfw}(\bar{x}^i_{\mfv,v})$. We consider $K$-qi section through $\tilde{x}^i_{\mfw,w}$ lying inside $\F l_K(X_u)\cap X_w$ over $B_w$ if $\Sigma_i\cap X_w=\emptyset$; otherwise, the qi section to be $\Sigma_i\cap X_w$. Then we have tripod of ladder inside $X_w$ as described in the initial step of the induction formed by these three $K$-qi sections through $\tilde{x}^i_{\mfw,w}$.\smallskip

\emph{Note $(2A)$}: By Fact $(2)$, $d_{\mfv\mfw}(P_{\mfw}(\bar{x}^i_{\mfv,v}),\tilde{x}^i_{\mfw,w})\le max\{2K,D\}$ for $i=1,2,3$.\smallskip	
	
\emph{Subcase} $(2B)$: Suppose $\bar{Y}_{\mfv,v}\sse\L^j_{\mfv,v}$ for some $j\in\{1,2,3\}$. We fix such $j$. In this subcase, we consider further division $(2B$-$A)$, $(2B$-$B)$, $(2B$-$C)$, $(2B$-$D)$ as follows.\smallskip
	
$(2B$-$A)$: Suppose $\Sigma_i\cap X_w\ne\emptyset$ for all $i\in\{1,2,3\}$. Then we go back to Subcase $(1B)$.\smallskip
	
$(2B$-$B)$: Suppose $\Sigma_{i\pm1}\cap X_w\ne\emptyset$ and $\Sigma_i\cap X_w=\emptyset$. Here in Subcase $(2B)$, depending on $j$, we have further following division $(2B$-$B$-$A)$ and $(2B$-$B$-$B)$.\smallskip
	
$(2B$-$B$-$A)$: Let $\bar{Y}_{\mfv,v}\sse\L^i_{\mfv,v}$, i.e., $j=i$. Then by Fact $(3)$, $d_{\mfv\mfw}(\bar{x}^i_{\mfv,v},\tilde{x}^i_{\mfw,w})\le K$. Let $\gm$ be a $K$-qi section through $\tilde{x}^i_{\mfw,w}$ lying inside $\F l_K(X_u)\cap X_w$ over $B_w$. We set $\bar{x}^{i\pm1}_{\mfv,v}=\Sigma_{i\pm1}(\mfv)$ and $\tilde{x}^{i\pm1}_{\mfw,w}=\Sigma_{i\pm1}(\mfw)$. In this subcase, we will have tripod of ladder inside $X_w$ formed by three $K$-qi sections $\Sigma_{i\pm1}\cap X_w$ and $\gm$. To set $\L^i$ and $\L^{ij}$ over $B_w$, we go back to the initial step of induction with qi sections $\Sigma_{i\pm1}\cap X_w$ and $\gm$.\smallskip
	
\emph{Note $(2B$-$B$-$A)$}: By Fact $(1)$, $d_{\mfv\mfw}(P_{\mfw}(x^i_{\mfv,v}),\tilde{x}^i_{\mfw,w})\le D$.\smallskip
	
$(2B$-$B$-$B)$: Suppose $\bar{Y}_{\mfv,v}\sse\L^j_{\mfv,v},~j\in\{i\pm1\}$ (i.e. $j\ne i$). Then by Fact $(4)$, $d_{\mfv\mfw}(z_{\mfv,v},Q_{\mfw,w})\le K+\phi(2K)+2\dl_0=D_2$ (say). We set $\bar{x}^i_{\mfv,v}\in\L^i_{\mfv,v}$ to be closest to $x^i_{\mfv,v}$ (in the fiber metric) such that $\bar{x}^i_{\mfv,v}$ is $D_2$-close in $F_{\mfv\mfw}$-metric to a point $\tilde{x}^i_{\mfw,w}\in Q_{\mfw,w}$. Let $\gm$ be a $K$-qi section through $\tilde{x}^i_{\mfw,w}$ lying inside $\F l_K(X_u)\cap X_w$ over $B_w$. We set $\bar{x}^{i\pm1}_{\mfv,v}=\Sigma_{i\pm1}(\mfv)$ and $\tilde{x}^{i\pm1}_{\mfw,w}=\Sigma_{i\pm1}(\mfw)$. Then we have tripod of ladder inside $X_w$ (as discussed in the initial step of induction) formed by qi sections $\Sigma_{i\pm1}\cap X_w$ and $\gm$.\smallskip

	$(2B$-$C)$: Suppose $\Sigma_i\cap X_w\ne\emptyset$ and $\Sigma_{i\pm1}\cap X_w=\emptyset$. Here in Subcase $(2B)$, depending on $j$, we have further following division $(2B$-$C$-$A)$ and $(2B$-$C$-$B)$.\smallskip
	
$(2B$-$C$-$A)$: $\bar{Y}_{\mfv,v}\sse\L^i_{\mfv,v}$, i.e., $j=i$. We set $\bar{x}^i_{\mfv,v}=\Sigma_i(\mfv)$, $\tilde{x}^i_{\mfw,w}=\Sigma_i(\mfw)$, $\bar{y}^i_{\mfv,v}=\bar{x}^{i\pm1}_{\mfv,v}$ and $\tilde{y}^i_{\mfw,w}=P_{\mfw}(\bar{y}^i_{\mfv,v})$. Now by Fact $(3)$, $d_{\mfv\mfw}(\bar{y}^i_{\mfv,v},\tilde{y}^i_{\mfw,w})\le K$. Let $\gm$ be a $K$-qi section through $\tilde{y}^i_{\mfw,w}$ lying inside $\F l_K(X_u)\cap X_w$ over $B_w$. Then we will have degenerate tripod of ladder inside $X_w$ formed by qi sections $\Sigma_i\cap X_w$ and $\gm$.

We set this as part of the ladder $\L^i$ with $top(\L^i_{a,w})=\gm(a),~bot(\L^i_{a,w})=\Sigma_i(a),~a\in B_w$. Also, we set $\L^{ii\pm1}\cap X_w=\L^i\cap X_w$ with the same orientation as $\L^i\cap X_w$ has and $\L^{i+1i-1}\cap X_w$ is empty over $B_{T_{wv}}$.\smallskip
	
\emph{Note $(2B$-$C$-$A)$}: By Fact $(1)$, $d_{\mfv\mfw}(P_{\mfw}(x^{i\pm1}_{\mfv,v}),\tilde{y}^i_{\mfw,w})\le D$ and $d_{\mfv\mfw}(P_{\mfw}(z_{\mfv,v}),\tilde{y}^i_{\mfw,w})\le D$. Also, we have $d_{\mfv\mfw}(P_{\mfw}(x^i_{\mfv,v}),\tilde{x}^i_{\mfw,w})\le 2K.$

$(2B$-$C$-$B)$: Let $\bar{Y}_{\mfv,v}\sse\L^j_{\mfv,v},~j\in\{i\pm1\}$, i.e., $j\ne i$. By Fact $(4)$, $d_{\mfv\mfw}(z_{\mfv,v},Q_{\mfw,w})\le K+\phi(2K)+2\dl_0=D_2$ (say). We maintain the order of $+$ and $-$ depending on $j=i\pm1$. Take a point $\bar{x}^{i\mp1}_{\mfv,v}\in\L^{i\mp1}_{\mfv,v}$ closest to $x^{i\mp1}_{\mfv,v}$ (in the fiber metric) such that $\bar{x}^{i\mp1}_{\mfv,v}$ is $D_2$-close in $F_{\mfv\mfw}$-metric to a point $\tilde{x}^{i\mp1}_{\mfv,v}\in Q_{\mfw,w}$. Let $\gm$ and $\gm\pr$ be $K$-qi sections through $\tilde{x}^{i-1}_{\mfw,w}$ and $\tilde{x}^{i+1}_{\mfw,w}$ lying inside $\F l_K(X_u)\cap X_w$ over $B_w$. We set $\bar{x}^i_{\mfv,v}=\Sigma_i(\mfv)$, $\tilde{x}^i_{\mfw,w}=\Sigma_i(\mfw)$. Then we will have tripod of ladder inside $X_w$ (as described in the initial step of induction) formed by three qi sections $\Sigma_i\cap X_w$, $\gm$ and $\gm\pr$.\smallskip
	
	
	
$(2B$-$D)$: Suppose $\Sigma_i\cap X_w=\emptyset$ for all $i\in\{1,2,3\}$. Note that $\bar{Y}_{\mfv,v}\sse\L^j_{\mfv,v}$ (we are in Subcase $(2B)$) and $\bar{x}^{j-1}_{\mfv,v}=\bar{x}^{j+1}_{\mfv,v}$. We set  $\bar{y}^j_{\mfv,v}=\bar{x}^{j\pm1}_{\mfv,v}$ and $\tilde{y}^j_{\mfw,w}=P_{\mfw}(\bar{y}^j_{\mfv,v})$, and $\tilde{x}^{j}_{\mfw,w}=P_{\mfw}(\bar{x}^j_{\mfv,v})$. By Fact $(3)$, we have $d_{\mfv\mfw}(\bar{x}^j_{\mfv,v},\tilde{x}^j_{\mfw,w})\le K$ and $d_{\mfv\mfw}(\bar{y}^j_{\mfv,v},\tilde{y}^j_{\mfw,w})\le K$, and $\tilde{x}^j_{\mfw,w},\tilde{y}^j_{\mfw,w}\in Q_{\mfw,w}$. Then we have degenerate tripod of ladder inside $X_w$ formed by two $K$-qi sections, say $\gm$ and $\gm\pr$, through $\tilde{x}^j_{\mfw,w}$ and $\tilde{y}^j_{\mfw,w}$ lying inside $\F l_K(X_u)\cap X_w$ over $B_w$ respectively.

We set this as part of $\L^j$ with $top(\L^j_{a,w})=\gm\pr(a),~bot(\L^j_{a,w})=\gm(a),~a\in B_w$. Also, we set $\L^{jj\pm1}\cap X_w=\L^j\cap X_w$ with the same orientation as $\L^j\cap X_w$ has and $\L^{j+1j-1}\cap X_w$ is empty over $B_{T_{wv}}$.\smallskip
	
\emph{Note $(2B$-$D)$}: By Fact $(1)$, $d_{\mfv\mfw}(P_{\mfw}(x^{j\pm1}_{\mfv,v}),\tilde{y}^j_{\mfw,w})$, $d_{\mfv\mfw}(P_{\mfw}(x^j_{\mfv,v}),\tilde{x}^j_{\mfw,w})$ and $d_{\mfv\mfw}(z_{\mfv,v},\tilde{y}^j_{\mfw,w})$ are bounded by $D$.\smallskip

{\bf Case 3}: Suppose $Y_{\mfv,v}$ has only one leg, i.e., $Y_{\mfv,v}=\L^i_{\mfv,v}=[x^i_{\mfv,v},y^i_{\mfv,v}]^f$ for some $i\in\{1,2,3\}$ and the pair $(Y_{\mfv,v},F_{\mfw,w})$ is $C$-cobounded. So $\L^i\cap X_v$ is degenerate tripod of ladder bounded by two $K$-qi sections, say $\gm_1$ and $\gm_2$, lying inside $\F l_K(X_u)\cap X_v$ over $B_v$. Then from our construction in Case $1$ and Case 2, either $\gm_1$ (say) is $\Sigma_i\cap X_v$ or both $\gm_1$ and $\gm_2$ are not restriction of $\Sigma_i$.

In the later case, we set $\L^i$ to be empty over $B_{T_{wv}}$. Of course, all other $\L^j$ and $\L^{ij}$ are empty over $B_{T_{wv}}$.

Now for the former case, we assume that $\gm_1=\Sigma_i\cap X_v$. We consider the following two subcases depending on the emptiness of $\Sigma_i\cap X_w$.\smallskip
	
\emph{Subcase} $(3A)$: Suppose $\Sigma_i\cap X_w=\emptyset$. Then we set $\L^i$ to be empty over $B_{T_{wv}}$. Of course, all other $\L^j$ and $\L^{ij}$ are empty over $B_{T_{wv}}$.\smallskip	
	
\emph{Subcase} $(3B)$: Suppose $\Sigma_i\cap X_w\ne\emptyset$. Note that $top(\L^i_{\mfv,v})=y^i_{\mfv,v}$ and $bot(\L^i_{\mfv,v})=x^i_{\mfv,v}=\Sigma_i(\mfv)$. We set $\tilde{x}^i_{\mfw,w}=\Sigma_i(\mfw)$, and $\bar{y}^i_{\mfv,v}\in\L^i_{\mfv,v}$ is the closest to $y^i_{\mfv,v}$ (in the fiber metric) such that $\bar{y}^i_{\mfv,v}$ is $K$-close in $F_{\mfv\mfw}$-metric to a point $\tilde{y}^i_{\mfw,w}\in Q_{\mfw,w}$. (Existence of such points are clear as $\Sigma_i\cap F_{\mfw,w}\ne\emptyset$.) Let $\gm$ be a $K$-qi section through $\tilde{y}^i_{\mfw,w}$ lying inside $\F l_K(X_u)\cap X_w$ over $B_w$. Then we have degenerate tripod of ladder inside $X_w$ bounded by qi sections $\Sigma_i\cap X_w$ and $\gm$. We set this as part of $\L^i$ with $top(\L^i_{a,w})=\gm(a)$ and $bot(\L^i_{a,w})=\Sigma_i(a)$, $a\in B_w$. We also set $\L^i\cap X_w$ is the part of the same $\L^{ij}$ as it was over $B_v$ with orientation same as $\L^i\cap X_w$.\smallskip
	
\emph{Note} $(3B)$: Since $\Sigma_i\cap X_w\ne\emptyset$ and the pair $(Y_{\mfv,v},F_{\mfw,w})$ is $C$-cobounded in $F_{\mfv\mfw}$, we have $d_{\mfv\mfw}(P_{\mfw}(x^i_{\mfv,v}),\tilde{x}^i_{\mfw,w})\le 2K$ and $d_{\mfv\mfw}(P_{\mfw}(y^i_{\mfv,v}),\tilde{y}^i_{\mfw,w})\le2K+C$ respectively.\smallskip

{\bf Case 4}: Suppose $Y_{\mfv,v}$ has only one leg, i.e., $Y_{\mfv,v}=\L^i_{\mfv,v}=[x^i_{\mfv,v},y^i_{\mfw,w}]^f$ for some $i\in\{1,2,3\}$ and the pair $(Y_{\mfv,v},F_{\mfw,w})$ is not $C$-cobounded in $F_{\mfv\mfw}$. Then we have two extreme points $\bar{x}_{\mfv,v}$ and $\bar{y}_{\mfv,v}$ of $\bar{Y}_{\mfv,v}\sse\L^i_{\mfv,v}$, which are $K$-close to points $\tilde{x}^i_{\mfw,w}$ and $\tilde{y}^i_{\mfw,w}$ of $Q_{\mfw,w}$ respectively in $F_{\mfv\mfw}$-metric (see Fact $(5)$). From our construction in Case 2, we can have only $\Sigma_i\cap X_w\ne\emptyset$. If $\Sigma_i\cap X_w\ne\emptyset$, then we set $\bar{x}^i_{\mfv,v}=\Sigma_i(\mfv)$ and $\tilde{x}^i_{\mfw,w}=\Sigma_i(\mfw)$. Suppose $\gm$ is a $K$-qi section through $\tilde{x}^i_{\mfw,w}$ lying inside $\F l_K(X_u)\cap X_w$ if $\Sigma_i\cap X_w=\emptyset$; otherwise,  $\gm=\Sigma_i\cap X_w$. Let $\gm\pr$ be a $K$-qi section through $\tilde{y}^i_{\mfw,w}$ lying inside $\F l_K(X_u)\cap X_w$ over $B_w$. Then we have degenerate tripod of ladder inside $X_w$ bounded by $\gm$ and $\gm\pr$ with orientation is same as described in Subcase $(3B)$. Further, we set $\L^i\cap X_w$ is the part of the same $\L^{ij}$ as it was over $B_v$ with the orientation same as $\L^i\cap X_w$.\smallskip
	
\emph{Note for Case 4}: We have $d_{\mfv\mfw}(P_{\mfw}(x^i_{\mfv,v}),\tilde{x}^i_{\mfw,w})\le \max\{2K,D\}$ and $d_{\mfv\mfw}(P_{\mfw}(y^i_{\mfv,v}),\tilde{y}^i_{\mfw,w})\le D$ (see Fact $(5)$).\medskip


Suppose (1) $\L^i=\cup_{v\in T,~b\in B_v}\L^i_{b,v}$ and $\L^{ij}=\cup_{v\in T,~b\in B_v}\L^{ij}_{b,v}$ where $\L^{ij}_{b,v}=[bot(\L^i_{b,v}),bot(\L^j_{b,v})]^f$ for all distinct $i,j\in\{1,2,3\}$; (2) $S_i:=\textrm{hull}(\pi(\L^i))$ and $B_i:=\pi_B^{-1}(S_i)$ for all $i\in\{1,2,3\}$, and $S_{123}=\cap_{1\le i\le 3}S_i$ and $B_{123}=\cap_{1\le i\le 3}B_i$; and (3) $\Xi=\{z_{b,v}:v\in S_{123},~b\in B_v\}$ where $z_{b,v}$ is $\dl_0$-center of geodesic triangle in the fiber $F_{b,v}$ with vertices $\{bot(\L^i_{b,v}):i=1,2,3\}$. From the construction, we have $\Xi\sse N^f_{5\dl_0}(\F l_K(X_u))$, $\L^i\sse N^f_{6\dl_0}(\F l_K(X_u))$ and $\L^{ij}\sse N^f_{2\dl_0}(\F l_K(X_u))$ for all distinct $i,j\in\{1,2,3\}$.

{\bf Step (3): Conclusion}

Step (2) completes the construction of $\L^i$ and $\L^{ij}$ for all distinct $i,j\in\{1,2,3\}$ as an union of geodesic segments in fibers. We are yet to show that they form ladders (Lemma \ref{Li} and Lemma \ref{Lij}) using Lemma \ref{promoting-of-ladder}. Before proving these lemmata, we first prove that $\Xi$ is a qi section.

Let $v\in S_{123}$. As we saw in the initial step of induction that $\{z_{b,v}:b\in B_v\}$ form a $k_{\ref{centers-forms-qi-section}}(K)$-qi section in $X_v$ over $B_v$. Therefore, to show $\Xi=\{z_{b,v}:v\in S_{123},~b\in B_v\}$ form a (uniform) qi section over $B_{123}$ (Corollary \ref{Xi-form-qi-section}), we only need to prove $d_{\mfv\mfw}(z_{\mfv,v},z_{\mfw,w})$ is uniformly bounded, where $v,w\in S_{123}$, $d_T(v,w)=1$ and $[\mfv,\mfw]$ is the edge joining $\mfv\in B_v$ and $\mfw\in B_w$. This is proved in the following lemma.
	
\begin{lemma}\textup{(\hspace{-.08mm}\cite[Lemma $5.8$]{ps-kap})}\label{zc_v-zc_w}
There exists a constant $k_{\ref{zc_v-zc_w}}$ such that $d_{\mfv\mfw}(z_{\mfv,v},z_{\mfw,w})\le k_{\ref{zc_v-zc_w}}$, where $v,w\in S_{123}$ such that $d_T(v,w)=1$ and $[\mfv,\mfw]$ is the edge joining $\mfv\in B_v$ and $\mfw\in B_w$.
\end{lemma}
\begin{proof}
This situation happens in Subcases $(1B)$, $(1C)$, $(2B$-$A)$, $(2B$-$B$-$A)$, $(2B$-$B$-$B)$, $(2B$-$C$-$B)$ and $(2A)$. We denote $\triangle_v$ and $\hat{\triangle}_v$ for geodesic triangle with vertices $\{\bar{x}^i_{\mfv,v}:i=1,2,3\}$ in $F_{\mfv,v}$ and $F_{\mfv\mfw}$ respectively. We also denote $\triangle_w$ and $\hat{\triangle}_w$ for geodesic triangle with vertices $\{\tilde{x}^i_{\mfw,w}:i=1,2,3\}$ in $F_{\mfw,w}$ and $F_{\mfv\mfw}$ respectively. In all these cases, $d_{\mfv\mfw}(\bar{x}^i_{\mfv,v},\tilde{x}^i_{\mfw,w})\le max\{D_1,D_2,K\}=D_3$ (say) for all $i\in\{1,2,3\}$. (See the {\em Notes} made after each of these cases and $D_2$ is mentioned in $(2B$-$B$-$B)$, $(2B$-$C$-$B)$.) Since $z_{\mfv,v}$ is $\dl_0$-center of $\triangle(x^1_{\mfv,v},x^2_{\mfv,v},x^3_{\mfv,v})$ in $F_{\mfv,v}$, so $z_{\mfv,v}$ is $3\dl_0$-center of $\triangle_v$ in $F_{\mfv,v}$. Thus $z_{\mfv,v}$ is $(3\dl_0+D_{\ref{ml}}(\dl_0,L\pr_0,L\pr_0))$-center $\hat{\triangle}_v$ in $F_{\mfv\mfw}$ (as the inclusion $F_{\mfv,v}\ri F_{\mfv\mfw}$ is a $L\pr_0$-qi embedding, see Lemma \ref{com-two-hyp-sps}). Since the corresponding end points of the triangles $\hat{\triangle}_v$ and $\hat{\triangle}_w$ are $D_3$-distance apart from each other in the path metric of $F_{\mfv\mfw}$, then by slimness of quadrilateral in $F_{\mfv\mfw}$, we get $z_{\mfv,v}$ is $(3\dl_0+D_{\ref{ml}}(\dl_0,L\pr_0,L\pr_0)+2\dl\pr_0+D_3)$-center of $\hat{\triangle}_w$. Again, as $z_{\mfw,w}$ is $\dl_0$-center of $\triangle_w$ and so is $(\dl_0+D_{\ref{ml}}(\dl_0,L\pr_0,L\pr_0))$-center of $\hat{\triangle}_w$. Thus $z_{\mfv,v}$ and $z_{\mfw,w}$ are two $D_4$-center of $\hat{\triangle}_w$ in the path metric of $F_{\mfv\mfw}$, where $D_4=3\dl_0+D_{\ref{ml}}(\dl_0,L\pr_0,L\pr_0)+2\dl\pr_0+D_3$. Hence, by \cite[Lemma $1.76$]{ps-kap}, we have $d_{\mfv\mfw}(z_{\mfv,v},z_{\mfw,w})\le2D_3+9\dl\pr_0=:k_{\ref{zc_v-zc_w}}$.
	\end{proof}
	
	\begin{cor}\label{Xi-form-qi-section}
We have a constant $k_{\ref{Xi-form-qi-section}}=max\{k_{\ref{centers-forms-qi-section}}(K),k_{\ref{zc_v-zc_w}}\}$ such that $\Xi=\bigcup\limits_{b\in B_v,~v\in S_{123}}\{z_{b,v}\}$ forms a $k_{\ref{Xi-form-qi-section}}$-qi section over $B_{123}$ with  $\Xi\sse N^f_{6\dl_0}(\F l_K(X_u))$.
	\end{cor}
Now we show that $\L^i$ and $\L^{ij}$ form ladders.
	
\begin{lemma}\label{Li}
There are constants $k_{\ref{Li}},~c_{\ref{Li}}$ and $\varepsilon_{\ref{Li}}$ such that $\L^i$ is a $(k_{\ref{Li}},c_{\ref{Li}},\varepsilon_{\ref{Li}})$-ladder with a central base $\mfB$ containing $\Sigma_i$, $i=1,2,3$.
\end{lemma}
	
\begin{proof}
We will show that $\L^i=\cup_{v\in T,~b\in B_v}\L^i_{b,v}$ satisfies all the conditions of Lemma \ref{promoting-of-ladder}, and that completes the proof.

\emph{Condition} (1): Note that by Lemma \ref{qi-sec-inside-lad-len-lift} (2), for all $v\in S_i$, $\L^i\cap X_v$ is a special $C_{\ref{qi-sec-inside-lad-len-lift}}(k_{\ref{centers-forms-qi-section}}(K))$-ladder over $B_v$ (see Definition \ref{K-metric-graph-bundle}). Let $v,w\in S_i$ such that $d_T(v,w)=1$ and $[\mfv,\mfw]$ is the edge joining $\mfv\in B_v$ and $\mfw\in B_w$. Suppose $d_T(u,v)<d_T(u,w)$.

According to our notation, $bot(\L^i_{\mfv,v})=x^i_{\mfv,v},~bot(\L^i_{\mfw,w})=\tilde{x}^i_{\mfw,w}$, and if $w\in S_{123}$, $top(\L^i_{\mfv,v})=z_{\mfv,v},~top(\L^i_{\mfw,w})=z_{\mfw,w}$; otherwise, $top(\L^i_{\mfw,w})=\tilde{y}^i_{\mfw,w}$.

If $v,w\in \mfT$, then $d_{\mfv\mfw}(top(\L^i_{\mfv,v}),top(\L^i_{\mfw,w}))=d_{\mfv\mfw}(z_{\mfv,v},z_{\mfw,w})\le k_{\ref{zc_v-zc_w}}$ (by Lemma \ref{zc_v-zc_w}) and $d_{\mfv\mfw}(bot(\L^i_{\mfv,v}),bot(\L^i_{\mfw,w}))=d_{\mfv\mfw}(\Sigma_i(\mfv),\Sigma_i(\mfw))\le K$.

Otherwise, suppose $d_T(v,\mfT)<d_T(w,\mfT)$.

Now let $v,w\in S_{123}$, i.e., both $Y_{\mfv,v}$ and $Y_{\mfw,w}$ have three legs. This happens in $(1B)$, $(1C)$, $(2B$-$A)$, $(2B$-$B$-$A)$, $(2B$-$B$-$B)$, $(2B$-$C$-$B)$ and $(2A)$. Note that $\tilde{x}^i_{\mfw,w}=bot(\L^i_{\mfw,w}),z_{\mfw,w}=top(\L^i_{\mfw,w})$. In all these cases, $d_{\mfv\mfw}(\bar{x}^i_{\mfv,v},\tilde{x}^i_{\mfw,w})\le max\{D_1,D_2,K\}=D_3$ (say). (See the {\em Notes} made after each of these cases and Figure \ref{heart}, and $D_2$ is mentioned in $(2B$-$B$-$B)$, $(2B$-$C$-$B)$.) Again $z_{\mfw,w}$ is $k_{\ref{zc_v-zc_w}}$-close to $\L^i_{\mfv,v}$. Note that max$\{k_{\ref{zc_v-zc_w}},D_3\}=k_{\ref{zc_v-zc_w}}$.

Now let $v\in S_{123}$ and $w\notin S_{123}$, i.e., $Y_{\mfv,v}$ has three legs but $Y_{\mfw,w}$ has one. This happens in $(1D)$, $(2B$-$C$-$A)$, $(2B$-$D)$. Then (by our construction) $top(\L^i_{\mfw,w})$ and $bot(\L^i_{\mfw,w})$ are $K$-close to $\L^i_{\mfv,v}$ in $d_{\mfv\mfw}$-metric.

Finally, $v,w\notin S_{123}$, i.e., both $Y_{\mfv,v}=\L^i_{\mfv,v}$ and $Y_{\mfw,w}=\L^i_{\mfw,w}$ have one leg (see Case 3 and Case 4), then $top(\L^i_{\mfw,w})$ and $bot(\L^i_{\mfw,w})$ are $K$-close to $\L^i_{\mfv,v}$.

Therefore, for the condition $(1)$ of Lemma \ref{promoting-of-ladder}, we can take $K\pr=max\{C_{\ref{qi-sec-inside-lad-len-lift}}(K_{\ref{centers-forms-qi-section}}(K)),k_{\ref{zc_v-zc_w}},K\}$.\smallskip
		
\emph{Condition} (2): Let $v,w\in S_i$ such that $d_T(v,w)=1$ and $d_T(v,\mfT)<d_T(w,\mfT)$. To show a uniform bound on $Hd_{\mfv\mfw}(P_{\mfw}(\L^i_{\mfv,v}),\L^i_{\mfw,w})$, we first show that $$d_{\mfv\mfw}(P_{\mfw}(bot(\L^i_{\mfv,v})),bot(\L^i_{\mfw,w})) \textrm{ and } d_{\mfv\mfw}(P_{\mfw}(top(\L^i_{\mfv,v})),top(\L^i_{\mfw,w}))$$ are uniformly bounded, and then we apply \cite[Corollary $1.116$]{ps-kap}. 
		
Suppose both $Y_{\mfv,v}$ and $Y_{\mfw,w}$ have three legs. This happens in $(1B)$, $(1C)$, $(2B$-$A)$, $(2B$-$B$-$A)$, $(2B$-$B$-$B)$, $(2B$-$C$-$B)$ and $(2A)$. In all these cases, $d_{\mfv\mfw}(P_{\mfw}(z_{\mfv,v}),z_{\mfw,w})\le2k_{\ref{zc_v-zc_w}}$ (by Lemma \ref{zc_v-zc_w}). So we need a bound only for $d_{\mfv\mfw}(P_{\mfw}(x^i_{\mfv,v}),\tilde{x}^i_{\mfw,w})$. In $(1B)$ and $(2B$-$A)$, $d_{\mfv\mfw}(P_{\mfw}(x^i_{\mfv,v}),\tilde{x}^i_{\mfw,w})\le2K$. In $(1C)$, $d_{\mfv\mfw}(P_{\mfw}(x^i_{\mfv,v}),\tilde{x}^i_{\mfw,w})\le max\{2K,2D_1+C\}$. In $(2B$-$B$-$A)$ and $(2A)$, $d_{\mfv\mfw}(P_{\mfw}(x^i_{\mfv,v}),\tilde{x}^i_{\mfw,w})\le max\{2K,D\}$. In $(2B$-$B$-$B)$ and $(2B$-$C$-$B)$, $d_{\mfv\mfw}(P_{\mfw}(x^i_{\mfv,v}),\tilde{x}^i_{\mfw,w})\le max\{2K,2(K+\phi(2K)+\dl_0)\}$.
		
Suppose $Y_{\mfv,v}$ has three legs and $Y_{\mfw,w}$ has one. This happens in $(1D)$, $(2B$-$C$-$A)$, $(2B$-$D)$. Then $d_{\mfv\mfw}(P_{\mfw}(x^i_{\mfv,v}),\tilde{x}^i_{\mfw,w})$ and $d_{\mfv\mfw}(P_{\mfw}(z_{\mfv,v}),\tilde{y}^i_{\mfw,w})$ are bounded by max$\{2K+C, 2K,D\}$ (see the {\em Notes} made after each of these cases and Figure \ref{heart}).
		
Finally, we assume that both $Y_{\mfv,v}$ and $Y_{\mfw,w}$ have one leg (see Case $3$ and Case $4$). Then $d_{\mfv\mfw}(P_{\mfw}(x^i_{\mfv,v}),\tilde{x}^i_{\mfw,w})$ and $d_{\mfv\mfw}(P_{\mfw}(y_{\mfv,v}),\tilde{y}^i_{\mfw,w})$ are bounded by $max\{2K,2K+C,2K+D\}$ (see {\em Note $(3B)$} and {\em Note for Case $4$}).
		
Let $D_5$ be the maximum of all the above constants. Now by \cite[Corollary $1.116$]{ps-kap}, there is a constant $C_1$ depending on $\dl\pr_0,\lm\pr_0$ and $L\pr_0$ such that $$Hd_{\mfv\mfw}(P_{\mfw}(\L^i_{\mfv,v}),[P_{\mfw}(bot(\L^i_{\mfv,v})),P_{\mfw}(top(\L^i_{\mfv,v}))]^f)\le C_1.$$ Since $\L^i_{\mfw,w}$ are $L\pr_0$-quasigeodesic in $F_{\mfv\mfw}$, by slimness of quadrilateral in $F_{\mfv\mfw}$, there is a constant $\ep\pr$ depending on $\dl\pr_0$, $D_5$, $C_1$ and $L\pr_0$ such that $Hd_{\mfv\mfw}(P_{\mfw}(\L^i_{\mfv,v}),\L^i_{\mfw,w})\le\epsilon\pr$.\smallskip
		
\emph{Condition} (3): Suppose $v\in S_i$ and $w\notin S_i$ such that $d_T(v,w)=1$ and $[\mfv,\mfw]$ is the edge joining $\mfv\in B_v$ and $\mfw\in B_w$. That is $\L^i_{\mfw,w}=\emptyset$. Then the pair $(\L^i_{\mfv,v},F_{\mfw,w})$ is $C\pr=max\{C_{\ref{other-pairs-are-cobounded}},C_{\ref{Lij-is-cobounded}}\}$-cobounded in $F_{\mfv\mfw}$ (see Lemmata \ref{other-pairs-are-cobounded} and \ref{Lij-is-cobounded}).
		
Therefore, to conclude the lemma, we can take $k_{\ref{Li}}=k_{\ref{promoting-of-ladder}}(K\pr)$, $c_{\ref{Li}}=c_{\ref{promoting-of-ladder}}(C\pr)$ and $\varepsilon_{\ref{Li}}=\varepsilon_{\ref{promoting-of-ladder}}(\epsilon\pr)$ for the above values of $K\pr,C\pr$ and $\epsilon\pr$.
	\end{proof}
	
	\begin{lemma}\label{Lij}
There are constants $k_{\ref{Lij}},~c_{\ref{Lij}}$ and $\varepsilon_{\ref{Lij}}$ such that $\L^{ij}$ is a $(k_{\ref{Lij}},c_{\ref{Lij}},\varepsilon_{\ref{Lij}})$-ladder with a central base $\mfB$ containing $\Sigma_i$ and $\Sigma_j$ for all distinct $i,j\in\{1,2,3\}$.
	\end{lemma}
	\begin{proof}
We will show that $\L^{ij}=\cup_{v\in T,~b\in B_v}\L^{ij}_{b,v}$ satisfies all the conditions of Lemma \ref{promoting-of-ladder}, and that completes the proof.

Suppose $v\in \pi(\L^{ij})$. In the construction, we have seen that $\L^{ij}\cap X_v$ matches with $\L^i\cap X_v$ or $\L^j\cap X_v$ unless both $\L^i\cap X_v$ and $\L^j\cap X_v$ are nonempty, i.e., we have genuine tripod in $X_v$. Note that we already have seen in Lemma \ref{Li} that $\L^i=\cup_{v\in T,~b\in B_v}\L^i_{b,v}$ satisfies Conditions (1), (2) and (3) of Lemma \ref{promoting-of-ladder} for some constants. Also, for all $v\in\pi(\L^{ij})$, $\L^{ij}\cap X_v$ is a special $C_{\ref{qi-sec-inside-lad-len-lift}}(K)$-ladder over $B_v$ (see Definition \ref{K-metric-graph-bundle}), and for all $v,w\in\mfT$ with $d_T(u,v)<d_T(u,w)$, we have $d_{\mfv\mfw}(x^i_{\mfv,v},\tilde{x}^i_{\mfw,w})\le K$ and $d_{\mfv\mfw}(x^j_{\mfv,v},\tilde{x}^j_{\mfw,w})\le K$. Hence, for Conditions (1) and (2) of Lemma \ref{promoting-of-ladder}, we only need to give the required bound in the following situations.

Suppose $v,w\in \pi(\L^{ij})$ such that $d_T(v,w)=1$, $d_T(\mfT,v)<d_T(\mfT,w)$ and $[\mfv,\mfw]$ is the edge joining $\mfv\in B_v$ and $\mfw\in B_w$.

{\em Case $(a)$}: Both $Y_{\mfv,v}$ and $Y_{\mfw,w}$ have three legs.

{\em Case $(b)$}: $Y_{\mfv,v}$ has three legs but $Y_{\mfw,w}$ has one.
		
\emph{Condition} (1): Let $v,w\in \pi(\L^{ij})$ be as above. Note that $Hd^f(\L^i_{\mfv,v}\cup\L^j_{\mfv,v},\L^{ij}_{\mfv,v})\le2\dl_0$. In Case $(a)$, by Lemma \ref{Li}, the points $\tilde{x}^i_{\mfw,w}$ and $\tilde{x}^j_{\mfw,w}$ are $K'$-close (in $d_{\mfv\mfw}$-metric) to $\L^i_{\mfv,v}\cup\L^j_{\mfv,v}$, where $K'$ is as in the proof of Lemma \ref{Li}. For Case $(b)$, without loss of generality, we assume that $Y_{\mfw,w}=\L^i_{\mfw,w}$. Hence $\L^{ii\pm1}_{\mfw,w}=\L^i_{\mfw,w}$ and $\L^{i-1i+1}_{\mfw,w}=\emptyset$. (See Subcases $(1D)$, $(2B$-$D)$ and $(2B$-$C$-$A)$.) In these cases, $\tilde{x}^i_{\mfw,w}$ and $\tilde{y}^i_{\mfw,w}$ are $K$-close to $\L^i_{\mfv,v}$. Note that $K'>K$.

So in Case $(a)$ (resp. in Case $(b)$), there are points, $x_1,y_1\in\L^{ij}_{\mfv,v}$ such that $d_{\mfv\mfw}(\tilde{x}^i_{\mfw,w},x_1)\le K'+3\dl_0$ and $d_{\mfv\mfw}(\tilde{x}^j_{\mfw,w},y_1)\le K'+3\dl_0$ (resp. $d_{\mfv\mfw}(\tilde{x}^i_{\mfw,w},x_1)\le K'+3\dl_0$ and $d_{\mfv\mfw}(\tilde{y}^i_{\mfw,w},y_1)\le K'+3\dl_0$) with an order $bot(\L^j_{\mfv,v})=bot(\L^{ij}_{\mfv,v})\le y_1\le x_1\le top(\L^{ij}_{\mfv,v})=bot(\L^i_{\mfv,v})$.
		
Therefore, for the condition $(1)$ of Lemma \ref{promoting-of-ladder}, we can take $K\pr_1=max\{K'+3\dl_0,C_{\ref{qi-sec-inside-lad-len-lift}}(K)\}$.\smallskip
		
\emph{Condition} (2): Let $v,w\in \pi(\L^{ij})$ be as above. We will prove that $Hd_{\mfv\mfw}(P_{\mfw}(\L^{ij}_{\mfv,v}),\L^{ij}_{\mfw,w})$ is uniformly bounded.
		
If both $Y_{\mfv,v}$ and $Y_{\mfw,w}$ have one leg, then $Hd_{\mfv\mfw}(P_{\mfw}(\L^{ij}_{\mfv,v}),\L^{ij}_{\mfw,w})\le\varepsilon_{\ref{Li}}$ (by Lemma \ref{Li}). If $Y_{\mfv,v}$ has three legs but $Y_{\mfw,w}$ has one (Case $(b)$), then $\L^s_{\mfw,w}=\emptyset$ for $s$ is equal to either $i$ or $j$. So by Lemma \ref{Li}, diam$(P_{\mfw}(\L^s_{\mfv,v}))\le c_{\ref{Li}}$ in $d_{\mfv\mfw}$-metric. Again if both $Y_{\mfv,v}$ and $Y_{\mfw,w}$ have three legs (Case $(a)$), $Hd_{\mfv\mfw}(P_{\mfw}(\L^s_{\mfv,v}),\L^s_{\mfw,w})\le\varepsilon_{\ref{Li}}$ for $s\in\{i,j\}$. Now $Hd^f(\L^i_{\mfw,w}\cup\L^j_{\mfw,w},\L^{ij}_{\mfw,w})\le2\dl_0$ implies that, in either case, $Hd_{\mfv\mfw}(P_{\mfw}(\L^i_{\mfv,v}\cup\L^j_{\mfv,v}),\L^{ij}_{\mfw,w})\le\varepsilon_{\ref{Li}}+2\dl_0+c_{\ref{Li}}$. Again $Hd^f(\L^i_{\mfv,v}\cup\L^j_{\mfv,v},\L^{ij}_{\mfv,v})\le2\dl_0$ and so by Lemma \ref{proj-on-qc} $(1)$, $Hd_{\mfv\mfw}(P_{\mfw}(\L^i_{\mfv,v}\cup\L^j_{\mfv,v}),P_{\mfw}(\L^{ij}_{\mfv,v}))\le C_{\ref{proj-on-qc}}(\dl\pr_0,\lambda\pr_0).(2\dl_0+1)$. Therefore, combining these inequalities, we have $$Hd_{\mfv\mfw}(P_{\mfw}(\L^{ij}_{\mfv,v}),\L^{ij}_{\mfw,w})\le\epsilon$$where $\epsilon=\varepsilon_{\ref{Li}}+2\dl_0+c_{\ref{Li}}+C_{\ref{proj-on-qc}}(\dl\pr_0,\lambda\pr_0).(2\dl_0+1)$.
		
Therefore, for condition $(2)$ of Lemma \ref{promoting-of-ladder}, we take $\epsilon\pr=max\{\epsilon,\varepsilon_{\ref{Li}}\}=\epsilon$.\smallskip
		
\emph{Condition} (3): So for the condition $(3)$ of Lemma \ref{promoting-of-ladder}, we can take  $C\pr=max\{C_{\ref{other-pairs-are-cobounded}},C_{\ref{Lij-is-cobounded}}\}$ (see Lemmata \ref{other-pairs-are-cobounded} and \ref{Lij-is-cobounded}).
		
Hence, to conclude the lemma, we can take $k_{\ref{Lij}}=k_{\ref{promoting-of-ladder}}(K\pr_1)$, $c_{\ref{Lij}}=c_{\ref{promoting-of-ladder}}(C\pr)$ and $\varepsilon_{\ref{Lij}}=\varepsilon_{\ref{promoting-of-ladder}}(\epsilon\pr)$ for the above values of $K\pr_1,~C\pr$ and $\epsilon\pr$.
	\end{proof}

\underline{{\bf Conclusion of Proposition \ref{existence-of-tripod-ladder}}}: 	
	
Therefore, we complete Proposition \ref{existence-of-tripod-ladder} by taking $k_{\ref{existence-of-tripod-ladder}}=max\{k_{\ref{Li}},k_{\ref{Lij}}\}$, $c_{\ref{existence-of-tripod-ladder}}=max\{c_{\ref{Li}},c_{\ref{Lij}}\}$ and $\varepsilon_{\ref{existence-of-tripod-ladder}}=max\{\varepsilon_{\ref{Li}},\varepsilon_{\ref{Lij}}\}$. Lemma \ref{Li}, Corollary \ref{Xi-form-qi-section} and Lemma \ref{Lij}, respectively, complete the proof of Proposition \ref{existence-of-tripod-ladder} $(1)$, $(1)$ $(a)$ and $(2)$. Note that Proposition \ref{existence-of-tripod-ladder} $(1)$ $(b)$ and $(c)$ follows from the construction of $\L^i$ and $\Xi$.
\end{proof}

As a consequence of Proposition \ref{existence-of-tripod-ladder}, we have the following.

\begin{cor}\label{ladder-in-flow-sp-for-two-pts}
Let $x,y\in\F l_K(X_u)$. Suppose $\Sigma_x$ and $\Sigma_y$ are $K$-qi sections through $x$ and $y$ lying inside $\F l_K(X_u)$ over $B_{[u,\pi(x)]}$ and $B_{[u,\pi(y)]}$ respectively. Then there is a $(k_{\ref{ladder-in-flow-sp-for-two-pts}},c_{\ref{ladder-in-flow-sp-for-two-pts}},\varepsilon_{\ref{ladder-in-flow-sp-for-two-pts}})$-ladder, say $\L_{xy}$, with a central base (possibly bigger than) $B_u$ containing the sections $\Sigma_x,\Sigma_y$ such that $bot(\L_{xy})\sse\F l_K(X_u)$, $top(\L_{xy})\sse\F l_K(X_u)$ and $\L_{xy}\sse N^f_{2\dl_0}(\F l_K(X_u))$, where $k_{\ref{ladder-in-flow-sp-for-two-pts}}=k_{\ref{existence-of-tripod-ladder}},~c_{\ref{ladder-in-flow-sp-for-two-pts}}=c_{\ref{existence-of-tripod-ladder}}$ and $\varepsilon_{\ref{ladder-in-flow-sp-for-two-pts}}=\varepsilon_{\ref{existence-of-tripod-ladder}}$.
\end{cor}

\subsection{Hyperbolicity of $\cup_{i=1}^{3}\L^i$}
In the following Lemma \ref{paths-are-slim}, we will investigate the (uniform) hyperbolicity of a (uniform) neighborhood of $\cup_{i=1}^{3}\L^i$ in a bit larger neighborhood of $\F l_K(X_u)$, where $\L^i$'s are the ladders obtained in Proposition \ref{existence-of-tripod-ladder}.

\begin{lemma}\label{paths-are-slim}
Given $R\ge2C^{(9)}_{\ref{qi-sec-inside-lad-len-lift}}(k_{\ref{existence-of-tripod-ladder}})$, there are constants $\dl_{\ref{paths-are-slim}}=\dl_{\ref{paths-are-slim}}(k_{\ref{existence-of-tripod-ladder}},R)$ and $L_{\ref{paths-are-slim}}=L_{\ref{paths-are-slim}}(k_{\ref{existence-of-tripod-ladder}},R)$ such that the following hold.
	\begin{enumerate}
\item $Y:=N_{R+2\dl_0}(\cup_{i=1}^{3}\L^i)$ is a $\dl_{\ref{paths-are-slim}}$-hyperbolic subspace (with the induced path metric) of $N_{(R+8\dl_0)}(\F l_K(X_u))$.
		
\item The inclusion $L^{ij}:=N_R(\L^{ij})\ri Y$ is a $L_{\ref{paths-are-slim}}$-qi embedding with their induced path metrics.
	\end{enumerate}
\end{lemma}

\begin{proof}
{\em (1)} Note that $\L^i$ is a $(k_{\ref{existence-of-tripod-ladder}},c_{\ref{existence-of-tripod-ladder}},\varepsilon_{\ref{existence-of-tripod-ladder}})$-ladder such that $\L^i\sse N^f_{6\dl_0}(\F l_K(X_u))$ for $i=1,2,3$, and so $Y\sse N_{(R+8\dl_0)}(\F l_K(X_u))$. Let $L^i=N_{R+2\dl_0}(\L^i),~i=1,2,3$. Since the tree of metric bundles $(X,B,T)$ satisfies $C^{(9)}_{\ref{qi-sec-inside-lad-len-lift}}(k_{\ref{existence-of-tripod-ladder}})$-flaring condition, by Theorem \ref{general-ladder-is-hyp}, $L^i$ is $\dl_1$-hyperbolic, where $\dl_1=\dl_{\ref{general-ladder-is-hyp}}(k_{\ref{existence-of-tripod-ladder}},R+2\dl_0)$. Now we apply Proposition \ref{combi-hyp-sps} twice on $L^i$'s, $i=1,2,3$; first on $L^1\cup L^2$ and then on $(L^1\cup L^2)\cup L^3$ to show $Y$ is hyperbolic. 
	
	\underline{{\bf \emph{$\bm{L^1\cup L^2}$ is Hyperbolic}}}: We verify the conditions of Proposition \ref{combi-hyp-sps} for $n=2$ (see Remark \ref{combi-hyp-sps-2}).
	
	$(1)$ $L^1$ and $L^2$ are $\dl_1$-hyperbolic.

	$(2)$ Note that $N_{2k_{\ref{existence-of-tripod-ladder}}}(\Xi)\sse L^1\cap L^2$, and by Lemma \ref{qi-sec-inside-lad-len-lift} $(3)$, $N_{2k_{\ref{existence-of-tripod-ladder}}}(\Xi)$ is $k_{\ref{existence-of-tripod-ladder}}(2k_{\ref{existence-of-tripod-ladder}}+1)$-qi embedded in both $L^1$ and $L^2$. Let $N^i_D(N_{2k_{\ref{existence-of-tripod-ladder}}}(\Xi))$ denote the $D$-neighborhood of $N_{2k_{\ref{existence-of-tripod-ladder}}}(\Xi)$ in $L^i$-metric for $i=1,2$. If $L^1\cap L^2\sse N^i_D(N_{2k_{\ref{existence-of-tripod-ladder}}}(\Xi))$, then by Lemma \ref{hd-imp-qi}, $L^1\cap L^2$ is $L_1$-qi embedded in both $L^1$ and $L^2$ for some $L_1$ depending on $D$ and $k_{\ref{existence-of-tripod-ladder}}(2k_{\ref{existence-of-tripod-ladder}}+1)$. Now we will find $D$.

	\emph{Finding} $D$: Let $x\in L^1\cap L^2$. Then $\exists~ x_i\in\L^i$ such that $d_{L^i}(x,x_i)\le R+2\dl_0,~i=1,2$. So $d_X(x_1,x_2)\le2(R+2\dl_0)$ and $d_B(a_1,a_2)\le2(R+2\dl_0)$, where $\pi_X(x_i)=a_i,i=1,2$. Let $v$ be the center of the tripod with vertices $\pi(x_1),\pi(x_2),u$ in $T$. If any one of $\pi(x_1)$, $\pi(x_2)$ is $u$ then we set $v$ to be $u$. Let $c\in B_v\cap[a_1,a_2]$. Then $d_B(a_i,c)\le2(R+2\dl_0), i=1,2$. Suppose $\gm_i$ is $k_{\ref{existence-of-tripod-ladder}}$-lift through $x_i$ of geodesic $[a_i,c]$ lying inside $\L^i$. Let $\gm_i(a_i)=c_i, i=1,2$. So $d_{L^i}(x_i,c_i)\le4(R+2\dl_0)k_{\ref{existence-of-tripod-ladder}}, i=1,2$. Thus $d_X(c_1,c_2)\le d_X(c_1,x_1)+d_X(x_1,x_2)+d_X(x_2,c_2)\le2(R+2\dl_0)(4k_{\ref{existence-of-tripod-ladder}}+1)$, and so $d^f(c_1,c_2)\le \phi(2(R+2\dl_0)(4k_{\ref{existence-of-tripod-ladder}}+1))=D_1$ (say). We also note that $v\in\mfB$ where $\mfB$ is a central base for all $\L^i$'s. Since $z_{c,v}\in\Xi$ is $\dl_0$-center of geodesic triangle with vertices $\{bot(\L^j_{c,v});j=1,2,3\}$ in $F_{c,v}$, so $z_{c,v}$ is $3\dl_0$-close (in fiber distance) to a point $c_3\in[c_1,c_2]^f$. Then, for $i=1,2$, $d_{L^i}(c_i,z_{c,v})\le d^f(c_i,z_{c,v})\le d^f(c_i,c_3)+d^f(c_3,z_{c,v})\le D_1+3\dl_0$. Hence, for $i=1,2$, $d_{L^i}(x,z_{c,v})\le d_{L^i}(x,x_i)+d_{L^i}(x_i,c_i)+d_{L^i}(c_i,z_{c,v})\le R+2\dl_0+4(R+2\dl_0)k_{\ref{existence-of-tripod-ladder}}+D_1+3\dl_0=:D$. Therefore, $L^1\cap L^2\sse N^i_D(\Xi)\sse N^i_D(N_{2k_{\ref{existence-of-tripod-ladder}}}(\Xi))$, where $N^i_D(\zeta)$ denotes $D$-neighborhood around $\zeta\in L^i$ in the path metric of $L^i$, $i=1,2$.
	
	
	Therefore, by Remark \ref{combi-hyp-sps-2}, we conclude that $L^1\cup L^2$ is $\dl_2$-hyperbolic, where $\dl_2=\dl_{\ref{combi-hyp-sps-2}}(\dl_1,L_1)$.
	
	The exact proof as above works mutatis mutandis for $(L^1\cup L^2)\cup L^3$, and we conclude that Y is uniformly hyperbolic with the induced path metric. Let the hyperbolicity constant for $Y$ be $\dl_{\ref{paths-are-slim}}$.
	
{\em (2)} Note that $L^{ij}\sse Y\sse N_{(R+8\dl_0)}(\F l_K(X_u))$ as $\L^{ij}\sse N^f_{\dl_0}(\L^i\cup\L^j)$. Since $\L^{ij}$ is a $(k_{\ref{existence-of-tripod-ladder}},c_{\ref{existence-of-tripod-ladder}},\varepsilon_{\ref{existence-of-tripod-ladder}})$-ladder (see Proposition \ref{existence-of-tripod-ladder}), then by Corollary \ref{semicts-is-qi-emb}, the inclusion $L^{ij}\ri X$ is a $L_{\ref{semicts-is-qi-emb}}(k_{\ref{existence-of-tripod-ladder}},R)$-qi embedding and so is the inclusion $L^{ij}\ri Y$. Therefore, we can take $L_{\ref{paths-are-slim}}=L_{\ref{semicts-is-qi-emb}}(k_{\ref{existence-of-tripod-ladder}},R)$.
\end{proof}	

\subsection{Geodesics in $Fl_{KL}(X_u)$ are coarsely well-defined}

Given a pair of points in $\F l_K(X_u)$, we get a ladder according to Corollary \ref{ladder-in-flow-sp-for-two-pts}. But this ladder is far from being canonical. In the following proposition, we show that different ladders for different choices of qi sections for the same pair of points give rise to uniform Hausdorff-close geodesic paths joining those points in the respective ladders. In fact, this is more generally true, but we prove it according to our requirements. 

\begin{prop}\label{paths-are-Hd-close}
There is a constant $D_{\ref{paths-are-Hd-close}}=D_{\ref{paths-are-Hd-close}}(k_{\ref{ladder-in-flow-sp-for-two-pts}},c_{\ref{ladder-in-flow-sp-for-two-pts}},\varepsilon_{\ref{ladder-in-flow-sp-for-two-pts}})$ such that the following holds.
	
Let $x,y\in\F l_K(X_u)$. Suppose $\Sigma_x$ and $\Sigma_y$ are $K$-qi sections through $x$ and $y$ lying inside $\F l_K(X_u)$ over $B_{[u,\pi(x)]}$ and $B_{[u,\pi(y)]}$ respectively. Let $\L^1_{xy}$ and $\L^2_{xy}$ be two $(k_{\ref{ladder-in-flow-sp-for-two-pts}},c_{\ref{ladder-in-flow-sp-for-two-pts}},\varepsilon_{\ref{ladder-in-flow-sp-for-two-pts}})$-ladders containing $\Sigma_x,\Sigma_y$ (see Corollary \ref{ladder-in-flow-sp-for-two-pts}). Further, we assume that $c^1(x,y)$ and $c^2(x,y)$ are geodesics joining $x,y$ inside $L^1_{xy}:=N_{r_2}(\L^1_{xy})$ and $L^2_{xy}:=N_{r_2}(\L^2_{xy})$ respectively, where the constant $r_2$ is defined in Lemma \ref{Li-bar-is-qi-emb-in-Li} below. Then $c^1(x,y)$ and $c^2(x,y)$ are $D_{\ref{paths-are-Hd-close}}$-Hausdorff-close in $X$.
\end{prop}
\begin{proof}
In the proof, we omit the subscript $xy$ when denoting the ladders to avoid excessive notation. We denote $\L^i:=\L^i_{xy}$, $L^i:=N_{r_2}(\L^i)$ and the fibers of $\L^i$ by $\L^i_{b,v}\sse F_{b,v},i=1,2$. Let $S_i=\textrm{hull}(\pi(\L^i))$ and $B_i=\pi_B^{-1}(S_i)$, $i=1,2$. Note that both the ladders $\L^1,~\L^2$ contain $\Sigma_x$ and $\Sigma_y$, so for $v\in\pi(\Sigma_x)\cup\pi(\Sigma_y)\textrm{ and }b\in B_v$, $d^f(\L^1_{b,v},\L^2_{b,v})\le\dl_0$; in particular, the pair $(\L^1_{b,v},\L^2_{b,v})$ is $5\dl_0$-close in $F_{b,v}$. 
	
The proof is divided into two steps. In Step $1$, we find a common base $\bar{B}$ and develop lemmata which are needed in Step $2$. In Step $2$, we show that there is a common subspace containing $x,y$ which is qi embedded in both $L^1,L^2$. Finally, we conclude the proposition.

{\bf Step 1}: {\bf \emph{Construction of common base $\bm{\bar{B}}$}}: Let $B\pr=\{b\in B_v:d^f(\L^1_{b,v},\L^2_{b,v})\le5\dl_0, v\in S_1\cap S_2\}$. Note that $B_u\sse\pi_X(\Sigma_x)\cup\pi_X(\Sigma_y)\sse B\pr$. Let $B\pr_v=\textrm{hull}(B\pr)\cap B_v$ and $\bar{B}_v=N_{\dl_0}(B\pr_v)\cap B_v$, where $v\in\pi_B(B\pr)$. Suppose $\bar{B}_1=\cup_{v\in\pi_B(B\pr)}^{}\bar{B}_v\cup\textrm{hull}(B\pr)$, and so $\pi_B(\bar{B}_1)$ is a subtree of $S_1\cap S_2$. Then by \cite[Lemma $1.93$]{ps-kap} and the fact that $B_v$'s are isometrically embedded in $B$, we note that $\bar{B}_1$ is $(1,6\dl_0)$-qi embedded in $B$. Finally, we will add a few more vertices and edges to $\bar{B}_1$ to complete the construction of $\bar{B}$. Suppose $v\in\pi_B(\bar{B}_1),w\notin\pi_B(\bar{B}_1)$ and $w\in S_1\cap S_2$ such that $d_T(v,w)=1$. Let $[\mfv,\mfw]$ be the edge joining $\mfv\in B_v$ and $\mfw\in B_w$. Further, we assume that $\mfv\in\bar{B}_v$ and the pair $(\L^1_{\mfv,v},\L^2_{\mfv,v})$ is $5\dl_0$-close in the fiber $F_{\mfv,v}$. Then we include only the vertex $\mfw$ and the edge $[\mfv,\mfw]$ to $\bar{B}_1$. We will use the same notation for these extra vertices, i.e., here $\bar{B}_w=\bar{B}\cap B_w=\{\mfw\}$. Notice that $\bar{B}$ is still $(1,6\dl_0)$-qi embedded in $B$. Let $\bar{S}=\pi_B(\bar{B})$.
	
Let $v\in S_1\cap S_2$ and $b\in B_v$. Let $P^i_{b,v}:F_{b,v}\ri\L^i_{b,v}$ be a nearest point projection map in $F_{b,v}$ (see Lemma \ref{proj-on-qc} $(1)$) and $\bar{P}^i_{b,v}:F_{b,v}\ri\L^i_{b,v}$ be modified projection map (see Definition \ref{modified-projection}) corresponding to $P^i_{b,v}$, $i=1,2$. We denote $\bar{\L}^1_{b,v}:=\bar{P}^1_{b,v}(\L^2_{b,v})\sse\L^1_{b,v}$ and $\bar{\L}^2_{b,v}:=\bar{P}^2_{b,v}(\L^1_{b,v})\sse\L^2_{b,v}$, and $\bar{\L}^i:=\bigcup\limits_{b\in\bar{B}_v,~v\in\bar{S}}\bar{\L}^i_{b,v},i=1,2$.

	\begin{note}\label{note}
$(i)$ Let $v\in \pi_B(B\pr)$ and $b\in B\pr\cap B_v$. Then $Hd^f(P^1_{b,v}(\L^2_{b,v}),P^2_{b,v}(\L^1_{b,v}))\le2\dl_0+5\dl_0=7\dl_0$ (by Remark \ref{3-re-in-one} $(2)$). Thus  $Hd^f(\bar{\L}^1_{b,v},\bar{\L}^2_{b,v})\le13\dl_0$ by Remark \ref{modi-geo}. However, we will prove below in Lemma \ref{Li's-bar-are-Hd-close} that $\fa~ v\in\bar{S}\textrm{ and }\fa~ b\in\bar{B}_v$, $Hd^f(\bar{\L}^1_{b,v},\bar{\L}^2_{b,v})$ is uniformly bounded.
		
$(ii)$ If $v\in S_1\cap S_2$, then by Lemma \ref{modified-projection-gives-qi-sections}, $\bar{\L}^{i}\cap X_v:=\bigcup_{b\in B_v}\bar{\L}^i_{b,v}$ forms a (uniformly) special $C_{\ref{qi-sec-inside-lad-len-lift}}(K_{\ref{modified-projection-gives-qi-sections}}(K))$-ladder (see Definition \ref{K-metric-graph-bundle}) bounded by $K_{\ref{modified-projection-gives-qi-sections}}(K)$-qi sections, $i=1,2$.
	\end{note}

	\begin{lemma}\label{Li's-bar-are-Hd-close}
With the above notations, there is a constant $R_{\ref{Li's-bar-are-Hd-close}}$ such that $\fa~ v\in\bar{S}\textrm{ and }\fa~ b\in\bar{B}_v$, $$Hd^f(\bar{\L}^1_{b,v},\bar{\L}^2_{b,v})\le R_{\ref{Li's-bar-are-Hd-close}}.$$
	\end{lemma}
	\begin{proof}
Let $v\in\bar{S},~c\in\bar{B}_v$. We divide the proof into following cases. First three cases deal with the vertices $v\in\pi_B(\bar{B}_1)=\pi_B(B\pr)$ and Case 4 with the extra vertices.
		
{\bf Case 1:} Suppose $c\in\textrm{hull}(B\pr)$ such that $c\in[a,b]_B$ for some $a,b\in B\pr$ and $a\in B_u$. Let $w=\pi_B(b),v=\pi_B(c)$. We prove that $d^f(\L^1_{c,v},\L^2_{c,v})$ is uniformly bounded and hence we are through by Remark \ref{3-re-in-one} $(2)$ and Remark \ref{modi-geo}. As the pair $(\L^1_{b,w},\L^2_{b,w})$ is $5\dl_0$-close, we take $x\in\L^1_{b,w}$ such that $d^f(x,\L^2_{b,w})\le5\dl_0$. Consider $k_{\ref{ladder-in-flow-sp-for-two-pts}}$-qi section, say $\gm_x$, through $x$ over $B_{[u,w]}$ in the ladder $\L^1$ (since $\L^i$'s are $(k_{\ref{ladder-in-flow-sp-for-two-pts}},c_{\ref{ladder-in-flow-sp-for-two-pts}},\varepsilon_{\ref{ladder-in-flow-sp-for-two-pts}})$-ladders). Now we apply Mitra's retraction (see Theorem \ref{mitra's-retraction-on-semicts-subsp}), $\rho_{\L^2}$ on $\gm_x$ and get a $\mathcal{R}_0(2k_{\ref{ladder-in-flow-sp-for-two-pts}}+1)$-qi section, say $\gm\pr_x$, over $B_{[u,w]}$ in $\L^2$, where $\mathcal{R}_0:=L_{\ref{mitra's-retraction-on-semicts-subsp}}(k_{\ref{ladder-in-flow-sp-for-two-pts}},c_{\ref{ladder-in-flow-sp-for-two-pts}},\varepsilon_{\ref{ladder-in-flow-sp-for-two-pts}})$. Then we have two $\mathcal{R}_0(2k_{\ref{ladder-in-flow-sp-for-two-pts}}+1)$-qi sections $\gm_x,\gm\pr_x$ over $B_{[u,w]}$ such that (as $a\in B_u$) $d^f(\gm_x(s),\gm\pr_x(s))\le5\dl_0,~s\in\{a,b\}$. Again the tree of metric bundles $(X,B,T)$ satisfies $\mathcal{R}_0(2k_{\ref{ladder-in-flow-sp-for-two-pts}}+1)$-flaring condition. Thus the restriction of $\gm_x,\gm\pr_x$ on geodesic $[a,b]$ and Lemma \ref{flaring-lemma} $(2)$ imply $d^f(\gm_x(c),\gm\pr_x(c))\le R_1$, where $R_1=R_{\ref{flaring-lemma}}(L_0(2k_{\ref{ladder-in-flow-sp-for-two-pts}}+1),5\dl_0)$. Hence $d^f(\L^1_{c,v},\L^2_{c,v})\le R_1$ and so by Remark \ref{3-re-in-one} $(2)$ and Remark \ref{modi-geo}, $Hd^f(\bar{\L}^1_{c,v},\bar{\L}^2_{c,v})\le8\dl_0+R_1$.
		
		{\bf Case 2:} Suppose $c\in\textrm{hull}(B\pr)$ such that $c\in[a,b]_B$ for some $a,b\in B\pr$ and none of $a,b$ belong to $B_u$. Let $v=\pi_B(c)$. More precisely, we assume that $v$ is the center of the geodesic triangle $\triangle(\pi_B(a),u,\pi_B(b))$, otherwise, it will land in Case $1$. Let $a\pr\in B_u$. Then by $\dl_0$-slimness of the geodesic triangle $\triangle(a,b,a\pr)$ and without loss of generality, we may assume that $d_B(c,c\pr)\le\dl_0$ for some $c\pr\in[a\pr,b]$. So by Case $1$, we have $d^f(\L^1_{c\pr,v},\L^2_{c\pr,v})\le R_1$. Let $x\in\L^1_{c\pr,v},y\in\L^2_{c\pr,v}$ such that $d^f(x,y)\le R_1$. Now we take $k_{\ref{ladder-in-flow-sp-for-two-pts}}$-qi lifts, say $\gm_1$, and $\gm_2$ of geodesic $[c,c\pr]$ in $\L^1$ and $\L^2$ through $x$ and $y$ respectively. Then $d_X(\gm_1(c),\gm_2(c))\le d_X(\gm_1(c),x)+d_X(x,y)+d_X(y,\gm_2(c))\le2k_{\ref{ladder-in-flow-sp-for-two-pts}}\dl_0+R_1+2k_{\ref{ladder-in-flow-sp-for-two-pts}}\dl_0=4k_{\ref{ladder-in-flow-sp-for-two-pts}}\dl_0+R_1$. Thus $d^f(\L^1_{c,v},\L^2_{c,v})\le \phi(4k_{\ref{ladder-in-flow-sp-for-two-pts}}\dl_0+R_1)$, where fibers are $\phi$-properly embedded in $X$. Therefore, by Remark \ref{3-re-in-one} $(2)$ and Remark \ref{modi-geo}, $Hd^f(\bar{\L}^1_{c,v},\bar{\L}^2_{c,v})\le8\dl_0+\phi(4k_{\ref{ladder-in-flow-sp-for-two-pts}}\dl_0+R_1)$.
		
{\bf Case 3:} Suppose $c\in\bar{B}_1$ and $\pi_B(c)=v$. Then by construction of $\bar{B}_1$, there exists $c_1\in B\pr_v\sse \textrm{hull}(B\pr)$ such that $d_B(c,c_1)=d_{B_v}(c,c_1)\le\dl_0$. Since $c_1\in$ hull$(B\pr)$, by Case 2, we know that $d^f(\L^1_{c_1,v},\L^2_{c_1,v})\le \phi(4k_{\ref{ladder-in-flow-sp-for-two-pts}}\dl_0+R_1)$. Now we use the same argument used in the last part of Case 2. Let $x\in\L^1_{c_1,v},y\in\L^2_{c_1,v}$ such that $d^f(x,y)\le \phi(4k_{\ref{ladder-in-flow-sp-for-two-pts}}\dl_0+R_1)$. By taking $k_{\ref{ladder-in-flow-sp-for-two-pts}}$-qi lifts through points $x$ and $y$ of the geodesic $[c,c_1]$ in the ladders $\L^1$ and $\L^2$ respectively, one can conclude that $d^f(\L^1_{c,v},\L^2_{c,v})\le \phi(4k_{\ref{ladder-in-flow-sp-for-two-pts}}\dl_0+\phi(4k_{\ref{ladder-in-flow-sp-for-two-pts}}\dl_0+R_1))$. Hence, by Remark \ref{3-re-in-one} $(2)$ and Remark \ref{modi-geo}, $Hd^f(\bar{\L}^1_{c,v},\bar{\L}^2_{c,v})\le8\dl_0+\phi(4k_{\ref{ladder-in-flow-sp-for-two-pts}}+\phi(4k_{\ref{ladder-in-flow-sp-for-two-pts}}\dl_0+R_1))$.
		
{\bf Case 4:} Finally, we will show the required Hausdorff bound for the points in $\bar{B}$ which do not appear in earlier cases. Suppose $c\in\bar{B}\setminus\bar{B}_1$. Let $\pi_B(c)=w$ and $[v,w]$ be the edge such that $d_T(u,v)<d_T(u,w)$. Let $[\mfv,\mfw]$ be the edge joining $\mfv\in B_v$ and $\mfw\in B_w$. Note that $c=\mfw$. Then by the construction of $\mfw\in\bar{B},~d^f(\L^1_{\mfv,v},\L^2_{\mfv,v})\le5\dl_0$ and so $Hd^f(\bar{\L}^1_{\mfv,v},\bar{\L}^2_{\mfv,v})\le13\dl_0$ (see Remark \ref{3-re-in-one} $(2)$). Suppose $x\in\bar{\L}^1_{\mfv,v}$ and $y\in\bar{\L}^2_{\mfv,v}$ such that $d^f(x,y)\le13\dl_0$. We consider $x\pr\in\L^1_{\mfw,w},~y\pr\in\L^2_{\mfw,w}$ such that $d_{\mfv\mfw}(P_{\mfw}(x),x\pr)\le\varepsilon_{\ref{ladder-in-flow-sp-for-two-pts}}$ and $d_{\mfv\mfw}(P_{\mfw}(y),y\pr)\le\varepsilon_{\ref{ladder-in-flow-sp-for-two-pts}}$ (since $\L^i$'s are $(k_{\ref{ladder-in-flow-sp-for-two-pts}},c_{\ref{ladder-in-flow-sp-for-two-pts}},\varepsilon_{\ref{ladder-in-flow-sp-for-two-pts}})$-ladders). Again $P_{\mfw}$ is $L\pr_1$-coarsely Lipschitz retraction in the metric $F_{\mfv\mfw}$ (see Lemma \ref{com-two-hyp-sps}$~(2)$) and so $d_{\mfv\mfw}(P_{\mfw}(x),P_{\mfw}(y))\le L\pr_1d_{\mfv\mfw}(x,y)+L\pr_1\le L\pr_1(13\dl_0+1)$. Thus by triangle inequality, $d_{\mfv\mfw}(x\pr,y\pr)\le2\varepsilon_{\ref{ladder-in-flow-sp-for-two-pts}}+L\pr_1(13\dl_0+1)$ $\Rightarrow d^f(x\pr,y\pr)\le \phi(2\varepsilon_{\ref{ladder-in-flow-sp-for-two-pts}}+L\pr_1(13\dl_0+1))$. Now $x\pr\in\L^1_{\mfw,w},~y\pr\in\L^2_{\mfw,w}$ and so by Remark \ref{3-re-in-one} $(2)$ and Remark \ref{modi-geo}, we have $d^f(\bar{\L}^1_{\mfw,w},\bar{\L}^2_{\mfw,w})\le8\dl_0+\phi(2\varepsilon_{\ref{ladder-in-flow-sp-for-two-pts}}+L\pr_1(13\dl_0+1))$.
		
		Therefore, we can take $R_{\ref{Li's-bar-are-Hd-close}}$ to be  the maximum of all four constants we get in four cases, i.e., $R_{\ref{Li's-bar-are-Hd-close}}=max\{\phi(2k_{\ref{ladder-in-flow-sp-for-two-pts}}+\phi(2k_{\ref{ladder-in-flow-sp-for-two-pts}}\dl_0+R_1)),8\dl_0+\phi(2\varepsilon_{\ref{existence-of-tripod-ladder}}+L\pr_1(13\dl_0+1))\}$.
	\end{proof}
	
	Next we show that $\bar{\L}^i$'s are more general ladders. More precisely, they are semicontinuous families and they need not satisfy the condition $(4)$ of Definition \ref{semicts-subgrphs} trivially. In other words, according to the notation of Definition \ref{semicts-subgrphs}, we possibly have $B\pr\subsetneq \pi_B^{-1}(T_{\mfY})$), and $\bar{\L}^i$ behave like ladders. 
	
	\begin{lemma}\label{Li-bar-behaves-like-ladder}
There are constants $k_{\ref{Li-bar-behaves-like-ladder}},~c_{\ref{Li-bar-behaves-like-ladder}}$ and $\varepsilon_{\ref{Li-bar-behaves-like-ladder}}$ such that the following hold.
		
Suppose $[v,w]$ is an edge in $T$ such that $d_T(u,v)<d_T(u,w)$ and $[\mfv,\mfw]$ is the edge joining $\mfv\in B_v$ and $\mfw\in B_w$. Let $\mfv,\mfw\in\bar{B}$. Then:
\begin{enumerate}
\item $\bar{\L}^i_{\mfw,w}\sse N^{\mfv\mfw}_{k_{\ref{Li-bar-behaves-like-ladder}}}(\bar{\L}^i_{\mfv,v})$.
			
\item $Hd_{\mfv\mfw}(P_{\mfw}(\bar{\L}^i_{\mfv,v}),\bar{\L}^i_{\mfw,w})\le\varepsilon_{\ref{Li-bar-behaves-like-ladder}}$.
			
\item Suppose $a\in\bar{B},~b\notin\bar{B}$ with $d_B(a,b)=1$ and $\pi_B(a)=s,~\pi_B(b)=t$. Then diam$^f(\bar{\L}^i_{a,s})\le c_{\ref{Li-bar-behaves-like-ladder}}$ if $s=t$ or the pair $(\bar{\L}^i_{a,s},F_{b,t})$ is $c_{\ref{Li-bar-behaves-like-ladder}}$-cobounded in the path metric of $F_{ab}:=\pi_X^{-1}([a,b])$ if $s\ne t$.
\end{enumerate}
In particular, $\bar{\L}^i$ is a $(k_{\ref{Li-bar-behaves-like-ladder}},c_{\ref{Li-bar-behaves-like-ladder}},\varepsilon_{\ref{Li-bar-behaves-like-ladder}})$-semicontinuous family, $i=1,2$.
\end{lemma}
	
\begin{proof}
$(1)$ We will only prove for $i=1$ as the proof for $i=2$ involves a simple change of indices. Since $\L^i$'s are $(k_{\ref{ladder-in-flow-sp-for-two-pts}},c_{\ref{ladder-in-flow-sp-for-two-pts}},\varepsilon_{\ref{ladder-in-flow-sp-for-two-pts}})$-ladders, for $x\in\bar{\L}^1_{\mfw,w}$, $\exists~x_1\in\L^1_{\mfv,v}$ such that $d_{\mfv\mfw}(x,x_1)\le k_{\ref{ladder-in-flow-sp-for-two-pts}}$. Let $y\in\bar{\L}^2_{\mfw,w},~y_1\in\L^2_{\mfv,v}$ such that $d^f(x,y)\le R_{\ref{Li's-bar-are-Hd-close}}$ (by Lemma \ref{Li's-bar-are-Hd-close}) and $d_{\mfv\mfw}(y,y_1)\le k_{\ref{ladder-in-flow-sp-for-two-pts}}$. Then $d_{\mfv\mfw}(x_1,y_1)\le d_{\mfv\mfw}(x_1,x)+d^f(x,y)+d_{\mfv\mfw}(y,y_1)\le2k_{\ref{ladder-in-flow-sp-for-two-pts}}+R_{\ref{Li's-bar-are-Hd-close}}$. Hence $d^f(x_1,y_1)\le R_2$, where $R_2=\phi(2k_{\ref{ladder-in-flow-sp-for-two-pts}}+R_{\ref{Li's-bar-are-Hd-close}})$. Note that $P^1_{\mfv,v}:F_{\mfv,v}\ri\L^1_{\mfv,v}$ is a nearest point projection map in the metric of $F_{\mfv,v}$. For simplicity, let $P=P^1_{\mfv,v}$. Then $d^f(P(x_1),P(y_1))=d^f(x_1,P(y_1))\le C_{\ref{proj-on-qc}}(\dl_0,\dl_0)(R_2+1)$ (see Lemma \ref{proj-on-qc} $(1)$). So, $d^f(x_1,\bar{\L}^1_{\mfv,v})\le C_{\ref{proj-on-qc}}(\dl_0,\dl_0)(R_2+1)$. Hence $d_{\mfv\mfw}(x,\bar{\L}^1_{\mfv,v})\le d_{\mfv\mfw}(x,x_1)+d^f(x_1,\bar{\L}^1_{\mfv,v})\le k_1$, where $k_1=k_{\ref{ladder-in-flow-sp-for-two-pts}}+C_{\ref{proj-on-qc}}(\dl_0,\dl_0)(R_2+1)$. Therefore, $\bar{\L}^i_{\mfw,w}\sse N^{\mfv\mfw}_{k_1}(\bar{\L}^i_{\mfv,v})$. Thus $k_1$ works for $(1)$ but for the second part of this lemma, we have defined $k_{\ref{Li-bar-behaves-like-ladder}}>k_1$ in the end.\smallskip
		
$(2)$ Here also we will only prove for $i=1$. Let $x\in\bar{\L}^1_{\mfv,v}$ and we take $y\in\bar{\L}^2_{\mfv,v}$ such that $d^f(x,y)\le R_{\ref{Li's-bar-are-Hd-close}}$ (by Lemma \ref{Li's-bar-are-Hd-close}). Suppose $x\pr\in\L^1_{\mfw,w},~y\pr\in\L^2_{\mfw,w}$ such that $d_{\mfv\mfw}(P_{\mfw}(x),x\pr)\le\varepsilon_{\ref{ladder-in-flow-sp-for-two-pts}}$ and $d_{\mfv\mfw}(P_{\mfw}(y),y\pr)\le\varepsilon_{\ref{ladder-in-flow-sp-for-two-pts}}$ (since $\L^i$'s are $(k_{\ref{ladder-in-flow-sp-for-two-pts}},c_{\ref{ladder-in-flow-sp-for-two-pts}},\varepsilon_{\ref{ladder-in-flow-sp-for-two-pts}})$-ladders). Again $P_{\mfw}$ is $L\pr_1$-coarsely Lipschitz in the metric of $F_{\mfv\mfw}$, so $d_{\mfv\mfw}(P_{\mfw}(x),P_{\mfw}(y))\le L\pr_1d_{\mfv\mfw}(x,y)+L\pr_1\le L\pr_1(R_{\ref{Li's-bar-are-Hd-close}}+1)$. Therefore, (by triangle inequality) $d_{\mfv\mfw}(x\pr,y\pr)\le2\varepsilon_{\ref{ladder-in-flow-sp-for-two-pts}}+L\pr_1(R_{\ref{Li's-bar-are-Hd-close}}+1)$ $\Rightarrow d^f(x\pr,y\pr)\le \phi(2\varepsilon_{\ref{ladder-in-flow-sp-for-two-pts}}+L\pr_1(R_{\ref{Li's-bar-are-Hd-close}}+1))=R_3$ (say). Note that $P^1_{\mfw,w}:F_{\mfw,w}\ri\L^1_{\mfw,w}$ is a nearest point projection map in the metric of $F_{\mfw,w}$. For simplicity, let $P=P^1_{\mfw,w}$. Since $d^f(x\pr,y\pr)\le R_3$, then by Lemma \ref{proj-on-qc} $(1)$, $d^f(x\pr,P(y\pr))=d^f(P(x\pr),P(y\pr))\le C_{\ref{proj-on-qc}}(\dl_0,\dl_0)(R_3+1)=R_4$ (say). So $d^f(x\pr,\bar{\L}^1_{\mfw,w})\le R_4$. Hence $d_{\mfv\mfw}(P_{\mfw}(x),\bar{\L}^1_{\mfw,w})\le d_{\mfv\mfw}(P_{\mfw}(x),x\pr)+d_{\mfv\mfw}(x\pr,\bar{\L}^1_{\mfw,w})\le\varepsilon_{\ref{ladder-in-flow-sp-for-two-pts}}+R_4=\varepsilon_1$ (say), i.e., $P_{\mfw}(\bar{\L}^1_{\mfv,v})\sse N^{\mfv\mfw}_{\varepsilon_1}(\bar{\L}^1_{\mfw,w})$.
		
For the other inclusion, let $x\in\bar{\L}^1_{\mfw,w}$. Then by $(1)$, $\exists~ x_1\in\bar{\L}^1_{\mfv,v}$ such that $d_{\mfv\mfw}(x,x_1)\le k_1$. Then $d_{\mfv\mfw}(P_{\mfw}(x_1),x)\le2k_1$. Hence $\bar{\L}^1_{\mfw,w}\sse N^{\mfv\mfw}_{2k_1}(\bar{\L}^1_{\mfw,w})$.
		
Therefore, we can take $\varepsilon_{\ref{Li-bar-behaves-like-ladder}}:=max\{\varepsilon_1,2k_1\}$.\smallskip

$(3)$ Suppose $s=t$. Then $a,b\in B_s$. Since $b\notin\bar{B}$, so $d^f(\L^1_{b,s},\L^2_{b,s})>5\dl_0$. Then by Remark \ref{3-re-in-one} $(1)$, diam$^f(\bar{\L}^i_{b,s})\le8\dl_0$ for $i=1,2$. Let $\bar{\L}^i_{a,s}=[\eta^i_{a,s},\zeta^i_{a,s}]$ for $i=1,2$. Since $d_B(a,b)=1$, by Note \ref{note} $(ii)$, $d_X(\eta^i_{a,s},\bar{\L}^i_{b,s})\le2K_{\ref{modified-projection-gives-qi-sections}}(K)$ and $d_X(\zeta^i_{a,s},\bar{\L}^i_{b,s})\le2K_{\ref{modified-projection-gives-qi-sections}}(K)$ for $i=1,2$. Then by triangle inequality, $d_X(\eta^i_{a,s},\zeta^i_{a,s})\le4K_{\ref{modified-projection-gives-qi-sections}}(K)+8\dl_0$ for $i=1,2$. Therefore, diam$^f(\bar{\L}^i_{a,s})\le \phi(4K_{\ref{modified-projection-gives-qi-sections}}(K)+8\dl_0)$ for $i=1,2$.
		
Now suppose $s\ne t$. Note that since $a\in\bar{B},~b\notin\bar{B}$, then $d_T(u,s)<d_T(u,t)$ and $[a,b]$ is the edge joining $a\in B_s$, $b\in B_t$. If $t\notin S_1\cup S_2$, then the pair $(\L^i_{a,s},F_{b,t})$ is $c_{\ref{ladder-in-flow-sp-for-two-pts}}$-cobounded in $F_{ab}$. So by Lemma \ref{small-imp-small}, there is a constant $C_1$ depending on $\dl\pr_0,\lm\pr_0$ and $c_{\ref{ladder-in-flow-sp-for-two-pts}}$ such that the pair $(\bar{\L}^i_{a,s},F_{b,t})$ is $C_1$-cobounded in $F_{ab}$.
		
Now let $t\in S_1\cap S_2$. Since $b\notin \bar{B}$, then by the construction of $\bar{B}$, $d^f( \L^1_{a,s},\L^2_{a,s})>5\dl_0$. Thus by Remark \ref{3-re-in-one} $(1)$, diam$^f(\bar{\L}^i_{a,s})\le8\dl_0$ for $i=1,2$. Therefore, by Lemma \ref{small-imp-small}, there is a constant $C_2$ depending on $\dl\pr_0,\lm\pr_0$ and $8\dl_0$ such that the pair $(\bar{\L}^i_{a,s},F_{b,t})$ is $C_2$-cobounded in $F_{ab}$ for $i=1,2$.
		
Finally, we assume that $t$ belong to only one of the $S_1$, $S_2$. Without of loss of generality, let $t\in S_1$ but $t\notin S_2$. Note that $s\in S_2$. Then the pair $(\L^2_{a,s},F_{b,t})$ is $c_{\ref{ladder-in-flow-sp-for-two-pts}}$-cobounded in $F_{ab}$. So by Lemma \ref{small-imp-small}, the pair $(\bar{\L}^2_{a,s},F_{b,t})$ is $C_1$-cobounded in $F_{ab}$ (where $C_1$ is defined above). Again, $Hd^f(\bar{\L}^1_{a,s},\bar{\L}^2_{a,s})\le R_{\ref{Li's-bar-are-Hd-close}}$ and a nearest point projection map $P:F_{ab}\ri F_{b,t}$ is $L\pr_1$-coarsely Lipschitz (see Lemma \ref{com-two-hyp-sps}$~(2)$) together imply that the diameter (in the metric of $F_{ab}$) of $\{P(\bar{\L}^1_{a,s})\}$ is bounded by $2L\pr_1(R_{\ref{Li's-bar-are-Hd-close}}+1)+C_1=D$ (say). Then by Lemma \ref{small-imp-small}, there is a constant $C_3$ depending on $\dl\pr_0$, $\lm\pr_0$ and $D$ such that the pair $(\bar{\L}^1_{a,s},F_{b,t})$ is $C_3$-cobounded in $F_{ab}$.
		
Therefore, we can take $c_{\ref{Li-bar-behaves-like-ladder}}=max\{\phi(4K_{\ref{modified-projection-gives-qi-sections}}(K)+8\dl_0),C_1,C_2,C_3\}$.
		
For second part, we note that $\bar{B}$ is $(1,6\dl_0)$-qi embedded in $B$. Now by Note \ref{note} $(ii)$, $\bar{\L}^i\cap X_v$ is special $C_{\ref{qi-sec-inside-lad-len-lift}}(K_{\ref{modified-projection-gives-qi-sections}}(K))$-ladder in $X_v,v\in S_1\cap S_2$. Therefore, by Lemma \ref{promoting-of-ladder}, $\bar{\L}^i$ is $(k_{\ref{Li-bar-behaves-like-ladder}},c_{\ref{Li-bar-behaves-like-ladder}},\varepsilon_{\ref{Li-bar-behaves-like-ladder}})$-ladder, where $k_{\ref{Li-bar-behaves-like-ladder}}=k_{\ref{promoting-of-ladder}}(k\pr)>k_1$ and $k\pr=max\{k_1,C_{\ref{qi-sec-inside-lad-len-lift}}(K_{\ref{modified-projection-gives-qi-sections}}(K))\}$, and $i=1,2$.
	\end{proof}

\begin{lemma}\label{Li-bar-is-qi-emb-in-Li}
Let $r_1=max\{2k_{\ref{Li-bar-behaves-like-ladder}},2\dl_0+1\},~r_2=max\{2C^{(9)}_{\ref{qi-sec-inside-lad-len-lift}}(k_{\ref{ladder-in-flow-sp-for-two-pts}}),R_{\ref{Li's-bar-are-Hd-close}}+r_1+\dl_0\}$ and $\bar{L}^i:=N_{r_1}(\bar{\L}^i),~L^i:=N_{r_2}(\L^i),~i=1,2$. Then there is a constant $L_{\ref{Li-bar-is-qi-emb-in-Li}}$ such that the inclusion $\bar{L}^i\ri L^i$ is a $L_{\ref{Li-bar-is-qi-emb-in-Li}}$-qi embedding. 
	\end{lemma}
	\begin{proof}
Note that $\bar{\L}^i$ is $(k_{\ref{Li-bar-behaves-like-ladder}},c_{\ref{Li-bar-behaves-like-ladder}},\varepsilon_{\ref{Li-bar-behaves-like-ladder}})$-semicontinuous family (see Lemma \ref{Li-bar-behaves-like-ladder}) and $\bar{L}^i\sse L^i$. Then by Corollary \ref{semicts-is-qi-emb}, the inclusion $\bar{L}^i\ri X$ is a $L_{\ref{semicts-is-qi-emb}}(k_{\ref{Li-bar-behaves-like-ladder}},r_1)$-qi embedding and so is the inclusion $\bar{L}^i\ri L^i,i=1,2$. Therefore, we can take $L_{\ref{Li-bar-is-qi-emb-in-Li}}=L_{\ref{semicts-is-qi-emb}}(k_{\ref{Li-bar-behaves-like-ladder}},r_1)$.
	\end{proof}
	
{\bf Step 2}: {\bf We fix $\bm{r_1}$ and $\bm{r_2}$ as in Lemma \ref{Li-bar-is-qi-emb-in-Li} for the rest of the proof}. Here we construct a common qi embedded subspace of $L^1$ and $L^2$ containing both $x$ and $y$ and which will show that $c^1(x,y)$ and $c^2(x,y)$ are uniformly Hausdorff-close.
	
{\bf \emph{Construction of common qi embedded subspace of $\bm{L^1}$ and $\bm{L^2}$}}: Let $v\in\bar{S}\textrm{ and } b\in\bar{B}_v$. Suppose $\mathcal{Z}_{b,v}:=\textrm{hull}(\bar{\L}^1_{b,v}\cup\bar{\L}^2_{b,v})\sse F_{b,v},~\mathcal{Z}:=\bigcup\mathcal{Z}_{b,v}$ and $Z:=N_{r_1}(\mathcal{Z})$, where the quasiconvex hull and its neighborhood is taken in the corresponding fiber. We also have $Hd^f(\bar{\L}^1_{b,v},\bar{\L}^2_{b,v})\le R_{\ref{Li's-bar-are-Hd-close}}$ (by Lemma \ref{Li's-bar-are-Hd-close}). Suppose $L^i$ and $\bar{L}^i$ are as in Lemma \ref{Li-bar-is-qi-emb-in-Li}. Then $\bar{L}^i\sse Z\sse N_{r_1+R_{\ref{Li's-bar-are-Hd-close}}+\dl_0}(\bar{\L}^i)=N_{R_{\ref{Li's-bar-are-Hd-close}}+\dl_0}(\bar{L}^i)$ in the metric $L^i$ and so $Hd_{L^i}(\bar{L}^i,Z)\le R_{\ref{Li's-bar-are-Hd-close}}+\dl_0$ for $i=1,2$. The subspace $Z$ is our required common subspace, and in the below lemma we will see that it is uniformly qi embedded in $L^i$ for $i=1,2$.
	
	\begin{lemma}\label{Y-is-qi-emb-in-Li}
		With the above notations, there is a (uniform) constant $L_{\ref{Y-is-qi-emb-in-Li}}$ such that the inclusion $Z\ri L^i$ is a $L_{\ref{Y-is-qi-emb-in-Li}}$-qi embedding in the path metric of $L^i,~i=1,2$.
	\end{lemma}
	
	\begin{proof}
By Lemma \ref{Li-bar-is-qi-emb-in-Li}, the inclusion $\bar{L}^i\ri L^i$ is a $L_{\ref{Li-bar-is-qi-emb-in-Li}}$-qi embedding. Since $Hd_{L^i}(\bar{L}^i,Z)\le R_{\ref{Li's-bar-are-Hd-close}}+\dl_0$, with the reference to \cite[Lemma $1.19$]{ps-kap}, our task is to show that $Z$ is uniformly properly embedded in $L^i$.
		
\emph{$Z$ is properly embedded in $L^i$}: Let $x,y\in Z$ and $r\in\R_{\ge0}$ such that $d_{L^i}(x,y)\le r$. Then $\exists~ x_1,y_1\in\bar{L}^i$ such that $d_Z(x,x_1)\le R_{\ref{Li's-bar-are-Hd-close}}+\dl_0,~d_Z(y,y_1)\le R_{\ref{Li's-bar-are-Hd-close}}+\dl_0$. So, $d_{L^i}(x_1,y_1)\le2(R_{\ref{Li's-bar-are-Hd-close}}+\dl_0)+r$ and by Lemma \ref{Li-bar-is-qi-emb-in-Li}, $d_Z(x_1,y_1)\le d_{\bar{L}^i}(x_1,y_1)\le (2(R_{\ref{Li's-bar-are-Hd-close}}+\dl_0)+r)L_{\ref{Li-bar-is-qi-emb-in-Li}}+L_{\ref{Li-bar-is-qi-emb-in-Li}}^2=D(r)$ (say). Thus,
\begin{eqnarray*}
d_Z(x,y)&\le&d_Z(x,x_1)+d_Z(x_1,y_1)+d_Z(y_1,y)\\
&\le&R_{\ref{Li's-bar-are-Hd-close}}+\dl_0+D(r)+R_{\ref{Li's-bar-are-Hd-close}}+\dl_0\\
&=&2(R_{\ref{Li's-bar-are-Hd-close}}+\dl_0)+D(r)=:g(r)\textrm{ (say) }
\end{eqnarray*}
So $Z$ is $g$-properly embedded in $L^i$ for the function $g:\R_{\ge0}\ri\R_{\ge0}$ sending $r\mapsto 2(R_{\ref{Li's-bar-are-Hd-close}}+\dl_0)+D(r)$. Therefore, for $i=1,2$, the inclusion $Z\ri L^i$ is a $L_{\ref{Y-is-qi-emb-in-Li}}$-qi embedding for some constant $L_{\ref{Y-is-qi-emb-in-Li}}$ depending on $L_{\ref{Li-bar-is-qi-emb-in-Li}},~R_{\ref{Li's-bar-are-Hd-close}}+\dl_0$ and $g$ (see \cite[Lemma $1.19$]{ps-kap}).
	\end{proof}
	
\underline{{\bf \emph{Conclusion of the proof of Proposition \ref{paths-are-Hd-close}}}} : Let $c^*(x,y),~c^1(x,y)$ and $c^2(x,y)$ be geodesic paths in $Z,~L^1$ and $L^2$ respectively joining $x,y$. Since $(X,B,T)$ satisfies $C^{(9)}_{\ref{qi-sec-inside-lad-len-lift}}(k_{\ref{ladder-in-flow-sp-for-two-pts}})$-flaring condition. So by Theorem \ref{general-ladder-is-hyp}, $L^i$ is $\dl_{\ref{general-ladder-is-hyp}}(k_{\ref{ladder-in-flow-sp-for-two-pts}},r_2)$-hyperbolic, $i=1,2$. Therefore, by the stability of quasigeodesic, $Hd_X(c^1(x,y),c^2(x,y))\le D_{\ref{paths-are-Hd-close}}:=2D_{\ref{ml}}(\dl_{\ref{general-ladder-is-hyp}}(k_{\ref{ladder-in-flow-sp-for-two-pts}},r_2),L_{\ref{Y-is-qi-emb-in-Li}},L_{\ref{Y-is-qi-emb-in-Li}})$.
\end{proof}

\subsection{Final proof: $Fl_{KL}(X_u)$ is hyperbolic}
Now, we are at a stage to show the hyperbolicity of a large enough neighborhood of $\mathcal Fl_K(X_u)$ with the induced path metric.

\begin{theorem}\label{flow-sp-is-hyp}
Suppose $\F l_K(X_u)$ is the flow space of $X_u$ with the parameters $k$ and $R$ as mentioned in the Introduction of this Section \ref{hyp-of-flow-sp}. Let $r_2$ be the constant as in Lemma \ref{Li-bar-is-qi-emb-in-Li}. Then for any $L\ge r_2+8\dl_0$ there is $\dl_{\ref{flow-sp-is-hyp}}=\dl_{\ref{flow-sp-is-hyp}}(K,L)$ such that $Fl_{KL}(X_u)$ $:=N_L(\F l_K(X_u))$ is $\dl_{\ref{flow-sp-is-hyp}}$-hyperbolic.
\end{theorem}

\begin{proof}
We show that $Fl_{KL}(X_u)$ satisfies all the conditions of Proposition \ref{combing}. Note that $\F l_K(X_u)$ is $L$-dense in $Fl_{KL}(X_u)$. For a point $x\in\F l_K(X_u)$, we fix once and for all a $K$-qi section $\Sigma_x$ through $x$ over $B_x:=B_{[u,\pi(x)]}$ lying inside $\F l_K(X_u)$. Now given a pair $(x^1,x^2)$ of distinct points in $\F l_K(X_u)$, by Corollary \ref{ladder-in-flow-sp-for-two-pts}, there is a $(k_{\ref{ladder-in-flow-sp-for-two-pts}},c_{\ref{ladder-in-flow-sp-for-two-pts}},\varepsilon_{\ref{ladder-in-flow-sp-for-two-pts}})$-ladder, say $\tilde{\L}^{12}$, containing $\Sigma_{x^1},~\Sigma_{x^2}$ such that $top(\tilde{\L}^{12})\sse\F l_K(X_u),~bot(\tilde{\L}^{12})\sse\F l_K(X_u)$ and $\tilde{\L}^{12}\sse N^f_{2\dl_0}(\F l_K(X_u))$.
	
We take $\tilde{c}(x^1,x^2)$ a geodesic path joining $x^1,x^2$ in $\tilde{L}^{12}:=N_{r_2}(\tilde{\L}^{12})$. For a given pair of points, once and for all, we fix this ladder and the geodesic path. These paths serve as family of paths for Proposition \ref{combing}.
	
Let us start with three points $x^i\in\F l_K(X_u),~i=1,2,3$ and geodesic paths $\tilde{c}(x^i,x^j)$ in the respective ladders $\tilde{L}^{ij}:=N_{r_2}(\tilde{\L}^{ij})$ for all distinct $i,j\in\{1,2,3\}$. Note that $\tilde{L}^{ij}\sse Fl_{KL}(X_u)$.
	
{\bf Condition (1)}: As $\tilde{L}^{ij}$ is $L_{\ref{semicts-is-qi-emb}}(k_{\ref{ladder-in-flow-sp-for-two-pts}},r_2)$-qi embedded in $X$ and so is in $Fl_{KL}(X_u)$. Then the path $\tilde{c}(x^i,x^j)$ is $h$-properly embedded in $Fl_{KL}(X_u)$, where $h:\R_{\ge0}\ri\R_{\ge0}$ sending $r\mapsto rL_{\ref{semicts-is-qi-emb}}(k_{\ref{ladder-in-flow-sp-for-two-pts}},r_2)+(L_{\ref{semicts-is-qi-emb}}(k_{\ref{ladder-in-flow-sp-for-two-pts}},r_2))^2$.
	
{\bf Condition (2)}: By Proposition \ref{existence-of-tripod-ladder}, given any three points $x^i,i=1,2,3$, there are $(k_{\ref{existence-of-tripod-ladder}},c_{\ref{existence-of-tripod-ladder}},\varepsilon_{\ref{existence-of-tripod-ladder}})$-ladders, say $\L^{ij}$, containing $\Sigma_i,~\Sigma_j$ such that $top(\L^{ij})\sse\F l_K(X_u),~bot(\L^{ij})\sse\F l_K(X_u)$ and $\L^{ij}\sse N^f_{2\dl_0}(\F l_K(X_u))$. Let $c(x^i,x^j)$ be a geodesic path joining $x^i,x^j$ in $L^{ij}:=N_{r_2}(\L^{ij})\sse Fl_{KL}(X_u)$. Note that $k_{\ref{existence-of-tripod-ladder}}=k_{\ref{ladder-in-flow-sp-for-two-pts}},~c_{\ref{existence-of-tripod-ladder}}=c_{\ref{ladder-in-flow-sp-for-two-pts}}$ and $\varepsilon_{\ref{existence-of-tripod-ladder}}=\varepsilon_{\ref{ladder-in-flow-sp-for-two-pts}}$ (see Lemma \ref{ladder-in-flow-sp-for-two-pts}). So by Proposition \ref{paths-are-Hd-close},  $Hd_X(\tilde{c}(x^i,x^j),c(x^i,x^j))\le D$, where $D=D_{\ref{paths-are-Hd-close}}(k_{\ref{ladder-in-flow-sp-for-two-pts}},c_{\ref{ladder-in-flow-sp-for-two-pts}},\varepsilon_{\ref{ladder-in-flow-sp-for-two-pts}})$. Thus by Proposition \ref{fsandgss-are-proper-emb}, their Hausdorff distance is bounded by $\eta_1(D)$ in the path metric of $Fl_{KL}(X_u)$, where $\eta_1:=\eta_{\ref{fsandgss-are-proper-emb}}(K,L)$.
	
Again by Lemma \ref{paths-are-slim}, there is a $\dl_{\ref{paths-are-slim}}(k_{\ref{existence-of-tripod-ladder}},r_2)$-hyperbolic space $Y(:=N_{r_2+2\dl_0}(\cup_{i=1}^{3}\L^i))$ such that the inclusion $L^{ij}\ri Y$ is $L_{\ref{paths-are-slim}}(k_{\ref{existence-of-tripod-ladder}},r_2)$-qi embedding. Also note that $Y\sse N_{(r_2+8\dl_0)}(\F l_K(X_u))\sse Fl_{KL}(X_u)$. Let $\dl_1=\dl_{\ref{paths-are-slim}}(k_{\ref{existence-of-tripod-ladder}},r_2)$ and $L_1=L_{\ref{paths-are-slim}}(k_{\ref{existence-of-tripod-ladder}},r_2)$. Then by Lemma \ref{ml}, the triangle formed by the paths $c(x^i,x^j)$, for all distinct $i,j\in\{1,2,3\}$, are $D_1$-slim in the path metric of $Y$ and so is in the path metric of $Fl_{KL}(X_u)$, where $D_1:=2D_{\ref{ml}}(\dl_1,L_1,L_1)+\dl_1$.
	
Hence the triangle formed by the paths $\tilde{c}(x^i,x^j)$, for all distinct $i,j\in\{1,2,3\}$, are $D_2$-slim in the path metric of $Fl_{KL}(X_u)$, where $D_2:=2\eta_1(D)+D_1$.
	
Therefore, by Proposition \ref{combing}, $Fl_{KL}(X_u)$ is $\dl_{\ref{flow-sp-is-hyp}}$-hyperbolic metric space with the induced path metric from $X$, where $\dl_{\ref{flow-sp-is-hyp}}=\dl_{\ref{combing}}(h,D_2,L)$, and $h$ and $D_2$ are defined above.
\end{proof}


\section{Hyperbolicity of {\boldmath$N_D(\F l_K(X_u)\cup\F l_K(X_v))$}}\label{union-of-two-flow-sp-is-hyp}

Let $u,v\in T$, and $\F l_K(X_u)$ and $\F l_K(X_v)$ are the flow spaces as described in the first paragraph of Section \ref{hyp-of-flow-sp}. We also assume that $\F l_K(X_u)\cap X_v\ne\emptyset$. In this section, we will prove that $N_D(\F l_K(X_u)\cup\F l_K(X_v))$ is uniformly hyperbolic where $D\ge0$ is large enough. For a vertex $w\in T$, we use the notation $\bm{Fl_{KD}(X_w):=N_D(\F l_K(X_w))}$ for $D\ge0$. Here, we require the tree of metric bundles $(X,B,T)$ to satisfy $(2(L\pr)^2(2K+1)+L\pr)$-flaring condition where $L\pr=L_{\ref{mitra's-retraction-on-fsandgss}}(K)$. 


So far we have the following properties $(\H0)-(\H6)$. We will use them in this section.

$(\mathcal{H}0)$ Suppose $w,w\pr\in T$ and $e$ is the edge on $[w,w\pr]$ incident on $w\pr$. Let $T\pr$ be the maximal subtree of $T$ containing $w\pr$ not containing $e$. Then $\F l_K(X_w)\cap X_{T\pr}\sse \F l_K(X_{w\pr})\cap X_{T\pr}$ by the construction of flow spaces. \smallskip

$(\mathcal{H}1)$ For all $w\in T$, we have $L\pr:=L_{\ref{mitra's-retraction-on-fsandgss}}(K)$-coarsely Lipschitz retraction $\rho_w:X\ri \F l_K(X_w)$ such that for all $x\in \F l_K(X_w)$, $\pi_X(\rho_w(x))=\pi_X(x)$ (see Proposition \ref{mitra's-retraction-on-fsandgss}).\smallskip 

$(\mathcal{H}2)$ Let $w\in T$. For all $x\in\F l_K(X_w)$ there is a $K$-qi section lying in $\F l_K(X_w)\cap\pi^{-1}([w,\pi(x)])$ through $x$ over $B_{[w,\pi(x)]}$.\smallskip

$(\mathcal{H}3)$ Let $w\in T$. For $L\ge2K$, there is a proper map $\eta(L):\R_{\ge0}\ri\R_{\ge0}$ such that the inclusion $Fl_{KL}(X_w)\ri X$ is a $\eta(L)$-proper embedding (see Proposition \ref{fsandgss-are-proper-emb}).\smallskip


$(\mathcal{H}4)$ Let $\mathcal{G}=\{\gm:\gm \textrm{ is a }(2KL\pr+L\pr)\textrm{-qi section over } B_{[u,v]}\}$. Note that $\G\ne\emptyset$ as $\F l_K(X_u)\cap X_v\ne\emptyset$. For $w\in [u,v]$, $b\in B_w$, let $H_{b,w}=\textrm{hull}\{\gm(b):\gm\in\G\}\sse F_{b,w}$ and $H=\bigcup_{w\in [u,v],~b\in B_w} H_{b,w}$. (Here quasiconvex hull is considered in the corresponding fiber.) Then by Lemma \ref{getting-metric-graph-bundles}, $H$ is $\kappa$-metric bundle over $B_{[u,v]}$ where $\kappa=K_{\ref{getting-metric-graph-bundles}}(2KL\pr+L\pr)\ge2KL\pr+L\pr$. Now we consider flow of $H$ with parameters $\kappa$ and $\kappa$ (see Definition \ref{gen-flow-sp-1}). According to our notation (see Notation \ref{iteration-fn-for-qi-sec}), we have the flow space $\F l_{\kappa^{(1)}}(H)$ satisfying the following. Let $w\in T$ and $T_{uvw}$ be the tripod with vertices $u,v,w$. Since $\kappa\ge2KL\pr+L\pr$, by  Lemma \ref{flow-sp-contains-qi-sec}, for any $(2KL\pr+L\pr)$-qi section $\gm$ over $B_{T_{uvw}}$, $\gm\sse\F l_{\kappa^{(1)}}(H)$.\smallskip

By Notation \ref{iteration-fn-for-qi-sec}, we also have flow spaces $\F l_{\kappa^{(2)}}(X_u)$ containing $\F l_{K}(X_u)\cup \F l_{\kappa^{(1)}}(H)$ and $\F l_{\kappa^{(2)}}(X_v)$ containing $\F l_{K}(X_v)\cup \F l_{\kappa^{(1)}}(H)$.\smallskip


$(\mathcal{H}5)$ Let $w\in T$ and let $R_0$ be large enough so that $Fl_{\kappa^{(2)}R_0}(X_w)$ is $\dl$-hyperbolic for some $\dl\ge0$ (see Theorem \ref{flow-sp-is-hyp}).\smallskip

$(\mathcal{H}6)$ Since flow spaces are semicontinuous family, for all $L\ge max\{2\kappa^{(1)},2\dl_0+1\}$ there is a constant $\bar{L}(L)$ such that the inclusions $Fl_{KL}(X_u)\ri X$, $Fl_{KL}(X_v)\ri X$ and $Fl_{\kappa^{(1)}L}(H)\ri X$ are $\bar{L}(L)$-qi embeddings (see Corollary \ref{semicts-is-qi-emb} and Remark \ref{3-in-one}).\smallskip

We know that uniform neighborhood of flow spaces are uniformly properly embedded in the total space (see $(\H3)$). In the following proposition, we prove the same for the union of two intersecting flow spaces.
\begin{prop}\label{union-of-two-flow-sp-is-proper-emb}
Let $k_{\ref{union-of-two-flow-sp-is-proper-emb}}=2(L\pr)^2(2K+1)+L\pr$. For all $L\ge M_{k_{\ref{union-of-two-flow-sp-is-proper-emb}}}(\ge2K)$ there exists a proper map $\eta_{\ref{union-of-two-flow-sp-is-proper-emb}}=\eta_{\ref{union-of-two-flow-sp-is-proper-emb}}(K,L):\R_{\ge0}\ri\R_{\ge0}$ such that the inclusion $N_L(\F l_K(X_u)\cup\F l_K(X_v))\ri X$ is a $\eta_{\ref{union-of-two-flow-sp-is-proper-emb}}$-proper embedding, where $M_{k_{\ref{union-of-two-flow-sp-is-proper-emb}}}$ is coming from $k_{\ref{union-of-two-flow-sp-is-proper-emb}}$-flaring condition.
\end{prop}

\begin{proof}
Our proof goes in the same methodology as in the book \cite{ps-kap} for trees of metric spaces (see \cite[Subsection $6.1.1$]{ps-kap}). We denote the induced path metric on $ Fl_{KL}(X_u)\cup Fl_{KL}(X_v)$ by $d\pr$. We divide the proof by reducing the tree $T$ to intervals and the general tree in the following three cases. We first prove in all the cases that for $r\in\R_{\ge0}$ and $x,y\in\F l_K(X_u)\cup\F l_K(X_v)$ with $d_X(x,y)=r$, $d\pr(x,y)$ is bounded in terms of $r$. In the end, we prove for the points in $ Fl_{KL}(X_u)\cup Fl_{KL}(X_v)$.
	
Let $x,y\in\F l_K(X_u)\cup\F l_K(X_v)$ such that $d_X(x,y)\le r$. Suppose $\pi(x)=u\pr,\pi(y)=v\pr,\pi_X(x)=\bar{x},~\pi_X(y)=\bar{y}$. As $L\ge 2K$, we may assume that $x\in\F l_K(X_u)\setminus\F l_K(X_v)$ and $y\in\F l_K(X_v)\setminus\F l_K(X_u)$, otherwise, by $(\H3)$, $d\pr(x,y)\le\eta(L)(r)$.
	
{\bf Case 1}: We first assume that $T=[u,v]$. Then $u\pr\ne v$, otherwise, $x\in X_v\sse\F l_K(X_v)$. Also $v\pr\ne u$, otherwise, $y\in X_u\sse\F l_K(X_u)$. Depending on positions of $u\pr,v\pr,x$ and $y$, we consider following subcases.\smallskip
	
\emph{Subcase} (1A): Suppose $u\pr=v\pr$ and $x=\rho_u(y)$ (see $(\H1)$). Consider a $K$-qi section, say $\gm_y$, over $B_{[v\pr,v]}$ through $y$ in $\F l_K(X_v)$ (see $(\H2)$). Since $\rho_u$ is $L\pr$-coarsely Lipschitz retraction (see $(\H1)$) and $\F l_K(X_u)\cap X_v\ne\emptyset$, so applying $\rho_u$ on $\gm_y$, we get a $(2KL\pr+L\pr)$-qi section, say $\bar{\gm}_y$, in $\F l_K(X_u)$ over $B_{[v\pr,v]}$ such that $x=\rho_u(y)=\bar{\gm}_y(\bar{y})$. Let $b$ be the nearest point projection of $\bar{x}$ on $B_v$. (Note that such $b$ exists as $v\pr=u\pr\ne v$.) Applying $\rho_v$ (see $(\H1)$) on $\bar{\gm}_y$, we get a $2.(2KL\pr+L\pr)L\pr+L\pr=k_{\ref{union-of-two-flow-sp-is-proper-emb}}$-qi section, say $\bar{\bar{\gm}}_y$, in $\F l_K(X_v)$ over $B_{[v\pr,v]}$. Note that $\bar{\gm}_y(b)=\bar{\bar{\gm}}_y(b)$ (as $\bar{\gm}_y(b)\in X_v$). Let $\bar{\bar{\gm}}_y(\bar{x})=y\pr$. Then $\rho_v(x)=y\pr$ and since $\rho_v(y)=y$, we have $d_X(y\pr,y)=d_X(\rho_v(x),\rho_v(y))\le L\pr d_X(x,y)+L\pr\le L\pr(r+1)$. Thus $d_X(x,y\pr)\le d_X(x,y)+d_X(y,y\pr)\le r(L\pr+1)+L\pr$ and so $d^f(x,y\pr)\le \phi(r(L\pr+1)+L\pr)$, where the fibers are $\phi$-properly embedded in total space. Here we have two $k_{\ref{union-of-two-flow-sp-is-proper-emb}}$-qi sections $\bar{\gm}_y$ and $\bar{\bar{\gm}}_y$ over $B_{[v\pr,v]}$ such that $\bar{\gm}_y(b)=\bar{\bar{\gm}}_y(b)$ and $d^f(\bar{\gm}_y(\bar{x}),\bar{\bar{\gm}}_y(\bar{x}))=d^f(x,y\pr)\le \phi(r(L\pr+1)+L\pr)$. Let $a$ be the point on $[\bar{x},b]_B$ closest to $\bar{x}$ such that $d^f(\bar{\gm}_y(a),\bar{\bar{\gm}}_y(a))\le M_{k_{\ref{union-of-two-flow-sp-is-proper-emb}}}$. Since $L\ge M_{k_{\ref{union-of-two-flow-sp-is-proper-emb}}}$, $d\pr(\bar{\gm}_y(a),\bar{\bar{\gm}}_y(a))\le M_{k_{\ref{union-of-two-flow-sp-is-proper-emb}}}$. Again, the tree of metric bundles $(X,B,T)$ satisfies flaring condition, so by Lemma \ref{flaring-lemma} $(1)$, $d_B(\bar{x},a)\le\tau_{\ref{flaring-lemma}}(k_{\ref{union-of-two-flow-sp-is-proper-emb}},C)$, where $C=max\{M_{k_{\ref{union-of-two-flow-sp-is-proper-emb}}},\phi(r(L\pr+1)+L\pr)\}$. Let $C_1=\tau_{\ref{flaring-lemma}}(k_{\ref{union-of-two-flow-sp-is-proper-emb}},C)$. Then by taking lifts of geodesic $[\bar{x},a]_B$ in $\bar{\gm}_y$ and $\bar{\bar{\gm}}_y$ (see Lemma \ref{qi-sec-inside-lad-len-lift} $(3)$), we get $d\pr(x,\bar{\gm}_y(a))\le 2k_{\ref{union-of-two-flow-sp-is-proper-emb}}C_1$ and $d\pr(y,\bar{\bar{\gm}}_y(a))\le2k_{\ref{union-of-two-flow-sp-is-proper-emb}}C_1$. Hence $d\pr(x,y)\le d\pr(x,\bar{\gm}_y(a))+d\pr(\bar{\gm}_y(a),\bar{\bar{\gm}}_y(a))+d\pr(\bar{\bar{\gm}}_y(a),y)\le 4k_{\ref{union-of-two-flow-sp-is-proper-emb}}C_1+M_{k_{\ref{union-of-two-flow-sp-is-proper-emb}}}=:\eta_1(r)$ for some $\eta_1:\R_{\ge0}\ri\R_{\ge0}$ sending $r\mapsto 4k_{\ref{union-of-two-flow-sp-is-proper-emb}}C_1+M_{k_{\ref{union-of-two-flow-sp-is-proper-emb}}}$.\smallskip

\emph{Subcase} (1B): Let $y\pr=\rho_u(y)$. In this subcase, $y\pr$ need not be equal to $x$. Since $\rho_u(x)=x$, $d_X(x,y\pr)=d_X(\rho_u(x),\rho_u(y))\le L\pr r +L\pr=L\pr(r+1)$. So $ d_X(y\pr,y)\le d_X(y\pr,x)+d_X(x,y)\le L\pr(r+1)+r$. Since $L\ge2K$, $ Fl_{KL}(X_u)$ is $\eta(L)$-properly embedded in $X$ (see $(\H3)$). So $d\pr(x,y\pr)\le\eta(L)(L\pr(r+1))$.
Note that $\pi_X(y\pr)=\pi_X(y)$ (as $\F l_K(X_u\cap X_v\ne\emptyset$) and $y\pr\in\F l_K(X_u)$. Then by \emph{Subcase} $(1A)$, $d\pr(y,y\pr)\le\eta_1(L\pr(r+1)+r)$. Hence $d\pr(x,y)\le d\pr(x,y\pr)+d\pr(y\pr,y)\le\eta(L)(L\pr(r+1))+\eta_1(L\pr(r+1)+r)$.
	
We assume $\bm{\zeta_1(r)}:=max\{\eta_1(r),\eta(L)(L\pr(r+1))+\eta_1(L\pr(r+1)+r)\}$, maximum distortion in this Case $1$ for some $\zeta_1:\R_{\ge0}\ri\R_{\ge0}$.
	
{\bf Case 2}: We now assume that $T=[w,w\pr]\supsetneq[u,v]$ such that $u$ is closest to $w$. Then $v\pr\notin[w,u]$, otherwise, by $(\H0)$, $y\in\F l_K(X_u)$. Also, $u\pr\notin[v,w\pr]$, otherwise, by $(\H0)$, $x\in\F l_K(X_v)$. We consider the following subcases depending on the position of $u\pr$ and $v\pr$.
	
{\em Let $S=[u,v]$ and $X_S:=\pi^{-1}(S)$. Let $d\prr$ denote the induced path metric on $L$-neighborhood (in $X_S$-metric) of $(\F l_K(X_u)\cup\F l_K(X_v))\cap X_S$ inside $X_S$. We note that the restriction $\pi_X|_{X_S}:X_S\ri B_S$ also satisfies the flaring condition $($Remark \ref{bigsmall-subsp} $(b))$.}\smallskip
	
\emph{Subcase} (2A): Suppose $u\pr\in[w,u]$ and $v\pr\in[v,w\pr]$. Let $b\pr$ be the nearest point projection of $\bar{x}$ on $B_u$ and $b\prr$ be that of $\bar{y}$ on $B_v$. Then $d_B(\bar{x},\bar{y})\le d_X(x,y)\le r$ implies $d_B(\bar{x},b\pr)\le r$, $d_B(b\pr,b\prr)\le r$ and $d_B(b\prr,\bar{y})\le r$. Let $\gm_x$ be $K$-qi lift through $x$ of geodesic $[\bar{x},b\pr]_B$ in $\F l_K(X_u)$ and $\gm_y$ be that through $y$ of $[b\prr,\bar{y}]_B$ in $\F l_K(X_v)$ (see $(\H2)$). Let $\gm_x(b\pr)=x\pr$ and $\gm_y(b\prr)=y\pr$. Then $d\pr(x,x\pr)\le2Kr$ and $d\pr(y,y\pr)\le2Kr$ (see Lemma \ref{qi-sec-inside-lad-len-lift} $(3)$). So by triangle inequality, $d_X(x\pr,y\pr)\le4Kr+r$, and that implies  $d_{X_S}(x\pr,y\pr)\le\eta_{\ref{subsp-is-proper-emb}}(4Kr+r)$ (see Proposition $\ref{subsp-is-proper-emb}$). Then by Case $1$, $d\prr(x\pr,y\pr)\le\zeta_1(\eta_{\ref{subsp-is-proper-emb}}(4Kr+r))$. Therefore, $d\pr(x,y)\le d\pr(x,x\pr)+d\prr(x\pr,y\pr)+d\pr(y\pr,y)\le 4Kr+\zeta_1(\eta_{\ref{subsp-is-proper-emb}}(4Kr+r))$.\smallskip
	
\emph{Subcase} (2B): Suppose $u\pr\in[u,v]\setminus\{u,v\}$ and $v\pr\in[v,w\pr]$. Let $b\prr$ be the nearest point projection of $\bar{y}$ on $B_v$. Then $d_B(\bar{x},\bar{y})\le d_X(x,y)\le r$ implies $d_B(\bar{y},b\prr)\le r$ (as $u\pr\in[u,v]\setminus\{u,v\})$. Let $\gm_y$ be a $K$-qi lift of the geodesic $[b\prr,\bar{y}]_B$ in $\F l_K(X_v)$ through $y$ (see $(\H2)$) and let $\gm_y(b\prr)=y\pr$. Then $d\pr(y,y\pr)\le 2Kr$ (see Lemma \ref{qi-sec-inside-lad-len-lift} $(3)$). Again, $d_X(y\pr,x)\le d_X(y\pr,y)+d_X(y,x)\le 2Kr+r$. So by Proposition \ref{subsp-is-proper-emb}, $d_{X_S}(y\pr,x)\le\eta_{\ref{subsp-is-proper-emb}}(2Kr+r)$. Hence by  \emph{Subcase} $(1A)$, $d\prr(x,y\pr)\le\zeta_1(\eta_{\ref{subsp-is-proper-emb}}(2Kr+r))$. So $d\pr(x,y)\le d\pr(x,y\pr)+d\pr(y\pr,y)\le d\prr(x,y\pr)+d\pr(y\pr,y)\le \zeta_1(\eta_{\ref{subsp-is-proper-emb}}(2Kr+r))+2Kr$.\smallskip
	
\emph{Subcase} (2C): Suppose $u\pr\in[w,u]$ and $v\pr\in[u,v]\setminus\{u,v\}$. Then this is a symmetry of \emph{Subcase} $(2B)$. So $d\pr(x,y)\le\zeta_1(\eta_{\ref{subsp-is-proper-emb}}(2Kr+r))+2Kr$.\smallskip
	
\emph{Subcase} (2D): Finally, we assume that $u\pr,v\pr\in[u,v]\setminus\{u,v\}$. Then by Proposition \ref{subsp-is-proper-emb}, $d_{X_S}(x,y)\le\eta_{\ref{subsp-is-proper-emb}}(r)$. So by Case $1$, $d\prr(x,y)\le\zeta_1(\eta_{\ref{subsp-is-proper-emb}}(r))$. Hence $d\pr(x,y)\le d\prr(x,y)\le\zeta_1(\eta_{\ref{subsp-is-proper-emb}}(r))$. 
	
We assume $\bm{\zeta_2(r)}:=max\{4Kr+\zeta_1(\eta_{\ref{subsp-is-proper-emb}}(4Kr+r)),\zeta_1(\eta_{\ref{subsp-is-proper-emb}}(2Kr+r))+2Kr, \zeta_1(\eta_{\ref{subsp-is-proper-emb}}(r))\}$, maximum distortion in this Case $2$ for some $\zeta_2:\R_{\ge0}\ri\R_{\ge0}$.

{\bf Case 3}: Here we consider the general case, i.e., $T$ is any tree. Depending on the position of $u,v,u\pr$ and $v\pr$, we consider the following subcases.
	
{\em Let $S$ be an interval in $T$ containing $u,v$ and $X_S:=\pi^{-1}(S)$. We denote the induced path metric on $L$-neighborhood (in $X_S$-metric) of $(\F l_K(X_u)\cup\F l_K(X_v))\cap X_S$ inside $X_S$ by $d\prr$. We will use this notation below. We note that the restriction $\pi_X|_{X_S}:X_S\ri B_S$ also satisfies the flaring condition $($see Remark \ref{bigsmall-subsp} $(b))$.}\smallskip
	
\emph{Subcase} (3A): Suppose $u,v,u\pr$ and $v\pr$ lie on an interval in $T$. We fix one such interval $S$ in $T$ containing $u,v,u\pr,v\pr$. So $d_{X_S}(x,y)\le\eta_{\ref{subsp-is-proper-emb}}(r)$ (by Proposition \ref{subsp-is-proper-emb}). Now we restrict the flow spaces to $X_S$. Then by Case $2$, $d\prr(x,y)\le\zeta_2(\eta_{\ref{subsp-is-proper-emb}}(r))$. So $d\pr(x,y)\le d\prr(x,y)\le\zeta_2(\eta_{\ref{subsp-is-proper-emb}}(r))$.
	
Now we consider the subcases when all of $u,v,u\pr$ and $v\pr$ do not lie on an interval.\smallskip
	
\emph{Subcase} (3B): Suppose there is no interval containing $u,v$ that contains both $u\pr,v\pr$; but there is an interval containing $u,v$ which contains one of $u\pr,v\pr$. We give a proof when an interval containing $u,v$ also contains $u\pr$, and leave the other case because its involves only a change in indices. We fix one such interval $S$ in $T$ containing $u,v$ and $u\pr$. Let $t$ be the nearest point projection of $v\pr$ on $S$ in $d_T$-metric and  $b\prr$ be that of $\bar{y}$ on $B_t$ in $d_B$-metric. Since $d_X(x,y)\le r$, then $d_B(\bar{y},b\prr)\le d_B(\bar{y},\bar{x})\le r$. Let $\gm_y$ be a $K$-qi lift of the geodesic $[b\prr,\bar{y}]_B$ through $y$ in $\F l_K(X_v)$ (see $(\H2)$). Suppose $\gm_y(b\prr)=y\pr$. Then $d\pr(y,y\pr)\le2Kr$ (see Lemma \ref{qi-sec-inside-lad-len-lift} $(3)$). Again $d_X(y\pr,x)\le d_X(y\pr,y)+d_X(y,x)\le2Kr+r$, and so by Proposition \ref{subsp-is-proper-emb}, $d_{X_S}(y\pr,x)\le\eta_{\ref{subsp-is-proper-emb}}(2Kr+r)$. Now we restrict the flow spaces to $X_S$. Hence by Case $2$, $d\prr(y\pr,x)\le\zeta_2(\eta_{\ref{subsp-is-proper-emb}}(2Kr+r))$. Therefore, $d\pr(x,y)\le d\pr(x,y\pr)+d\pr(y\pr,y)\le d\prr(x,y\pr)+d\pr(y\pr,y)\le \zeta_2(\eta_{\ref{subsp-is-proper-emb}}(2Kr+r))+2Kr$.\smallskip

	
\emph{Subcase} (3C): Suppose there is no interval containing $u,v$ that contains either of $u\pr,v\pr$. We fix $S=[u,v]$. Let $t_1$ and $t_2$ be the nearest point projections of $u\pr$ and $v\pr$ on $S$ respectively. Then $t_1,t_2\in[u,v]\setminus\{u,v\}$, otherwise, it will land in Subcase $(3B)$. We divide the proof into two cases depending on whether $t_1,t_2$ are same or not.\smallskip
	
{\bf (a)} Suppose $t_1\ne t_2$. Let $b\pr$  be the nearest point projection of $\bar{x}$ on $B_{t_1}$ and $b\prr$ be that of $\bar{y}$ on $B_{t_2}$. Since $d_B(\bar{x},\bar{y})\le d_X(x,y)\le r$, then $d_B(\bar{x},b\pr)\le r$ and $d_B(\bar{y},b\prr)\le r$. Let $\gm_x$ be a  $K$-qi lift of the geodesic $[\bar{x},b\pr]_B$ through $x$ in $\F l_K(X_u)$ and $\gm_y$ be that of $[\bar{y},b\prr]_B$ through $y$ in $\F l_K(X_v)$ (see $(\H2)$). Let $x\pr=\gm_x(b\pr)$ and $y\pr=\gm_y(b\prr)$. Then $d\pr(x,x\pr)\le2Kr$ and $d\pr(y,y\pr)\le2Kr$. So $d_X(x\pr,y\pr)\le d_X(x\pr,x)+d_X(x,y)+d_X(y,y\pr)\le4Kr+r$. Then by Proposition \ref{subsp-is-proper-emb}, $d_{X_S}(x\pr,y\pr)\le\eta_{\ref{subsp-is-proper-emb}}(4Kr+r)$. Note that $x\pr\in\F l_K(X_u)$ and $y\pr\in\F l_K(X_v)$. Now we restrict the flow spaces to $X_S=X_{[u,v]}$. Hence by Case $1$, $d\prr(x\pr,y\pr)\le\zeta_1(\eta_{\ref{subsp-is-proper-emb}}(4Kr+r))$, and so $d\pr(x\pr,y\pr)\le\zeta_1(\eta_{\ref{subsp-is-proper-emb}}(4Kr+r))$. Therefore, $d\pr(x,y)\le d\pr(x,x\pr)+d\pr(x\pr,y\pr)+d\pr(y\pr,y)\le 4Kr+\zeta_1(\eta_{\ref{subsp-is-proper-emb}}(4Kr+r))$.\smallskip

	
{\bf (b)} Suppose $t_1=t_2=t$ (say). Let $s$ be the center of $\bigtriangleup(u\pr,t,v\pr)$ and $c\in[\bar{x},\bar{y}]\cap B_s$. Since $d_B(\bar{x},\bar{y})\le d_X(x,y)\le r$, so $d_B(\bar{x},c)\le r$ and $d_B(c,\bar{y})\le r$. Let $\gm_1$ be a $K$-qi lift through $x$ of the geodesic $[\bar{x},c]_B$ in $\F l_K(X_u)$ and $\gm_2$ be that through $y$ of the geodesic $[\bar{y},c]_B$ in $\F l_K(X_v)$ (see $(\H2)$). Let $x_1=\gm_1(c)$ and $y_1=\gm_2(c)$. Then $d\pr(x,x_1)\le2Kr$ and $d\pr(y,y_1)\le2Kr$.
	
Now we only need to show that $d\pr(x_1,y_1)$ is bounded in terms of $r$. We will apply the same trick as in Case $1$. Let $\gm_y$ be a $K$-qi section over $B_{[s,v]}$ through $y_1$ in $\F l_K(X_v)$ (see $(\H2)$). Now we apply $\rho_u$ (see $(\H1)$) on $\gm_y$ and get a $L\pr(2K+1)$-qi section, say $\bar{\gm}_y$, over $B_{[s,v]}$ in $\F l_K(X_u)$ (see Figure \ref{proj}). (This is possible as $[s,v]\sse \pi(\F l_K(X_u))$.) By triangle inequality, $d_X(x_1,y_1)\le4Kr+r$. Let $\rho_u(y_1)=x_2$. Since $\rho_u(x_1)=x_1$, then $d_X(x_1,x_2)=d_X(\rho_u(x_1),\rho_u(y_1))\le L\pr d_X(x_1,y_1)+L\pr\le L\pr(4Kr+r+1)$. Then $d_X(y_1,x_2)\le d_X(y_1,x_1)+d_X(x_1,x_2)\le4Kr+r+L\pr(4Kr+r+1)=(4Kr+r)(L\pr+1)+L\pr$. Again we apply $\rho_v$ (see $(\H1)$) on $\bar{\gm}_y$ and get a $k_{\ref{union-of-two-flow-sp-is-proper-emb}}$-qi section, say $\bar{\bar{\gm}}_y$, over $B_{[s,v]}$ in $\F l_K(X_v)$ (see Figure \ref{proj}). Let $\rho_v(x_2)=y_2$. Since $\rho_v(y_1)=y_1$, $d_X(y_1,y_2)\le d_X(\rho_v(y_1),\rho_v(x_2))\le L\pr d_X(y_1,x_2)+L\pr\le L_1$, where $L_1=L\pr((4Kr+r)(L\pr+1)+L\pr)+L\pr$.  Then $d_X(x_2,y_2)\le d_X(x_2,y_1)+d_X(y_1,y_2)\le L_2$, where $L_2=(4Kr+r(L\pr+1)+L\pr+L_1$. So $d^f(x_2,y_2)\le\phi(L_2)$.
	
	\begin{figure}[h]
		\includegraphics[width=10cm]{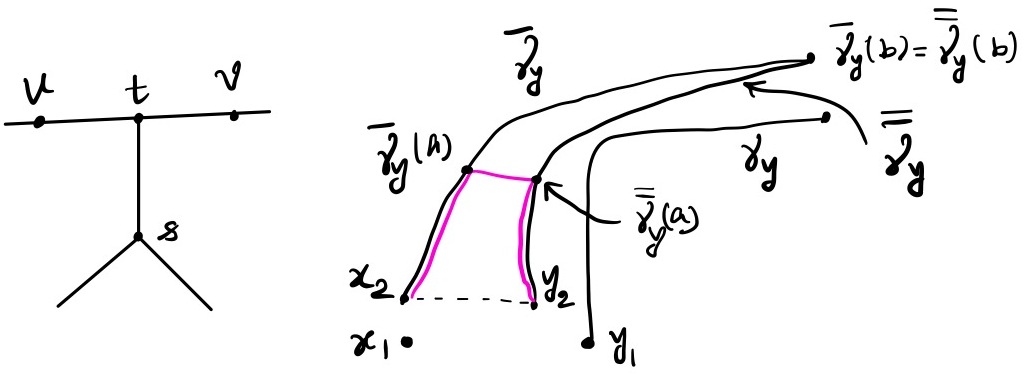}
		\centering
		\caption{}
		\label{proj}
	\end{figure}
	
Note that $d_X(x_1,x_2)\le L\pr(4Kr+r+1)$ and $d_X(y_1,y_2)\le L_1$. Again $x_1,x_2\in\F l_K(X_u)$ and $y_1,y_2\in\F l_K(X_v)$, so by $(\H3)$, $d\pr(x_1,x_2)\le\eta(L)(L\pr(4Kr+r+1))$ and $d\pr(y_1,y_2)\le\eta(L)(L_1)$. So to get a bound on $d\pr(x_1,y_1)$, we need to get a bound on $d\pr(x_2,y_2)$; which we will show now.
	
	Let $b$ be the nearest point projection of $c$ on $B_v$. Then $\bar{\gm}_y(b)\in X_v$ and so $\bar{\gm}_y(b)=\bar{\bar{\gm}}_y(b)$. Note that $\bar{\gm}_y$ and $\bar{\bar{\gm}}_y$ are two $k_{\ref{union-of-two-flow-sp-is-proper-emb}}$-qi sections over $B_{[s,v]}$ such that $d^f(\bar{\gm}_y(c),\bar{\bar{\gm}}_y(c))=d^f(x_2,y_2)\le \phi(L_2)$ and $\bar{\gm}_y(b)=\bar{\bar{\gm}}_y(b)$. Now we restrict the qi sections $\bar{\gm}_y$ and $\bar{\bar{\gm}}_y$ on the geodesic $[c,b]_B\sse B$. Let $a$ be the point on $[c,b]$ closest to $c$ such that $d^f(\bar{\gm}_y(a),\bar{\bar{\gm}}_y(a))\le M_{k_{\ref{union-of-two-flow-sp-is-proper-emb}}}$. Since the tree of metric bundles $(X,B,T)$ satisfies flaring condition, by Lemma \ref{flaring-lemma} $(1)$, $d_B(c,a)\le\tau_{\ref{flaring-lemma}}(k_{\ref{union-of-two-flow-sp-is-proper-emb}},D)$, where $D=max\{M_{k_{\ref{union-of-two-flow-sp-is-proper-emb}}},\phi(L_2)\}$. Let $D_1=\tau_{\ref{flaring-lemma}}(k_{\ref{union-of-two-flow-sp-is-proper-emb}},D)$. Then by taking $k_{\ref{union-of-two-flow-sp-is-proper-emb}}$-qi lifts of the geodesic $[c,a]_B$ in $\bar{\gm}_y$ and $\bar{\bar{\gm}}_y$, we get $d\pr(x_2,\bar{\gm}_y(a))\le2k_{\ref{union-of-two-flow-sp-is-proper-emb}}D_1$ and $d\pr(\bar{\bar{\gm}}_y(a),y_2)\le2k_{\ref{union-of-two-flow-sp-is-proper-emb}}D_1$ (note that $\bar{\gm}_y(c)=x_2,~\bar{\bar{\gm}}_y(c)=y_2$). Again $L\ge M_{k_{\ref{union-of-two-flow-sp-is-proper-emb}}}$ implies $d\pr(\bar{\gm}_y(a),\bar{\bar{\gm}}_y(a))\le M_{k_{\ref{union-of-two-flow-sp-is-proper-emb}}}$. Hence $d\pr(x_2,y_2)\le4k_{\ref{union-of-two-flow-sp-is-proper-emb}}D_1+M_{k_{\ref{union-of-two-flow-sp-is-proper-emb}}}$ (by triangle inequality).
	
	Again by triangle inequality, $d\pr(x_1,y_1)\le d\pr(x_1,x_2)+d\pr(x_2,y_2)+d\pr(y_2,y_1)\le L_3$, where $L_3=\eta(L)(L\pr(4Kr+r+1))+4k_{\ref{union-of-two-flow-sp-is-proper-emb}}D_1+M_{k_{\ref{union-of-two-flow-sp-is-proper-emb}}}+\eta(L)(L_1)$. So $d\pr(x,y)\le d\pr(x,x_1)+d\pr(x_1,y_1)+d\pr(y_1,y)\le 4Kr+L_3$.
	
	Let $\bm{\zeta_3(r)}:=max\{\zeta_2(\eta_{\ref{subsp-is-proper-emb}}(r)),~\zeta_2(\eta_{\ref{subsp-is-proper-emb}}(2Kr+r))+2Kr,~ \zeta_1(\eta_{\ref{subsp-is-proper-emb}}(4Kr+r))+4Kr,~4Kr+L_3\}$, maximum distortion in this Case $3$ for some $\zeta_3:\R_{\ge0}\ri\R_{\ge0}$.
	
	Let $\zeta:\R_{\ge0}\ri\R_{\ge0}$ such that $\zeta(r):=max\{\eta(L)(r),\zeta_1(r),\zeta_2(r),\zeta_3(r)\}$ for $r\in\R_{\ge0}$. We have proved that if $x,y\in\F l_K(X_u)\cup\F l_K(X_v)$ are at most $r$-distance apart in the metric of $X$, then they are at most $\zeta(r)$-distance apart in the induced path metric on $ Fl_{KL}(X_u)\cup Fl_{KL}(X_v)$. Now we take points $x,y\in Fl_{KL}(X_u)\cup Fl_{KL}(X_v)$ such that $d_X(x,y)\le r$ for $r\in\R_{\ge0}$. Let $x_1,y_1\in\F l_K(X_u)\cup\F l_K(X_v)$ such that $d\pr(x,x_1)\le L$ and $d\pr(y,y_1)\le L$. Then $d_X(x,y)\le d_X(x,x_1)+d_X(x_1,y_1)+d_X(y_1,y)\le r+2L$. So $d\pr(x_1,y_1)\le\zeta(r+2L)$. Hence $d\pr(x,y)\le d\pr(x,x_1)+d\pr(x_1,y_1)+d\pr(y_1,y)\le 2L+\zeta(r+2L)$.
	
	Therefore, we can take $\eta_{\ref{union-of-two-flow-sp-is-proper-emb}}:\R_{\ge0}\ri\R_{\ge0}$ sending $r\mapsto\zeta(r+2L)+2L$.
\end{proof}

To show the hyperbolicity of $N_D(\F l_K(X_u)\cup\F l_K(X_v))$, we construct a bigger uniformly hyperbolic space $Y=Y_1\cup Y_2$ containing both $\F l_K(X_u)$ and $\F l_K(X_v)$ as uniformly quasiconvex subsets. As $\F l_K(X_u)\cap X_v\ne\emptyset$, so a uniform neighborhood of $\F l_K(X_u)\cup\F l_K(X_v)$ in $Y$ is uniformly hyperbolic. Let $N\pr_D(y)$ denote $D$-neighborhood of $y\in Y$ in the path metric of $Y$. Next, we show that $N\pr_D(\F l_K(X_u)\cup\F l_K(X_v))\sse Y$ and $N_D(\F l_K(X_u)\cup\F l_K(X_v))\sse X$ are (uniformly) quasiisometric, and that completes the proof.

\underline{{\bf \emph{Construction of the space $Y$}}}: By $(\H5)$, $Fl_{\kappa^{(2)}R_0}(X_u)$ is a $\dl$-hyperbolic metric space. Now by $(\H6)$, $Fl_{KR_0}(X_u)$ is $\bar{L}(R_0)$-qi embedded in $X$ and so is in $Fl_{\kappa^{(2)}R_0}(X_u)$. Then by Lemma \ref{quasi-goes-to-quasi} $(1)$, there is $K_1$ depending on $\dl$ and $\bar{L}(R_0)$ such that
$Fl_{KR_0}(X_u)$ is $K_1$-quasiconvex in $Fl_{\kappa^{(2)}R_0}(X_u)$. So $\F l_K(X_u)$ is $(K_1+R_0)$-quasiconvex in $Fl_{\kappa^{(2)}R_0}(X_u)$. Also by $(\H6)$, $Fl_{\kappa^{(1)}R_0}(H)$ is $\bar{L}(R_0)$-qi embedded in $X$ and so is in $Fl_{\kappa^{(2)}R_0}(X_u)$. Then by Lemma \ref{quasi-goes-to-quasi} $(1)$, there is $K_2$ depending on $\dl$ and $\bar{L}(R_0)$ such that $Fl_{\kappa^{(1)}R_0}(H)$ is $K_2$-quasiconvex in $Fl_{\kappa^{(2)}R_0}(X_u)$. So $\F l_{\kappa^{(1)}}(H)$ is $(K_2+R_0)$-quasiconvex in $Fl_{\kappa^{(2)}R_0}(X_u)$. Let $K_3=max\{K_1+R_0,K_2+R_0\}$.
As $\F l_K(X_u)\cap\F l_{\kappa^{(1)}}(H)\ne\emptyset$, so $\F l_K(X_u)\cup\F l_{\kappa^{(1)}}(H)$ is $(K_3+\dl)$-quasiconvex in $Fl_{\kappa^{(2)}R_0}(X_u)$. Let $Y\pr_{1R}:=N\pr_R(\F l_K(X_u)\cup\F l_{\kappa^{(1)}}(H))$ be $R$-neighborhood of $\F l_K(X_u)\cup\F l_{\kappa^{(1)}}(H)\sse Fl_{\kappa^{(2)}R_0}(X_u)$ in the induced path metric on $Fl_{\kappa^{(2)}R_0}(X_u)$ where
\begin{eqnarray}\label{value-of-R}
	\bm{R}&=&max\{K_3+\dl+1,~M_{k_{\ref{union-of-two-flow-sp-is-proper-emb}}}\}
\end{eqnarray}
Hence, by Lemma \ref{qi-emb-in-Y-X}, there is a constant $L_1$ depending on $\dl$, $K_3$ and $R$ such that the inclusion $Y\pr_{1R}\ri Fl_{\kappa^{(2)}R_0}(X_u)$ is a $L_1$-qi embedding.

{\bf We fix this $\bm{R}$ for the rest of this section}.  Thus there is $\dl_1$ depending on $\dl$ and $L_1$ such that $Y\pr_{1R}$ is $\dl_1$-hyperbolic with the induced path metric. Again, the inclusion $Fl_{\kappa^{(2)}R_0}(X_u)\ri X$ is a $\bar{L}(R_0)$-qi embedding (see $(\H6)$). Thus the inclusion $Y\pr_{1R}\ri X$ is a $L_2$-qi embedding for some $L_2$ depending on $L_1$ and $\bar{L}(R_0)$.

Let $Y\pr_{2R}:=N\pr_R(\F l_K(X_v)\cup\F l_{\kappa^{(1)}}(H))$ be $R$-neighborhood of $\F l_K(X_v)\cup\F l_{\kappa^{(1)}}(H)$ $\sse Fl_{\kappa^{(2)}R_0}(X_v)$ in the induced path metric on $Fl_{\kappa^{(2)}R_0}(X_v)$. Then by similar argument, we can show that $Y\pr_{2R}$ is $\dl_1$-hyperbolic and the inclusion $Y\pr_{2R}\ri Fl_{\kappa^{(2)}R_0}(X_v)$ is a $L_2$-qi embedding.

We take $Y:=Y_{1R}\cup Y_{2R}$ where $Y_{1R}:=N_R(\F l_K(X_u)\cup\F l_{\kappa^{(1)}}(H))\sse X$ and $Y_{2R}:=N_R(\F l_K(X_v)\cup\F l_{\kappa^{(1)}}(H))\sse X$. Note that these neighborhoods are considered in $X$.\smallskip

\underline{{\bf \emph{Hyperbolicity of $Y$}}}:
\begin{lemma}\label{Yir-is-hyp}
There is a constant $\dl_{\ref{Yir-is-hyp}}\ge0$ depending on $\dl_1,~L_2$ and $R$ such that $Y_{iR}$ is $\dl_{\ref{Yir-is-hyp}}$-hyperbolic metric space with the induced path metric, where $i=1,2$.
\end{lemma}
\begin{proof}
Since the inclusion $Y\pr_{iR}\ri X$ is a $L_2$-qi embedding, so is the inclusion $Y\pr_{iR}\ri Y_{iR}$. Also the inclusion $Y\pr_{iR}\ri Y_{iR}$ is $R$-coarsely surjective. Hence the inclusion $Y\pr_{iR}\ri Y_{iR}$ is a $(L_2,L_2,R)$-quasiisometry for $i=1,2$ (see Subsection \ref{some-basic-concepts}). Since the hyperbolicity is quasiisometry invariant and $Y\pr_{iR}$ is $\dl_1$-hyperbolic, $Y_{iR}$ is $\dl_{\ref{Yir-is-hyp}}$-hyperbolic for $i=1,2$, where $\dl_{\ref{Yir-is-hyp}}$ depends on $\dl_1,L_2,R$.
\end{proof}

\begin{lemma}\label{Yir-is-proper-emb}
There exists a proper function $\eta_{\ref{Yir-is-proper-emb}}:\R_{\ge0}\ri\R_{\ge0}$ depending on $L_2$ and $R$ such that the inclusion $Y_{iR}\ri X$ is a $\eta_{\ref{Yir-is-proper-emb}}$-proper embedding, where $i=1,2$.
\end{lemma}

\begin{proof}
	We prove it only for $i=1$ as the proof for $i=2$ is similar. We denote the induced path metric on $Y_{1R}$ and $Y\pr_{1R}$ by $d_1$ and $d\pr_1$ respectively. Let $x,y\in Y_{1R}$ such that $d_X(x,y)=r$ for $r\in\R_{\ge0}$. We take points $x_1,y_1\in\F l_K(X_u)\cup\F l_{\kappa^{(1)}}(H)\sse Y\pr_{1R}$ such that $d_1(x,x_1)\le R,~d_1(y,y_1)\le R$. So $d_X(x_1,y_1)\le r+2R$. Since $Y\pr_{1R}$ is $L_2$-qi embedded in $X$, then $d\pr_1(x_1,y_1)\le(r+2R)L_2+L_2^2$. Since $Y\pr_{1R}\sse Y_{1R}$ and so $d_1(x,y)\le d_1(x,x_1)+d\pr_1(x_1,y_1)+d_1(y_1,y)\le(r+2R)L_2+L_2^2+2R=:\eta_{\ref{Yir-is-proper-emb}}(r)$.
\end{proof}


\begin{lemma}\label{containment}
Let $d_i$ denote the induced path metric on $Y_{iR}$ for $i=1,2$. There is a constant $D_{\ref{containment}}\ge0$ depending on $R$ and $L'$ (see Introduction of Section \ref{union-of-two-flow-sp-is-hyp} for $L'$) such that $Y_{1R}\cap Y_{2R}\sse N^i_{D_{\ref{containment}}}(Y_0)$, where $N^i_{D_{\ref{containment}}}(Y_0)$ denotes the $D_{\ref{containment}}$-neighborhood of $Y_0$ in $d_i$-metric.
\end{lemma}

\begin{proof}
	Let $x\in Y_{1R}\cap Y_{2R}$. Then there exist $x_1\in\F l_K(X_u)\cup\F l_{\kappa^{(1)}}(H)$ and $x_2\in\F l_K(X_v)\cup\F l_{\kappa^{(1)}}(H)$ such that $d_i(x,x_i)\le R,~i=1,2$. So $d_X(x_1,x_2)\le2R$. Without loss of generality, we assume that $x_1\in\F l_K(X_u)\setminus\F l_{\kappa^{(1)}}(H),~x_2\in\F l_K(X_v)\setminus\F l_{\kappa^{(1)}}(H)$, otherwise, $x\in Y_0:=Fl_{\kappa^{(1)}R}(H)$. Let $\pi_X(x_i)=\bar{x}_i,~\pi(x_i)=t_i$ for $i=1,2$. Let $w_i$ be the nearest point projection of $t_i$ on $[u,v],i=1,2$. 
	We consider the following two cases, depending on whether $w_1=w_2$ or $w_1\ne w_2$.
	
{\bf Case $\bm{1}$}: Suppose $w_1\ne w_2$. Let $\bar{y}_i$ be the nearest point projection of $\bar{x}_i$ on $B_{w_i}$ for $i=1,2$. Then $d_B(\bar{x}_1,\bar{x}_2)\le d_X(x_1,x_2)\le2R$ implies $d_B(\bar{x}_i,\bar{y}_i)\le2R$ for $i=1,2$. Let $\gm_{x_1}$ be a $K$-qi section through $x_1$ inside $\F l_K(X_u)$ over $B_{[t_1,u]}$ and $\gm_{x_2}$ be that through $x_2$ inside $\F l_K(X_v)$ over $B_{[t_2,v]}$ (see $(\H2)$). Suppose $\gm_{x_i}(\bar{y}_i)=y_i,~ i=1,2$. Then by taking lift of the geodesic $[\bar{x}_i,\bar{y}_i]_B$ in $\gm_{x_i}$, we get $d_i(x_i,y_i)\le2K.2R=4KR$ (see Lemma \ref{qi-sec-inside-lad-len-lift} $(3)$) for $i=1,2$. Now we restrict the $K$-qi section $\gm_{x_2}$ over $B_{[w_2,v]}$ and apply $\rho_u$ (see $(\H1)$) on this restriction of $\gm_{x_2}$ over $B_{[w_2,v]}$. We set this projection as $\bar{\gm}_2$. Note that $B_{[w_2,v]}\sse \pi(\F l_K(X_u))$, then $\bar{\gm}_2$ is a $(2KL\pr+L\pr)$-qi section over $B_{[w_2,v]}$ inside $\F l_K(X_u)$. As $\F l_K(X_u\cap X_v\ne\emptyset$, then we can extend $\bar{\gm}_2$ to a $(2KL\pr+L\pr)$-qi section over $B_{[u,v]}$. Then in particular, we have $\bar{\gm}_2\sse H$ (see $(\H4)$). Again, $d_X(y_1,y_2)\le d_X(y_1,x_1)+d_X(x_1,x_2)+d_X(x_2,y_2)\le8KR+2R$. Note that $\rho_u(y_1)=y_1$ and $\rho_u(y_2)=\bar{\gm}_2(\bar{y}_2)$. Since  $\rho_u$ is $L\pr$-coarsely Lipschitz retraction (see $(\H1)$), $d_X(y_1,\bar{\gm}_2(\bar{y}_2))\le L\pr d_X(y_1,y_2)+L\pr\le L\pr(8KR+2R)+L\pr$. Since $y_1,\bar{\gm}_2(\bar{y}_2)\in\F l_K(X_u)\sse Y_{1R}$, by Lemma \ref{Yir-is-proper-emb}, $d_1(y_1,\bar{\gm}_2(\bar{y}_2))\le\eta_{\ref{Yir-is-proper-emb}}(L\pr(8KR+2R)+L\pr)$. Now $\bar{\gm}_2\sse H$ implies $d_1(x,Y_0)\le d_1(x,\bar{\gm}_2(\bar{y}_2))\le d_1(x,x_1)+d_1(x_1,y_1)+d_1(y_1,\bar{\gm}_2(\bar{y}_2))\le R+4KR+\eta_{\ref{Yir-is-proper-emb}}(L\pr(8KR+2R)+L\pr)$.
	
Again, $d_X(x_2,\bar{\gm}_2(\bar{y}_2))\le d_X(x_2,x_1)+d_X(x_1,y_1)+d_X(y_1,\bar{\gm}_2(\bar{y}_2))\le2R+4KR+L\pr(8KR+2R)+L\pr$. Since $x_2\in\F l_K(X_v)\sse Y_{2r}$ and $\bar{\gm}_2\sse H\sse Y_{2r}$, so by Lemma \ref{Yir-is-proper-emb}, $d_2(x_2,\bar{\gm}_2(\bar{y}_2))\le \eta_{\ref{Yir-is-proper-emb}}(2R+4KR+L\pr(8KR+2R)+L\pr)$. Thus $d_2(x,Y_0)\le d_2(x,x_2)+d_2(x_2,\bar{\gm}_2(\bar{y}_2))\le R+\eta_{\ref{Yir-is-proper-emb}}(2R+4KR+L\pr(8KR+2R)+L\pr)$.
	
	{\bf Case $\bm{2}$}: Suppose $w_1=w_2$. Let $w$ be the center of the tripod $\triangle(t_1,t_2,w_1)$. Suppose $\bar{y}_i$ is the nearest point projection of $\bar{x}_i$ on $B_w$ for $i=1,2$. Let $\gm_{x_1}$ be a $K$-qi section through $x_1$ inside $\F l_K(X_u)$ over $B_{[t_1,u]}$ and $\gm_{x_2}$ be that through $x_2$ inside $\F l_K(X_v)$ over $B_{[t_2,v]}$ (see $(\H2)$). Again $d_B(\bar{x}_1,\bar{x}_2)\le2R$ implies $d_B(\bar{x}_i,\bar{y}_i)\le2R$ for $i=1,2$.  Let $\gm_{x_i}(\bar{y}_i)=y_i,~i=1,2$. Then by taking lift of the geodesic $[\bar{x}_i,\bar{y}_i]_B$ in $\gm_{x_i}$, we have $d_i(x_i,y_i)\le2K.2R=4KR$ (see Lemma \ref{qi-sec-inside-lad-len-lift} $(3)$). Now let us restrict the $K$-qi section, $\gm_{x_2}$, over $B_{[w,v]}$ and apply $\rho_u$ (see $(\H1)$) on this restriction of $\gm_{x_2}$ over $B_{[w,v]}$. We denote the image under $\rho_u$ by $\gm_2$. Since $B_{[w,v]}\sse \pi(\F l_K(X_u))$, $\gm_2$ is a $(2KL\pr+L\pr)$-qi section over $B_{[w,v]}$. Let  $T_{uvw}$ is the tripod in $T$ with vertices $u,v,w$ and $B_{T_{uvw}}:=\pi_B^{-1}(T_{uvw})$. As $\F l_K(X_u)\cap X_v\ne\emptyset$, we can extend $\gm_2$ to a $(2KL\pr+L\pr)$-qi section over $B_{T_{uvw}}$. Then in particular, we have $\gm_2\sse \F l_{\kappa^{(1)}}(H)$ (see $(\H4)$). Now we apply line by line argument as in $\textrm{Case }1$. Note that $d_X(y_1,y_2)\le d_X(y_1,x_1)+d_X(x_1,x_2)+d_X(x_2,y_2)\le8KR+2R$ and $\rho_u(y_1)=y_1$, $\rho_u(y_2)=\gm_2(\bar{y}_2)$. So $d_X(y_1,\gm_2(\bar{y}_2))\le L\pr d_X(y_1,y_2)+L\pr\le L\pr(8KR+2R)+L\pr$. Since $y_1,\gm_2(\bar{y}_2)\in\F l_K(X_u\sse Y_{1R}$, by Lemma \ref{Yir-is-proper-emb}, $d_1(y_1,\gm_2(\bar{y}_2))\le\eta_{\ref{Yir-is-proper-emb}}(L\pr(8KR+2R)+L\pr)$. Again, we have $\gm_2\sse\F l_{\kappa^{(1)}}(H)$, so $d_1(x,Y_0)\le d_1(x,x_1)+d_1(x_1,y_1)+d_1(y_1,\gm_2(\bar{y}_2))\le R+4KR+\eta_{\ref{Yir-is-proper-emb}}(L\pr(8KR+2R)+L\pr)$.
	
	Again, $d_X(x_2,\gm_2(\bar{y}_2))\le d_X(x_2,x_1)+d_X(x_1,y_1)+d_X(y_1,\gm_2(\bar{y}_2))\le2R+4KR+L\pr(8KR+2R)+L\pr$. Since $x_2\in\F l_K(X_v)\sse Y_{2R}$ and $\gm_2\sse\F l_{\kappa^{(1)}}(H)\sse Y_{2R}$, so by Lemma \ref{Yir-is-proper-emb}, $d_2(x_2,\gm_2(\bar{y}_2))\le \eta_{\ref{Yir-is-proper-emb}}(2R+4KR+L\pr(8KR+2R)+L\pr)$. Thus $d_2(x,Y_0)\le d_2(x,x_2)+d_2(x_2,\gm_2(\bar{y}_2))\le R+\eta_{\ref{Yir-is-proper-emb}}(2R+4KR+L\pr(8KR+2R)+L\pr)$.
	
	Therefore, we can take $D_{\ref{containment}}=max\{R+4KR+\eta_{\ref{Yir-is-proper-emb}}(L\pr(8KR+2R)+L\pr),R+\eta_{\ref{Yir-is-proper-emb}}(2R+4KR+L\pr(8KR+2R)+L\pr)\}$.
\end{proof}

\begin{lemma}\label{Y-is-hyperbolic}
There is a uniform constant $\dl_{\ref{Y-is-hyperbolic}}=\dl_{\ref{Y-is-hyperbolic}}(R)$ such that $Y$ is $\dl_{\ref{Y-is-hyperbolic}}$-hyperbolic metric space.
\end{lemma}

\begin{proof}
	We verify all the conditions of Proposition \ref{combi-hyp-sps} for $n=2$ (see Remark \ref{combi-hyp-sps-2}). Note that $Y=Y_{1R}\cup Y_{2R}$.
	
	$(1)$ $Y_{iR},~i=1,2$ are $\dl_{\ref{Yir-is-hyp}}$-hyperbolic.
	
	$(2)$ $Y_0$ is $\bar{L}(R)$-qi embedded in $X$ (see $(\H6)$), so is in both $Y_{1R}$ and $Y_{1R}$. Again $Y_{1R}\cap Y_{1R}\sse N_{D_{\ref{containment}}}(Y_0)$ (see Lemma \ref{containment}), so by Lemma \ref{hd-imp-qi}, the inclusion $Y_{1R}\cap Y_{1R}\ri Y_{iR}$ is a $L_{\ref{hd-imp-qi}}(\bar{L}(R),D_{\ref{containment}})$-qi embedding for $i=1,2$.
	%
	%
	%
	
	%
	%
	
	Therefore, $Y$ is $\dl_{\ref{Y-is-hyperbolic}}:=\dl_{\ref{combi-hyp-sps-2}}(\dl_{\ref{Yir-is-hyp}},L_{\ref{hd-imp-qi}}(\bar{L}(R),D_{\ref{containment}}),1)$-hyperbolic.
\end{proof}

\begin{lemma}\label{Y-is-proper-emb}
The inclusion $Y\ri X$ is a $\eta_{\ref{Y-is-proper-emb}}=\eta_{\ref{Y-is-proper-emb}}(R)$-proper embedding for some uniform proper function $\eta_{\ref{Y-is-proper-emb}}:\R_{\ge0}\ri\R_{\ge0}$.
\end{lemma}

\begin{proof}
	Let $r\in\R_{\ge0}$ and $x,y\in Y$ such that $d_X(x,y)\le r$. Then $\exists~x_1,y_1\in\F l_K(X_u)\cup\F l_{\kappa^{(1)}}(H)\cup\F l_K(X_v)$ such that $d_Y(x,x_1)\le R$ and $d_Y(y,y_1)\le R$. So by triangle inequality, $d_X(x_1,y_1)\le r+2R$. Without loss of generality, we may assume that $x_1\in\F l_K(X_u)$ and $y_1\in\F l_K(X_v)$. Otherwise, by Lemma \ref{Yir-is-proper-emb}, $d_Y(x_1,y_1)\le\eta_{\ref{Yir-is-proper-emb}}(r+2R)$. Again $R\ge M_{k_{\ref{union-of-two-flow-sp-is-proper-emb}}}$, and thus by Proposition \ref{union-of-two-flow-sp-is-proper-emb}, $d_Y(x_1,y_1)\le\eta_{\ref{union-of-two-flow-sp-is-proper-emb}}(K,R)(r+2R)$. Therefore, in either case, $d_Y(x,y)\le2R+max\{\eta_{\ref{Yir-is-proper-emb}}(r+2R),\eta_{\ref{union-of-two-flow-sp-is-proper-emb}}(K,R)(r+2R)\}=:\eta_{\ref{Y-is-proper-emb}}(r)$.
\end{proof}

\begin{prop}\label{union-of-flow-sp-is-hyp}
There exists a uniform constant $D_{\ref{union-of-flow-sp-is-hyp}}(R)$ such that for all $D\ge D_{\ref{union-of-flow-sp-is-hyp}}(R)$ we have $\dl_{\ref{union-of-flow-sp-is-hyp}}=\dl_{\ref{union-of-flow-sp-is-hyp}}(D)$ for which $N_D(\F l_K(X_u)\cup\F l_K(X_v))$ is a $\dl_{\ref{union-of-flow-sp-is-hyp}}$-hyperbolic metric space with the induced path metric from $X$.
\end{prop} 

\begin{proof}
In the proof, we define $D_{\ref{union-of-flow-sp-is-hyp}}$. For $D\ge D_{\ref{union-of-flow-sp-is-hyp}}$, we denote the induced path metric on $N_D(\F l_K(X_u)\cup\F l_K(X_v))$ by $\bar{d}$. By $(\H6)$, $Fl_{KR}(X_u)$ is $\bar{L}(R)$-qi embedded in $X$ and so is in $Y$. Hence $Fl_{KR}(X_u)$ is $K_1$-quasiconvex in $Y$ for some $K_1$ depending on $\dl_{\ref{Y-is-hyperbolic}}(R)$ and $\bar{L}(R)$. So $\F l_K(X_u$ is $K_2$-quasiconvex in $Y$, where $K_2=K_1+R$. Also, by the similar argument $\F l_K(X_v)$ is $K_2$-quasiconvex in $Y$. Since $Y$ is $\dl_{\ref{Y-is-hyperbolic}}(R)$-hyperbolic and $\F l_K(X_u)\cap\F l_K(X_v)\ne\emptyset$, so $\F l_K(X_u)\cup\F l_K(X_v)$ is $(K_2+\dl_{\ref{Y-is-hyperbolic}}(R))$-quasiconvex in $Y$. Let $N\pr_D(\F l_K(X_u)\cup\F l_K(X_v))$ be the $D$-neighborhood of $\F l_K(X_u)\cup\F l_K(X_v)$ (inside $Y$) in the path metric on $Y$. We set $D_{\ref{union-of-flow-sp-is-hyp}}>K_2+\dl_{\ref{Y-is-hyperbolic}}(R)+1$. Thus for $D\ge D_{\ref{union-of-flow-sp-is-hyp}}$, (by Lemma \ref{qi-emb-in-Y-X}) the inclusion $N\pr_D(\F l_K(X_u)\cup\F l_K(X_v))\ri Y$ is a $L_1$-qi embedding, where $L_1=L_{\ref{qi-emb-in-Y-X}}(\dl_{\ref{Y-is-hyperbolic}}(R),K_2+\dl_{\ref{Y-is-hyperbolic}}(R),D)$. Therefore, $N\pr_D(\F l_K(X_u)\cup\F l_K(X_v))$ is $\dl_1$-hyperbolic, where $\dl_1$ depends on $\dl_{\ref{Y-is-hyperbolic}}(R)$ and $L_1$.
	
Now we show that the subset $N\pr_D(\F l_K(X_u)\cup\F l_K(X_v))\sse N_D(\F l_K(X_u)\cup\F l_K(X_v))$ satisfies all the conditions of Proposition \ref{combing}. Note that $N\pr_D(\F l_K(X_u)\cup\F l_K(X_v))$ is a $D$-dense in $N_D(\F l_K(X_u)\cup\F l_K(X_v))$ in the $\bar{d}$-metric. For any pair $(x,y)$ of distinct points in $N\pr_D(\F l_K(X_u)\cup\F l_K(X_v))$, we fix once and for all a geodesic path, say $c(x,y)$, joining $x$ and $y$ in the $\dl_1$-hyperbolic space $N\pr_D(\F l_K(X_u)\cup\F l_K(X_v))$. These paths serve as family of paths for Proposition \ref{combing}. Then any triangle formed by these paths are $\dl_1$-slim in the induced path metric of $N\pr_D(\F l_K(X_u)\cup\F l_K(X_v))$ and so is in $N_D(\F l_K(X_u)\cup\F l_K(X_v))$. Hence we are left to show the proper embedding of these paths in $N_D(\F l_K(X_u)\cup\F l_K(X_v))$. Indeed, suppose $x,y\in N\pr_D(\F l_K(X_u)\cup\F l_K(X_v))$ such that $\bar{d}(x,y)\le r$ for some $r\in\R_{\ge0}$. Then $d_X(x,y)\le r$ and by Lemma \ref{Y-is-proper-emb}, $d_Y(x,y)\le\eta_{\ref{Y-is-proper-emb}}(r)$. Since $N\pr_D(\F l_K(X_u)\cup\F l_K(X_v))$ is $L_1$-qi embedded in $Y$, the path $c(x,y)$ is $\eta_1$-properly embedded, where $\eta_1:\R_{\ge0}\ri\R_{\ge0}$ sending $r\mapsto\eta_{\ref{Y-is-proper-emb}}(r)L_1+L^2_1$.
	
Therefore, $N_D(\F l_K(X_u)\cup\F l_K(X_v))$ is $\dl_{\ref{union-of-flow-sp-is-hyp}}$-hyperbolic metric space with the induced path metric from $X$, where $\dl_{\ref{union-of-flow-sp-is-hyp}}=\dl_{\ref{combing}}(\eta_1,\dl_1,D)$.
\end{proof}

\section{Proof of Theorem \ref{main-theorem-com}}\label{hyp-tmb-sec}
{\bf Strategy.} We think of the tree of metric bundles $(X,B,T)$ as a tree of metric spaces $\pi:X\ri T$ as explained in Remark \ref{crusial-words}. Suppose $u\in V(T)$, and $\F l_K(X_u)$ is the flow space of $X_u$ considered in Section \ref{hyp-of-flow-sp} (with the parameters mentioned there). Then we will show below that the subspaces $\F l_K(X_u),~u\in V(T)$ satisfy conditions $(\mathcal P0)-(\mathcal P4)$ of Theorem \ref{treeofsps-com-thm} (see Section \ref{hyp-tree-sps}). This will conclude Theorem \ref{main-theorem-com}.\smallskip

$\bm{(\mathcal P0)}$: The flow space $\F l_K(X_u)$ was constructed by induction on $d_T(u,v),~v\in V(T)$. Recall that if an edge $[v,w]\sse\pi(\F l_K(X_u))$, then $\F l_K(X_u)\cap F_{\mfw,w}$ is a quasiconvex hull in $F_{\mfw,w}$ of the set $N_R(\F l_K(X_u)\cap F_{\mfv,v})\cap F_{\mfw,w}$ where the neighborhood is considered in $F_{\mfv,\mfw}$-metric. On the other hand, by definition $\F l_K(X_v)\cap F_{\mfw,w}$ is the quasiconvex hull in $F_{\mfw,w}$ of the set $N_R(F_{\mfv,v})\cap F_{\mfw,w}$. Hence, irrespective of whether the edge $[v,w]$ belongs to $\pi(\F l_K(X_u))$ or not, $\F l_K(X_u)\cap F_{\mfw,w}\sse\F l_K(X_v)\cap F_{\mfw,w}$, and so $\F l_K(X_u)\cap X_w\sse\F l_K(X_v)\cap X_w$. Therefore, $\F l_K(X_u),~u\in V(T)$ satisfy the condition $(\mathcal P0)$.\smallskip

$\bm{(\mathcal P1)}$: By Proposition \ref{mitra's-retraction-on-fsandgss}, we can take $L\pr=L_{\ref{mitra's-retraction-on-fsandgss}}(K)$ such that for all $u\in V(T)$, we have $L'$-coarsely Lipschitz retraction $\rho_u:X\map\F l_K(X_u)$ where $\rho_u=\rho_{\ref{mitra's-retraction-on-fsandgss}}$. Again we take $C=D_{\ref{R-sep-D-cobdd}}(\dl\pr_0,L\pr_0)$ for which $diam\{\rho_u(X_S)\}\le C$ for any subtree $S\sse T$ such that $S\cap\pi(\F l_K(X_u))=\emptyset$.\smallskip

We set $L>max\{D_{\ref{union-of-flow-sp-is-hyp}}(R),r_2\}$ where $R$ is defined in equation \ref{value-of-R} (see Section \ref{union-of-two-flow-sp-is-hyp}) and $r_2$ as in Theorem \ref{flow-sp-is-hyp}.

$\bm{(\mathcal P2)}$: By Proposition \ref{fsandgss-are-proper-emb}, for all $u\in V(T)$, the inclusion $Fl_{KL}(X_u)\map X$ is a $\eta_1$-proper embedding where $\eta_1=\eta_{\ref{fsandgss-are-proper-emb}}(K,L):\R_{\ge0}\map\R_{\ge0}$.\smallskip

$\bm{(\mathcal P3)}$: Suppose $u,v\in V(T)$ such that $\F l_K(X_u)\cap X_v\ne\emptyset$. Then by Proposition \ref{union-of-two-flow-sp-is-proper-emb}, the inclusion $N_L(\F l_K(X_u))\cup N_L(\F l_K(X_v))\map X$ is a $\eta_2$-proper embedding where $\eta_2=\eta_{\ref{union-of-two-flow-sp-is-proper-emb}}(K,L):\R_{\ge0}\map\R_{\ge0}$.

We define $\eta:\R_{\ge0}\map\R_{\ge0}$ such that $\eta(r)=max\{\eta_1(r),\eta_2(r)\}$, $r\in\R_{\ge0}$.\smallskip

$\bm{(\mathcal P4)}$: By Proposition \ref{union-of-flow-sp-is-hyp}, for all such $u,v$ as in $(\mathcal P3)$ above, the space $N_L(\F l_K(X_u)\cup\F l_K(X_v))$ is $\dl$-hyperbolic where $\dl=\dl_{\ref{union-of-flow-sp-is-hyp}}(L)$.


\section[Applications]{Applications to complexes of groups}\label{applications}

We refer to \cite[Chapter III$.C$]{bridson-haefliger}, \cite{haefliger-cplx} and \cite{corson1} for basic notions of developable complexes of groups. All the groups we consider here are finitely generated. 

The construction of a tree of metric bundles for a given complex of groups as in the setup $\mathcal{C}$ (see Definition \ref{setup C}) follows from the idea of  \cite{corson1}, \cite{haefliger-cplx}. We briefly discuss the same below.

Suppose $\Y$ is a finite connected simplicial complex and $\G(\Y)$ is a developable complex of groups over $\Y$. Let $G$ be the fundamental group of $\G(\Y)$. As in \cite{corson1} and more generally, in \cite[Theorem $3.4.1$]{haefliger-cplx}, we consider a cellular aspherical realization (see \cite[Definition $3.3.4$]{haefliger-cplx}) $\mathcal{X}$ of the complex of groups $\G(\Y)$ with cellular map $p:\mathcal{X}\ri\Y$. Note that $\mathcal{X}$ is constructed by gluing along the Eilenberg-Mac Lane complexes of the local groups of the complex of groups $\G(\Y)$; where for each local group $G_{\sigma}$ corresponding to a face $\sigma$ of $\Y$, the $0$-skeleton of $K(G_{\sigma},1)$ is a point $x_{\sigma}$ and the $1$-skeletons form wedge of circles where each circle corresponds to an element in a fixed finite generating set of $G_{\sigma}$. Now we consider the universal covering $\pi:\tilde{\mathcal{X}}\ri\mathcal{X}$ with the standard CW-complex structure on $\tilde{\mathcal{X}}$ coming from $\mathcal{X}$. We identify $G$ with the group of deck transformation of the covering map $\pi:\tilde{\mathcal{X}}\ri\mathcal{X}$. Let $y\in\Y$ and $\sigma$ be the face containing $y$ in its interior, and let $a$ be the barycenter of $\sigma$. Now we collapse each connected component of $\{(p\circ\pi)^{-1}(y)\}$ to a point. Note that since $\G(\Y)$ is developable,  the inclusion $p^{-1}(a)\ri \mathcal{X}$ is $\pi_1$-injective and hence $\{(p\circ\pi)^{-1}(y)\}$ are copies of universal cover of $\mathcal{X}_a:=p^{-1}(a)$. We do this for all $y\in\Y$. Let $\B$ be the space we get after collapsing and $q:\tilde{\mathcal{X}}\ri\B$ be the quotient map. Here $\B$ is called the {\em universal cover} of $\G(\Y)$. Thus we get a quotient map $\bar{\pi}:\B\ri\B/G=\Y$ and a commutative diagram as follows.

\begin{figure}[H]
	\begin{tikzcd}
		\tilde{\mathcal{X}} \arrow{r}{q} \arrow[swap]{d}{\pi} \arrow[dr, phantom, "\curvearrowright"] & \B \arrow{d}{\bar{\pi}} \\
		\mathcal{X} \arrow{r}{p}& \Y
	\end{tikzcd}
	\centering
	\caption{}
	\label{}
\end{figure}


Let $X:=\tilde{\mathcal{X}}^{(1)}$ and $B:=\B^{(1)}$, where $Z^{(1)}$ denotes the $1$-skeleton of a CW-complex $Z$. Assume that each edge in $X$ and $B$ has length $1$. Put the length metric on $X$ and $B$. Since the groups are finitely generated, by covering space argument, we have the following facts.

\begin{enumerate}
	\item The restriction of the map $q$ on $X$, $q|_X:X\ri B$ is $G$-equivariant, surjective and $1$-Lipschitz.	
	
	\item The action of $G$ on $X$ is proper and cocompact, and the action of $G$ on $B$ is cocompact (but not necessarily proper unless all local groups are finite).
	
	\item There is an isomorphism of graphs $r:B/G\ri\Y^{(1)}$ such that for all $\sigma_0\in\Y^{(0)}$ and $a\in \{r^{-1}(\sigma_0)\}$, $G_a$ (the stabilizer subgroup of $a\in B^{(0)}$) is conjugate to $G_{\sigma_0}$ in $G$.
	
	
	\item Let $a\in B^{(0)}$ and  $F_a:=(q|_X)^{-1}(a)$. Then the action of $G_a$ (the stabilizer subgroup of $a$) on $F_a^{(0)}$ is transitive and on $F_a^{(1)}$ is uniformly cofinite. In particular, $G_a$ is uniformly quasiisometric to $F_a$ (with the induced path metric from $X$).
	
	\item Since a finitely generated subgroup of a finitely generated group is properly embedded (with respect to their finite generating sets), for all $a\in B^{(0)}$, the inclusion $F_a\ri X$ is uniformly properly embedding, where $F_a:=(q|_X)^{-1}(a)$.
\end{enumerate}

\underline{Complexes of groups as in the setup $\mathcal{C}$}: Now suppose $\G(\Y,Y)$ is a complex of groups over $\Y$ as in Definition \ref{setup C}. 
Then we have a natural graph of groups structure on $\G(\Y,Y)$, say $(\G,Y)$, over $Y$ such that the vertex groups are $G_s:=\pi_1(\G_s(\Y_s)),\fa~ s\in Y^{(0)}$ and the edge groups are $G_e$ for $e=[u,v]$ joining two vertices $u,v\in\Y$ so that restriction of $p_Y$ on $e$ is injective. Note that the monomorphisms from edge groups to the corresponding vertex groups are the restriction of $\G(\Y,Y)$.

As a corollary of \cite[Proposition $3.9$, III$.C$]{bridson-haefliger}, we have the following lemma.

\begin{lemma}\label{com-grps-based-graph-dev}
	The complex of groups $\G(\Y,Y)$ is developable.
\end{lemma}
\subsection{Trees of metric bundles coming from complexes of groups}\label{graphs-tree-gra-grphs-comp}
Consider the above discussion on complexes of groups $\G(\Y,Y)$ as in the setup $\mathcal{C}$ (Definition \ref{setup C}). For now onward, we denote the restriction map $q|_X$ by $\pi_X$. Let $s\in Y^{(0)}$ and $G_s=\pi_1(\G_s(\Y_s))$. Consider the corresponding graph of groups $(\G,Y)$ as explained above. Then note that $X$ is the corresponding tree of metric spaces over the Base-Serre tree of the graph of groups $(\G,Y)$. 
Let $T$ be the Base-Serre tree and $\pi:X\ri T$ be the projection map. Again vertex spaces of $X$ are acted (properly and cocompactly) upon by the conjugates of $G_s$ in $G$, $s\in Y^{(0)}$. Now by condition $(2)$ of Definition \ref{setup C}, it follows from \cite[Section $3.3$]{ps-krishna} that the vertex spaces of $\pi:X\ri T$ are metric graph bundles (see \cite[Definition 1.5]{pranab-mahan}) with uniform parameters. For instance, if $s\in Y^{(0)}$, $g\in G$, $u=gG_s\in T^{(0)}$ and $B_s$ is the $1$-skeleton of the universal cover of $G_s(\Y_s)$, then $X_u:=\pi^{-1}(u)$ is the metric graph bundle over $B_u=gB_s\sse B$ and the subgroup $gG_sg^{-1}$ acts on $X_u$ properly and cocompactly. Also, note that $T$ is obtained by collapsing the universal cover of $\G_s(\Y_s)$ (for $s\in\Y^{(0)}$) and its $G$-translates (in $B$) to points. Let $\pi_B:B\ri T$ be the projection map. Again the maps $\pi_X:X\ri B$ and $\pi_B:B\ri T$ are $G$-equivariant. Therefore, we have the following.

\begin{prop}\label{main-const-of-treeofmetricbundles}
Suppose $G$ is the fundamental group of $\G(\Y,Y)$. Then there is a natural tree of metric bundles $(X,B,T)$ and an action of $G$ by isometries on both $X$ and $B$ such that the following hold.
\begin{enumerate}
\item The map $\pi_X$ is $G$-equivariant.
		
\item  The action of $G$ on $X$ is proper and cocompact, and the action of $G$ on $B$ is cocompact (but not necessarily proper unless all local groups are finite).
		
\item There is an isomorphism of graphs $r:B/G\ri\Y^{(1)}$ such that for all $\sigma_0\in\Y^{(0)}$ and $a\in \{r^{-1}(\sigma_0)\}$, $G_a$ (the stabilizer subgroup of $a\in B^{(0)}$) is conjugate to $G_{\sigma_0}$ in $G$.
		
		
\item Let $a\in B^{(0)}$ and $u=\pi_B(a),~F_{a,u}:=\pi_X^{-1}(a)=q^{-1}(a)^{(1)}$. Then the action of $G_a$ on $F_{a,u}^{(0)}$ is transitive and on $F_{a,u}^{(1)}$ is uniformly cofinite. In particular, if $\sigma_0\in\Y^{(0)}$ and $G_{\sigma_0}$ is hyperbolic, then for all $a\in\{r^{-1}(\sigma_0)\},~F_{a,u}$ is uniformly hyperbolic, where $u=\pi_B(a)$.
\end{enumerate} 
\end{prop}
It is worth mentioning that condition $(2)$ of the setup $\mathcal{C}$ (see Definition \ref{setup C}) in Proposition \ref{main-const-of-treeofmetricbundles} is necessary to get trees of metric bundles. 
Also note that $\pi_X:X\ri B$ satisfies axiom {\bf H} (see Definition \ref{axiomH}) by the assumption on $\G(\Y,Y)$ (see Remark \ref{nonelementary barycenter map}).\smallskip

Now we are ready to prove the main application of Theorem \ref{main-theorem-com}.\smallskip

{\bf \emph{Proof of Theorem \ref{application-2}}} : It follows from Proposition \ref{main-const-of-treeofmetricbundles} and Theorem \ref{main-theorem-com}.\qed\smallskip

{\bf \emph{Proof of Corollary \ref{necessity-flaring-com-grp}}}: It follows from Proposition \ref{main-const-of-treeofmetricbundles} and Remark \ref{necessity-of-flaring}.\qed

\subsection{Examples} Now we will see an example (Example \ref{exp explanation}) that satisfies all the assumptions of our main theorem (Theorem \ref{main-theorem-com}). We mention here that we not proving the hyperbolicity of groups in Example \ref{exp explanation} by applying Theorem \ref{main-theorem-com} instead Theorem  \ref{hyp exp} (see Remark \ref{hard remark}). 
In the setting of relatively hyperbolic groups, Matsuda and Oguni (\cite{mats-oguni}) constructed these types of groups to show the nonexistence of Cannon--Thurston maps. For completeness, we are including a proof of Theorem \ref{hyp exp} following the argument given in \cite{mats-oguni}.

Recall that a subgroup $A$ of a group $B$ is said to be {\em malnormal} (respectively, {\em weakly malnormal}) if for all $b\in B\setminus A$ we have $A\cap bAb^{-1}=\{1\}$ (respectively, $A\cap bAb^{-1}$ is finite).

\begin{theorem}\textup{(\cite{mats-oguni})}\label{hyp exp}
Suppose $G_1$ and $G_2$ are hyperbolic groups and $H$ is an infinite common subgroup in them giving the free product with amalgamation $G=G_1*_HG_2$. Let $H$ be weakly malnormal and quasiconvex in $G_1$. Then $G$ is hyperbolic. Moreover, $G_2$ is weakly malnormal and quasiconvex in $G$.	
\end{theorem}

\begin{proof}
Since $H$ is infinite, weakly malnormal and quasiconvex in $G_1$, so by \cite[Theorem 7.11]{bowditch-relhyp-publish}, $G_1$ is hyperbolic relative to $H$. Then by \cite[Theorem 0.1 (2)]{dahmani-com}, $G$ is hyperbolic relative to $G_2$. Again, since $G_2$ is hyperbolic, so $G$ is hyperbolic by \cite[Corollary 2.41]{osin-2006}. On the other hand, since $G$ is hyperbolic and hyperbolic relative to $G_2$, so $G_2$ is quasiconvex and weakly malnormal in $G$ by \cite[Theorem 1.5]{osin-bddgen}.
\end{proof}
In the following example, $F_i$ and $S_g$ denote respectively the free group of rank $i\in\N$ and closed orientable surface of genus $g\ge2$. 
\begin{example}\label{exp explanation}
Let $\pi_1(S_g)\rtimes F_m$ and $F_n\rtimes F_{m'}$ be hyperbolic groups where $n,m,m'\in\N$ and $n\ge3$. Suppose $F_2$ is a common subgroup of $\pi_1(S_g)$ and $F_n$ $($normal subgroup of $F_n\rtimes F_{m'})$ giving the amalgamation $G=(\pi_1(S_g)\rtimes F_{m})*_{F_2}(F_n\rtimes F_{m'})$. By \textup{\cite[Theorem 6.7]{kapovich-nonqc}}, after passing to a subgroup, if necessary, we assume that $F_2$ is malnormal and quasiconvex in $\pi_1(S_g)\rtimes F_m$. Then by Theorem \ref{hyp exp}, the group $G$ is hyperbolic.\smallskip


Let $\pi_X:X\map B$ be the tree of metric bundles induced by the amalgamation $G$ as explained in Example \ref{trees of metric bundles exps}. Now we will see that it satisfies all the conditions of Theorem \ref{main-theorem-com} as follows. 

$(1)$ Since fibers are uniformly quasiisometric to $\pi_1(S_g)$ or $F_n$, then fibers are all uniformly hyperbolic. Again, since $\pi_1(S_g)$ and $F_n$ are nonelementary hyperbolic groups, the barycenter maps for the fibers are uniformly coarsely surjective.

$(2)$ It is standard that finitely generated subgroups of a free group and a surface group are quasiconvex. So the subgroup $F_2$ is quasiconvex in both $\pi_1(S_g)$ and $F_n$ $($normal subgroup of $F_n\rtimes F_{m'})$.

$(3)$ Note that bases of metric bundles corresponding to $\pi_1(S_g)\rtimes F_m$ and $F_n\rtimes F_{m'}$ are isometric copies of Cayley graphs of $F_m$ and $F_{m'}$ respectively. Then the space $B$ is constructed by gluing these Cayley graphs in a tree-like way $($see the construction in Example \ref{trees of metric bundles exps}$)$. So $B$ is hyperbolic.

$(4)$ Since $G$ is hyperbolic group, by Remark \ref{necessity-of-flaring}, the tree of metric bundles $\pi_X:X\map B$ satisfies the Bestvina--Feighn hallway flaring condition.
\end{example}

Note that the groups $\pi_1(S_g)\rtimes F_m$ and $F_n\rtimes F_{m'}$ used in Example \ref{exp explanation} are known due to \cite{mosher-hypextns} and \cite{BFH-lam} respectively. We also mention that the amalgamating subgroup $F_2$ in Example \ref{exp explanation} is quasiconvex in both $\pi_1(S_g)\rtimes F_m$ and $F_n\rtimes F_{m'}$ under suitable assumption on quotient groups $F_m$ and $F_{m'}$ as follows (see \cite{mj-rafi}, also \cite{scottswar,mitra-pams}): $F_m$ and $F_{m'}$ are convex cocompact in Teichm\text{\"u}ller space and outer space respectively. In the later case, one assume that $F_{m'}$ is purely hyperbolic. Crucial point here is that the amalgamating subgroup $F_2$ has infinite index in both $\pi_1(S_g)$ and $F_n$.

\begin{remark}\label{hard remark}
Potential examples of hyperbolic groups, which could be obtained as applications of Theorem \ref{main-theorem-com}, can be constructed similarly to Example \ref{trees of metric bundles exps}. However, verifying the Bestvina–Feighn hallway flaring condition is challenging.
\end{remark}

\bibliographystyle{amsalpha}
\bibliography{Ubib}

\end{document}